\documentclass[reqno]{amsart}[12pt]
\usepackage[paper=letterpaper,margin=.9in]{geometry} 
\usepackage[utf8]{inputenc}
\usepackage{lmodern}
\usepackage{array}
\usepackage{adjustbox}
\usepackage{graphicx}
\usepackage{amssymb, amsfonts, amsthm, amsmath}
\usepackage[english]{babel}
\usepackage[pdfborder={0 0 0}]{hyperref}
\usepackage{float}
\usepackage{comment}
\usepackage{tikz}
\usepackage{bbm}
\usepackage{caption,subcaption}
\usetikzlibrary{decorations.markings}
\usetikzlibrary{calc,angles,positioning}
\usetikzlibrary{decorations.pathreplacing}
\newcommand{\poleset}{\mathcal{P}}
\newcommand{\corej}{\mathfrak{D}}

\newcommand{\statmargin}{\pi_\rho}
\newcommand{\hmue}{\hat{\mu}_\ep}
\newcommand{\hlambdae}{\hat{\lambda}_\ep}

\newcommand{\kpzt}{\widetilde{\kpz}}
\newcommand{\A}{{\fontfamily{qpl}\selectfont\textbf{(A)}}}
\newcommand{\B}{{\fontfamily{qpl}\selectfont\textbf{(B)}}}
\usepackage{acronym}
\acrodef{KPZ}{Kardar--Parisi--Zhang}
\acrodef{SHE}{Stochastic Heat Equation}
\acrodef{SHS6V}{stochastic higher spin six vertex}
\acrodef{ASEP}{Asymmetric Simple Exclusion Process}

\newcommand{\etapart}[2]{\eta_{#2}(#1)}
\newcommand{\etapartt}[2]{\etat_{#2}(#1)}
\newcommand{\errr}{\text{ error term }}
\newcommand{\mut}[2]{#1(#2)}
\newcommand{\alphat}[2]{#1(#2)}
\newcommand{\lambdat}[2]{#1(#2)}
\newcommand{\rmatrix}{L_{\alpha}^{(1)}}
\newcommand{\rem}{\mathfrak{R}}
\newcommand{\reme}{\mathfrak{R}_\ep}
\newcommand{\remt}{\widetilde{\mathfrak{R}}}
\newcommand{\N}{N^{\mathsf{f}}}
\newcommand{\n}{N^{\mathsf{uf}}}
\newcommand{\rj}{L_{\alpha}^{(J)}}

\newcommand{\anglebrack}[1]{\langle #1 \rangle}
\newcommand{\floor}[1]{\lfloor #1 \rfloor}

\newcommand{\ZZ}{\mathbb{Z}}
\newcommand{\ZZn}{\mathbb{Z}_{\geq 0}}
\newcommand{\PP}{\mathbb{P}}
\renewcommand{\AA}{\mathcal{A}}
\newcommand{\independent}{\text{ is independent with }}

\newcommand{\EE}{\mathbb{E}}
\newcommand{\RR}{\mathbb{R}}

\newcommand{\hlambda}{\hat{\lambda}}
\newcommand{\hmu}{\hat{\mu}}
\newcommand{\M}{\mathcal{M}}
\newcommand{\Q}{\mathcal{Q}}
\newcommand{\du}{\widetilde{D}}

\newcommand{\NN}{\ZZ_{\geq 1}}
\newcommand{\C}{\mathcal{C}}
\newcommand{\CC}{\mathbb{C}}

\newcommand{\FF}{\mathcal{F}}
\newcommand{\G}{\mathcal{G}}
\newcommand{\GG}{\mathbb{G}}
\newcommand{\OO}{\mathcal{O}}
\newcommand{\MM}{\mathbb{M}}

\newcommand{\HH}{\mathcal{H}}
\newcommand{\HHe}{\mathcal{H}_\ep}

\newcommand{\YY}{\mathbb{Y}}

\newcommand{\im}{\mathbf{i}}

\newcommand{\lc}{\C_R}
\newcommand{\pset}{\mathcal{P}_{23}}
\newcommand{\ue}{u_\epsilon}

\newcommand{\rad}{r^*}

\newcommand{\wwr}{\widetilde{r}}

\newcommand\numberthis{\addtocounter{equation}{1}\tag{\theequation}}
\newcommand{\LL}{L}

\newcommand{\wc}{\Rightarrow}
\newcommand{\iset}{\{0, 1, \dots, I\}}
\newcommand{\jset}{\{0, 1, \dots, J\}}
\newcommand{\occupc}{\vec{\eta}}
\newcommand{\occupv}[1]{\eta_{#1}}

\newcommand{\dist}{\text{dist\,}}
\newcommand{\vv}{\big|}
\newcommand{\vvv}{\bigg|}

\newcommand{\occuppr}{\vec{\eta}(t)}

\newcommand{\rlocc}{\vec{y}}

\renewcommand{\mod}{\mathrm{mod}_J}

\newcommand{\var}{\text{Var}}

\newcommand{\rlocpr}{\vec{y}(t)}
\newcommand{\occupvr}[2]{\eta_{#1} (#2)}
\newcommand{\bivr}[2]{\etapart{#1}{#2}}

\newcommand{\tonepart}{\mathsf{p}}
\newcommand{\toneparte}{\mathsf{p}_\epsilon}

\newcommand{\rhzr}{\mathbf{P}_{\overleftarrow{\text{SHS6V}}}}
\newcommand{\rhzrt}{\mathbf{V}}

\newcommand{\rhzrte}{\mathbf{V}_\ep}

\newcommand{\rhzrfr}{\rhzrte^{\text{fr}}}

\newcommand{\rhzrin}{\rhzrte^{\text{in}}}

\newcommand{\rhzrb}{\rhzrt^{\text{blk}}_\ep}

\newcommand{\rhzrr}{\rhzrt^{\text{res}}_\ep}

\newcommand{\height}[2]{N(#1, #2)}
\newcommand{\weyl}[1]{\mathbb{W}_I^{#1}}

\newcommand{\x}{\vec{x}}
\newcommand{\idc}{\mathbf{1}}
\newcommand{\mqv}[1]{m_{q,v}(#1)}
\newcommand{\qhalfint}[1]{[#1]_{q^{\frac{1}{2}}}}

\newcommand{\etat}{\widetilde{\eta}}

\newcommand{\ep}{\epsilon}

\newcommand{\eph}{\epsilon^{\frac{1}{2}}}

\newcommand{\bracket}[1]{\langle #1 \rangle}

\newcommand{\semigr}{\widetilde{\mathcal{B}}}

\newcommand{\semigrst}[2]{\widetilde{\mathcal{B}}_{#1} \cdots \widetilde{\mathcal{B}}_{#2}}

\newcommand{\ZZZ}{\mathcal{Z}}

\newcommand{\ber}{\text{Ber\,}}

\newcommand{\rfunc}{\mathsf{u}}
\newcommand{\ppradw}{r_{\ep, w}}
\newcommand{\pprad}{r_{\ep, t-s}}
\newcommand{\pprads}{r_{\ep, t-s}}
\newcommand{\ppradt}{\wwr_{\ep, t-s}}

\newcommand{\stat}{\bigotimes \pi_\rho}

\newcommand{\norm}[1]{\big\Vert #1 \big\Vert }

\newcommand{\bignorm}[1]{\bigg\Vert #1  \bigg\Vert}

\newcommand{\xstar}{x^\star}

\newcommand{\convol}{\tonepart(t+1, t) * Z(t)}

\newcommand{\res}{\mathsf{Res}}

\newcommand{\na}{\nabla}

\newcommand{\polt}{\widetilde{\mathfrak{s}}}

\newcommand{\pol}{\mathfrak{s}}
\newcommand{\pole}{\mathfrak{s}_\ep}

\newcommand{\ppole}{\mathfrak{p}_\ep}

\newcommand{\pstar}{\mathfrak{p}_{*}}

\newcommand{\sstar}{\mathfrak{s}_{*}}

\newcommand{\kpz}{\mathcal{H}}

\newcommand{\she}{\mathcal{Z}}

\newcommand{\noise}{\xi(t, x)}

\newcommand{\sign}{\text{sgn}}

\newcommand{\test}{C_c^\infty (\RR)}
\newcommand{\hk}{p}

\newcommand{\bb}{\mathcal{B}_\epsilon}

\newcommand{\err}{\mathcal{E}_\ep}

\newcommand{\funcspace}{C([0, \infty), C(\RR)) }
\renewcommand{\Re}{\text{Re}\,}

\newcommand{\lambdae}{\lambda_\epsilon}
\newcommand{\mue}{\mu_\epsilon}
\newcommand{\alphae}{\alpha_\epsilon}
\newcommand{\qe}{q_\epsilon}

\newcommand{\zt}{Z^\na}
\newcommand{\leig}{\Psi_{\vec{w}}^\ell}
\newcommand{\reig}{\Psi_{\vec{w}}^r}

\newcommand{\coret}{\widetilde{\mathfrak{D}}}

\newcommand{\interactt}{\widetilde{\mathfrak{F}}}

\newcommand{\planm}{dm_\lambda^{q}}
\newcommand{\planmm}[1]{dm_{#1}^q}
\newcommand{\core}{\mathfrak{D}}
\newcommand{\coreh}{\mathfrak{D}}

\newcommand{\coree}{\mathfrak{D}_\epsilon}
\newcommand{\coreej}{\mathfrak{D}_\epsilon}
\newcommand{\corelim}{\mathfrak{D}_{*}}
\newcommand{\interact}{\mathfrak{F}}
\newcommand{\interacte}{\mathfrak{F}_\ep}

\newcommand{\pcoree}{\mathfrak{H}_\ep}
\newcommand{\pcorelim}{\mathfrak{H}_{*}}
\newcommand{\jprod}{\mathfrak{J}_\ep}

\newcommand{\contone}{\Gamma(t-s, \ep)}
\newcommand{\contonelim}{\Gamma_{*}}
\newcommand{\mcont}{\mathcal{M}}
\newcommand{\mcontu}{\mathcal{M}(u)}
\newcommand{\mcontp}{\mathcal{M}'}
\newcommand{\qint}[1]{[#1]_q}

\newcommand{\qbinom}[2]{\binom{#1}{#2}}
\newcommand{\scontone}{\gamma}

\renewcommand{\a}{a}
\newcommand{\onegrad}[1]{\mathcal{Y}_{\nabla} (#1)}
\newcommand{\twograd}[1]{\mathcal{Y}_{\nabla, \nabla} (#1)}
\newcommand{\gradsquare}[1]{\widetilde{\mathcal{Y}}(#1)}

\newcommand{\onegradz}[1]{Z_{\nabla} (#1)}
\newcommand{\twogradz}[1]{Z_{\na, \na}(#1)}
\newcommand{\cat}{\tau(t)}
\newcommand{\cas}{\tau(s)}

\newcommand{\pa}{\partial}
\theoremstyle{plain}
\newtheorem{theorem}{Theorem}[section]
\newtheorem{cond}[theorem]{Condition}
\newtheorem{defin}[theorem]{Definition}
\newtheorem{prop}[theorem]{Proposition}
\newtheorem{cor}[theorem]{Corollary}
\newtheorem{lemma}[theorem]{Lemma}

\newtheorem{remark}[theorem]{Remark}

\author{Yier Lin}
\thanks{Department of Mathematics, Columbia University, Email: yl3609@columbia.edu}

\begin{document}
\numberwithin{equation}{section}
\title{KPZ equation limit of stochastic higher spin six vertex model}

\begin{abstract}
We consider the  \ac{SHS6V} model introduced by Corwin and Petrov \cite{CP16} with general integer spin parameters $I, J$. Starting from \emph{near stationary initial condition}, we prove that the \ac{SHS6V} model converges to the \ac*{KPZ} equation under \emph{weakly asymmetric scaling}. This generalizes the result in \cite[Theorem 1.1]{CGST18} from $I = J =1$ to general $I, J$.   
\end{abstract}
\maketitle
\setcounter{tocdepth}{1}
\tableofcontents
\section{Introduction}
\subsection{KPZ equation and weak KPZ universality}\label{sec:kpzequation}
The \ac{KPZ} equation is the following non-linear stochastic partial differential equation (SPDE) introduced 
in the seminal work \cite{KPZ86}, which describes the random evolution of an interface that has the property of relaxation and lateral growth
\begin{equation}\label{eq:kpz}
\partial_t \kpz(t, x) = \frac{\delta}{2} \pa_x^2 \kpz(t, x) + \frac{\kappa}{2} \big(\pa_x \kpz(t, x) \big)^2 + \sqrt{D} \xi(t, x).
\end{equation}
Here $\xi(t, x)$ is the \emph{space time white noise}, which could be formally understood as a Gaussian field with covariance function $\EE\big[\xi(t, x) \xi(s, y)\big] = \delta(t-s) \delta(x-y)$, where  $\delta$ is the Dirac delta function.
\bigskip
\\
Care is needed to make sense of \eqref{eq:kpz} due to the nonlinearity $(\pa_x \HH(t, x))^2$. The Hopf-Cole solution to the \ac{KPZ} equation is defined by
\begin{equation}\label{eq:hftransform}
\kpz(t, x) = \frac{\delta}{\kappa} \log \she(t, x),
\end{equation}
where $\she(t, x)$ is the \emph{mild solution} of the \ac{SHE}
\begin{equation*}
\partial_t \she(t, x) = \frac{\delta}{2} \pa_x^2 \she(t, x) + \frac{\kappa \sqrt{D}}{\delta} \she(t, x) \xi(t, x). 
\end{equation*}
So long as $\mathcal{Z}(0, x)$ is (almost surely) positive, \cite{mueller91} proved that $\mathcal{Z}(t, x)$ remains positive for all $t > 0$ and $x$. 
This justifies the well-definedness of \eqref{eq:hftransform}. Other equivalent definitions of the solution are given by  regularity structure \cite{Hai14}, paracontrolled distribution \cite{GP17} or the notion of energy solution \cite{GJ14, GMP18}.  
\bigskip
\\
It is well-known that there is no non-trivial scaling under which the \ac{KPZ} equation is invariant in law. More precisely, if we define $\mathcal{H}_\ep (t, x) =  \ep^{z} \mathcal{H}(\ep^{-b} t, \ep^{-1} x)$, using the scaling of space-time white noise $\xi(\ep^{-b} t, \ep^{-1} x) = \ep^{\frac{b+1}{2}} \xi(t, x)$ (in law), then
\begin{equation}\label{eq:scaledkpz}
\pa_t \HHe(t, x) = \frac{\delta}{2} \ep^{2 - b} \pa_{x}^2 \HHe(t, x) + \frac{\kappa}{2} \ep^{-z+ 2 - b} (\pa_{x} \mathcal{H}_\ep (t, x))^2 + \ep^{z + \frac{1}{2} - \frac{b}{2}} \sqrt{D} \xi(t, x).
\end{equation}
It is clear that there is no $b, z$ such that the coefficients in the above equation match with those in \eqref{eq:kpz}. However, if we simultaneously scale some of the parameters $\delta$, $\kappa$, $D$, it is possible that the \ac{KPZ} equation remains unchanged: such scaling is called \emph{weak scaling}. 
It is thus natural to believe that the KPZ equation is the weak scaling limit of microscopic models with similar properties such as relaxation and lateral growth. Roughly speaking, this is the  \emph{weak universality of the \ac{KPZ} equation}, see \cite{Cor12, Qua11} for an extensive survey.  We emphasize that the weak universality of the \ac{KPZ} equation should be distinguished from \emph{\ac{KPZ} universality}, which says that without tuning of the parameter of the model, the microscopic system converges to a universal limit called \emph{\ac{KPZ} fixed point} under $[1:2:3]$ scaling, see \cite{MQR16, DOV18, BL19} for some recent progress and breakthroughs in identifying the \ac{KPZ} fixed point.
\bigskip
\\
The weak universality of the \ac{KPZ} equation has been verified for a number of interacting particle systems.  The first result was given in the work of \cite{BG97}, for \ac{ASEP}. For more results of the weak universality of \ac{KPZ} equation, see Section 1.5.3 of \cite{CGST18} for a brief review.
%
\bigskip
\\
Recently \cite[Theorem 1.1]{CGST18} proved that under weak asymmetric scaling (which corresponds to taking $b = 2, z = \frac{1}{2}$ and $\kappa \to \sqrt{\ep} \kappa$ in \eqref{eq:scaledkpz}), the stochastic six vertex model converges to the \ac{KPZ} equation. In this paper, we consider stochastic higher spin six vertex model (SHS6V) model introduced in \cite{CP16}\footnote{The \ac{SHS6V} model has vertical and horizontal spin parameters $I, J \in \ZZ_{\geq 1}$. The stochastic six vertex model is a degeneration of it by taking $I = J = 1$.}. We prove that under similar weak asymmetric scaling, the \ac{SHS6V} model converges to the \ac{KPZ} equation. This extends the result of \cite[Theorem 1.1]{CGST18} to the full generality. We like to emphasize that there are some significant new complications in our case compared with \cite{CGST18}, see Section \ref{sec:method} for discussion.
\bigskip
\\
Before ending this section, we remark that there might be other SPDEs (besides the \ac{KPZ} equation) arising from the vertex model. For instance, it was shown in \cite{BG18, ST18} that under a different scaling, the stochastic six vertex model converges to the solution of the stochastic telegraph equation. It is interesting to ask whether the \ac{SHS6V} model converges to other SPDEs, this question is left for future work.
\subsection{The \ac{SHS6V} model}\label{sec:model}
The SHS6V model introduced in \cite{CP16}  (also see \cite{Bor17}) belongs to the family of vertex models which themselves are examples of quantum integrable systems. 
In general, the $R$-matrix (which can be thought of as the weights associated to the vertex) are not stochastic.
\cite{GS92, BCG16} studied the stochastic six vertex model, which is a stochastic version of the six vertex model introduced by \cite{pauling35}. The authors of \cite{CP16} worked with the $L$-matrices, which is a  stochastic version of the $R$-matrices\footnote{See \cite[Remark 2.2]{CP16} for more discussion of the relation between $L$-matrices and $R$-matrices.} and they defined the \ac{SHS6V} model. The stochasticity allows us to define the vertex model on the entire line as an interacting particle system which follows sequential Markov update rule. Moreover, the $L$-matrices in \cite{CP16} satisfy the Yang-Baxter equation which implies the integrability of the model. In particular, the transfer matrices are diagonalizable by a complete set of Bethe ansatz eigenfunctions \cite{BCPS15, CP16}. The model also enjoys Markov duality. The stochastic $R$-matrices of the \ac{SHS6V} model have four parameters, by specifying which the \ac{SHS6V} model degenerates to known integrable systems such as stochastic six vertex model, \ac{ASEP}, q-Hahn TASEP, q-TASEP. Indeed, it is on top of a hierarchy of \ac{KPZ} class integrable probabilistic systems. Recent studies of the \ac{SHS6V} model and its dynamical version include \cite{OP17, Agg18, Bor18, BP18, IMS19}. 
\bigskip
\\
Let us recall the definition of the \ac{SHS6V} model from \cite{CP16}. Fix $I, J \in \ZZ_{\geq 1}, \alpha, q\in \RR$,  we define the $L$-matrix $\rj:\ZZn^4 \to \RR$ via
\begin{align*}
\rj(i_1, j_1; i_2, j_2) = &\idc_{\{i_1 + j_1 = i_2 + j_2\}} q^{\frac{2j_1 - j_1^2}{4} - \frac{2j_2 - j_2^2}{4} + \frac{i_2^2 + i_1^2}{4} + \frac{i_2 (j_2 - 1) + i_1 j_1}{2}}\\
\numberthis \label{eq:rjmatrix}
&\times \frac{\nu^{j_1 - i_2} \alpha^{j_2 - j_1 + i_2} (-\alpha\nu^{-1}; q)_{j_2 - i_1}}{(q;q)_{i_2} (-\alpha; q)_{i_2 + j_2} (q^{J+1 - j_1 }; q)_{j_1 - j_2}}  { }_{4} \bar{\phi}_3 \bigg(\begin{matrix}
q^{-i_2};q^{-i_1}, -\alpha q^J, -q\nu \alpha^{-1}\\ \nu, q^{1 + j_2 - i_1}, q^{J + 1 - i_2 - j_2} 
\end{matrix} \bigg| q, q\bigg).
\end{align*} 
Here, $\nu = q^{-I}$ and ${}_4 \bar{\phi}_3$ is the regularized terminating basic hyper-geometric series defined by 
\begin{align*}
 _{r+1} \bar{\phi}_r \bigg(\begin{matrix}
 q^{-n}, a_1, \dots, a_r\\ b, \dots, b_r 
\end{matrix}
\bigg|q, z \bigg) &= \sum_{k=0}^n z^k \frac{(q^{-n}; q)_k}{(q; q)_k} \prod_{i=1}^r (a_i; q)_k (b_i q^k; q)_{n-k},
\end{align*}
where we recall the $q$-Pochhammer symbols $(a, q)_n$ (here $n$ is allowed to be negative) are defined by
\begin{equation*}
(a; q)_n := 
\begin{cases}
\prod_{i=1}^n (1 - a q^{i-1}), & n > 0,
\\
1, & n = 0,\\
\prod_{k=0}^{-n-1} (1 - a q^{n+k})^{-1}, & n < 0.
\end{cases}
\end{equation*}
We view $\rj$ as a matrix with row indexed by $(i_1, j_1) \in \ZZn^2$ and column indexed by $(i_2, j_2) \in \ZZn^2$.
Note that the $L$-matrix in \eqref{eq:rjmatrix} actually depends on four generic parameters $\alpha, q, I, J$, we suppress the dependence on $q, I$ in the notation of $\rj$ to simplify the notation.  
\bigskip
\\
It is straightforward by definition that for $(i_1, j_1) \in \iset \times \jset$ (using $\nu = q^{-I}$) 
\begin{equation*}
\rj(i_1, j_1; i_2, j_2) = 0, \qquad \text{ for all } (i_2, j_2) \in \ZZ_{\geq 0}^2 \backslash \iset \times \jset,
\end{equation*}
which means there is no way to transition out of $\iset \times \jset$ from itself. Therefore, in the following we restrict ourselves to the block  with $(i_1, j_1), (i_2, j_2) \in \iset \times \jset$.
\bigskip
\\
When $J=1$, by straightforward calculation, the $L$-matrix defined above simplifies to
\begin{equation}\label{eq:temp33}
\begin{split}
\rmatrix(m, 0; m, 0) &= \frac{1+\alpha q^m}{1+\alpha}, \hspace{2em}\rmatrix(m, 0; m-1, 1) = \frac{\alpha(1-q^m)}{1+\alpha}, \\
\rmatrix(m, 1; m+1, 0) &= \frac{1 - \nu q^m}{1+\alpha}, \hspace{3.8em} \rmatrix(m, 1; m, 1) = \frac{\alpha+\nu q^m}{1+\alpha}.
\end{split}
\end{equation}
For the history of the expression \eqref{eq:rjmatrix}, we remark that more intricate expressions for a quantity similar to the $\rj$ had been known in the context of quantum integrable systems since the work of \cite{KR87}. Relatively compact expressions of $\rj$ became available only in more recent times after the work of \cite{Man14}. \cite{CP16} also provides a probabilistic proof for this expression.
\bigskip
\\ 
From our perspective, we will think of  $\rj(i_1, j_1; i_2, j_2)$ as the weight associated to a vertex configuration with $i_1$ input lines from south, $j_1$ input lines from west, $i_2$ output lines to the north and $j_2$ output lines to the east see Figure \ref{fig:vertex}. Since we have restricted $\rj(i_1, j_1; i_2, j_2)$ to $(i_1, j_1), (i_2, j_2) \in \iset \times \jset$, we can have at most $I$ vertical lines and $J$ horizontal lines in the vertex configuration. Note that due to the indicator in \eqref{eq:rjmatrix}, all non-zero vertex weights $\rj(i_1, j_1; i_2, j_2)$ satisfy $i_1 + j_1 = i_2 + j_2$, a property that we consider as conservation of lines.
\bigskip
\\
In this paper, we always assume the following condition. 
\begin{cond}\label{cond:1} 
We take $
q > 1, \alpha < -q^{-(I + J - 1)}
$ and as we noted before, $\nu = q^{-I}$. 
\end{cond} 
It follows from \cite{CP16} that under Condition \ref{cond:1}, $\rj$ is a stochastic matrix on  $\iset \times \jset$. In other words, for any fixed $(i_1, j_1) \in \iset \times \jset$, $\rj(i_1, j_1; \cdot, \cdot)$ defines a probability measure on $\iset \times \jset$. Although in this paper we will not investigate the range of parameters out of Condition \ref{cond:1}, it is worth remarking that there are other choices of parameters which make $\rj$ stochastic, a few of them are provided in \cite[Proposition 2.3]{CP16}.
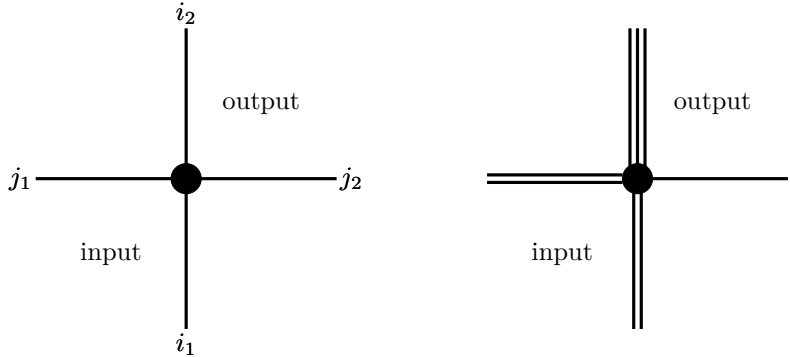
\begin{figure}[ht]\label{fig:rmatrix}
\begin{tikzpicture}
\draw[fill] (2, 2) circle (0.2);
\draw[very thick] (2, 0) -- (2, 4);
\draw[very thick] (0, 2) -- (4, 2);
\node at (2, -0.2) {$i_1$};
\node at (2, 4.2) {$i_2$};
\node at (-0.2, 2) {$j_1$};
\node at (4.2, 2) {$j_2$};
\draw[fill] (8, 2) circle (0.2);
\draw[very thick] (7.95, 0) -- (7.95, 1.9);
\draw[very thick] (8.05, 0) -- (8.05, 1.9);
\draw[very thick] (8, 2.1) -- (8, 4);
\draw[very thick] (7.9, 2.1) -- (7.9, 4);
\draw[very thick] (8.1, 2.1) -- (8.1, 4);
\draw[very thick] (6, 1.95) -- (7.8, 1.95);
\draw[very thick](6, 2.05) -- (7.8, 2.05);
\draw[very thick] (8.2, 2) -- (10, 2);
\node at (2, -0.2) {$i_1$};
\node at (2, 4.2) {$i_2$};
\node at (-0.2, 2) {$j_1$};
\node at (4.2, 2) {$j_2$};
\node at (1, 1) {input};
\node at (3, 3) {output};
\node at (7, 1) {input};
\node at (9, 3) {output};
\end{tikzpicture}
\caption{\textbf{Left:} The vertex configuration labeled by four tuples of integer $(i_1, j_1; i_2, j_2) \in \ZZn^4$ (from bottom and then in the clockwise order) has weight $\rj(i_1, j_1; i_2, j_2)$,  which takes $i_1$ vertical input lines and $j_1$ horizontal input lines,  and produce $i_2$ vertical output lines and $j_2$ horizontal output lines.  \textbf{Right:} The representation of the vertex configuration $(i_1, j_1; i_2, j_2) = (2, 2; 3, 1)$ in terms of lines.}
\label{fig:vertex}
\end{figure}
\bigskip
\\
There are several equivalent ways to define the \ac{SHS6V} model. In this paper, we view the \ac{SHS6V} model as a one-dimensional interacting particle system, which follows a sequential update rule. We proceed to give a precise definition of it. Denote by the space of left-finite particle configuration
\begin{equation}\label{eq:leftfinitestate}
\GG = \{\vec{g} = (\dots, g_{-1}, g_0, g_1 \dots): \text{ all } g_i \in \iset \text{ and there exists } x \in \ZZ \text{ such that } g_i = 0 \text{ for all } i < x.\},
\end{equation} 
where $g_x$ should be understood as the number of particles at position $x$. We define a discrete time Markov process $\vec{g}(t) = (g_x(t))_{x \in \ZZ} \in\GG$ as follows. 
\begin{defin}[left-finite fused \ac{SHS6V} model]\label{def:fused}
For any state $\vec{g} = (g_x)_{x\in \ZZ} \in \GG$, we specify the update rule from state $\vec{g}$ to $\vec{g}'$ as follows:
Assume the leftmost particle in the configuration $\vec{g}$ is at $x$ (i.e. $g_x > 0$ and $g_z = 0$ for all $z < x$). Starting from $x$, we update $g_x$ to $g'_x$ by setting $h_x = 0$ and randomly choosing $g'_x$ according to the probability $\rj(g_x, h_x = 0; g'_x, h_{x+1})$ where $h_{x+1} := g_x - g'_x$. Proceeding sequentially, we update $g_{x+1}$ to $g'_{x+1}$ according to the probability $\rj(g_{x+1}, h_{x+1}; g'_{x+1}, h_{x+2})$ where $h_{x+2} := g_{x+1} + h_{x+1} - g'_{x+1}$. Continuing for $g_{x+2}, g_{x+3}, \dots$, we have defined the update rule from $\vec{g}$ to $\vec{g}' = (g'_x)_{x \in \ZZ}$, see Figure \ref{fig:jmodel} for visualization of the update procedure. We call the discrete \textbf{time-homogeneous} Markov process  $\vec{g}(t) \in \GG$ with the update rule defined above
\textbf{the left-finite fused \ac{SHS6V} model}.\footnote{Note that in Definition \ref{def:fused}, although the update from $\vec{g}$ to $\vec{g}'$ may never stop as it goes to the right, the process is well-defined since we only care about the sigma algebra generated by $(g_x)_{x \leq z, x\in \mathbb{Z}}$ for all $z \in \mathbb{Z}$.}
\end{defin}
\begin{figure}[ht]
\begin{tikzpicture}
\draw[very thick] (0, 4.5) -- (12, 4.5);
\draw[very thick] (1.5, 3) -- (1.5, 6);
\draw[very thick] (4.5, 3) -- (4.5, 6);
\draw[very thick] (7.5, 3) -- (7.5, 6);
\draw[very thick] (10.5, 3) -- (10.5, 6);
\node at (1.5, 2.7)  {$g_x = 3$};
\node at (4.5, 2.7)  {$g_{x+1} = 2$};
\node at (7.5, 2.7)  {$g_{x+2} = 1$};
\node at (1.5, 6.3)  {$g'_x = 1$};
\node at (4.5, 6.3)  {$g'_{x+1} = 3$};
\node at (7.5, 6.3)  {$g'_{x+2} =2 $};
\node at (0, 4.8) {$h_x = 0$};
\node at (3, 4.8) {$h_{x+1} = 2$};
\node at (6, 4.8) {$h_{x+2} = 1$};
\node at (9, 4.8) {$h_{x+3} = 0$};
\node at (10.5, 2.7)  {$g_{x+3} = 3$};
\node at (10.5, 6.3)  {$g'_{x+3} = 2$};
\node at (12, 4.8) {$h_{x+4} = 1$};
\node at (12.5, 4.5) {\dots};
\draw[very thick] (0, -1) -- (12, -1);
\
\draw[very thick] (0, 4.5) -- (12, 4.5);

\draw[fill]  (1.5, -0.8) circle (0.2);
\draw[fill = gray]  (1.5, -0.4) circle (0.2);
\draw[fill = gray]  (1.5, 0) circle (0.2);
\draw[fill]  (4.5, -0.8) circle (0.2);

\draw[fill = gray]  (4.5, -0.4) circle (0.2);
\draw[fill]  (7.5, -0.8) circle (0.2);
\draw[fill]  (10.5, -0.8) circle (0.2);
\draw[fill]  (10.5, -0.4) circle (0.2);
\draw[fill = gray]  (10.5, 0) circle (0.2);
\node at (1.5, -1.3) {$x$};
\node at (4.5, -1.3) {$x+1$};
\node at (7.5, -1.3) {$x+2$};
\node at (10.5, -1.3) {$x+3$};
\node at (1.5, 1.3) {$\rj(3, 0; 1, 2)$};
\node at (4.5, 1.3) {$\rj(2, 2; 3, 1)$};
\node at (7.5, 1.3) {$\rj(1, 1; 2, 0)$};
\node at (10.5, 1.3) {$\rj(3, 0; 2, 1)$};
\node at (12, -0.5) {\dots};
\draw[bend left = 60, very thick, ->]  (1.5, 0.3) to (4.5, 0.3);
\draw[bend left = 60, very thick, ->]  (4.5, -0.1) to (7.5, -0.1);
\draw[bend left = 60, very thick, ->]  (7.5, -0.4) to (10.5, 0.4);
\draw[bend left = 60, very thick, ->]  (10.5, 0.3) to (13.5, 0.3);
\end{tikzpicture}
\caption{The visualization of the sequential update rule for the left-finite fused \ac{SHS6V} model in Definition \ref{def:fused}.  Assuming $x$ is the location of the leftmost particle, we update sequentially for positions $x, x+1, x+2, \dots$ according to the stochastic matrix $\rj$, the gray particles in the picture above will move one step to the right.}
\label{fig:jmodel}
\end{figure}
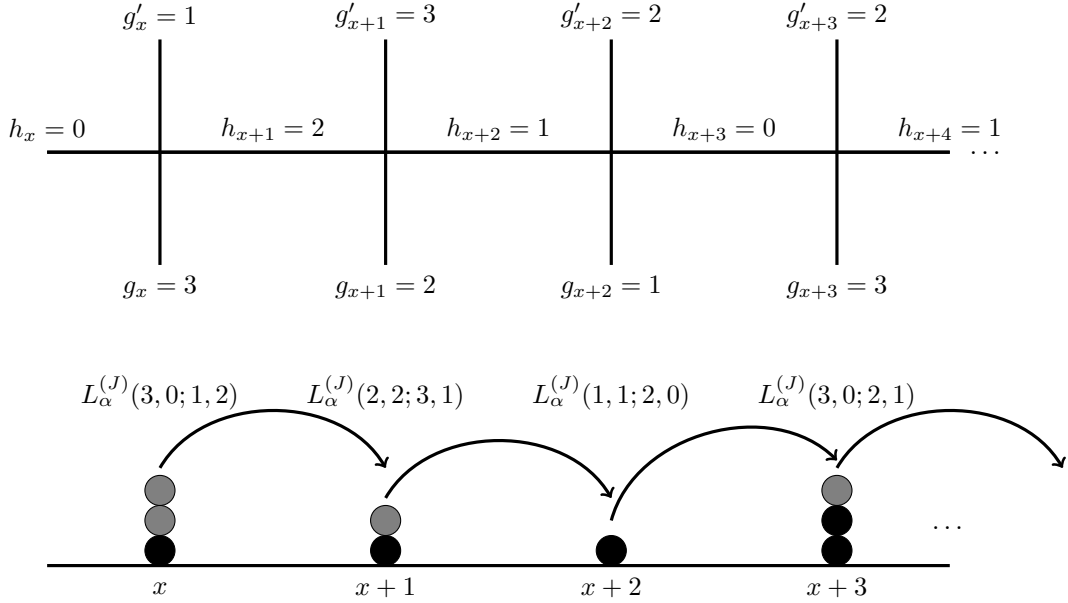
For $s \in \ZZ_{\geq 0}$, we define $\mod(s) := s - J\floor{s/J}$. For instance, $$\big(\mod(0), \mod(1), \dots, \mod(J-1), \mod(J), \mod(J+1), \dots\big) = \big(0, 1, \dots, J-1, 0, 1, \dots\big).$$
We further define $\alphat{\alpha}{t} = \alpha q^{\mod(t)}$ for $t \in \ZZn$.
\begin{defin}[left-finite unfused \ac{SHS6V} model]\label{def:unfused}
For all state $\vec{\eta} \in \GG$, we specify the update rule at time $t$ from state $\vec{\eta}$ to $\vec{\eta}' \in \GG$ as follows.
Assume the leftmost particle in the configuration $\vec{\eta}$ is at $x$. Starting from $x$, we update $\eta_x$ to $\eta'_x$ by setting $h_x = 0$ and randomly choosing $\eta'_x$ according to the probability $L_{\alpha(t)}^{(1)} (\eta_x, h_x; \eta'_x, h_{x+1})$ where $h_{x+1} := \eta_x + h_x - \eta'_x$. Proceeding sequentially, we update $\eta_{x+1}$ to $\eta_{x+1}$ according to the probability $L_{\alpha(t)}^{(1)}(\eta_{x+1}, h_{x+1}; \eta'_{x+1}, h_{x+2})$ where $h_{x+2} := \eta_{x+1} + h_{x+1} - \eta'_{x+1}$. Continuing for $\eta_{x+2}, \eta_{x+3}, \dots$, we have defined the update rule from $\vec{\eta}$ to $\vec{\eta}' = (\eta'_x)_{x \in \ZZ}$.
We call the discrete \textbf{time-inhomogeneous} Markov process $\vec{\eta}(t) \in \GG$ with the update rule defined above \textbf{the left-finite unfused \ac{SHS6V} model}. 
\end{defin}
\begin{remark}
It is straightforward to check that under Condition \ref{cond:1}, for all $t \in \ZZn$, $L_{\alpha(t)}^{(1)}$ in \eqref{eq:temp33} is a stochastic matrix which transfers $\iset \times \{0, 1\}$ to itself.
\end{remark}
In this paper, as a notational convention, we always use $\vec{g}(t)$ to denote the fused \ac{SHS6V} model and $\vec{\eta}(t)$ to denote the unfused one. The connection between them is specified in the following proposition.
\begin{prop}[\cite{CP16}, Theorem 3.15]
\label{prop:fusion}
Consider the left-finite fused \ac{SHS6V} model $\vec{g}(t)$ and the left-finite unfused \ac{SHS6V} model $\vec{\eta}(t)$. If $\vec{g}(0) = \vec{\eta}(0)$ in law, then
$$(\vec{g}(t), t \geq 0) = (\vec{\eta}(Jt), t \geq 0) \quad \text{ in law }.$$
\end{prop}
By Proposition \ref{prop:fusion}, we can construct the \ac{SHS6V} model with higher horizontal spin ($J \in \NN$) from those with horizontal spin $J=1$. This procedure is called \emph{fusion}, which goes back to the work of \cite{KR87}.
Thanks to Proposition \ref{prop:fusion}, for any left-finite \ac{SHS6V} model $\vec{g}(t)$, we can couple it with a left-finite unfused \ac{SHS6V} model $\vec{\eta}(t)$ so that $\vec{g}(t) = \vec{\eta}(Jt)$. We will extend the definition of unfused \ac{SHS6V} model $\vec{\eta}(t)$ in Lemma \ref{lem:biinfinite} so that it takes value in a larger space of bi-infinite particle configuration $\iset^\ZZ$ (thus  extend as well the definition of the fused \ac{SHS6V} model using the relation $\vec{g}(t) = \vec{\eta}(Jt)$).
\bigskip
\\
For the particle configuration $\vec{g} \in \GG$, define 
\begin{equation}\label{eq:temp13}
N_x (\vec{g})  = \sum_{y \leq x} g_y.
\end{equation}
For the left-finite unfused \ac{SHS6V} model $\vec{\eta}(t) \in \GG$, we define the \emph{unfused height function} as 
\begin{equation}\label{eq:temp14}
\n(t, x ) = N_{x}(\vec{\eta}(t)) - N_0 (\vec{\eta}(0)).
\end{equation}
Note that in the notation of unfused height function, we suppress the underlying process $\vec{\eta}(t)$.
Similarly, we define the \emph{fused height function} $\N(t, x)$ for the left-finite fused \ac{SHS6V} model $\vec{g}(t) \in \GG$ as
$$\N(t, x) = N_{x}(\vec{g}(t)) - N_0 (\vec{g}(0)).$$
Since $\vec{g}(t) = \vec{\eta}(Jt)$, certainly one  has for all $t \in \ZZn$ and $x \in \ZZ$,
$\N(t, x) =  \n(Jt, x).$
\bigskip
\\
We will state our result for the fused height function $\N(t, x)$ though we will mainly work with the unfused height function $\n(t, x)$ in our proof. In the future, the notation of $\n (t, x)$ will often be shortened to $N(t, x)$.
\bigskip
\\ 
Having defined $\N(t, x)$ (respectively, $\n(t, x)$) on the lattice, we  linearly interpolate it first in space variable $x$ then in time variable $t$, which makes $\N(t, x)$ (respectively, $\n(t, x)$) a $C([0, \infty), C(\RR))$-valued process. 
For construction of height functions of the bi-infinite version of the fused or unfused \ac{SHS6V} model, see Lemma \ref{lem:biinfinite}.
\subsection{Result}
The main result
of our paper shows that the fluctuation of the fused height function $\N(t, x)$ converges weakly to the solution of the \ac{KPZ} equation. Fix $\rho \in (0, I)$, define 
\begin{equation}\label{eq:temp1}
\lambda =\frac{1 + \alpha - q^\rho (\alpha + \nu)}{1 + \alpha q^J - q^\rho (\alpha q^J + \nu)},\qquad \mu = \frac{\alpha q^\rho (1-q^J) (1-\nu)}{(1 +\alpha q^J - q^\rho (\alpha q^J + \nu))(1+\alpha - q^\rho (\alpha + \nu))}.
\end{equation}
As a matter of convention, we endow the space $C(\RR)$ and $C([0, \infty), C(\RR))$ with the topology of uniform convergence over compact subsets, and write $``\Rightarrow"$ for the weak convergence of probability laws. We present our main theorem.
\begin{theorem}\label{thm:main}
Fix $b \in \big(\frac{I+J-2}{I+J-1}, 1\big)$, $I \geq 2$ and $J \geq 1$, for small $\ep > 0$, wet $q = e^{\sqrt{\ep}}$ and define $\alpha$ via $b = \frac{1 + \alpha q}{1 + \alpha}$. We call this weakly asymmetric scaling. 
Assume that $\{\N_\epsilon (0, x)\}_{\epsilon > 0}$ is nearly stationary with density $\rho$ (see Definition \ref{def:nearstationary}) and 
\begin{equation*}
\sqrt{\epsilon} \Big(\N_\epsilon (0, \epsilon^{-1} x) - \rho \epsilon^{-1} x\Big) \wc \HH^{ic}  (x) \text{ in } C(\RR) \text{ as } \epsilon \downarrow 0,
\end{equation*}
then
\begin{align*}\numberthis \label{convergence}
\sqrt{\epsilon} \Big(\N_\epsilon (\epsilon^{-2} t, \epsilon^{-1} x + \epsilon^{-2} t \mu_\epsilon) - \rho (\epsilon^{-1} x + \epsilon^{-2} \mu_\epsilon  t )\Big)  - t \log \lambda_\epsilon  \wc \HH(t, x)&
\\
&\text{ in } C([0, \infty), C(\RR)) \text{ as } \ep \downarrow 0,
\end{align*}
where $\HH(t, x)$ is the Hopf-Cole solution of the  \ac{KPZ} equation 
\begin{equation}\label{eq:KPZ equation}
\partial_t \kpz(t, x) = \frac{J V_*}{2} \pa_x^2 \kpz(t, x) - \frac{J V_*}{2} \big(\pa_x \kpz(t, x) \big)^2 + \sqrt{J D_*} \xi(t, x),
\end{equation}
with initial condition $\HH^{ic} (x)$, where the coefficients are given by
\begin{align*}
\numberthis\label{eq:vstar}
V_* 
&= \frac{(I + J) b - (I+ J-2)}{I^2 (1-b)},\\
\numberthis \label{eq:dstar}
D_* 
&= \frac{\rho (I -\rho)}{I} \frac{(I + J) b - (I+ J-2)}{I^2 (1-b)}.
\end{align*} 
\end{theorem}
Note that the restriction of $b \in (\frac{I+J-2}{I+J-1}, 1)$ in Theorem \ref{thm:main} is necessary and sufficient to ensure that Condition \ref{cond:1} holds for $\ep$ small enough. In Appendix \ref{app:kpzscaling}, we will demonstrate how our theorem agrees with the non-rigorous \ac{KPZ} scaling theory used in  physics\footnote{The KPZ scaling theory is a non-rigorous physics method used to compute the constants (the coefficients of the \ac{KPZ} equation \eqref{eq:KPZ equation} as in our case) arising in limit theorems for the models in the \ac{KPZ}
universality class \cite{krug92, Spohn12}, which has been confirmed in a few cases such as \cite{DT16, Gho17}.}. 
\begin{remark}
In a different setting where $0 < q, \nu < 1$ (in contrast to Condition \ref{cond:1}, there is no $I \in \NN$ such that $\nu = q^{-I}$) and $\alpha \geq 0$, one can show that $\rj$ is a stochastic matrix on $\ZZ_{\geq 0} \times \jset$ (instead of $\iset \times \jset$ for our case). 
In this regime, the \ac{SHS6V} model allows arbitrary number of particles at each site (instead of at most $I$ particles for our case). \cite{CT17} proves the weak universality of the \ac{SHS6V} model\,\footnote{In the context of \cite{CT17},  the authors prove the weak universality for the higher spin exclusion process defined in \cite[Definition 2.10]{CP16}, which is equivalent to the \ac{SHS6V} model after a gap-particle transform. We describe their result in the language of the \ac{SHS6V} model here.} under a different type of weak scaling that corresponds to taking $b=3, z=1, \delta \to \ep \delta, \kappa \to \ep^2 \kappa $ in \eqref{eq:scaledkpz}. Under this scaling, the number of particles at each site diverges to infinity with rate $\ep^{-1}$. This simplifies considerably the control of the quadratic variation of the martingale in the discrete \ac{SHE} \eqref{eq:temp12}, which is the main complexity in our analysis.
\end{remark}
\begin{remark}
Taking $I = J = 1$, Theorem \ref{thm:main}  recovers \cite[Theorem 1.1]{CGST18}.
We assume $I \geq 2$ in Theorem \ref{thm:main} merely due to some technical subtleties we met in  Section \ref{sec:asymptotic analysis}. The proof for $I = 1$ needs particular modification and we do not pursue it here.
\end{remark}
The proof of Theorem \ref{thm:main} will be given in the end of Section \ref{sec:microshe}, as a corollary of Theorem \ref{thm:main1}.
\subsection{Method}\label{sec:method}
In this section, we explain the method used in proving Theorem \ref{thm:main}. Although initially our methods follow \cite{CGST18}, rather quickly, we encounter novel complexities that are not present in \cite{CGST18} which require new ideas. 
\bigskip
\\
As illustrated in Section \ref{sec:model}, via fusion, to study the fused \ac{SHS6V} model, it suffices to work with the unfused version. Similar to \cite{CGST18}, the first step is to perform a microscopic Hopf-Cole transform of the \ac{SHS6V} model \eqref{eq:microshe1}. The existence of the microscopic Hopf-Cole transform is guaranteed by one particle version of the duality \eqref{eq:fextenddual1}  (which goes back to \cite[Theorem 2.21]{CP16}). The microscopic Hopf-Cole transform $Z(t, x)$, which is essentially an exponential version of the unfused height function $N(t, x)$, satisfies 
 a discrete version of \ac{SHE}
\begin{equation}\label{eq:temp12}
dZ = \mathcal{L} Z dt + dM.
\end{equation}
Here $\mathcal{L}$ is an operator which approximates the Laplacian and $M$ is a martingale. 
Owing to the definition of the Hopf-Cole solution to the \ac{KPZ} equation, Theorem \ref{thm:main} is equivalent to showing that the above discrete \ac{SHE} converges to its continuum version (Theorem \ref{thm:main1}). The proof of Theorem \ref{thm:main1} reduces to three steps: 
\bigskip
\\
1). Showing tightness.
\\ 
2). Identifying the limit of the linear martingale problem.
\\
3). Identifying the limit of the quadratic martingale problem.
\bigskip
\\
Steps 1) and 2) follow from a similar approach as in \cite{CGST18}. Step 3) is the difficult part; Proposition \ref{prop:timedecorr}
does this by proving a form of self-averaging for the quadratic variation of the martingale $M$. We will focus on discussing the method for proving this self-averaging result in the rest of the section. We remark that other recent \ac{KPZ} equation convergence results using the Hopf-Cole transform include ASEP-$(q, j)$ \cite{CST18}, Hall-Littlewood
PushTASEP \cite{Gho17}, weakly asymmetric bridges \cite{Lab17}, open ASEP \cite{CS18, Par19}.
\bigskip
\\
We will explain what is self-averaging in a moment, but first introduce two tools used in proving it. The first tool is the \emph{Markov duality} and the second is the \emph{exact formula of two particle transition probability} of the \ac{SHS6V} model. 
\bigskip
\\ 
The stochastic six vertex model enjoys two Markov dualities \cite[Theorem 2.21]{CP16} and \cite[Theorem 1.5]{Lin19}\footnote{The Markov duality proved in \cite{Lin19} first appears in \cite[Theorem 2.23]{CP16}. In fact \cite[Theorem 2.23]{CP16} claims a more general Markov duality for the \ac{SHS6V} model. In discussions with the authors of [CP16], we recognized a gap in that proof as well as a counter-example to the result when $I>1$, see \cite{CP19} for detail.}, which are exploited in proving the self-averaging \cite[Proposition 5.6]{CGST18}.
The Markov duality in \cite[Theorem 2.21]{CP16} also works for the \ac{SHS6V} model (Proposition \ref{prop:firstdual} in our paper), yet it is unknown whether there exists a generalization of  \cite[Theorem 1.5]{Lin19} for the \ac{SHS6V} model. \cite[Theorem 4.10]{kuan18} discovers a general duality for the multi-species \ac{SHS6V} model using the algebraic machinery\footnote{As a remark, the functional in \cite[Theorem 4.10]{kuan18} also serves as the duality functional for a multi-species version of ASEP(q, j), see \cite{CGRS16, kuan17}.}. At first glance, the duality functional written in \cite[Theorem 4.10]{kuan18} takes a rather complicated form, but we only need a two particle version of this duality, in which case the duality functional simplifies greatly (Proposition \ref{prop:secdual} in our paper) and is applicable for proving the desired self-averaging. 
We remark that this is the first application of \cite[Theorem 4.10]{kuan18} as far as we know.
\bigskip
\\
In \cite[Theorem 3.6]{BCG16}, an integral formula was obtained for general $k$ particle transition probability of the stochastic six vertex model via a generalized Fourier theory (Bethe ansatz), using a complete set of eigenfunction of the stochastic six vertex model transition matrix obtained in \cite[Theorem 3.4]{BCG16} together with the Plancherel identity  \cite[Theorem 2.1]{TW08}. \cite{CGST18} applies the steepest descent analysis to a two particle version of this formula to extract a space-time bound, which is the key to control the quadratic variation of the martingale in \eqref{eq:temp12}.  
\bigskip
\\
For the \ac{SHS6V} model, it is natural to expect that the similar method should apply, since we also have a set of eigenfunctions from \cite[Proposiiton 2.12]{CP16} and a generalized Plancherel identity from \cite[Corollary 3.13]{BCPS15}. However, the Plancherel identity was originally designed only for $0 < q, \nu < 1$ and there is a technical issue in extending this identity to $q > 1, \nu = q^{-I}$ which has not been addressed in the existing literatures\footnote{\cite[Proposition A.3]{CP16} claims the Plancherel identity for $\nu = q^{-I}$ can be obtained by analytic continuation of \cite[Corollary 3.13]{BCPS15}. After discussions with the authors of \cite{CP16}, they agree that there is an issue in this analytic continuation (and the resulting identity) due to poles encountered along the way \cite{CP19}.} (see Remark \ref{rmk:ana}). 
Fortunately, we find that when $I \geq 2$ and there are only two particles, such analytic continuation does work, which produces an integral formula for the two particle \ac{SHS6V} model transition probability (Theorem \ref{thm:intformulap}). In terms of large contours, the integral formula consists of two double contour integrals and one single contour integral. We find that the single contour integral can be expressed as a residue of one of the double contour integrals. This simplifies our analysis since via certain contour deformation, the single contour integral will be canceled out.
\bigskip
\\
We will analyze (a tilted version of) this integral formula (Corollary \ref{cor:intformulav}) in Section \ref{sec:asymptotic analysis} using steepest descent analysis and obtain a very precise estimate of the (tilted) two particle transition probability $\rhzrt$ defined in \eqref{eq:four:tilttransition}. Compared with the analysis for stochastic six vertex model in \cite[Section 6]{CGST18}, one difficulty is to find (and justify) the contours for different $I, J$ such that the steepest descent analysis applies. Also in certain cases (Section \ref{sec:pp}) the steepest descent contour can only be implicitly defined (compared with \cite[Section 6]{CGST18} where all the steepest descent contour are circles), which complicates our analysis.
\bigskip
\\
Now let us explain what is self-averaging and how these two tools can be applied to prove it. 
Denote the discrete gradient by $\na f(x): = f(x+1) -  f(x)$. Roughly speaking, the terminology ``self-averaging" refers to the phenomena that as $\ep \downarrow 0$
\\
{\fontfamily{qpl}\selectfont\textbf{(A)}  For $x_1 \neq x_2$,  the average of $\ep^{-1} \na Z(t, x_1)  \na Z(t, x_2)$  over a long time interval of length $O(\ep^{-2})$ will vanish.\\
\textbf{(B) } There exists a positive constant $\lambda$ such that the average of $(\ep^{-\frac{1}{2}} \na Z(t, x))^2 - \lambda Z(t, x)^2$ over a long time interval of length $\OO(\ep^{-2})$ will vanish.}
\bigskip
\\
The proofs of \A\  and \B\   are given in Lemma \ref{lem:timedecorr1} and Lemma \ref{lem:timedecorr2} respectively, let us make a brief discussion about our strategy here. As we will see in 
\eqref{eq:seven1:temp2}, 
under weakly asymmetric scaling, 
\begin{equation}\label{eq:temp10}
\ep^{-\frac{1}{2}} \na Z(t, x) = (\rho - \etat_{x+1} (t)) Z(t, x) + \errr.
\end{equation}
where $\rho \in (0, I)$ is the density, $\etat_x (t) = \eta_{x + \hmu(t)} (t)$ and $\hmu(t)$ is some constant defined in \eqref{eq:temp11}. Pointwisely, $\ep^{-\frac{1}{2}} \na Z(t, x)$ is of the same order as $Z(t, x)$. But {\fontfamily{qpl}\selectfont\textbf{(A)}} tells that after averaging over a long time interval (we will just say "averaging" afterwards for short), $\ep^{-1} \na Z(t, x_1) \na Z(t, x_2)$ vanishes for $x_1 \neq x_2$, this explains the terminology of ``self-averaging". To prove \A, by the first duality in Lemma \ref{lem:dualt} (which goes back to Proposition \ref{prop:firstdual}), one is able to write down the conditional quadratic variation in terms of the summation of (a tilted version of) two particle transition  probability $\rhzrt$, i.e. for $x_1 \leq x_2$ 
\begin{equation}\label{eq:temp9}
\EE\big[Z(t, x_1) Z(t, x_2) \vv \FF(s)\big] = \sum_{y_1 \leq y_2} \rhzrt\big((x_1, x_2), (y_1, y_2), t, s\big) Z(s, y_1) Z(s, y_2)
\end{equation} 
This allows us to move the gradients from $Z(t, x_1)$ and $Z(t, x_2)$ to $\rhzrt$. We proceed by using a very precise estimate of $\rhzrt$ from Proposition \ref{prop:semiestimate}
(which is proved by making use of the steepest descent analysis of the integral formula of $\rhzrt$). Referring to Proposition \ref{prop:semiestimate}, each gradient on $\rhzrt\big((x_1, x_2), (y_1, y_2), t, s\big)$ gives an extra decay of $\frac{1}{\sqrt{t-s+1}}$, which helps us to conclude \A.  We remark that for demonstrating {\fontfamily{qpl}\selectfont\textbf{(A)}}, our argument is actually simpler than  that of \cite{CGST18}. Since we assume $I \geq 2$, \eqref{eq:temp9} holds for all $x_1 \leq x_2$, while in the situation of the stochastic six vertex model $(I=1)$, \eqref{eq:temp9} holds only for $x_1 < x_2$, due to the exclusion restriction (i.e. two particles can not stay in the same site). In fact, \cite{CGST18} needs both of the duality \cite[Theorem 2.21]{CP16} and \cite[Theorem 1.5]{Lin19} to prove {\fontfamily{qpl}\selectfont\textbf{(A)}}. 
\bigskip
\\
For {\fontfamily{qpl}\selectfont\textbf{(B)}}, there are two tasks: Identifying $\lambda$ and proving the self-averaging. These were done simultaneously for the stochastic six vertex model \cite{CGST18}: Note that by \eqref{eq:temp10}, 
\begin{equation}\label{eq:temp15}
(\ep^{-\frac{1}{2}} \na Z(t, x))^2 = (\etat_{x+1} (t)- \rho)^2 Z(t, x)^2 + \errr.
\end{equation}
For the stochastic six vertex model, $\etat_{x} (t) \in \{0, 1\}$ for all $t, x$, hence $\etat_x (t)^2 = \etat_x (t)$. \cite[Lemma 7.1]{CGST18} uses this crucial observation  to obtain
\begin{align*}
(\etat_{x+1} (t) - \rho)^2 Z(t, x)^2 &= \rho^2 Z(t, x)^2 + (1-2\rho) \etat_{x+1} (t) Z(t, x)\\
&=\rho(1-\rho) Z(t, x)^2 + \ep^{-\frac{1}{2}} (2\rho - 1) \na Z(t, x) Z(t, x) + \errr.
\end{align*} 
By similar method used in demonstrating \A, it is not hard to prove that  $\ep^{-\frac{1}{2}} \na Z(t, x) Z(t, x)$ vanishes after averaging, implying that $\lambda = \rho(1-\rho)$.
\bigskip
\\
For our case, first we note that $\etat_x (t) \in \iset$ with $I \geq 2$, so the $\etat_x (t)^2  = \etat_x (t)$ identity obviously fails. We need to find another way to determine $\lambda$ and prove the self-averaging. We proceed by first guessing the $\lambda$. Via \eqref{eq:temp15}, the average of $\ep^{-1} (\na Z(t, x))^2$ over a long time interval can be approximated by the average of $(\etat_x (t) - \rho)^2 Z(t, x)^2$. In Appendix \ref{sec:stat}, we derive a family of stationary distribution of the \ac{SHS6V} model, which is a product measure $\stat$, where $\pi_\rho$ is a probability measure on $\iset$ indexed by its mean $\rho \in (0, I)$. 
Starting the \ac{SHS6V} model $\vec{\eta}(t)$ from $\vec{\eta}(0) \sim \stat$, 
it is clear that $\etat_x (t) \sim \pi_\rho$ for all $t, x$. In a heuristic level, one can approximate the average of $(\etat_{x+1} (t) - \rho)^2 Z(t, x)^2$ by that of the $\EE_{\pi_\rho} \big[(\etat_{x+1} (t) - \rho)^2\big] Z(t, x)^2$. Under weakly asymmetric scaling, one computes that  $$\lim_{\ep \downarrow 0} \EE_{\pi_\rho}\big[(\etat_{x+1} (t) - \rho)^2\big] =  \frac{\rho(I-\rho)}{I},$$
which suggests $\lambda = \frac{\rho(I-\rho)}{I}$.
\bigskip
\\ 
To prove \B\  with $\lambda = \frac{\rho(I-\rho)}{I}$, note that the second duality in Lemma \ref{lem:dualt} (which goes back to Proposition \ref{prop:secdual}) implies 
\begin{equation}\label{eq:temp32}
\EE\big[D(t, x, x) \vv \FF(s)\big]  = \sum_{y_1 \leq y_2} D(s, y_1, y_2) \rhzrt\big((x, x), (y_1, y_2), t, s\big)
\end{equation}
where approximately\footnote{In fact, the functional $D(s, y_1, y_2)$ below is only an approximate version of the duality functional defined in \eqref{eq:dfunc}, we use this approximate version here to avoid extra  notations and make our argument more intuitive.}
\begin{equation}\label{eq:temp30}
D(s, y_1, y_2) = 
\begin{cases}
Z(s,y_1)^2   \big(I - \etapartt{s}{y_1}\big)  \big(I - 1 - \etapartt{s}{y_1}\big)  \qquad &\text{if}\  y_1 = y_2,\\
\frac{I-1}{I} Z(s, y_1) Z(s, y_2) \big(I - \etapartt{s}{y_1}\big)  \big(I - \etapartt{s}{y_2}\big)   \qquad  &\text{if}\  y_1 < y_2
\end{cases}
\end{equation}
Note that the expression of $D(s, y_1, y_2)$ is  different depending on whether $y_1 = y_2$, which is crucial to our proof. Rewriting $(\ep^{-\frac{1}{2}} \na Z(t, x))^2 - \frac{\rho(I-\rho)}{I} Z(t, x)^2$ in terms of the two duality functionals in \eqref{eq:temp9} and  \eqref{eq:temp30}
\begin{align*}
&(\ep^{-\frac{1}{2}} \na Z(t, x))^2 - \frac{\rho(I - \rho)}{I} Z(t, x)^2 = \bigg((\etat_{x+1} (t) - \rho)^2 - \frac{\rho(I - \rho)}{I}\bigg) Z(t, x)^2 + \errr  \\ 
&=  \bigg((I - \etat_{x+1}(t)) (I - 1- \etat_{x+1}(t)) - \frac{(I-1)(I-\rho)^2}{I}\bigg) Z(t, x+1)^2 - (2\rho +1 - 2I ) \ep^{-\frac{1}{2}} \na Z(t, x) Z(t, x) + \errr,\\
&= \bigg(D(t, x+1, x+1) - \frac{(I-1)(I-\rho)^2}{I}Z(t, x+1)^2\bigg) - (2\rho +1 - 2I ) \ep^{-\frac{1}{2}} \na Z(t, x) Z(t, x) + \errr.
\end{align*}
It is not hard to show that the second term  $\ep^{-\frac{1}{2}} \na Z(t, x) Z(t, x)$ vanishes after averaging. For the first term above, we combine both of the dualities \eqref{eq:temp9},  \eqref{eq:temp32} and get
\begin{align*}
&\EE\bigg[D(t, x+1, x+1) - \frac{(I-1)(I-\rho)^2}{I} Z(t, x+1)^2 \vvv \FF(s)\bigg] \\ \numberthis \label{eq:temp31}
&= \sum_{y_1 \leq y_2} \rhzrt\big(x+1, x+1, y_1, y_2, t, s\big) \bigg(D(s, y_1, y_2) - \frac{(I-1)(I-\rho)^2}{I} Z(s, y_1) Z(s, y_2)\bigg).  
\end{align*}
The number of pairs $(y_1, y_2)$ such that $y_1 = y_2$ compared with $y_1 < y_2$ is negligible in the summation above so it suffices to study for $y_1 < y_2$ 
\begin{align*}
&D(s, y_1, y_2) - \frac{(I-1)(I-\rho)^2}{I} Z(s, y_1) Z(s, y_2) \\
&= \bigg(\frac{I-1}{I}  (I - \etapartt{s}{y_1})  (I - \etapartt{s}{y_2}) - \frac{(I-1)(I-\rho)^2}{I}\bigg) Z(s, y_1) Z(s, y_2)\\
&=(I - \etat_{y_2}(s)) (\ep^{-\frac{1}{2}} \nabla Z(s, y_1))  Z(s, y_2) + (I-\rho) (\ep^{-\frac{1}{2}} \na Z(s, y_2)) Z(s, y_1) + \errr.
\end{align*}
Inserting this expression into the RHS of \eqref{eq:temp31} and using the summation by part formula (see \eqref{eq:sumbypart1}), we can move the gradient from $Z$ to $\rhzrt$. Similar to the argument for \A, applying the estimate in Proposition \ref{prop:semiestimate}  completes the proof of \B.
\subsection{Outline}
The paper will be organized as follows. In Section \ref{sec:sec2} we give an equivalent definition of SHS6V model through fusion. At the beginning, we require the existence of a leftmost particle. After that we extend the definition to a bi-infinite version of the \ac{SHS6V} model (Lemma \ref{lem:biinfinite}), which is the object that we study for the rest of the paper. 
In Section \ref{sec:duality}, we introduce two Markov dualities enjoyed by the model. The first one is taken directly from the \cite[Theorem 2.21]{CP16}. The second one is a certain degeneration from a general duality in  \cite[Theorem 4.10]{kuan18}. Section \ref{sec:exactformula} contains the derivation of integral formula for the two point transition probability of the \ac{SHS6V} model. In Section 5, we define the microscopic Hopf-Cole transform and prove that it satisfies a discrete version of \ac{SHE}. Due to the definition of the Hopf-Cole solution to the KPZ equation, it suffices to prove that the discrete \ac{SHE} converges to its continuum version. In Section \ref{sec:tightness}, we prove this result in two steps. First, we establish the tightness of the discrete \ac{SHE}.  Second, we show that any limit point is the solution to the \ac{SHE} in continuum, assuming the self-averaging property (Proposition \ref{prop:timedecorr}). The last two sections are devoted to the proof of Proposition \ref{prop:timedecorr}.
In Section \ref{sec:asymptotic analysis}, we obtain a very precise estimate for the two point transition probability by applying steepest descent analysis to the integral formula that we obtain in Section \ref{sec:exactformula}. In Section \ref{sec:timedecorr}, we prove Proposition \ref{prop:timedecorr} using the Markov duality and our estimate of the two point transition probability. 
\subsection{Notation} In this paper, we denote $\ZZ_{\geq i} = \{n\in \ZZ: n \geq i\}$. $\idc_E$ denotes the indicator function of an event $E$. We use $\EE$ (respectively, $\PP$) to denote the expectation (respectively, probability) with respect to the process or random variable that follow. The symbol $\C_r$ represents a circular contour centered at the origin with radius $r$. All contours, unless otherwise specified, are counterclockwise.  
\subsection{Acknowledgment}
The author heartily thanks his advisor Ivan Corwin for suggesting this problem, his encouragement and helpful discussion, as well as reading part of the manuscript. We thank Jeffrey Kuan for the helpful discussion over the degeneration of \cite[Theorem 4.10]{kuan18} to the stochastic higher spin six vertex model. We also wish to thank Promit Ghosal and Li-Cheng Tsai, who provided very useful discussions about the technical part of the paper. We are also grateful to two anonymous referees for their valuable suggestions. The author was partially supported by the Fernholz Foundation's
``Summer Minerva Fellow" program and also received summer support from Ivan Corwin's
NSF grant DMS:1811143.
\section{The bi-infinite \ac{SHS6V} model}
\label{sec:sec2}
The main goal of this section is to 
extend the definition of the left-finite unfused (fused) \ac{SHS6V} model in Definition \ref{def:unfused} (Definition \ref{def:fused}) to the space of bi-infinite configurations $\iset^\ZZ$. The motivation of such extension is to include one important class of initial condition called \emph{near stationary initial condition} as in \cite{BG97}. We will proceed following the idea of \cite{CGST18}, which goes back to \cite{CT17}. By fusion (Proposition \ref{prop:fusion}), it suffices to extend the left-finite unfused \ac{SHS6V} model $\vec{\eta}(t)$. The extension of the fused version $\vec{g}(t)$ follows easily by taking $\vec{g}(t) = \vec{\eta}(Jt)$. 
\bigskip
\\
For the extension, the first step is to restate the \ac{SHS6V} model in a parallel update rule. To this end, we equip the probability space with a family of independent Bernoulli random variables $B(t, y, \eta), B'(t, y, \eta)$ such that  
\begin{equation}\label{eq:randomenvir}
B(t, y, \eta) \sim \ber\bigg(\frac{\alphat{\alpha}{t}(1- q^{\eta})}{1 + \alphat{\alpha}{t}}\bigg), \qquad B'(t, y, \eta) \sim \ber\bigg(\frac{\alphat{\alpha}{t} + \nu q^{\eta}}{1 + \alphat{\alpha}{t}}\bigg),
\end{equation}
for $t \in \ZZ_{\geq 0}$, $y \in \ZZ$ and $\eta \in \iset$, recall that $\alpha(t) = \alpha q^{\mod(t)}$.
\bigskip
\\
Treating these Bernoulli random variables as a random environment, we find an equivalent way to define the left-finite unfused \ac{SHS6V} model, through recursion. Given initial state $\vec{\eta}(0) \in \GG$, define $N(0, x) := N_x (\vec{\eta}(0)) - N_0 (\vec{\eta}(0))$ (recall the notation from \eqref{eq:temp13}) and recursively for $t = 0, 1, \dots,$
\begin{align*}
\numberthis \label{eq:leftfrecur1}
 N(t+1, y) &:=
\begin{cases}
N(t, y) - B(t, y, \occupvr{y}{t})   \qquad \text{ if } N(t, y-1) - N(t+1, y-1) = 0,\\
N(t, y) - B'(t, y, \occupvr{y}{t})  \qquad  \text{ if } N(t, y-1) - N(t+1, y-1) = 1.
\end{cases}
\\
\occupvr{y}{t+1} &:= N(t+1, y)  - N(t+1, y-1).
\end{align*}
It is straightforward to see that $\vec{\eta}(t) = (\eta_y (t))_{y \in \ZZ}$ is a left-finite unfused \ac{SHS6V} model and $N(t, x)$ is indeed its height function defined by \eqref{eq:temp14}.
\bigskip
\\
The recursion  \eqref{eq:leftfrecur1} is equivalent to 
\begin{equation}\label{eq:two:recur1}
N(t, y) - N(t+1, y) = \Big(N(t, y-1) - N(t+1, y-1)\Big) \Big(B'(t, y, \eta_y(t)) - B(t, y, \eta_y(t))\Big) + B(t, y, \eta_y (t)).
\end{equation}
Iterating \eqref{eq:two:recur1} implies
\begin{align}\label{eq:two:recur}
N(t, y) - N(t+1, y) 
&= \sum_{y' = -\infty}^y \prod_{z = y'+1}^y \Big(B'(t, z, \occupvr{z}{t}) - B(t, z, \occupvr{z}{t})\Big) B(t, z, \occupvr{z}{t}).
\end{align}
Note that the summation above is finite. The reason is that since  $\vec{\eta}(t) \in \GG$, there exists $w$ such that $\eta_z(t) = 0$ for all $z < w$, which implies $B(t, z, \eta_z(t)) = 0$ for all $z < w$.
\bigskip
\\
In light of \eqref{eq:two:recur}, we extend the Definition \ref{def:unfused} to the space of  bi-infinite particle configuration $\iset^\ZZ$.
\begin{lemma}\label{lem:biinfinite}
For any bi-infinite particle configuration $\vec{\eta}(0) \in \iset^\ZZ$, define the initial height function 
\begin{align*}
N(0, x) 
&= \idc_{\{x > 0\}} \sum_{i=1}^x \eta_i (0) -  \idc_{\{x < 0\}} \sum_{i=1}^{-x} \eta_{-i} (0).
\end{align*}
Note that if $\vec{\eta}(0) \in \GG$, $N(0, x)$ defined above coincides with that defined in \eqref{eq:temp14}. 
We inductively define the $\vec{\eta}(t)$ and $N(t, x)$ for $t = 0, 1, \dots$ via the recursion
\begin{align}
\label{eq:birecur1}
N(t, y) - N(t+1, y) &:= \sum_{y' = -\infty}^y \prod_{z = y'+1}^y \Big(B'(t, z, \occupvr{z}{t}) - B(t, z, \occupvr{z}{t})\Big) B(t, z, \occupvr{z}{t}),
\\
\label{eq:birecur2}
\eta_y (t+1) &:= N(t+1, y) - N(t+1, y-1).
\end{align}
For $p \geq 1$, the infinite sum in \eqref{eq:birecur1} converges almost surely and in $\LL^p$  to a $\{0, 1\}$-valued random variable. Furthermore, consider left-finite initial configuration $\vec{\eta}^w (0) = (\eta_i (0) \idc_{\{i \geq w\}})_{i \in \ZZ}$ and the height function $N^w (t, y)$ inductively defined by \eqref{eq:birecur1} and \eqref{eq:birecur2}, then  for all $t \in \ZZ_{\geq 0}$ and $y \in \ZZ$
$$\lim_{w \to -\infty} N^w (t, y) = N(t, y) \text{ in } \LL^p.$$
\end{lemma}
\begin{remark}
It is clear that via \eqref{eq:birecur1}, one can recover the recursion \eqref{eq:leftfrecur1} since
\begin{align*}
N(t, y) - N(t+1, y) &= \sum_{y' = -\infty}^y \prod_{z = y'+1}^y \Big(B'(t, z, \occupvr{z}{t}) - B(t, z, \occupvr{z}{t})\Big) B(t, z, \occupvr{z}{t})\\
&=B(t, y, \eta_y (t)) + \Big(B'(t, y, \eta_y (t)) - B'(t, y, \eta_y (t))\Big) \sum_{y' = -\infty}^{y-1} \prod_{z = y'+1}^{y-1}  \Big(B'(t, z, \occupvr{z}{t}) - B(t, z, \occupvr{z}{t})\Big)\\
&=B(t, y, \eta_{y} (t)) + \Big(B'(t, y, \eta_y (t)) - B'(t, y, \eta_y (t))\Big) \Big(N(t, y-1) - N(t, y)\Big).
\end{align*}
In particular,
if
$\vec{\eta}(0) \in \GG$, the $\vec{\eta}(t)$ defined in Lemma \ref{lem:biinfinite} is a left-finite unfused \ac{SHS6V} model. Therefore, Lemma \ref{lem:biinfinite} truly extends the scope of Definition \ref{def:unfused}.
\end{remark}
\begin{proof}[Proof of Lemma \ref{lem:biinfinite}]
Define the canonical filtration 
\begin{align*}
\label{eq:filtration}
\FF(t) = \sigma\Big(\vec{\eta}(0), B(s, z, \eta), B'(s, z, \eta), 0 \leq s \leq t-1\Big). 
\end{align*}
It is not hard to see (via  \eqref{eq:birecur1} and \eqref{eq:birecur2})
that $N(t, y)$ and $\vec{\eta}(t)$ are adapted to this filtration.
\bigskip
\\
Let us first justify the convergence of the infinite summation \eqref{eq:birecur1}. To simplify notation, we denote by $\EE'\big[\cdot\big] = \EE\big[\cdot \vv \FF(t)\big]$. For $x < y \in \ZZ$, denote by 
\begin{align*}
K_{x, y} (t) := \sum_{y' = x}^{y} \prod_{z = y'+1}^y \Big(B'(t, z, \eta_z (t)) - B(t, z, \eta_z(t))\Big) B(t, y', \eta_{y'} (t))
\end{align*}
Observing that $K_{x, y} (t) \in \{0, 1\}$ for all realization of $B, B' \in \{0, 1\}$. Therefore, as $x \to -\infty$, the $L^p$ convergence of $K_{x, y} (t)$  implies the almost sure convergence.  Note that $B, B'$ are independent Bernoulli random variables with  mean given in \eqref{eq:randomenvir}. As a consequence, there exists constant $\delta > 0$ such that 
\begin{equation*}
\PP \big(B'(t, z, \eta)  - B(t, z, \eta) = 0\big) > \delta, \quad \forall\, (t, z, \eta) \in \ZZ_{\geq 0} \times \ZZ \times \iset.
\end{equation*}
Since $|B'(t, z, \eta) - B(t, z, \eta)| \leq 1$,
\begin{equation*}
\EE'\big[\big(B'(t, z, \eta_z (t)) - B(t, z, \eta_z (t))\big)^p\big] \leq 1 - \delta.
\end{equation*}
Furthermore, note that conditioning on $\FF(t)$, $B(t, z, \eta_z (t)), B'(t, z, \eta_z (t))$ are all independent, which yields
\begin{align*}
&\EE'\bigg[\bigg(B(t, y', \eta_{y'} (t)) \prod_{z = y'+1}^y \Big(B'(t, z, \eta_z (t) ) - B(t, z, \eta_z (t))\Big) \bigg)^p \bigg]\\ \numberthis \label{eq:two:temp3} 
&=\EE'\big[B(t, y', \eta_{y'} (t))^p \big] \prod_{z = y'+1}^y \EE'\big[\big(B'(t, z, \eta_z (t) ) - B(t, z, \eta_z (t))\big)^p\big]
\leq (1 - \delta)^{y - y'}.
\end{align*}
Taking expectation on both sides of \eqref{eq:two:temp3}, by tower property,
\begin{equation*}
\bignorm{\bigg(\prod_{z = y'+1}^y \Big(B'(t, z, \eta_z (t) ) - B(t, z, \eta_z (t))\Big) B(t, y', \eta_{y'} (t))\bigg)}_p \leq (1 - \delta)^{\frac{y - y'}{p}},
\end{equation*}
which implies the convergence of $K_{x, y} (t)$ in $\LL^p$ as  $x \to -\infty$.  
\bigskip
\\
We proceed to justify
\begin{equation}\label{eq:two:temp2}
\lim_{w \to -\infty} N^w (t, y) = N(t, y) \quad \text{ in } \LL^p. 
\end{equation}
We prove this by applying induction on $t$. The $t = 0$ case is immediately checked. Assuming that we have a proof for $t = s$, we show that \eqref{eq:two:temp2} also holds for $t = s + 1$. 
Note that for all $y \in \ZZ$,
\begin{equation*}
\eta^w_y (s) = N^w(s, y) - N^w(s, y-1) \to N(s, y) - N(s, y-1) =  \eta_y (s) \text{ in } \LL^p \quad \text{ as } w \to -\infty.
\end{equation*} 
Since both $\eta^w_y(s), \eta_y (s)$ take value in  $\iset$, we obtain
\begin{equation*}
\lim_{w \to -\infty} \PP\big(\eta^w_y (s) = \eta_y (s)\big) = 1. 
\end{equation*}
Taking $w \to -\infty$, one achieves
\begin{align*}
N^w(s, y) - N^w(s+1, y) = \sum_{y' = -\infty}^y \prod_{z = y'+1}^y \bigg(B'(s, z,\eta^w_z(s)) - B(s, z, \eta^w_z(s))\bigg) B(s, z, \eta^w_z(s)).
\end{align*}
Therefore, $\lim_{w \to -\infty} N^w(s, y) - N^w(s+1, y)  = N(s, y) - N(s, y+1)$ in $\LL^p$. Since we have assumed \eqref{eq:two:temp2} for $t = s$, we have
\begin{equation*}
\ N^w (s+1, y) \to N(s+1, y) \quad \text{ in } \LL^p,
\end{equation*} 
which completes the induction.
\end{proof}
\begin{defin}\label{def:biinfinite}
We call the $\vec{\eta}(t) \in \iset^\ZZ$ defined in Lemma \ref{lem:biinfinite} \textbf{the bi-infinite unfused \ac{SHS6V} model} and associate it with the height function $N(t, x)$ defined in Lemma \ref{lem:biinfinite}. We simply define \textbf{the bi-infinite fused \ac{SHS6V} model} $\vec{g}(t) $ and its height function $\N(t, x)$ via $$\vec{g}(t) := \vec{\eta}(Jt),\qquad \N(t, x) := N(Jt, x).$$  
\end{defin}
It is clear that to prove Theorem \ref{thm:main}, it suffices to work with the bi-infinite unfused \ac{SHS6V} model. Unless specified otherwise, the \ac{SHS6V} model now means the bi-infinite unfused \ac{SHS6V} model $\vec{\eta}(t)$. We associate it with the canonical filtration $\FF(t) = \sigma\Big(\vec{\eta}(0), B(s, z, \eta), B'(s, z, \eta), 0 \leq s \leq t-1\Big)$. 
\section{Markov duality}\label{sec:duality}
One main tool that we rely on to prove Theorem \ref{thm:main} is the Markov duality. This is a powerful property which has been found for different interacting particle systems including the contact process, voter model and symmetric simple exclusion process \cite{Lig12, Lig13}. Using Markov duality, Spitzer and Liggett showed that the only extreme translation invariant measures for the SSEP on $\ZZ^d$ are the Bernoulli product measure.
\bigskip
\\
In this section, we first state two Markov dualities for the $J=1$ version of left-finite \ac{SHS6V} model, which comes form \cite[Theorem 2.21]{CP16} and \cite[Theorem 4.10]{kuan18} respectively. The extension of them to the unfused left-finite SHS6V model is immediate since the transition operators of the model are commute. Finally we explain how to extend these dualities to the bi-infinite unfused  \ac{SHS6V} model constructed in the previous section. 
\bigskip\\
Let us recall the definition of Markov duality in the first place. 
\begin{defin}\label{def:duality}
Given two discrete time  Markov processes $X(t) \in U$ and $Y(t) \in V$ (might be time inhomogeneous)  and a function $H: U \times V \to \RR$, we say that $X(t)$ and $Y(t)$ are dual with respect to $H$ if for any $x \in U, y\in V$ and $s \leq t \in \ZZ_{\geq 0}$, we have 
\begin{equation*}
\EE\big[H(X(t), y) \vv X(s) = x\big] = \EE \big[H(x, Y(t)) \vv Y(s) = y\big]. 
\end{equation*}
\end{defin}
The Markov dualities that we are going to present are between the unfused \ac{SHS6V} model and the $k$-particle reversed unfused \ac{SHS6V} model location process. To define the latter process, let us first introduce several state spaces.
\begin{defin}\label{def:statespace}
	Recall the space of left-finite particle configuration $\GG$ from \eqref{eq:leftfinitestate}.
	We likewise define the space of right-finite particle configuration 
	\begin{equation*}
	\MM = \{\vec{m} = (\dots, m_{-1}, m_0, m_1, \dots): \text{ all } m_i \in \iset \text{ and there exists } x\in \ZZ \text{ such that } m_i = 0  \text{ for all } i > x\}.
	\end{equation*} 
	When there are finite number of $k$ particles, we restrict $\GG$ and $\MM$ to 
	\begin{equation*}
	\GG^k = \{\vec{g} \in \GG: \sum_{i} g_i = k\}, \qquad \MM^k = \{\vec{m} \in \MM: \sum_{i} m_i = k\}.
	\end{equation*}
	In terms of particle positions, the spaces $\GG^k$ and $\MM^k$ are in bijection with 
	\begin{equation*}
	\weyl{k} =  \left\{\vec{y} = (y_1 \leq \dots \leq y_k): \vec{y} \in \ZZ^k, \max_{1 \leq i \leq M(\vec{y})} c_i \leq I \right\},
	\end{equation*}
	where $(c_1, \dots, c_{M(\vec{y})})$ denotes the cluster number in $\vec{y}$, i.e. $\vec{y} = (y_1 = \dots = y_{c_1} < y_{c_1 +1 } = \dots = y_{c_1 + c_2} < \dots )$. $(y_1 \leq \dots \leq y_k)$ should be understood as the location of $k$ particles in a non-decreasing order.
	In particular, we denote by $\varphi: \weyl{k} \to \GG^k$
	and $\phi: \weyl{k} \to \MM^k$ to be 
	the bijective maps respectively.
\end{defin}
\begin{defin}\label{def:lfhzr}
	When $J=1$, it is clear that Definition \ref{def:fused} and Definition \ref{def:unfused} define the same Markov process. We call it
	the left-finite $J=1$ \ac{SHS6V} model. 
	In addition, we call $\vec{\xi}(t) = (\xi_{x} (t))_{x \in \ZZ} \in \MM$ the reversed $J=1$  \ac{SHS6V} model if $\vec{\xi}'(t) = (\xi_{-x} (t))_{x \in \ZZ} \in \GG$ is a left-finite $J=1$ \ac{SHS6V} model.
\end{defin}
Since the \ac{SHS6V} model preserves the number of particles, we can consider \ac{SHS6V} model
with $k$ particles as a process on the particle locations. 
\begin{defin}\label{def:three:locationpr}
We define the $k$ particle $J=1$ \ac{SHS6V} model location process $\vec{x}(t) =  \big(x_1 (t) \leq \dots \leq x_k (t)\big) \in \weyl{k}$ if $\varphi({\vec{x}(t)})$ (recall the bijective map $\varphi: \weyl{k} \to \GG^k$ from Definition \ref{def:statespace}) is the $J=1$ left-finite \ac{SHS6V} model. 
We say that $\vec{y}(t) = (y_1 (t) \leq \dots \leq y_k (t)) \in \weyl{k}$ is a $k$-particle reversed $J=1$  \ac{SHS6V} model location process if $-\vec{y}(t) = (-y_{k}(t) \leq \dots \leq -y_1 (t) )$ is a $k$-particle $J=1$ \ac{SHS6V} model location process. In addition, for $\vec{y}, \vec{y}' \in \weyl{k}$, we denote by $\semigr_{\alpha}(\vec{y}, \vec{y}')$ to be the transition probability from $\vec{y}$ to $\vec{y}'$ of the $k$-particle reversed $J=1$ \ac{SHS6V} model location process. As a matter of convention, $\semigr_\alpha$ could be seen as an operator acting on function $f: \weyl{k} \to \RR$ in the manner that
\begin{equation*}
(\semigr_\alpha f)(\vec{y}) := \sum_{\vec{y}' \in \weyl{k}} \semigr_\alpha (\vec{y}, \vec{y}') f(\vec{y}').
\end{equation*} 
\end{defin}
\begin{defin}\label{def:unfusedreverse}
We define the $k$-particle unfused \ac{SHS6V} model location process $\vec{x}(t) = (x_1 (t) \leq \dots \leq x_k (t))$ so that $\varphi(\vec{x}(t))$ is the left-finite unfused \ac{SHS6V} model. We say  $\vec{y}(t) = (y_1 (t) \leq \dots \leq y_k (t))$ is a $k$-particle reversed   unfused \ac{SHS6V} model location process if $-\vec{y}(t) = (-y_{k}(t) \leq \dots \leq -y_1 (t) )$ is a $k$-particle unfused \ac{SHS6V} model location process.
\end{defin}
Note that  
for the reversed $k$-particle \ac{SHS6V} model $\vec{y}(t)$, we denote by $\rhzr\big(\vec{x}, \vec{y}, t, s\big)$ the transition probability from $\vec{y}(s) = \vec{x}$ to $\vec{y}(t) = \vec{y}$. Apparently, one has 
\begin{equation*}
\rhzr\big(\vec{x}, \vec{y}, t, s\big) = \big(\semigrst{\alpha(s)}{\alpha(t-1)}\big)(\vec{x}, \vec{y}).
\end{equation*}
It follows from \cite[Corollary 2.14]{CP16} (or the Yang-Baxter equation \cite[Section 3]{BP18}) that $\semigr_{\alpha(i)}$ commutes with itself for different values of $i$ (i.e. $\semigr_{\alpha(i)} \semigr_{\alpha(j)} = \semigr_{\alpha(j)}  \semigr_{\alpha(i)}$).  
Consequently, 
\begin{equation}\label{eq:three:temp7}
\rhzr\big(\vec{x}, \vec{y}, t, s\big) = \big(\semigrst{\alpha(t-1)}{\alpha(s)}\big)(\vec{x}, \vec{y}).
\end{equation}
Let us first state the $J=1$ version of Markov duality.
\begin{prop}[\cite{CP16}, Proposition 2.21]\label{prop:firstdual}
For all $k \in \NN$, the $J=1$ left-finite \ac{SHS6V} model $\occuppr \in \GG$ (Definition \ref{def:lfhzr}) and $k$-particle $J =1$ reversed   \ac{SHS6V} model location process $\rlocpr$ (Definition \ref{def:three:locationpr}) are dual with respect to the functional $H: \GG \times \YY^k \to \RR$ 
\begin{equation}\label{eq:three:fdual}
H(\occupc, \rlocc) = \prod_{i=1}^k q^{-N_{y_i} (\occupc)},
\end{equation}
recall $N_{y}(\vec{\eta}) = \sum_{i \leq y} \eta_i$.
\end{prop}
In \cite{kuan18}, the author discovers a Markov duality for a multi-species version of the \ac{SHS6V} model.
For our application, we explain how to degenerate this result to a two particle \ac{SHS6V} model duality. Before stating the proposition, let us recall the notation of $q$-deformed quantity
\begin{equation*}
\qint{n} := \frac{q^n - q^{-n}}{q - q^{-1}}, \quad \qint{n}^! := \prod_{i=1}^n \qint{i},\quad \qbinom{n}{k}_q := \frac{\qint{n}^!}{\qint{k}^! \qint{n - k}^!}.
\end{equation*}
\begin{prop}\label{prop:secdual}
The $J =1$ left-finite  \ac{SHS6V} model $\occuppr$ and the two particle $J =1$ reversed   \ac{SHS6V} model location process $\vec{y}(t)$ are dual with respect to   
\begin{equation}\label{eq:three:sdual}
G(\occupc, (y_1, y_2)) = 
\begin{cases}
q^{-2N_{y_1} (\occupc)}    \qhalfint{I- \occupv{y_1}}  \qhalfint{I - 1 - \occupv{y_1}} q^{\occupv{y_1}} \qquad &\text{if}\  y_1 = y_2;\\
\frac{\qhalfint{I-1} }{\qhalfint{I}} q^{-N_{y_1}(\vec{\eta})} q^{-N_{y_2} (\vec{\eta})} \qhalfint{I- \occupv{y_1}}  \qhalfint{I- \occupv{y_2}}     q^{\frac{1}{2} \eta_{y_1}} q^{\frac{1}{2} \eta_{y_2}} \qquad  &\text{if}\  y_1 < y_2. 
\end{cases}
\end{equation}
\end{prop}
We remark that there is a misstatement in \cite[Theorem 4.10]{kuan18}. The particles in the process $\ZZZ$ and  $\ZZZ_{rev}$ were stated to jump to the left and to the right respectively. However, after discussing with the author, we realize that the right statement is that the particles in $\ZZZ$ jump to the right and those in $\ZZZ_{rev}$ jump to the left.
\begin{proof}
This is a degeneration from \cite[Theorem 4.10]{kuan18}. By taking the species number $n=1$, the spin parameter $m_x = I$ for all $x \in \ZZ$ as well as replacing $q$ by $q^{\frac{1}{2}}$, the multi-species \ac{SHS6V} model considered in \cite{kuan18} degenerates to the $J=1$ \ac{SHS6V} model (see Section 2.6.2 of \cite{kuan18} for detail). Then Theorem 4.10 of \cite{kuan18} reduces to: The $J=1$ left-finite \ac{SHS6V} model $\vec{\xi}(t)$ and the $J=1$ reversed  \ac{SHS6V} model $\vec{\eta}(t)$ are dual with respect to the functional 
\begin{equation*}
G_1 (\vec{\xi}, \vec{\eta}) = \prod_{x \in \ZZ} \qhalfint{\eta_x}^! \qhalfint{I - \eta_x}^! \qbinom{I-  \xi_x}{\eta_x}_{q^{\frac{1}{2}}} q^{-\frac{1}{2} \xi_x (\sum_{z > x} 2 \eta_z + \eta_x)}. 
\end{equation*}
Swapping the role of left and right, which makes the particles in $\vec{\xi}(t)$ jump to the left and those in $\vec{\eta}(t)$ jump to the right. Then $\vec{\eta}(t)$ becomes the $J=1$ left-finite \ac{SHS6V} model and $\vec{\xi}(t)$ becomes the $J=1$ reversed  \ac{SHS6V} model. They are dual with respect to the functional
\begin{align*}
G_2(\vec{\eta}, \vec{\xi}) &=  \prod_{x \in \ZZ} \qhalfint{\eta_x}^! \qhalfint{I - \eta_x}^! \qbinom{I-  \xi_x}{\eta_x}_{q^{\frac{1}{2}}} q^{-\frac{1}{2} \xi_x (\sum_{z < x} 2 \eta_z + \eta_x)},\\ 
\numberthis \label{eq:three:temp1}
&=\prod_{x \in \ZZ} \qhalfint{\eta_x}^! \qhalfint{I - \eta_x}^! \qbinom{I-  \xi_x}{\eta_x}_{q^{\frac{1}{2}}} q^{-\xi_x N_x (\vec{\eta}) + \frac{1}{2} \xi_x \eta_x}.
\end{align*}
Assuming $\vec{\xi} (t)$ has two particles, recall the bijective map $\phi: \weyl{2} \to \MM^2$ (take $k=2$) in Definition \ref{def:statespace}, then $\vec{y}(t) = \phi^{-1}(\vec{\xi} (t))$ is the  $J=1$  reversed two particle location process. The $J=1$ left-finite \ac{SHS6V} model $\vec{\eta}(t)$ and the  two particle $J=1$ reversed \ac{SHS6V} model location process $\vec{y}(t) = (y_1 (t) \leq y_2 (t))$ are dual with respect to  $G_2 (\vec{\eta}, \phi^{-1}(y_1, y_2))$, where $\vec{\xi} = (\xi
_x)_{x \in \mathbb{Z}} = \phi(y_1, y_2)$ is given by 
\begin{equation*}
\xi_x =
\begin{cases}
2 \idc_{\{x = y_1\}} &\text{ if } y_1 = y_2, \\
\idc_{\{x = y_1\}} + \idc_{\{x = y_2\}} &\text{ if } y_1 < y_2.
\end{cases}
\end{equation*}
In addition, note that
\begin{equation}\label{eq:three:temp8}
\qhalfint{\eta_x}^! \qhalfint{I - \eta_x}^!\binom{I - \xi_x}{\eta_x}_{q^{\frac{1}{2}}} = 
\begin{cases}
\qhalfint{I} & \text{ if } \xi_x = 0,\\ 
\qhalfint{I - \eta_x} & \text{ if } \xi_x = 1,\\ 
\frac{\qhalfint{I - \eta_x}\qhalfint{I - 1 - \eta_x}}{ \qhalfint{I-1}} & \text{ if } \xi_x = 2.
\end{cases}
\end{equation}
When $\vec{\xi}  = \phi (y_1, y_2)$,  there are at most two values for $x \in \mathbb{Z}$ so that $\xi_x \neq 0$.
To make sense of the infinite product in \eqref{eq:three:temp1}, one needs to normalize $G_2(\vec{\eta}, \vec{\xi})$ by dividing each factor in the product \eqref{eq:three:temp1} by $\qhalfint{I}$. After such normalization, it is straightforward via \eqref{eq:three:temp8} that $G_2 (\vec{\eta}, \phi(y_1, y_2))$ equals the functional $G(\vec{\eta}, (y_1, y_2))$ in \eqref{eq:three:sdual} up to a constant factor.
\end{proof}
We note that the duality functionals in \eqref{eq:three:fdual} and \eqref{eq:three:sdual} do not depend on the parameter $\alpha$. By Markov property and the property that  $\semigr_{\alpha(i)}$ commutes for different value of $i$, it is clear that the same Markov dualities in Proposition \ref{prop:firstdual} and Proposition \ref{prop:secdual} apply for the left-finite unfused \ac{SHS6V} model.
\begin{cor}\label{cor:biduality1}
For all $k \in \ZZ_{\geq 1}$, the left-finite unfused  \ac{SHS6V} model $\vec{\eta}(t) \in \GG$ (Definition \ref{def:unfused}) and the reversed $k$-particle unfused \ac{SHS6V} model location process $\vec{y}(t) \in \weyl{k}$ (Definition \ref{def:unfusedreverse}) are dual with respect to the functional $H$ in \eqref{eq:three:fdual}. The left-finite \ac{SHS6V} model $\vec{\eta}(t)$ and the two particle reversed  unfused \ac{SHS6V} model location process $\vec{y}(t)$ are dual with respect to the functional $G$ in \eqref{eq:three:sdual}.
\end{cor}
\begin{proof}
Due to Proposition \ref{prop:firstdual}, we see that for all $\vec{\eta} \in \GG$ and $\vec{y} \in \weyl{k}$,
\begin{equation}\label{eq:three:temp2}
\EE \big[ H(\vec{\eta}(t), \vec{y}) \vv \vec{\eta}(t-1) = \vec{\eta}\big] = \sum_{\vec{x} \in \weyl{k}} \semigr_{\alpha(t-1)}(\vec{y}, \vec{x}) H(\vec{\eta}, \vec{x}).
\end{equation}
Using Markov property and applying \eqref{eq:three:temp2}  repetitively, we see that 
\begin{align*}
\EE \big[ H(\vec{\eta}(t), \vec{y}) \vv \vec{\eta}(s) = \vec{\eta}\big] 
&= \sum_{\vec{x} \in \weyl{k}} \big(\semigr_{\alpha(s)} \cdots  \semigr_{\alpha(t-1)}\big) (\vec{y}, \vec{x})  H(\vec{\eta}, \vec{x})\\
&= \sum_{\vec{x} \in \weyl{k}} \rhzr\big(\vec{y}, \vec{x}, t, s\big) H(\vec{\eta}, \vec{x}) \\
&= \EE \big[H(\vec{\eta}, \vec{y}(t)) \vv \vec{y}(s) = \vec{y}\big].
\end{align*}
Here, the second equality follows from \eqref{eq:three:temp7}.
This proves the desired duality with respect to the functional $H$. The duality with respect to the functional $G$ follows by a similar argument. 
\end{proof}
For our application, we like to extend the dualities stated in Proposition \ref{prop:firstdual} and Proposition \ref{prop:secdual} to the  bi-infinite \ac{SHS6V} model. Denote by
\begin{equation}\label{eq:dual2}
\du(t, y_1, y_2) =
\begin{cases}
q^{-2 \height{t}{y_1}}    \qhalfint{I- \bivr{t}{y_1}}  \qhalfint{I - 1 -\bivr{t}{y_1}} q^{\bivr{t}{y_1}} \qquad &\text{if}\  y_1 = y_2;\\
\frac{\qhalfint{I-1} }{\qhalfint{I}} q^{-\height{t}{y_1}} q^{-\height{t}{y_2}} \qhalfint{I- \bivr{t}{y_1}}  \qhalfint{I- \bivr{t}{y_2}}     q^{\frac{1}{2} \bivr{t}{y_1}} q^{\frac{1}{2} \bivr{t}{y_2}} \qquad  &\text{if}\  y_1 < y_2. 
\end{cases} 
\end{equation}
Here $\vec{\eta}(t) = (\eta_x (t))_{x \in \ZZ}$ is the bi-infinite unfused \ac{SHS6V} model defined in Definition \ref{def:biinfinite} and $N(t, y)$ is the associated height function.
\begin{cor}\label{cor:biduality}
For the  bi-infinite unfused \ac{SHS6V} model $\occuppr$, for $\vec{y} = (y_1 \leq \dots \leq y_k) \in \weyl{k}$ one has
\begin{equation}\label{eq:fextenddual1}
\EE \big[\prod_{i=1}^k q^{-\height{t}{y_i}} \vv \FF(s) \big] = \sum_{\x \in \weyl{k}} \rhzr \big(\rlocc, \x, t, s\big) \prod_{i=1}^k q^{-\height{s}{x_i}}.
\end{equation}
For $y_1 \leq y_2 \in \ZZ$ (Since $I \geq 2$, this is equivalent to the condition $(y_1, y_2) \in \weyl{2}$)
\begin{equation}
\label{eq:sextenddual1}
\EE\big[\du(t, y_1, y_2)  \vv \FF(s)\big] = \sum_{x_1 \leq x_2 \in \ZZ^2} \rhzr \big((y_1, y_2), (x_1, x_2), t, s\big) \du(s, x_1, x_2).
\end{equation}
\end{cor}
\begin{proof}
Let us prove \eqref{eq:fextenddual1} in the first place. 
Given initial condition of the bi-infinite unfused \ac{SHS6V} model $\vec{\eta}(0)$,
we construct a sequence of left-finite \ac{SHS6V} model $\vec{\eta}^w (t)$  with initial condition $\vec{\eta}^w (0) := (\eta_i (0) \idc_{\{i \geq w\}})_{i \in \ZZ}$. We denote by $N^w (t, y)$ the associated height function. The first duality in Corollary \ref{cor:biduality1} implies that for any $w \in \ZZ$
\begin{equation}\label{eq:three:temp4}
\EE \big[\prod_{i=1}^k q^{-N^w (t, y_i)} \vv \FF(s)\big] = \sum_{\vec{x} \in \weyl{k}} \rhzr \big(\vec{y}, \vec{x}, t, s\big) \prod_{i=1}^k q^{-N^w (s, x_i)}.
\end{equation}
Let us show that the LHS and RHS of \eqref{eq:three:temp4} approximate those of \eqref{eq:fextenddual1} as $w \to -\infty$.
\bigskip
\\
For the approximation of the LHS, as $|\etapart{0}{x}| \leq I$ for all $x \in \ZZ$, we have $|N^w(0, y_i)| \leq I |y_i|$. Moreover, in a single time step, $N^w (t, y_i)$ may change by at most one, hence for all $w \in \ZZ$,
\begin{align*}\numberthis \label{eq:three:temp9}
|N^w (t, y_i)| &\leq |N^w (0, y_i)| +t \leq y_i I + t. 
\end{align*}
Therefore, for fixed $t \in \ZZn$ and $q > 1$, $\prod_{i=1}^k q^{-N^w (t, y_i)}$ is uniformly bounded. Via Lemma \ref{lem:biinfinite}, we know that $N^w (t, y_i) \to N(t, y_i)$ in probability, by conditional dominated convergence theorem, one has
\begin{equation*}
\lim_{w \to -\infty} \EE \big[\prod_{i=1}^k q^{-N^w (t, y_i)} \vv \FF(s)\big] = \EE \big[\prod_{i=1}^k q^{-N (t, y_i)} \vv \FF(s)\big].
\end{equation*}  
For the RHS approximation, according to Definition \ref{def:unfusedreverse}, when there is only one particle in the reversed \ac{SHS6V} model location process, it jumps to the left (at time $t$) as a geometric random variables with parameter $\frac{\nu + \alpha(t)}{1 + \alpha(t)}$. When there are $k$ particles, they jump to the left (at time $t$) as $k$ independent geometric random variables with parameter $\frac{\nu + \alpha(t)}{1 + \alpha(t)}$ except when one hits another. So  there exists constant $C$  such that for all $t, \vec{x}, \vec{y}$ 
\begin{equation*}
\rhzr\big(\vec{y}, \vec{x}, t+1, t\big) \leq C \prod_{i=1}^k \bigg(\frac{\nu + \alpha(t)}{1 + \alpha(t)}\bigg)^{|y_i - x_i|}.
\end{equation*}
Denote by $\theta = \sup_{t \in \ZZn} \frac{\nu + \alpha(t)}{1 + \alpha(t)}$,
one has
\begin{equation}\label{eq:three:temp5}
\rhzr\big(\vec{y}, \vec{x}, t+1, t\big) \leq C \prod_{i=1}^k \theta^{|y_i - x_i|}.
\end{equation} 
For fixed $s \leq t$, observing that $\rhzr\big(\vec{y}, \vec{x}, t, s\big)$ can be written as a $(t-s)$-fold convolution of one-step transition probability. The convolution can be expanded into a sum over all trajectories from $\vec{y} = (y_1, \dots, y_k)$ to $\vec{x} = (x_1, \dots, x_k)$. The contribution of each trajectories can be bounded by the product of $t  - s$ one-step transition probability, which is upper bounded by the RHS of \eqref{eq:three:temp5}. As the particles in the reversed \ac{SHS6V} model can only jump to the left, the number of the trajectories can be upper bounded by $\prod_{i = 1}^k \binom{|x_i - y_i| + t - s}{t - s}$. We obtain 
\begin{align}\label{eq:three:temp6}
\rhzr\big(\vec{y}, \vec{x}, t, s\big) &\leq C \prod_{i=1}^k \binom{|x_i - y_i| + t - s}{t - s} \theta^{|y_i - x_i|}.
\end{align} 
Furthermore, it is readily verified that under Condition \ref{cond:1},
$$q^I  \theta = \sup_{t \in \ZZ_{\geq 0}} \frac{1 + q^I \alpha(t)}{1 + \alpha(t)} < 1.$$ 
Using the bounds in \eqref{eq:three:temp9} and \eqref{eq:three:temp6}, fix $s \leq t \in \ZZ_{\geq 0}$ and $\vec{y} \in \weyl{k}$, we have for all $\vec{x} \in \weyl{k}$,
\begin{align*}
\rhzr \big(\vec{y}, \vec{x}, t, s\big) q^{-N^w (s, x_i)} &\leq C \prod_{i=1}^k \binom{|x_i - y_i| + t - s}{t - s} \theta^{|y_i - x_i|}  q^{I |x_i|},\\
& \leq C \prod_{i=1}^k \binom{|x_i - y_i| + t - s}{t - s} (q^I \theta)^{|y_i - x_i|},\\
& \leq C \prod_{i=1}^k {\delta}^{|y_i - x_i|},
\end{align*}
for some constant $0 < \delta < 1$. Since $N^w (s, x_i) \to N(s, x_i)$ in probability, we find that
\begin{equation*}
\sum_{x \in \weyl{I}} \rhzr \big(\vec{y}, \vec{x}, t, s\big) \prod_{i=1}^k q^{-N^w (s, x_i)} \longrightarrow \sum_{x \in \weyl{k}} \rhzr \big(\vec{y}, \vec{x}, t, s\big) \prod_{i=1}^k q^{-N (s, x_i)}  \quad \text{ in probability}.
\end{equation*}
Therefore, We conclude \eqref{eq:fextenddual1}. The proof of \eqref{eq:sextenddual1} is similar to \eqref{eq:fextenddual1}, where we consider instead
\begin{equation*}
\du^w (t, y_1, y_2) = 
\begin{cases}
q^{-2 N^w(t, y_1)}    \qhalfint{I- \eta^w_{y_1} (t)}  \qhalfint{I - 1 - \eta^w_{y_1} (t)} q^{\eta^w_{y_1}(t)} \qquad &\text{if}\  y_1 = y_2;\\
\frac{\qhalfint{I-1} }{\qhalfint{I}} q^{-N^w(t, y_1)} q^{-N^w(t, y_2)} \qhalfint{I- \eta_{y_1}^w (t)}  \qhalfint{I- \eta_{y_2}^w (t)}     q^{\frac{1}{2} \eta_{y_1}^w (t)} q^{\frac{1}{2} \eta_{y_2}^w (t)} \qquad  &\text{if}\  y_1 < y_2. 
\end{cases}
\end{equation*}
Applying the second duality in Corollary \ref{cor:biduality1}, we find that 
\begin{equation*}
\EE\big[\du^w (t, y_1, y_2) \vv \FF(s)\big] = \sum_{x_1 \leq x_2  \in \ZZ^2} \rhzr\big((y_1, y_2), (x_1, x_2), t, s\big) \du^w (s, x_1, x_2). 
\end{equation*}
By taking $w \to -\infty$ and using similar approximation, we conclude \eqref{eq:sextenddual1}.
\end{proof}
\section{Integral formula for the two particle transition probability}\label{sec:exactformula} 
In this section, we give an explicit integral formula for $\rhzr\big((x_1, x_2), (y_1, y_2), t, s\big)$ (note that for the rest of the paper, we prefer to swap the order of $(x_1, x_2)$ and $(y_1, y_2)$ in the notation compared with the RHS of \eqref{eq:sextenddual1}). Our approach is to utilize the generalized Fourier theory (Bethe ansatz) developed in \cite{BCPS15}. Let us review a few results obtained in \cite{BCPS15} and \cite{CP16} on which we rely  to derive the integral formula.
\begin{defin}\label{def:app:betheansatz}
For $\vec{y} \in (y_1 \leq \dots \leq y_k) \in \ZZ^{k}$, we define the left and right Bethe ansatz eigenfunction\footnote{Comparing with the original definition for Bethe ansatz function defined in (2.11) and (2.14) of \cite{BCPS15}, we reverse the order of components in the vector: We prefer to write $\vec{y} = (y_1 \leq \dots \leq y_k)$ instead of $\vec{y} = (y_1 \geq \dots \geq y_k)$.}
\begin{align*}
\leig(\vec{y}) &= \sum_{\sigma \in S_k} \prod_{1 \leq B < A \leq k} \frac{w_{\sigma(A)} - q w_{\sigma(B)}}{w_{\sigma(A)} - w_{\sigma(B)}} \prod_{i=1}^k \left(\frac{1-w_{\sigma(j)}}{1- \nu w_{\sigma(j)}}\right)^{-y_{k+1-j}},\\
\reig(\vec{y}) &= (-1)^k (1-q)^k q^{\frac{k(k-1)}{2}} \mqv{\vec{y}} \sum_{\sigma \in S_k}  \prod_{1 \leq B < A \leq k} \frac{w_{\sigma(A)} - q^{-1} w_{\sigma(B)}}{w_{\sigma(A)} - w_{\sigma(B)}} \prod_{i=1}^k \left(\frac{1-w_{\sigma(j)}}{1- \nu w_{\sigma(j)}}\right)^{y_{k+1-j}},
\end{align*}
where $S_k$ is the permutation group of $\{1,\dots, k\}$ and  
\begin{equation}\label{eq:mqv}
\mqv{\vec{y}}:= \prod_{i=1}^{M(\vec{y})} \frac{(\nu; q)_{c_i}}{(q; q)_{c_i}},
\end{equation}
where $(c_1, \dots, c_{M(\vec{y})})$ denotes the cluster number in $\vec{y}$, i.e. $\vec{y} = (y_1 = \dots = y_{c_1} < y_{c_1 +1 } = \dots = y_{c_1 + c_2} < \dots )$.
\end{defin}
It turns out that $\leig$ are the eigenfunctions of the operator $ \semigr_\alpha$ 
defined in Definition \ref{def:three:locationpr}. 
\begin{lemma}[Proposition 2.12 of \cite{CP16}]\label{lem:app:eigenfunc}
For all $k \in \NN$ and $\vec{w} = (w_1, \dots, w_k) \in \CC^k$ such that for all $i \in \{1, \dots, k\}$, $\big|\frac{1 - w_i}{1 - \nu w_i} \frac{\alpha + \nu}{1 + \alpha}\big| < 1$. Then,
\begin{equation*}
\big(\semigr_\alpha \Psi_{\vec{w}}^\ell\big)(\vec{y})  = \bigg(\prod_{i=1}^k \frac{1 + \alpha q w_i}{1+ \alpha w_i}\bigg) \leig(\vec{y}).
\end{equation*}
\end{lemma}

\cite{BCPS15} shows that the left and right Bethe ansatz eigenfunctions enjoy the following bi-orthogonal relation. 
\begin{lemma}[Corollary 3.13 of \cite{BCPS15}]\label{lem:plancherel}
For $0 < q, \nu < 1$ and $k \in \NN$  $\vec{x} = (x_1 \leq \dots \leq x_k) \in \ZZ^k$   and $\vec{y} = (y_1 \leq \dots \leq y_k) \in \ZZ^k$,
\begin{align}\label{eq:plancherel}
\sum_{\lambda \vdash k} \oint_{\scontone} \dots \oint_{\scontone} \planm(\vec{w}) \prod_{i=1}^{\ell(\lambda)} \frac{1}{(w_i, q)_{\lambda_j} (\nu w_i, q)_{\lambda_j}} \Psi_{\vec{w} \circ \lambda}^\ell (\vec{x}) \Psi_{\vec{w} \circ \lambda}^r (\vec{y}) =\idc_{\left\{\vec{x} = \vec{y}\right\}} 
\end{align}
Some notations must be specified here:  $\gamma$ is a very small circular contour around $1$ so as to exclude all the poles of the integrand except $1$. The Plancherel measure is defined as 
\begin{equation}\label{eq:def:planmeasure}
\planm(\vec{w}) = \frac{(-1)^k (1-q)^k q^{-k(k-1)/2}}{m_1 ! m_2 ! \dots } \text{det} \left[\frac{1}{w_i q^{\lambda_i} - w_i}\right]_{i,j=1}^{\ell(\lambda)} \prod_{i=1}^{k} q^{\lambda_i (\lambda_i-1)/2} w_i^{\lambda_i} \frac{dw_i}{2\pi \im},
\end{equation}
where the sum in \eqref{eq:def:planmeasure} is taken over the partition $\lambda$ of $k$, that is to say, $\lambda = (\lambda_1 \geq \dots \geq \lambda_{s}) \in \ZZ_{\geq 1}^{s}$ with $\sum_{i=1}^{s} \lambda_i = k$ and $\ell(\lambda) = s$ is the \emph{length} of the partition $\lambda$. For instance, the partitions of $k = 3$ are given by $(2, 1)$ and $(1, 1, 1)$. We denote by $m_j$ to be number of components that equal $j$ in $\lambda$ so that $\lambda = 1^{m_1} 2^{m_2} \dots$. Furthermore, we define
\begin{equation*}
\vec{w} \circ \lambda := (w_1, \dots, q^{\lambda_1 - 1} w_1, w_2, \dots, q^{\lambda_2 - 1} w_2, \dots, w_{s}, \dots, q^{\lambda_s -1} w_{s}). 
\end{equation*}
\end{lemma}
We are in a position to present our formula.
\begin{theorem}\label{thm:intformulap}
Assume $I \geq 2$, for any $x_1 \leq x_2 \in \ZZ$ and $y_1 \leq y_2 \in \ZZ$, the two point transition probability of reversed \ac{SHS6V} model admits the following integral formula
\begin{align*}
&\rhzr \big((x_1, x_2), (y_1, y_2), t, s\big)\\  
&= c(y_1, y_2) \bigg[ \oint_{\lc} \oint_{\lc} \prod_{i=1}^2 \coret(z_i)^{\floor{\frac{t-s}{J}}} \remt(z_i, t, s) z_i^{x_i - y_i } \frac{dz_i}{2\pi \im z_i} -    \oint_{\lc} \oint_{\lc} \interactt(z_1, z_2) \prod_{i=1}^2\coret(z_i)^{\floor{\frac{t-s}{J}}} \remt(z_i, t, s) z_i^{x_{3-i} - y_i } \frac{dz_i}{2\pi \im z_i}\\
\numberthis \label{eq:intformulap}
&+ \res_{z_1 = \polt(z_2)} \oint_{\lc} \oint_{\lc} \interactt(z_1, z_2) \prod_{i=1}^2 \coret(z_i)^{\floor{\frac{t-s}{J}}} \remt(z_i, t, s) z_i^{x_{3-i} - y_i } \frac{dz_i}{2\pi \im z_i}\bigg],
\end{align*}
where $\lc$ is a circle centered at zero with large enough radius $R$ so as to include all the poles of all the integrands. In addition,
\begin{align*}
\numberthis \label{eq:cy1}
c(y_1, y_2) &:= \idc_{\{y_1 < y_2\}} + \frac{1 - q\nu}{(1+q)(1-\nu)} \idc_{\{y_1 = y_2\}},\\
\coret(z) &:= \frac{(1 + \alpha q^J) z - (\nu + \alpha q^J)}{(1 + \alpha) z - (\nu + \alpha) },\\
\remt(z, t, s) &:= \prod_{k= s +J \floor{\frac{t-s}{J}}}^{t-1} \frac{(1 + \alpha(k) q)z - (\nu + \alpha(k) q) }{(1 + \alpha(k)) z - (\nu + \alpha(k)) },\\
\interactt(z_1 , z_2) &:= \frac{q\nu - \nu + (\nu-q) z_2 + (1-q\nu)z_1 + (q-1)z_1 z_2}{q\nu - \nu + (\nu-q) z_1 + (1-q\nu)z_2 + (q-1)z_1 z_2},\\
\polt (z) &:= \frac{(1-q \nu) z - \nu (1-q)}{(q-\nu)  + (1-q)  z}.
\end{align*}
Note that $z_1 = \polt(z_2)$ corresponds to the pole produced by the denominator of $\interactt(z_1, z_2)$ and 
$$\res_{z_1 = \polt(z_2)} \oint_{\lc} \oint_{\lc} \interactt(z_1, z_2) \prod_{i=1}^2 \coret(z_i, t, s) z_i^{x_{3-i} - y_i } \frac{dz_i}{2\pi \im z_i}$$ denotes the residue of the double contour integral above at the pole $z_1 = \polt(z_2)$.
\end{theorem}
\begin{proof}[Proof of Theorem \ref{thm:intformulap}]
The first step to prove Theorem \ref{thm:intformulap} is utilizing the bi-orthogonality of the Bethe ansatz function.
Taking $k=2$ in the previous lemma, since the possible partition is either $\lambda = (1, 1)$ or $\lambda = (2)$, we obtain 
\begin{align*}
\idc_{\left\{(x_1, x_2) = (y_1, y_2)\right\}} =& \oint_{\scontone}  \oint_{\scontone} \planmm{(1,1)}(w_1, w_2) \prod_{i=1}^{2} \frac{1}{(1-w_i) (1-\nu w_i)} \Psi_{(w_1, w_2)}^\ell (x_1, x_2) \Psi_{(w_1, w_2)}^r (y_1, y_2) \\  \numberthis \label{eq:plidentityk2}
&+\oint_{\scontone}  \planmm{(2)}(w) \frac{1}{(w, q)_{2} (\nu w, q)_{2}} \Psi_{(w, q w)}^\ell (x_1, x_2) \Psi_{(w, q w)}^r (y_1, y_2).  
\end{align*}
Note that according to the previous lemma, \eqref{eq:plidentityk2} holds only for $0 < q, \nu <1$, we want to extend this identity to $q > 1$ and $\nu = q^{-I}$. This extension can be justified by analytic continuation. Note that  the RHS of \eqref{eq:plidentityk2} is an analytic function of $q, \nu$ in a suitable domain which connects $\{(q, \nu): (q, \nu) \in (0, 1)^2\}$ and $\{(q, \nu): q>1, \nu = q^{-I}\}$. The reason behind is that after plugging in $\nu = q^{-I}$, there is no new pole of integrand generated inside $\gamma$ (Here we use the assumption $I \geq 2$, this analytic continuation argument is not valid when $I=1$, see Remark \ref{rmk:ana}).
%
\bigskip
\\
Let us now fix $y_1 \leq y_2 \in \ZZ$ on both side of \eqref{eq:plidentityk2} and treat both sides as functions of $(x_1, x_2)$. We denote  by the operator $$\semigr_\alpha (s, t):= \semigr_\alpha (s) \cdots  \semigr_\alpha (t-1).$$ Applying the operator $\semigr_\alpha (s, t)$ on both side of \eqref{eq:plidentityk2}. For the LHS, it is clear that  $$\big(\semigr_\alpha (s, t) 1_{\{\cdot = (y_1, y_2)\}}\big)(x_1, x_2) = \rhzr\big((x_1, x_2), (y_1, y_2), t, s\big).$$ 
For the RHS, we move $\semigr_\alpha (s, t)$ inside the integrand, which yields
\begin{align*}
\rhzr\big((x_1, x_2), (y_1, y_2), t, s\big) = &\oint_{\scontone} \oint_{\scontone} \planmm{(1,1)}(w_1, w_2) \prod_{i=1}^{2} \frac{1}{(1-w_i) (1-\nu w_i)} \big(\semigr_\alpha (s, t) \Psi_{(w_1, w_2)}^\ell\big) (x_1, x_2) \Psi_{(w_1, w_2)}^r (y_1, y_2) \\ 
\numberthis \label{eq:intp}
&+\oint_{\scontone} \planmm{(2)}(w) \frac{1}{(w, q)_{2} (\nu w, q)_{2}} \big(\big(\semigr_\alpha (s, t) \Psi_{(w, qw)}^\ell\big) (x_1, x_2) \Psi_{(w, qw)}^r (y_1, y_2).
\end{align*}
Due to Lemma \ref{lem:app:eigenfunc} (note that $\gamma$ is a small circle around $1$, hence $w_1, w_2$ satisfy the condition of Lemma \ref{lem:app:eigenfunc}),
\begin{align*}
\big(\semigr_\alpha (s, t) \Psi_{(w_1, w_2)}^\ell\big)(x_1, x_2) &= \prod_{i=1}^2 \bigg(\prod_{k=s}^{t-1} \frac{1 + \alpha(k) q w_i}{1 + \alpha(k) w_i}\bigg) \Psi_{(w_1, w_2)}^\ell (x_1, x_2),\\
\big(\semigr_\alpha (s, t) \Psi_{(w, qw)}^\ell\big)(x_1, x_2) &=  \prod_{k=s}^{t-1} \bigg(\frac{1 + \alpha(k) q w}{1 + \alpha(k) w} \cdot \frac{1 + \alpha(k) q^2 w}{1 + \alpha(k) q w}\bigg) \Psi_{(w_1, w_2)}^\ell (x_1, x_2),\\
& = \prod_{k=s}^{t-1} \bigg(\frac{1 + \alpha(k) q^2 w}{1 + \alpha(k) w}\bigg) \Psi_{(w_1, w_2)}^\ell (x_1, x_2).
\end{align*}
We name the first term on the RHS of \eqref{eq:intp} $I_1$ and the second term $I_2$,
\begin{align*}
\numberthis \label{eq:i1}
I_1 &= \oint_{\scontone} \oint_{\scontone} \planmm{(1,1)}(w_1, w_2) \prod_{i=1}^{2} \frac{1}{(1-w_i) (1-\nu w_i)}   \bigg(\prod_{k=s}^{t-1} \frac{1 + \alpha(k) q w_i}{1 + \alpha(k) w_i}\bigg)  \Psi_{(w_1, w_2)}^\ell (x_1, x_2) \Psi_{(w_1, w_2)}^r (y_1, y_2),\\
\numberthis \label{eq:i2}
I_2 &= \oint_{\scontone} \planmm{(2)}(w) \frac{1}{(w, q)_{2} (\nu w, q)_{2}} \prod_{k = s}^{t-1} \bigg(\frac{1 + \alpha(k) q^2 w}{1 + \alpha(k) w}\bigg)  \Psi_{(w, qw)}^\ell (x_1, x_2) \Psi_{(w, qw)}^r (y_1, y_2).
\end{align*}
We compute $I_1$ in the first place. In the integrand of \eqref{eq:i1}, the function $\Psi_{(w_1, w_2)}^\ell (x_1, x_2)$
is a symmetrization of $$\frac{w_2 - q w_1}{w_2 - w_1} \prod_{i=1}^2 \bigg(\frac{1 - w_i}{1 - \nu w_i}\bigg)^{-x_{3-i}}$$ 
Furthermore, all other terms of the integrand \eqref{eq:i1} are symmetric function of $w_1, w_2$. In addition,  we are integrating $w_1, w_2$ along the same contour, this allows us to desymmetrize the integrand 
\begin{align}\label{eq:i11}
I_1 = 2 \oint_{\scontone} \oint_{\scontone} \planmm{(1,1)}(w_1, w_2) \prod_{i=1}^{2} \bigg(\frac{1}{(1-w_i) (1-\nu w_i)}   \prod_{k=s}^{t-1} \frac{1 + \alpha(k) q w_i}{1 + \alpha(k) w_i}\bigg)  \frac{w_2 - q w_1}{w_2 - w_1} \prod_{i=1}^2 \bigg(\frac{1 - w_i}{1 - \nu w_i}\bigg)^{-x_{3-i}} \Psi_{(w_1, w_2)}^r (y_1, y_2).
\end{align}
We readily calculate
\begin{align*}\numberthis \label{eq:i1a}
dm_{(1, 1)}^q (w_1, w_2) &=\frac{ (1 - q)^2 q^{-1}}{2} \text{det} \left[\frac{1}{w_i q - w_j}\right]_{i,j=1}^{2} \prod_{i=1}^2  \frac{w_i dw_i}{2\pi \im} = \frac{(w_1 - w_2)^2}{2(w_2 - q w_1) (q w_2  - w_1)} \prod_{i=1}^2 \frac{d w_i}{2 \pi \im}\\
\Psi_{\vec{w}}^r (y_1, y_2) &= q (1-q)^2 \mqv{y} \sum_{\sigma \in S_2}  \prod_{1 \leq B < A \leq 2} \frac{w_{\sigma(A)} - q^{-1} w_{\sigma(B)}}{w_{\sigma(A)} - w_{\sigma(B)}} \prod_{i=1}^2 \bigg(\frac{1-w_{\sigma(i)}}{1- \nu w_{\sigma(i)}}\bigg)^{y_{3-i}}\\
\numberthis  \label{eq:i1b}
&= (1-q)^2 \mqv{y} \bigg( \frac{q w_2 -  w_1}{w_2 - w_1} \prod_{i=1}^2  \bigg(\frac{1-w_{i}}{1- \nu w_{i}}\bigg)^{y_{3-i}} + \frac{q w_1 - w_2}{w_1 - w_2} \prod_{i=1}^2 \bigg(\frac{1 -  w_i}{1 - \nu w_i}\bigg)^{y_i} \bigg)  
\end{align*}
Replacing the terms $dm_{(1, 1)}^q (w_1, w_2)$ and $\Psi_{\vec{w}}^r (y_1, y_2)$ in the integrand of \eqref{eq:i11} by the RHS of \eqref{eq:i1a} and \eqref{eq:i1b}, one sees that 
\begin{align*}
I_1 = &(1-q)^2 \mqv{y_1, y_2} \bigg[  \oint_{\scontone} \oint_{\scontone} \prod_{i=1}^2 \frac{1}{(1-w_i) (1-\nu w_i)}   \bigg(\prod_{k=s}^{t-1} \frac{1 + \alpha(k) q w_i}{1 + \alpha(k) w_i}\bigg) \bigg(\frac{1-w_i}{1 - \nu w_i}\bigg)^{y_{3- i} - x_{3-i}} \frac{dw_i}{2 \pi \im } \\ 
&-\oint_{\scontone}  \oint_{\scontone} \frac{q w_1 - w_2}{q w_2 - w_1} \prod_{i=1}^2 \frac{1}{(1-w_i) (1-\nu w_i)}   \bigg(\prod_{k=s}^{t-1} \frac{1 + \alpha(k) q w_i}{1 + \alpha(k) w_i}\bigg)\bigg(\frac{1-w_i}{1 - \nu w_i}\bigg)^{y_{i} - x_{3-i}} \frac{dw_i}{2 \pi \im }\bigg],\\
= &(1-q)^2 \mqv{y_1, y_2} \bigg[  \oint_{\scontone} \oint_{\scontone} \prod_{i=1}^2 \frac{1}{(1-w_i) (1-\nu w_i)}   \bigg(\prod_{k=s}^{t-1} \frac{1 + \alpha(k) q w_i}{1 + \alpha(k) w_i}\bigg) \bigg(\frac{1-w_i}{1 - \nu w_i}\bigg)^{y_{i} - x_{i}} \frac{dw_i}{2 \pi \im }\\ 
\numberthis \label{eq:eight:temp1}
&-\oint_{\scontone}  \oint_{\scontone} \frac{q w_1 - w_2}{q w_2 - w_1} \prod_{i=1}^2 \frac{1}{(1- w_i) (1 - \nu w_i)} \bigg(\prod_{k=s}^{t-1} \frac{1 + \alpha(k) q w_i}{1 + \alpha(k) w_i}\bigg) \bigg(\frac{1-w_i}{1 - \nu w_i}\bigg)^{y_{i} - x_{3-i}} \frac{dw_i}{2 \pi \im }\bigg].
\end{align*}
For the second equality above, we changed $\big(\frac{1 - w_i}{1 - \nu w_i}\big)^{y_{3-i} - x_{3-i}}$ to $\big(\frac{1 - w_i}{1 - \nu w_i}\big)^{y_{i} - x_{i}}$, due to the symmetry of $w_1, w_2$.  
\bigskip
\\
We proceed to compute $I_2$, by a straightforward calculation
\begin{align*}
&m_{(2)}^q (w) = \frac{(q-1)w}{q+1} \frac{dw}{2\pi i}, \qquad
\Psi_{w, qw}^\ell (x_1, x_2) = (1+q)  \left(\frac{1 - w}{1 - \nu w}\right)^{-x_1}  \left(\frac{1 - q w}{1 - \nu q w}\right)^{-x_2},  \\
&\Psi_{w, qw}^r (y_1, y_2) = (1-q)^2 \mqv{y} (1+q)    \left(\frac{1 - w}{1 - \nu w}\right)^{y_2} \left(\frac{1 - q w}{1 - q \nu w}\right)^{y_1}.
\end{align*}
Inserting these expressions into the integrand of \eqref{eq:i2} gives 
\begin{equation*}
I_2 = (1-q)^2 \mqv{y_1, y_2}   \oint_{\scontone} \frac{(q^2-1)w}{(w, q)_{2} (\nu w, q)_{2}} \prod_{k = s}^{t-1} \bigg(\frac{1 + \alpha(k) q^2 w}{1 + \alpha(k) w}\bigg) \bigg(\frac{1 - w}{1 - \nu w}\bigg)^{y_2 - x_1} \bigg(\frac{1 - q w}{1 - q\nu w}\bigg)^{y_1 - x_2} \frac{dw}{2 \pi \im }.
\end{equation*}
A crucial observation is that one can verify directly
\begin{equation}\label{eq:eight:temp2}
I_2  = - (1-q)^2 \mqv{y_1, y_2} \res_{w_1 = q w_2} \oint_{\scontone}  \oint_{\scontone} \frac{q w_1 - w_2}{q w_2 - w_1} \prod_{i=1}^2 \frac{1}{(1- w_i) (1 - \nu w_i)} \bigg(\prod_{k=s}^{t-1} \frac{1 + \alpha(k) q w_i}{1 + \alpha(k) w_i}\bigg) \bigg(\frac{1-w_i}{1 - \nu w_i}\bigg)^{y_{i} - x_{3-i}} \frac{dw_i}{2 \pi \im },
\end{equation}
Note that $\rhzr\big((x_1, x_2), (y_1, y_2), t, s\big) = I_1 + I_2$, using \eqref{eq:eight:temp1} and \eqref{eq:eight:temp2} one has
\begin{align*}
&\rhzr\big((x_1, x_2), (y_1, y_2), t, s\big) \\ 
&=(1-q)^2 \mqv{y_1, y_2} \bigg[  \oint_{\scontone} \oint_{\scontone} \prod_{i=1}^2 \frac{1}{(1-w_i) (1-\nu w_i)}   \bigg(\prod_{k=s}^{t-1} \frac{1 + \alpha(k) q w_i}{1 + \alpha(k) w_i}\bigg) \bigg(\frac{1-w_i}{1 - \nu w_i}\bigg)^{y_{i} - x_{i}} \frac{dw_i}{2 \pi \im }\\ 
&\quad -\oint_{\scontone}  \oint_{\scontone} \frac{q w_1 - w_2}{q w_2 - w_1} \prod_{i=1}^2 \frac{1}{(1- w_i) (1 - \nu w_i)} \bigg(\prod_{k=s}^{t-1} \frac{1 + \alpha(k) q w_i}{1 + \alpha(k) w_i}\bigg) \bigg(\frac{1-w_i}{1 - \nu w_i}\bigg)^{y_{i} - x_{3-i}} \frac{dw_i}{2 \pi \im } \\ 
&\quad -\res_{w_1 = q w_2} \oint_{\scontone}  \oint_{\scontone} \frac{q w_1 - w_2}{q w_2 - w_1} \prod_{i=1}^2 \frac{1}{(1- w_i) (1 - \nu w_i)} \bigg(\prod_{k=s}^{t-1} \frac{1 + \alpha(k) q w_i}{1 + \alpha(k) w_i}\bigg) \bigg(\frac{1-w_i}{1 - \nu w_i}\bigg)^{y_{i} - x_{3-i}} \frac{dw_i}{2 \pi \im } \bigg].
\end{align*}
Recall that $\alpha(k) = \alpha q^{\mod(k)}$ for all $k$, we can simplify the telescoping product in the integrand via 
\begin{equation*}
\prod_{k=s}^{t-1} \frac{1 + \alpha(k) q w_i}{1 + \alpha(k) w_i} = \bigg(\frac{1 + \alpha q^J w_i}{1 + \alpha w_i}\bigg)^{\floor{\frac{t-s}{J}}}  \prod_{k = s + J \floor{\frac{t-s}{J}}}^{t-1} \frac{1 + \alpha(k) q w_i}{1 + \alpha(k) w_i}.
\end{equation*}
Furthermore, referring to the expression \eqref{eq:mqv} and \eqref{eq:cy1}, we notice that $(1-q)^2 \mqv{y_1, y_2} = c(y_1, y_2)$. Thereby, 
\begin{align*}
&\rhzr\big((x_1, x_2), (y_1, y_2), t, s\big) \\ 
&= c(y_1, y_2) \bigg[  \oint_{\scontone} \oint_{\scontone} \prod_{i=1}^2 \frac{1}{(1-w_i) (1-\nu w_i)}   \bigg(\frac{1 + \alpha q^J w_i}{1 + \alpha w_i}\bigg)^{\floor{\frac{t-s}{J}}}  \bigg(\prod_{k = s + J \floor{\frac{t-s}{J}}}^{t-1} \frac{1 + \alpha(k) q w_i}{1 + \alpha(k) w_i}\bigg) \bigg(\frac{1-w_i}{1 - \nu w_i}\bigg)^{y_{i} - x_{i}} \frac{dw_i}{2 \pi \im }\\ 
&\quad -\oint_{\scontone}  \oint_{\scontone} \frac{q w_1 - w_2}{q w_2 - w_1} \prod_{i=1}^2 \frac{1}{(1- w_i) (1 - \nu w_i)} \bigg(\frac{1 + \alpha q^J w_i}{1 + \alpha w_i}\bigg)^{\floor{\frac{t-s}{J}}}  \bigg(\prod_{k = s + J \floor{\frac{t-s}{J}}}^{t-1} \frac{1 + \alpha(k) q w_i}{1 + \alpha(k) w_i}\bigg) \bigg(\frac{1-w_i}{1 - \nu w_i}\bigg)^{y_{i} - x_{3-i}} \frac{dw_i}{2 \pi \im } \\ 
\numberthis \label{eq:eight:temp3}
&\quad -\res_{w_1 = q w_2} \oint_{\scontone}  \oint_{\scontone} \frac{q w_1 - w_2}{q w_2 - w_1} \prod_{i=1}^2 \frac{1}{(1- w_i) (1 - \nu w_i)} \bigg(\frac{1 + \alpha q^J w_i}{1 + \alpha w_i}\bigg)^{\floor{\frac{t-s}{J}}}  \bigg(\prod_{k = s + J \floor{\frac{t-s}{J}}}^{t-1} \frac{1 + \alpha(k) q w_i}{1 + \alpha(k) w_i}\bigg) \bigg(\frac{1-w_i}{1 - \nu w_i}\bigg)^{y_{i} - x_{3-i}} \frac{dw_i}{2 \pi \im } \bigg].
\end{align*}
Lastly, we transform the small circle $\gamma$ surrounding $1$  into the big circle $\lc$ via a change of variable 
\[w_i = \Xi(z_i) =  \frac{1 - z_i}{\nu - z_i} \quad (\text{equivalently } z_i = \frac{1 - \nu w_i}{1 - w_i}), \qquad i=1,2.\] 
By the following relations
\begin{align*}
&\frac{q \Xi(z_1)  - \Xi(z_2)}{q \Xi(z_2) - \Xi(z_1)} = \interactt(z_1, z_2), \hspace{10em} \frac{1 - \Xi(z_i)}{1 - \nu \Xi(z_i)} = z_i^{-1},
\\ 
&\frac{1 + \alpha q^J \Xi(z_i)}{1 + \alpha \Xi(z_i)} = \coret(z_i), 
\hspace{10em} \prod_{k=s + J \floor{\frac{t-s}{J}}}^{t-1} \frac{1 + \alpha(k) q  \Xi(z_i)}{1 + \alpha(k)  \Xi(z_i)} = \remt(z_i, t, s),
\\
&\frac{d \Xi (z_i)}{(1 - \Xi (z_i)) (1 - \nu \Xi (z_i))} = \frac{d z_i}{(1 - \nu) z_i},
\end{align*}
we obtain 
\begin{align*}
&\rhzr \big((x_1, x_2), (y_1, y_2), t, s\big)\\ &=  
c(y_1, y_2) \bigg[ \oint_{\lc} \oint_{\lc} \prod_{i=1}^2 \coret(z_i)^{\floor{\frac{t-s}{J}}} \remt(z_i, t, s) z_i^{x_i - y_i } \frac{dz_i}{2\pi \im z_i} -  \oint_{\lc} \oint_{\lc} \interactt(z_1, z_2) \prod_{i=1}^2\coret(z_i)^{\floor{\frac{t-s}{J}}} \remt(z_i, t, s) z_i^{x_{3-i} - y_i } \frac{dz_i}{2\pi \im z_i}\\ 
\numberthis \label{eq:temp8}
&\quad + 
\res_{z_1 = \polt(z_2)} \oint_{\lc} \oint_{\lc} \interactt(z_1, z_2) \prod_{i=1}^2 \coret(z_i)^{\floor{\frac{t-s}{J}}} \remt(z_i, t, s) z_i^{x_{3-i} - y_i } \frac{dz_i}{2\pi \im z_i}\bigg].
\end{align*}
This concludes the proof of Theorem \ref{thm:intformulap}. Note that we change the sign in front of the residue from \eqref{eq:eight:temp3} to \eqref{eq:temp8}.
This is due to the fact that, before employing the change of variable, the set of the poles $\{q w_1: w_1 \in \gamma\}$ lies outside the $w_2$-contour $\gamma$, while after the change of variable, the set of the pole $\{\polt(z_1): z_1 \in \C_R\}$ lies inside the $z_2$-contour $\C_R$, since $R$ is chosen to be sufficiently large.
\end{proof}
\begin{remark}\label{rmk:ana}
We remark that our argument in proving that  \eqref{eq:plidentityk2} holds for $q > 1$ and $\nu = q^{-I}$ does not work when $I= 1$. The reason is as follows: Note that the factor  $\frac{1}{(\nu z_1, q)_{2}}$
in the integrand of \eqref{eq:plidentityk2} gives a pole for the $z_1$-contour at  $z_1 = \nu^{-1} q$. Before the substitution of $\nu = q^{-1}$, this pole lies outside the contour $\gamma$. Yet after substituting $\nu = q^{-1}$, the pole becomes $z_1 = 1$, which runs inside the contour $\gamma$, hence the argument of analytic continuation fails. This issue is also addressed in \cite{BCPS19}, when the authors tried to reproduce the integral formula for the $k$ particle ASEP transition probability (which first appears in \cite[Theorem 2.1]{TW08}) via analytic continuation of \eqref{eq:plancherel}. For a similar reason, our method does not yield the general $k$ particle transition probability formula of the \ac{SHS6V} model.
\end{remark}
\section{Microscopic Hopf-Cole transform and SHE}\label{sec:microshe}
In this section, we first define the microscopic Hopf-Cole transform $Z(t, x)$, which is an exponential transform of the height function $N(t, x)$. Using $k =1$ version of duality of \eqref{eq:fextenddual1}, it turns out that $Z(t, x)$ satisfies a discrete version of \ac{SHE}. As the Hopf-Cole solution to the \ac{KPZ} equation is the logarithm of the mild solution of the \ac{SHE}, this reduces the proof of Theorem \ref{thm:main} to proving that $Z(t, x)$ converges to the solution of \ac{SHE}. We will derive two Markov dualities for $Z(t, x)$ in Lemma \ref{lem:dualt}, as a tilted version of \eqref{eq:fextenddual1}. This will be used in the proof of self-averaging property Proposition \ref{prop:timedecorr}. 
\subsection{Microscopic Hopf-Cole Transform}\label{sec:microhf}
We first study a one particle version of the unfused \ac{SHS6V} model location process (Definition \ref{def:unfusedreverse}). When there is only one particle,
it performs a random walk $X'(t) = \sum_{k=0}^{t-1} R'(k)$ where $R'(k)$ are independent (but not same distributed)  $\ZZ_{\geq 0}$-valued random variables with distribution 
\begin{equation*}\label{eq:originrw}
\PP\big(R'(k) =  n\big) = 
\begin{cases}
\frac{1 + q \alphat{\alpha}{k}}{1 + \alphat{\alpha}{k}} \qquad &\text{if } n=0;\\
\frac{\alpha(k)(1-q)}{1 + \alphat{\alpha}{k}} \big(1- \frac{\nu + \alphat{\alpha}{k}}{1 + \alphat{\alpha}{k}}\big) \big(\frac{\nu +\alphat{\alpha}{k}}{1 + \alphat{\alpha}{k}}\big)^{n-1} \qquad &\text{if } n \in \ZZ  \\
0 \qquad &\text{else}.
\end{cases}
\end{equation*}
By tilting and centering $R'(k)$  with respect to $\EE\big[q^{\rho R'(k)} \idc_{\left\{R'(k) = \cdot \right\}}\big]$, we define a tilted random walk $X(t) = \sum_{k=0}^{t-1} R(k)$, where $R(k)$ are independent $\ZZ_{\geq 0} - \mut{\mu}{k}$ valued with distribution\footnote{The tilted and centered random walk $X(t)$ provides the heat kernel $\mathsf{p}(t+1, t)$ for the discrete SHE \eqref{eq:microshe} satisfied by the microscopic Hopf-Cole transform \eqref{eq:microshe1}, which is an exponential transform of the LHS of \eqref{convergence}.}
\begin{equation}\label{eq:rwdistribution}
\PP\big(R(k) = n -\mut{\mu}{k} \big) = 
\begin{cases}
\lambdat{\lambda}{k} \frac{1 + q \alphat{\alpha}{k}}{1 + \alphat{\alpha}{k}} \qquad &\text{if } n=0;\\
\lambdat{\lambda}{k} \frac{\alpha(k)(1-q)}{1 + \alphat{\alpha}{k}} \big(1- \frac{\nu + \alphat{\alpha}{k}}{1 + \alphat{\alpha}{k}}\big) \big(\frac{\nu +\alphat{\alpha}{k}}{1 + \alphat{\alpha}{k}}\big)^{n-1}  q^{\rho n} \qquad &\text{if } n \in \NN \\
0 & \qquad \text{else}.
\end{cases}
\end{equation}
Here, $\lambdat{\lambda}{k} = \big(\EE\big[q^{\rho R(k)}\big]\big)^{-1}$ is the normalizing parameter and $\mut{\mu}{k}$ is the centering parameter which makes $\EE\big[R(k)] = 0$. Under straightforward calculation, we see that
\begin{align*}
\numberthis
\label{eq:lamexpan}
\lambdat{\lambda}{k} &= \frac{1 + \alphat{\alpha}{k}  - q^\rho (\alphat{\alpha}{k} + \nu)}{1 + a(k) q  - q^\rho (\alphat{\alpha}{k} q + \nu)}, \\
\numberthis \label{eq:muxpan}
\mut{\mu}{k} &= \frac{\alphat{\alpha}{k} (1 - q) (1 - \nu) q^\rho }{(1 + \alphat{\alpha}{k} q - q^\rho (\alphat{\alpha}{k} q + \nu)) (1 + \alphat{\alpha}{k} - q^\rho (\alphat{\alpha}{k} + \nu))}.
\end{align*}
We remark that $\lambda(k)$ (respectively $\mu(k)$) are $J$ periodic in the sense that $\lambda(k) = \lambda(J+k)$ (respectively $\mu(k) = \mu(J+k)$).
Denote by 
\begin{equation}\label{eq:temp11}
\hat{\lambda}(t) := \prod_{k=0}^{t-1} \lambda(k), \qquad \hat{\mu} (t) := \sum_{k=0}^{t-1} \mut{\mu}{k},  \qquad \Xi(t, s) := \ZZ - \hmu(t) + \hmu(s),\qquad  \Xi(t):= \Xi(t, 0).
\end{equation}
It can be verified that the parameter $\lambda$, $\mu$ defined in \eqref{eq:temp1} satisfies 
$$\lambda = \hlambda(J),\qquad \mu = \hmu(J),$$
hence, one has 
\begin{equation}\label{eq:four:temp5}
\hlambda(Jt) = \lambda^t, \qquad \hmu(Jt) = \mu t.
\end{equation}
We define the \emph{microscopic Hopf-Cole transform} for $x \in \Xi(t)$ as
\begin{equation}\label{eq:microshe1}
Z(t, x) := \hlambda(t) q^{- (N(t, x+ \hmu(t)) - \rho(x+\hmu(t))) }.
\end{equation}
For $x \in \Xi(t, s)$, we set $\tonepart (t, s, x) := \PP\big(X(t) - X(s) = x\big)$. Denote by the convolution 
$$
(\tonepart(t, s) * f(s)) (x) := \sum_{y \in \Xi(s)} \tonepart(t, s, x-y) f(s, y).
$$
We set
\begin{equation*}
K(t, x) := N(t, x) - N(t+1, x), \qquad \overline{K}(t, x) :=  K(t, x) - \EE\big[K(t, x) \vv \FF(t)\big].
\end{equation*}
We sometimes call $K(t, x)$ the \emph{flux}, since it records the number of particles (either zero or one) that move across the position $x$ between time $t$ and $t+1$. Now we present the discrete \ac{SHE} satisfied by the microscopic Hopf-Cole transform of the unfused \ac{SHS6V} model.
\begin{prop}\label{prop:microscopicshe}
For $t \in \ZZ_{\geq 0}$ and $x \in \Xi(t)$, $Z(t, x)$ satisfies the following discrete \ac{SHE}
\begin{equation}\label{eq:microshe}
Z(t+1, x-\mut{\mu}{t}) = (\tonepart(t+1, t) * Z(t))(x-\mut{\mu}{t}) + M(t, x), 
\end{equation}
where
\begin{equation}\label{eq:four1:temp1}
M(t, x) 
= \lambdat{\lambda}{t} (q - 1) Z(t, x + \hmu(t)) \overline{K}(t, x + \hmu(t)).
\end{equation}
Furthermore, $M(t,x)$ is a martingale increment, i.e. $\EE\big[M(t, x) \vv \FF(t) \big] = 0$. The conditional quadratic variation of $M(t, x)$ equals 
\begin{align}\label{eq:four1:quadratic}
\EE\big[M(t, x_1) M(t, x_2) \vv \FF(t)\big] = \bigg( q^{\rho} \frac{\nu + \alphat{\alpha}{t}}{1 + \alphat{\alpha}{t}}\bigg)^{|x_1 - x_2|} \Theta_1 (t, x_1 \wedge x_2) \Theta_2 (t, x_1 \wedge x_2), \quad x_1, x_2 \in \Xi(t),
\end{align}
where
\begin{align}\label{eq:Thetaone}
\Theta_1 (t, x) &:= q \lambdat{\lambda}{t} Z(t, x
) - \big(\convol \big) (x - \mu(t)),\\ \label{eq:Thetatwo}
\Theta_2 (t, x) &:= -\lambdat{\lambda}{t} Z(t, x) + \big(\convol\big)(x - \mu(t)).
\end{align}
\end{prop}
\begin{proof}
We first show that $M(t, x)$ is a martingale increment. Note by \eqref{eq:microshe}, 
\begin{equation*}
M(t, x) = Z(t+1, x-\mut{\mu}{t}) = (\tonepart(t+1, t) * Z(t))(x-\mut{\mu}{t}).
\end{equation*}
Taking $k=1$ in the duality \eqref{eq:fextenddual1}, one has
\begin{equation*}
\EE\big[Z(t+1, x-\mut{\mu}{t}) \vv \FF(t) \big]  = (\convol) (x-\mut{\mu}{t}).
\end{equation*}
Hence, 
\begin{align*}
\numberthis \label{eq:four1:temp2}
M(t,x) 
&= Z(t+1, x-\mut{\mu}{t}) - \EE\big[Z(t+1, x-\mut{\mu}{t}) \vv \FF(t)\big],
\end{align*} 
which implies $\EE\big[M(t, x) \vv \FF(t)\big]  = 0$.
\bigskip
\\ 
We turn to justify \eqref{eq:four1:temp1}. Note that by  \eqref{eq:microshe1}
\begin{equation*}
Z(t+1, x-\mut{\mu}{t}) = \lambdat{\lambda}{t} Z(t, x) q^{N(t, x + \hmu(t)) - N(t+1, x + \hmu(t))} = \lambda(t) Z(t, x) q^{K(t, x+ \hmu(t))}. 
\end{equation*}
Since $K(t, x+\hmu(t)) \in \{0, 1\}$,
\begin{equation}\label{eq:four1:temp3}
Z(t+1, x-\mut{\mu}{t}) = \lambdat{\lambda}{t} Z(t, x) + \lambda(t) (q - 1) Z(t, x) K(t, x + \hmu(t)).
\end{equation}  
Combining with \eqref{eq:four1:temp2} gives
\begin{align*}
M(t, x) 
&= \lambdat{\lambda}{t} (q- 1) Z(t, x) \big(K(t, x + \hmu(t) ) - \EE\big[K(t, x+ \hmu(t)) \big| \FF(t)\big]\big),\\
\numberthis \label{eq:four:temp1}
&=  \lambdat{\lambda}{t} (q- 1) Z(t, x) \overline{K}(t, x+ \hmu(t)),
\end{align*}
which gives the desired equality.
\bigskip
\\
We turn our attention to  \eqref{eq:four1:quadratic}.  Define the short notation  $E'\big[\cdot\big] := \EE \big[\cdot \big| \FF(t)\big]$ and write $\text{Var}'$, $\text{Cov}'$  to be the corresponding conditional variance and covariance. We assume without loss of generosity $x_1 \leq x_2$ and use shorthand notation $x'_i := x_i + \hmu(t) \in \ZZ, i=1, 2$. Owing to \eqref{eq:four:temp1},
\begin{align*}
\EE'\big[M(t, x'_1) M(t, x'_2) \big] &=  \lambdat{\lambda}{t}^2 (q - 1)^2 Z(t, x_1 ) Z(t, x_2) \EE'\big[\overline{K}(t, x'_1) \overline{K}(t, x'_2) \big],\\
\numberthis \label{eq:four1:tempa}
&= \lambdat{\lambda}{t}^2 (q - 1)^2 Z(t, x_1) Z(t, x_2) \text{Cov}'\big(K(t, x'_1), K(t, x'_2)\big).
\end{align*}
Define
\begin{align*}
\numberthis \label{eq:four:temp2}
L_{x'_1, x'_2} (t) &= \prod_{z=x'_1 +1}^{x'_2} \bigg(B'(t, z, \eta_z (t)) - B(t, z, \eta_z (t))\bigg),\\
K_{x'_1, x'_2} (t) &= \sum_{y' = x'_1 +1}^{x'_2} \prod_{z = y'+1}^{x'_2} \bigg(B'(t, z, \eta_z (t)) - B(t, z, \eta_z (t)\bigg) B(t, z, \eta_z (t)),
\end{align*}
where $B, B'$ are defined in \eqref{eq:randomenvir}. Since $B, B'$ are all independent, due to the expression \eqref{eq:birecur1} of $K(t, x'_1) = N(t, x'_1) - N(t+1, x'_1)$ provided by \eqref{eq:birecur1}, it is straightforward that conditioning on $\FF(t)$, $(K_{x'_1, x'_2} (t), L_{x'_1, x'_2} (t))$ are independent with $K(t, x'_1)$. Furthermore, \eqref{eq:birecur1} implies
\begin{align*}
K(t, x'_2) 
&= K_{x'_1, x'_2} (t) + L_{x'_1, x'_2} (t) K(t, x'_1).
\end{align*}
By the independence, 
we see that 
\begin{align}\label{eq:four1:temp5}
\text{Cov}'\big(K(t, x'_1), K(t, x'_2)\big) 
=\EE'\big[L_{x'_1, x'_2} (t)\big]\text{Var}'\big(K(t, x'_1)\big).
\end{align}
Referring to \eqref{eq:four:temp2},  
\begin{equation*}
\EE'\big[L_{x'_1 , x'_2} (t)\big] = \prod_{z = x'_1 +1}^{x'_2} \EE'\big[B'(t, z, \eta_z (t)) - B(t, z, \eta_z (t))\big] =  \bigg(\frac{\nu + \alphat{\alpha}{t}}{1 +\alphat{\alpha}{t}}\bigg)^{x'_2 - x'_1} \prod_{z = x'_1 +1}^{x'_2} q^{\eta_{z} (t)}.
\end{equation*} 
Inserting this into the RHS of \eqref{eq:four1:temp5}, we find that 
\begin{align*}
\text{Cov}'\big(K(t, x_1), K(t, x_2)\big) &= \bigg(\frac{\nu + \alphat{\alpha}{t}}{1 + \alphat{\alpha}{t}}\bigg)^{x'_2 - x'_1} \prod_{z = x'_1 + 1}^{x'_2} q^{\eta_z (t)} \big(\EE'\big[K^2(t, x'_1)\big] - \EE'\big[K(t, x'_1)\big]^2\big),  \\
\numberthis \label{eq:four1:temp6}
&= \bigg(\frac{\nu + \alphat{\alpha}{t}}{1 + \alphat{\alpha}{t}}\bigg)^{x_2 - x_1} \prod_{z = x'_1 + 1}^{x'_2} q^{\eta_z (t)} \EE'\big[K(t, x'_1)\big] \big(1 - \EE'\big[K(t, x'_1)\big]\big). 
\end{align*}
Here, the last equality follows from the fact $K(t, x'_1)^2 = K (t, x'_1) $. Furthermore, due to \eqref{eq:four1:temp3},
\begin{equation*}
 \EE'\big[K(t, x'_1) \big] =   
\frac{\EE \big[Z(t+1, x_1 - \mut{\mu}{t}) - \lambdat{\lambda}{t} Z(t, x_1) \vv \FF(t)\big]}{\lambdat{\lambda}{t} (q - 1) Z(t, x_1)} = \frac{(\convol) (x_1 - \mut{\mu}{t}) - \lambdat{\lambda}{t} Z(t, x_1)}{\lambdat{\lambda}{t} (q - 1) Z(t, x_1)}.
\end{equation*}
Inserting this into the RHS of \eqref{eq:four1:temp6} yields 
\begin{align*}
\text{Cov}'\big(K(t, x_1), K(t, x_2)\big)  &=  \bigg(\frac{\nu + \alphat{\alpha}{t}}{1 + \alpha(t)}\bigg)^{x_2 - x_1} \frac{(\convol) (x_1 - \mut{\mu}{t}) - \lambdat{\lambda}{t} Z(t, x_1)}{\lambdat{\lambda}{t} (q - 1) Z(t, x_1)} \\
&\quad \times\bigg(1 - \frac{(\convol) (x_1 - \mut{\mu}{t}) - \lambdat{\lambda}{t} Z(t, x_1)}{\lambdat{\lambda}{t} (q - 1) Z(t, x_1)}\bigg) \prod_{z = x'_1 + 1}^{x'_2} q^{\eta_z (t)},\\
&= \bigg(\frac{\nu + \alphat{\alpha}{t}}{1 + \alpha(t)}\bigg)^{x_2 - x_1} \frac{\Theta_2 (t, x_1)}{\lambdat{\lambda}{t} (q - 1) Z(t, x_1)} \cdot \frac{\Theta_1 (t, x_1)}{\lambdat{\lambda}{t} (q - 1) Z(t, x_1)} \prod_{z = x'_1 + 1}^{x'_2} q^{\eta_z (t)}.
\end{align*}
Using the fact $Z(t, x_2) = q^{\rho (x_2 - x_1)} Z(t, x_1) \prod_{z=x'_1 +1}^{x'_2} q^{-\eta_z (t)} $, we obtain 
\begin{equation*}
\text{Cov}'\big(K(t, x_1), K(t, x_2)\big) = \bigg(q^\rho \frac{\nu + \alphat{\alpha}{t}}{1 + \alphat{\alpha}{t}}\bigg)^{x_2 - x_1} \frac{\Theta_1 (t, x_1)}{\lambda(t) (q - 1) Z(t, x_1)} \cdot \frac{\Theta_2 (t, x_1)}{\lambda(t) (q - 1) Z(t, x_2)}.
\end{equation*}
Combining with \eqref{eq:four1:tempa}, we arrive at the desired \eqref{eq:four1:quadratic}.
\end{proof}
For $x \in \Xi(t)$, define $$\etapartt{t}{x} :=  \etapart{t}{x + \hmu(t)}.$$ 
We consider a tilted version of the duality functional $\du$ in \eqref{eq:dual2}, for $y_1 \leq y_2 \in \Xi(t)$, define
\begin{align}\label{eq:dfunc}
D(t, y_1, y_2) := 
\begin{cases}
Z(t,y_1)^2   \left[I - \etapartt{t}{y_1}\right]_{q^{\frac{1}{2}}}  \left[I - 1 - \etapartt{t}{y_1}\right]_{q^{\frac{1}{2}}} q^{\etapartt{t}{y_1}} \qquad &\text{if}\  y_1 = y_2,\\
\frac{\left[I-1\right]_{q^{\frac{1}{2}}} }{\left[I\right]_{q^{\frac{1}{2}}}} Z(t, y_1) Z(t, y_2) \qhalfint{I - \etapartt{t}{y_1}}  \qhalfint{I - \etapartt{t}{y_2}}    q^{\frac{1}{2} \etapartt{t}{y_1}} q^{\frac{1}{2} \etapartt{t}{y_2}} \qquad  &\text{if}\  y_1 < y_2. 
\end{cases}
\end{align}
We further define for $x_1, x_2 \in \Xi(t)$ and $y_1, y_2 \in \Xi(s)$,
\begin{align}\label{eq:four:tilttransition}
\rhzrt\big((x_1, x_2), (y_1, y_2), t, s\big) := \bigg(\frac{\hlambda(t)}{\hlambda(s)}\bigg)^2 q^{\rho \left(x_1 + x_2 - y_1 -y_2 + 2 (\hmu(t) - \hmu(s))\right)} \rhzr \big(x_1 + \hmu(t), x_2  + \hmu(t), y_1 +\hmu(s), y_2 + \hmu(s), t, s\big).
\end{align}
Observe that $Z(t, x)$ is a tilted version of $q^{-N(t, x)}$, thus it is clear that  it inherits the two dualities stated in Corollary \ref{cor:biduality}. 
\begin{lemma}\label{lem:dualt}
For $s \leq t \in \ZZ_{\geq 0}$ and $x_1 \leq x_2 \in \Xi(t)$,
\begin{align}\label{eq:fdualt}
&\EE\big[Z(t, x_1) Z(t, x_2) \vv \FF(s)\big] = \sum_{y_1 \leq y_2 \in \Xi(s)} \rhzrt\big((x_1, x_2), (y_1, y_2), t, s\big) Z(s, y_1) Z(s, y_2),\\
\label{eq:sdualt}
&\EE\big[D(t, x_1, x_2) \vv \FF(s)\big] = \sum_{y_1 \leq y_2 \in \Xi(s)} \rhzrt\big((x_1, x_2), (y_1, y_2), t, s\big) D(s, y_1, y_2).
\end{align}
\end{lemma}
\begin{proof}
We use the shorthand notation $x'_i := x_i + \hmu(t)$. Referring to  \eqref{eq:microshe1},
\begin{align}\label{eq:four2:temp1}
\EE\big[Z(t, x_1) Z(t, x_2) \vv \FF(s)\big] &=\hlambda(t)^2 q^{\rho(x'_1 + x'_2)} \EE\big[ q^{-N(t, x'_1)} q^{-N(t, x'_2)}  \vv \FF(s)\big]
\end{align}
Using Corollary \ref{cor:biduality}, we have
\begin{align*}
&\EE\big[ q^{-N(t, x'_1)} q^{-N(t, x'_2)}  \vv \FF(s)\big] = \sum_{y'_1 \leq y'_2 \in \ZZ^2} \rhzr \big((x'_1, x'_2), (y'_1, y'_2), t, s\big) q^{-N(s, y'_1)} q^{-N(s, y'_2)},\\
&=\sum_{y_1 \leq y_2 \in \Xi(s)^2} \rhzr\big((x_1 + \hmu(t), x_2 + \hmu(t), (y_1 + \hmu(s), y_2 + \hmu(s), t, s\big) q^{-N(s, y_1 + \hmu(s))}  q^{-N(s, y_2 + \hmu(s))}, \\
&=\sum_{y_1 \leq y_2 \in \Xi(s)^2} \rhzr\big((x_1 + \hmu(t), x_2 + \hmu(t), (y_1 + \hmu(s), y_2 + \hmu(s), t, s\big) \frac{Z(s, y_1) Z(s, y_2)}{\hlambda(s)^2} q^{-2 \hmu(s)}.
\end{align*}
Inserting this into the RHS of \eqref{eq:four2:temp1}, 
via a straightforward computation, we  conclude \eqref{eq:fdualt}. The second duality \eqref{eq:sdualt} follows from a similar argument, we do not repeat here.
\end{proof}
The following corollary follows from Theorem \ref{thm:intformulap}. 
\begin{cor}\label{cor:intformulav}
For all $x_1 \leq x_2 \in \Xi(t)$ and  $y_1 \leq y_2 \in \Xi(s)$, we have
\begin{align*}
\rhzrt\big((x_1, x_2), (y_1, y_2), t, s\big) &= c(\vec{y}) \bigg[ \oint_{\lc} \oint_{\lc} \prod_{i=1}^2 \coreh(z_i, t, s) z_i^{x_i - y_i } \frac{dz_i}{2\pi \im z_i} -   \oint_{\lc} \oint_{\lc} \interact(z_1, z_2) \prod_{i=1}^2 \coreh(z_i, t, s) z_i^{x_{3-i} - y_i } \frac{dz_i}{2\pi \im z_i}  \\
\numberthis \label{eq:tilpr:intformula}
&\quad+\res_{z_1 = \pol(z_2)} \oint_{\lc} \oint_{\lc} \interact(z_1, z_2) \prod_{i=1}^2 \coreh(z_i, t, s) z_i^{x_{3-i} - y_i } \frac{dz_i}{2\pi \im z_i}\bigg].
\end{align*}
where $\lc$ is a circle centered at zero with a large enough radius $R$ so as to include all the poles of the integrands, $c(\vec{y})$ is defined in \eqref{eq:cy1} and
\begin{align*}
\numberthis \label{eq:four2:temp2}
\corej (z) &:= \lambda z^{\mu}  \frac{(1 + \alpha q^J)q^{-\rho} z - (\nu + \alpha q^J)}{(1 + \alpha) q^{-\rho} z - (\nu + \alpha)},\\  
\numberthis \label{eq:four2:temp3}
\mathfrak{R} (z, t, s) &:= \prod_{k = s + J \floor{\frac{t-s}{J}}}^{t-1} \lambda(k) z^{\mu(k)}  \frac{(1 + \alpha(k) q) q^{-\rho} z - (\nu + \alpha(k) q) }{(1 + \alpha(k)) q^{-\rho} z - (\nu + \alpha(k))},\\
\numberthis \label{eq:interactexpression}
\interact(z_1 , z_2) &:= \frac{q\nu - \nu + (\nu-q) q^{-\rho} z_2 + (1-q\nu) q^{-\rho} z_1 + (q-1) q^{-2\rho} z_1 z_2}{q\nu - \nu + (\nu-q) q^{-\rho} z_1 + (1-q\nu) q^{-\rho} z_2 + (q-1) q^{-2\rho} z_1 z_2},\\
\numberthis \label{eq:pole1expression}
\pol (z) &:= \frac{(1-q \nu) q^{-\rho} z - \nu  (1-q)}{(q-\nu) q^{-\rho}  + (1-q) q^{-2\rho}  z}.
\end{align*}
\end{cor}
\begin{proof}
Note that the integral formula for $\rhzr$ is given by  \eqref{eq:intformulap}, referring to \eqref{eq:four:tilttransition}, we find that 
\begin{align*}
&\rhzrt\big((x_1, x_2), (y_1, y_2), t, s\big) =\bigg(\frac{ \hlambda(t)}{\hlambda(s)}\bigg)^2 q^{\rho (x_1 + x_2 - y_1 -y_2 + 2\hmu(t) - 2\hmu(s))} \rhzr \big(x_1 + \hmu(t), x_2  + \hmu(t), y_1 +\hmu(s), y_2 + \hmu(s), t, s\big),\\
&=c(\vec{y}) \cdot \bigg(\frac{ \hlambda(t)}{\hlambda(s)}\bigg)^2 q^{\rho (x_1 + x_2 - y_1 -y_2 + 2\hmu(t) - 2\hmu(s))} \bigg[ \oint_{\lc} \oint_{\lc} \prod_{i=1}^2 \coret(z_i)^{\floor{\frac{t-s}{J}}} \remt(z_i, t, s) z_i^{x_i - y_i } \frac{dz_i}{2\pi \im z_i} \\ 
&\quad- \oint_{\lc} \oint_{\lc} \interactt(z_1, z_2)\prod_{i=1}^2\coret(z_i)^{\floor{\frac{t-s}{J}}} \remt(z_i, t, s) z_i^{x_{3-i} - y_i } \frac{dz_i}{2\pi \im z_i}
\\
&\quad+ \res_{z_1 = \polt(z_2)} \oint_{\lc} \oint_{\lc} \interactt(z_1, z_2) \prod_{i=1}^2 \coret(z_i)^{\floor{\frac{t-s}{J}}} \remt(z_i, t, s) z_i^{x_{3-i} - y_i } \frac{dz_i}{2\pi \im z_i}\bigg].
\end{align*}
We refer to the context of Theorem \ref{thm:intformulap} for the notation.
Multiplying the constant $\big(\frac{ \hlambda(t)}{\hlambda(s)}\big)^2 q^{\rho \left(x_1 + x_2 - y_1 -y_2 + 2 \hmu(t) -2 \hmu(s)\right)}$ to each term inside the square bracket above and applying change of variable $z_i \to q^{-\rho} z_i$ readily yield the desired formula.
\end{proof}
\subsection{The SHE}\label{sec:she}
Consider the KPZ equation with parameter $V_*$ and $D_*$ given in \eqref{eq:vstar} and \eqref{eq:dstar},
\begin{equation}
\label{eq:KPZ equation1}
\partial_t \kpz(t, x) = \frac{V_*}{2} \pa_x^2 \kpz(t, x) - \frac{V_*}{2} \big(\pa_x \kpz(t, x) \big)^2 + \sqrt{D_*} \xi(t, x).
\end{equation}
As mentioned in Section \ref{sec:kpzequation}, via formally applying Hopf-Cole transform, 
we say that $\HH(t, x)$ is a Hopf-Cole solution of \eqref{eq:KPZ equation1} if  
\begin{equation*}
\kpz(t, x) = - \log \she(t, x),
\end{equation*}
where $\she(t, x)$ is a \emph{mild solution} of the \ac{SHE}
\begin{equation*}
\pa_t \she(t,x) = \frac{V_*}{2} \pa_x^2 \she(t,x) + \sqrt{D_*} \noise \she(t,x)
\end{equation*} 
in the sense that it satisfies the following Duhamel integral form
\begin{equation*}
\she(t,x) = \int_\RR \hk(V_* t, x-y) \she^{\text{ic}}(y) dy + \int_0^t \int_\RR \hk( V_* (t-s), x-y) \she(s, y) \sqrt{D_*} \xi(s, y) ds dy,
\end{equation*}
where $p(t, x) = \frac{1}{\sqrt{2 \pi t}} e^{-\frac{x^2}{2t}}$ is the heat kernel.
The stochastic heat equation has a unique mild solution $\she (t, x)$, see \cite{Cor12} and references therein. 
\bigskip
\\
We recall the \emph{weakly asymmetric scaling} for the \ac{SHS6V} model stated in Theorem \ref{thm:main}: \begin{equation}\label{eq:scaling}
\text{For } \ep > 0, \text{ fix } I \in \ZZ_{\geq 2}, J \in \ZZ_{\geq 1} \text{ and } b \in \bigg(\frac{I + J - 2}{I + J - 1}, 1\bigg), \text{ set } q = e^{\sqrt{\ep}} \text{ and define } \alpha \text{ via } b = \frac{1 + \alpha q}{1 + \alpha}.  
\end{equation}
Such scaling corresponds to taking $b=2, z = \frac{1}{2}, \kappa \to \sqrt{\ep} \kappa$ and keeping $\delta, D$ unchanged in \eqref{eq:scaledkpz}. Note that all parameters in the \ac{SHS6V} model rely on the generic parameters $q, b, I, J, \rho$, since under weakly asymmetry scaling, $b, I, J, \rho$ are all fixed and $q = e^{\sqrt{\ep}}$,  
the evolution of the entire model depends on $\ep$. As we will let $\ep$ go to zero, it suffices to consider all $\ep > 0$ \emph{small enough}, which means that we only consider $\ep \in (0, \ep_0)$ for some generic but fixed threshold $\ep_0 > 0$. 
\begin{lemma}\label{lem:wsc}
Under weakly asymmetric scaling \eqref{eq:scaling},  we have the following asymptotics near $\ep = 0$
\begin{align*}
\frac{\nu + \alpha(t)}{1 + \alpha(t)} &= \frac{b (I + \mod(t)) - (I + \mod(t) - 1)}{ b  \mod(t) - (\mod(t) - 1)} + \OO(\ep^{\frac{1}{2}}),\\  
\frac{\nu + q \alpha(t)}{1 + \alpha(t)} &= \frac{b (I + 1 + \mod(t)) - (I + \mod(t))}{b  \mod(t) - (\mod(t) - 1)} + \OO(\ep^{\frac{1}{2}}),\\
\frac{1 + q\alpha(t)}{1 + \alpha(t)} &= \frac{b (1 + \mod(t)) - \mod(t)}{b  \mod(t) - (\mod(t) - 1)} + \OO(\ep^{\frac{1}{2}}),\\
\mu(t) &= \frac{1}{I} + \OO(\ep^{\frac{1}{2}}), \quad
\quad \lambda(t) = 1 -  \frac{\rho \ep^{\frac{1}{2}}}{I} + \OO(\ep).
\end{align*}
As notational convention, we denote $\OO(a)$ to be a generic quantity such that  $\sup_{0 < a< 1} |\OO(a)| a^{-1} < \infty$.
\end{lemma}
\begin{proof}
For every $\ep > 0$, we have $q  =e^{\sqrt{\ep}}$, $\nu = e^{-I \sqrt{\ep}}$ and $\alpha(t) = \alpha q^{\mod(t)} = \frac{1 - b}{b - e^{\sqrt{\ep}}} e^{\sqrt{\ep} \mod(t)}$, where $b, I, J, \rho$ are fixed. The relation of $\lambda(t)$ and $\mu(t)$ with $\ep$ is implied by \eqref{eq:lamexpan} and \eqref{eq:muxpan} 
The verification of the above asymptotic is  then   straightforward.
\end{proof}
To highlight the dependence on $\ep$ under weakly asymmetric scaling, 
we denote by the microscopic Hopf-Cole transform $Z_\epsilon (t,x) := Z(t, x)$. Note that presently $Z_\ep (t, x)$ is only defined for $t \in \ZZ_{\geq 0}$ and $x \in \Xi(t)$, we extend  $Z_\epsilon(t, x)$ to be a $C([0, \infty), C(\RR))$-valued process by first linearly interpolating in $x \in \ZZ$, then in $t \in \ZZ_{\geq 0}$. This is slightly different from exponentiating the interpolated height function $N(t, x)$. Nevertheless, under the weak asymmetric scaling $q = e^{\sqrt{\ep}}$, it is straightforward to see that the difference between these two interpolation schemes is negligible as $\ep \downarrow 0$.
\bigskip
\\
As a notational convention, we write $\norm{X}_p := (\EE |X|^p)^{\frac{1}{p}}$ for $p \geq 1$. 
Following the work of [BG97], we define the \emph{near stationary initial data} for the unfused/fused \ac{SHS6V} model.
\begin{defin}\label{def:nearstationary}
Fix $\rho \in (0, I)$, we call the initial data $N_\epsilon (0, x)$ (equivalently $\N_\ep (0, x)$) \textbf{near stationary with density $\rho$} if for any $n \in \NN$ and $a \in (0, \frac{1}{2})$, there exists constant $u:= u(n, a)$ and $C:= C(n, a)$ such that for all $x, x' \in \ZZ$
\begin{equation*}
\norm{Z_\epsilon (0, x)}_n \leq C e^{u\epsilon |x|}, \qquad \norm{Z_\epsilon (0, x) - Z_\epsilon (0, x')}_n \leq C (\epsilon|x - x'|)^\a e^{u \epsilon (|x| + |x'|)},
\end{equation*}
holds for $\epsilon > 0$ small enough.
\end{defin}
\begin{theorem}
\label{thm:main1}
Under weakly asymmetric scaling, assuming that $N_\epsilon (0, x)$ is near stationary with density $\rho$ and for some $C(\RR)$-valued process $\she^{ic} (x)$ 
\begin{equation*}
Z_\epsilon (0, x) \Rightarrow \she^{ic}(0, x) \text{ in } C(\RR) \text{ as } \epsilon \downarrow 0,
\end{equation*}
then
\begin{equation*}
Z_\epsilon (\epsilon^{-2} t, \epsilon^{-1} x) \Rightarrow \she(t, x) \text{ in } C\big([0, \infty), C(\RR)\big) \text{ as } \epsilon \downarrow 0,
\end{equation*}
where $\she (t, x)$ is the  mild solution to the \ac{SHE}
\begin{equation}\label{she}
\pa_t \she(t,x) = \frac{V_*}{2} \pa_x^2 \she(t,x) + \sqrt{D_*} \noise \she(t,x),
\end{equation} 
with initial condition  $\she^{ic} (x)$.
\end{theorem}
As a consequence of the preceding theorem, we prove Theorem \ref{thm:main}.
\begin{proof}[Proof of Theorem \ref{thm:main}]
Via the discussion in Section \ref{sec:she}, $\kpz(t, x) = - \log \she(t, x)$ solves the \ac{KPZ} equation 
\begin{equation*}
\partial_t \kpz(t, x) = \frac{V_*}{2} \pa_x^2 \kpz(t, x) - \frac{V_*}{2} \big(\pa_x \kpz(t, x) \big)^2 + \sqrt{ D_*} \xi(t, x).
\end{equation*}
One has by \eqref{eq:microshe}, 
\begin{align*}
Z_{\ep}(\ep^{-2} t, \ep^{-1} x)  &= \hlambdae (t) e^{-\sqrt{\ep} \big(N_\ep (\ep^{-2} t, \ep^{-1} x + \ep^{-2} \hmue (t)) - \rho(\ep^{-1} x + \ep^{-2} \hmue (t) \big) }\\
&= e^{-\sqrt{\ep} \big(N_\ep (\ep^{-2} t, \ep^{-1} x + \ep^{-2} \hmue(t)) - \rho(\ep^{-1} x + \ep^{-2} \hmue(t)) \big) + \log \hlambdae (t)}.
\end{align*}
By Theorem \ref{thm:main1} and continuous mapping theorem, we obtain
\begin{equation*}
-\log Z_{\ep} (\ep^{-2} t, \ep^{-1} x) \Rightarrow \kpz(t, x) \text{ in } C([0, \infty), C(\RR)).
\end{equation*}
In other words,
\begin{equation}\label{eq:four:temp6}
\sqrt{\ep} \big(N_\ep (\ep^{-2} t, \ep^{-1} x + \ep^{-2} \hmue (t)) - \rho(\ep^{-1} x + \ep^{-2} \hmue (t))\big) - \log \hlambdae (t) \Rightarrow \kpz(t, x) \text{ in } C([0, \infty), C(\RR)).
\end{equation}
Note that we have 
$\N_\ep (t, x) = N_\ep (Jt, x)$ (although in fact, they only equal on the lattice due to different linear interpolation scheme, but it is obvious that the difference between them is negligible). Moreover, via \eqref{eq:four:temp5}
$$\qquad \hlambdae(J t) = \lambdae^t, \qquad \hmue(J t) = \mue^t.$$
Therefore, replacing the time variable $t$ with $Jt$ in \eqref{eq:four:temp6},
\begin{equation*}
\sqrt{\ep} \big(\N_\ep (\ep^{-2} t, \ep^{-1} x + \ep^{-2} \mu_\ep t  ) - \rho (\ep^{-1} x + \ep^{-2} \mu_\ep t)\big) - t \log \lambda_\ep\Rightarrow \kpzt(t, x) \text{ in } C([0, \infty), C(\RR)),
\end{equation*}
where $\kpzt(t, x) := \kpz (J t, x)$. It is straightforward to check that $\kpzt(t, x)$ satisfies the \ac{KPZ} equation 
\begin{equation*}
\partial_t \kpzt(t, x) = \frac{J V_*}{2} \pa_x^2 \kpzt(t, x) - \frac{J V_*}{2} \big(\pa_x \kpzt(t, x) \big)^2 + \sqrt{J D_*} \xi(t, x),
\end{equation*}
which concludes the proof of Theorem \ref{thm:main}.
\end{proof}
\section{Tightness and proof of Theorem \ref{thm:main1}}\label{sec:tightness}
 In this section, we prove Theorem \ref{thm:main1} assuming Proposition \ref{prop:timedecorr}, whose proof is postponed to Section \ref{sec:timedecorr}. First of all,  we prove the tightness of $\{Z_\ep (\ep^{-2} \cdot, \ep^{-1} \cdot)\}_{0 < \ep < 1}$, which indicates that as $\ep \downarrow 0 $, $Z_\ep (\ep^{-2} \cdot, \ep^{-1} \cdot)$ converges weakly along a subsequence. To identify the limit as well as proving the convergence of the entire sequence, we appeal to the martingale problem of \ac{SHE} that was first introduced in the work of \cite{BG97}. Using approximation from the microscopic \ac{SHE} \eqref{eq:microshe} to the \ac{SHE} in continuum, we show that  any subsequential limit of $Z_\ep (\ep^{-2} \cdot, \ep^{-1} \cdot)$ satisfies the same martingale problem, hence is the mild solution of \ac{SHE}.
\bigskip
\\
Hereafter, we always assume that we are under weakly asymmetric scaling \eqref{eq:scaling}. In general, we will not specify the dependence of parameters on $\ep$. 
We will also write $q_\ep$,$\nu_\ep$,  etc. 
when we do want to emphasize the dependence.
The dependence on $I \in \ZZ_{\geq 2} , J \in \ZZ_{\geq 1}$, $b = \frac{1 + \alpha q}{1 + \alpha} \in (\frac{I+J-2}{I + J-1}, 1)$, $\rho \in (0, I)$ 
will not be indicated as they are fixed. 
\bigskip
\\
For the ensuing discussion, we will usually write $C$ for constants. We might not generally specify when irrelevant terms are being absorbed into the constants. We might also write $C(T), C(\beta, T), \dots$ when we want to specify which parameters the constant depends on. We say ``for all $\ep > 0$ small
enough" if the referred statement holds for all $0 < \ep < \ep_0$ for some generic but fixed threshold $\ep_0 > 0$ that may change from line to line.
\subsection{Moment bounds and tightness}
The goal of this section is to prove the following Kolmogorov-Chentsov type bound for the microscopic Hopf-Cole transform.
\begin{prop}\label{prop:tightness}
Assume that we start the \ac{SHS6V} model with near stationary initial data with density $\rho \in (0, I)$. Given $n \in \NN$, $a \in (0, \frac{1}{2})$, $T > 0$.   There exists positive constants $C:= C(n, a, T)$, $u:= u(n, a)$ such that 
\begin{align}
\numberthis \label{eq:tightness1}
&\norm{Z(t, x)}_{2n} \leq C e^{u \epsilon |x| },\\
\numberthis \label{eq:tightness2}
&\norm{Z(t, x) - Z(t, x')}_{2n} \leq C |\epsilon(x - x')|^{a} e^{u\epsilon (|x| + |x'|)},\\
\numberthis \label{eq:tightness3}
&\norm{Z(t, x) - Z(t', x)}_{2n} \leq C |\epsilon^2 (t - t')|^{\frac{\a}{2}} e^{2u\epsilon |x|},
\end{align}
for all $t, t' \in [0, \epsilon^{-2} T]$ and $x, x' \in \RR$. 
\end{prop}
We immediately deduce the tightness of $Z_\ep (\ep^{-2} \cdot, \ep^{-1}  \cdot)$ once we have the moment bound above.
\begin{cor}
The law of $\funcspace$-valued process $\{Z_\epsilon (\epsilon^{-2} \cdot, \epsilon^{-1} \cdot)\}_{0 < \epsilon < 1}$ is tight. 
\end{cor}
\begin{proof}
\eqref{eq:tightness1}, \eqref{eq:tightness2} and \eqref{eq:tightness3} indicate that with large probability $\{Z_\epsilon (\epsilon^{-2} \cdot, \epsilon^{-1} \cdot)\}_{0 < \epsilon < 1}$ is uniformly bounded, uniformly spatially and uniformly temporally  H\"{o}lder continuous. Applying Arzela-Ascoli theorem together with Prokhorov's theorem \cite{Bil13} yields the desired result.  
\end{proof}
For the proof of Proposition \ref{prop:tightness}, we will basically follow the framework developed in \cite{CGST18}. Let us begin with a technical lemma which will be frequently used for the rest of the paper. 
\begin{lemma}\label{lem:five:usefullem}
Fix $T > 0$, for any $u > 0$, there exists $\beta_0 > 0$ such that for all $\beta > \beta_0$ and $C(\beta) > 0$,  there exists $\ep_0$ such that for all positive  $\ep < \ep_0$, $t \in [0, \ep^{-2} T] \cap \ZZ$ and $x \in \Xi(t)$, the following inequality holds\footnote{Here $C(\beta)$ can be any positive constant, though for application, the choice of it usually depends on the value of $\beta$.}
\begin{equation*}
\sum_{y \in \Xi(t)} e^{-\frac{\beta |x - y|}{\sqrt{t+1} + C(\beta)}} e^{u \ep |y|} \leq 2 \sqrt{t+1} e^{u \ep |x|}.
\end{equation*}
\end{lemma}
\begin{proof}
Take $\beta_0 = 4 \sqrt{T} u$, for $\beta > \beta_0$ and arbitrary $C(\beta) > 0$, due to $t \in [0, \ep^{-2} T]$, one has
$$\frac{\beta |x|}{\sqrt{t+1} + C(\beta)} \geq \frac{\beta \epsilon |x|}{\sqrt{T + \epsilon^2} + C(\beta) \epsilon} \geq 2 u \epsilon |x|$$
holds for $\ep < \ep_0$, where is $\ep_0$ is to be chosen small enough. Thereby,
\begin{align*}
\sum_{y \in \Xi(t)} e^{-\frac{\beta |x - y|}{\sqrt{t+1} + C(\beta)}} e^{u \ep |y|}
&\leq e^{u \ep |x|} \sum_{y \in \Xi(t)} e^{-\frac{\beta |x - y|}{\sqrt{t+1} + C(\beta)}} e^{u \ep |x - y|},\\
&\leq e^{u \ep |x|} \sum_{y \in \ZZ} e^{-\frac{\beta |y|}{\sqrt{t+1} + C(\beta)}} e^{u \ep |y|}\\
&\leq e^{u \ep |x|} \sum_{y \in \ZZ} e^{-\frac{\beta |y|}{2 (\sqrt{t+1} + C(\beta))}}
\\
&\leq  2 \sqrt{t+1} e^{u \ep |x|}.
\end{align*}
Here, the last inequality follows from
\begin{equation*}
\sum_{x \in \Xi(t)} e^{-\frac{\beta |y|}{2 (\sqrt{t+1} + C(\beta))}} \leq \frac{2}{1 - e^{-\frac{\beta}{2 (\sqrt{t+1} + C(\beta))}}} \leq 2\sqrt{t+1}.
\end{equation*}
Thus, we conclude the lemma.
\end{proof}
The following estimate for the one particle transition probability will be useful in proving Proposition \ref{prop:tightness}.
\begin{lemma}\label{thm:onepartestimate}
For any $u, T \in (0, \infty)$ and $a \in (0, 1)$, there exists constant $C$ (depending on $a, u, T$) such that 
\begin{align*}
&(i)\   \tonepart (t, s, x) \leq C (t - s + 1)^{-\frac{1}{2}}, \hspace{9.2em} (ii)\ \sum_{x \in \Xi(t, s)} \tonepart (t, s, x) e^{u \epsilon  |x|} \leq C, \\
&(iii)\ \sum_{x \in \Xi(t, s)} |x|^a \tonepart(t, s, x) e^{u \epsilon |x|} \leq C (t - s + 1)^{\frac{a}{2}}, \qquad
(iv)\  |\tonepart(t, s, x)  - \tonepart(t, s, x')| \leq C |x-x'|^{a} {(t-s+1)}^{-\frac{a+1}{2}}.
\end{align*}
for $\epsilon > 0$ small enough and $s \leq t \in [0, \epsilon^{-2} T] \cap \ZZ$.
\end{lemma}
\begin{proof}
The proof is more or less analogous to \cite[Lemma 5.1]{CGST18}. 
We first claim that  $\tonepart(t, s, x)$ admits the following  integral formula
\begin{equation}\label{eq:intformulaonetpart}
\tonepart(t, s, x)  =  \oint_{\lc} \big(\corej (z)\big)^{\floor{\frac{t-s}{J}}} \rem (z, t, s)  z^{x} \frac{dz}{2 \pi \im z},
\end{equation}
where $\corej (z)$, $\rem (z, t, s)$ are defined in \eqref{eq:four2:temp2} and \eqref{eq:four2:temp3} respectively and $R$ is large enough so that the circle $\C_R$ includes all the singularities of the integrand.
This claim can be proved by observing 
\begin{align*}
\EE\big[z^{-R(k)}\big] &= \sum_{n=0}^{\infty} \PP\big(R(k) = n - \mut{\mu}{k}\big) z^{\mut{\mu}{k} - n}, \\
&= \lambdat{\lambda}{k} \frac{1 + q \alphat{\alpha}{k} }{1 + \alphat{\alpha}{k}} z^{\mut{\mu}{k}} + \sum_{n=1}^{\infty} \lambdat{\lambda}{k}\bigg(1- \frac{1 + q \alphat{\alpha}{k}}{1 + \alphat{\alpha}{k}}\bigg) \bigg(1- \frac{\nu + \alphat{\alpha}{k}}{1 + \alphat{\alpha}{k}}\bigg) \bigg(\frac{\nu + \alphat{\alpha}{k}}{1 + \alphat{\alpha}{k}}\bigg)^{n-1} q^{\rho n} z^{\mut{\mu}{k} - n}, \\
\numberthis \label{eq:mgf}
&= \lambdat{\lambda}{k} z^{\mut{\mu}{k}} \frac{1 + \alphat{\alpha}{k} q - (\nu + \alphat{\alpha}{k} q) q^{\rho} z^{-1}}{1 + \alphat{\alpha}{k} - (\nu + \alphat{\alpha}{k}) q^{\rho} z^{-1}}.
\end{align*}
This implies 
\begin{equation*}
\EE\big[z^{-(X(t) - X(s))} \big] = \prod_{k=s}^{t-1} \EE\big[z^{-R(k)}\big] = \big(\corej (z)\big)^{\floor{\frac{t-s}{J}}} \rem (z, t, s).
\end{equation*}
Via Fourier inversion formula, we have 
\begin{equation*}
\tonepart(t, s, x) = \PP\big(X(t) - X(s) = x\big) \oint_{\C_R} \EE\big[z^{-(X(t) - X(s)}\big] z^x \frac{dz}{2\pi \im z} =  \oint_{\lc} \big(\corej (z)\big)^{\floor{\frac{t-s}{J}}} \rem (z, t, s) \frac{dz}{2 \pi \im z},
\end{equation*}
In Section \ref{sec:asymptotic analysis}, we will obtain an upper bound of $\tonepart(t, s, x)$ by applying steepest descent analysis to the integral formula above and we use this upper bound here in advance.
Referring to \eqref{eq:six:onepartbound}, by taking $x_i - y_i \to x$, we obtain for all $\beta, T > 0$, there exists positive constant $C(\beta), C(\beta, T)$ such that for $\ep > 0$ small enough
\begin{equation}\label{eq:five:temp1}
\tonepart (t, s, x) \leq \frac{C(\beta, T)}{\sqrt{t - s+ 1}} e^{-\frac{\beta |x|}{\sqrt{t-s+1} + C(\beta)}}, \qquad t \in [0, \ep^{-2} T] \cap \ZZ.
\end{equation} 
which gives (i). 
Using \eqref{eq:five:temp1} together with Lemma \ref{lem:five:usefullem} gives (ii)
\begin{equation*}
\sum_{x \in \Xi(t, s)} \tonepart (t, s, x) e^{u\ep |x|} \leq \sum_{x \in \Xi(t, s)} \frac{C(\beta, T)}{\sqrt{t - s +1}} e^{-\frac{\beta |x|}{\sqrt{t - s +1} + C(\beta)}} e^{u\ep |x|} \leq C.
\end{equation*}
For (iii), we see that 
\begin{align*}
\sum_{x \in \Xi(t, s)} |x|^a \tonepart (t, s, x) e^{u \epsilon |x|} \leq \sum_{x \in \Xi(t, s)} C(\beta, T) |x|^a e^{-\frac{\beta |x|}{2 (\sqrt{t - s +1} + C(\beta))}} \leq C \big(\sqrt{t - s +1} +C(\beta)\big)^{a+1} \leq C (t - s +1)^{\frac{a + 1 }{2}}.
\end{align*}
For the second inequality above, we used the inequality 
\begin{equation*}
\sum_{x \in \Xi(t, s)} |x|^a e^{-b |x|} \leq C \int_{0}^{\infty} x^a e^{-b x} dx \leq C b^{-a-1}. 
\end{equation*}
Finally, to prove (iv), one has by \eqref{eq:six:onepartgradbound} (taking $\beta = 1$) 
\begin{equation*}
|\na p(t, s, x)| = |\tonepart(t, s, x+1) - \tonepart(t, s, x)| \leq \frac{C(T)}{t - s +1} e^{-\frac{|x|}{\sqrt{t - s +1} + C}}.
\end{equation*}
Summing the above equation  over $[x, x' - 1]$ (assuming with out loss of generosity that  $x < x'$), we obtain 
\begin{equation*}
\big|\tonepart(t, s, x) - \tonepart(t, s, x')\big| \leq \frac{C(T)}{t - s +1} \sum_{y =x}^{x' - 1} e^{-\frac{|y|}{\sqrt{t -s +1} + C}}
\end{equation*}
If we bound each term in the geometric sum by $1$, we have
$\big|\tonepart(t, s, x) - \tonepart(t, s, x')\big| \leq \frac{C}{t-s+1} |x' - x|$. In addition, we can bound the geometric sum by 
$$\sum_{y= x}^{x' - 1} e^{-\frac{|y|}{\sqrt{t - s +1} + C}} \leq 2 \sum_{y= 0}^{\infty} e^{-\frac{|y|}{\sqrt{t - s +1} + C}} = \frac{2}{1 - e^{-\frac{1}{\sqrt{t -s +1} + C}}} \leq C \sqrt{t - s +1},$$
which implies 
\begin{equation*}
\big|\tonepart(t, s, x) - \tonepart(t, s, x')\big| \leq \frac{C}{\sqrt{t -s +1}}.
\end{equation*} 
Thereby, 
\begin{equation*}
\big|\tonepart(t, s, x)\ - \tonepart(t, s, x') \big| \leq \min\bigg(\frac{C}{t -s +1} |x - x'|, \frac{C}{\sqrt{t -s +1}}\bigg) \leq C |x - x'|^a (t -s +1)^{-\frac{a + 1}{2}},
\end{equation*}
which concludes the proof of (iv).
\end{proof}
Recall the discrete \ac{SHE} in Proposition \ref{prop:microscopicshe}
\begin{equation}\label{eq:five:temp}
Z(t, x) = (\tonepart(t, t-1) * Z(t-1)) (x) + M(t-1, x+ \mut{\mu}{t-1}).
\end{equation}
Iterating \eqref{eq:five:temp} for $t$ times yields 
\begin{equation}\label{eq:five:microSHE}
Z(t, x ) = (\tonepart(t, 0) * Z(0))(x) + Z_{mg}(t),
\end{equation}
where the martingale $Z_{mg}(t)$ equals
\begin{equation}\label{eq:five:temp8}
Z_{mg}(t) = \sum_{s = 0}^{t-1} \big(\tonepart(t, s+1) * M(s)\big) (x+\mu(s)).
\end{equation}
To estimate $Z(t, x)$, it suffices to estimate $(\tonepart(t, 0) * Z(0))(x)$ and $Z_{mg}(t)$ respectively. In general, the former one is easier to bound due to  Lemma \ref{thm:onepartestimate}, while controlling the latter one is much harder.
Following the style of \cite{CGST18}, to  estimate $Z_{mg} (t)$, we need to establish the following two lemmas, which are in analogy with Lemma 5.2 and Lemma 5.3 of \cite{CGST18}. 
\bigskip
\\
Let $\pset(n)$ denote the set of the partitions into intervals of $2$ or $3$ elements. Here, the interval refers to the set of form $U = [a, b] := [a, b] \cap{\ZZ}$, $a \leq b \in \ZZ$. For example, 
\begin{equation*}
\pset(6) = \left\{\{[1,2], [3,4], [5, 6], \{[1, 2], [3, 6]\}, \{[1,4], [5, 6]\}, \{[1, 3], [4, 6]\}\right\}. 
\end{equation*}
For $\vec{y} = (y_1 \leq \dots \leq y_n)$ and $U = [a, b]$, we define $|\vec{y}|_U = y_b - y_a$. 
\begin{lemma}
Fix $n \in \ZZ_{> 0}$, for all $t \in \ZZ_{\geq 0}$ and $y_1 \leq \dots \leq y_n \in \ZZ$, we have  
\begin{equation*}
\bigg|\EE\bigg[\prod_{i=1}^n \overline{K}(t, y_i) \vvv \FF(t) \bigg] \bigg| \leq C(n) \sum_{\pi \in \pset(n)} \prod_{U \in \pi} e^{-\frac{1}{C(n)} |\vec{y}|_U}.
\end{equation*}
\end{lemma}
\begin{proof}
\cite[Lemma 5.2]{CGST18} proved this inequality for $I =1$. When $I \geq 2$,  the proof is almost the same. Let us denote by $\EE'\big[\cdot\big] = \EE \big[\cdot \vv \FF(t)\big]$ and
\begin{equation*}
I(y', y) = \prod_{z = y' +1}^{y}\big( B'(t, z, \eta_z (t)) - B(t, z, \eta_z (t))\big) B(t, y', \eta_{y'} (t)).
\end{equation*} 
Due to \eqref{eq:two:temp3}, there exists $C > 0$ such that 
\begin{equation*}
\big|\EE'\big[I(y', y)^\ell\big]\big| \leq  C e^{-\frac{1}{C} |y - y'|}, \qquad \ell \in \NN.
\end{equation*}
This gives bound similar to (5.10) in \cite[Lemma 5.2]{CGST18}. The rest of the proof is the same as in \cite[Lemma 5.2]{CGST18}, we do not repeat it here.
\end{proof}
\begin{lemma}\label{lem:burkholder}
Fix $n \in \NN$, recall the martingale increment $M(t, x)$ from \eqref{eq:microshe} and let $f(t, x)$ be a deterministic function defined on $t \in [t_1, t_2] \cap \ZZ$ and $x \in \Xi(t)$. Write $f_\infty (t):=\sup_{x \in \Xi(t)} |f(t, x)|$, we have 
\begin{equation*}
\bignorm{\sum_{t = t_1}^{t_2 -1} \sum_{x \in \Xi(t)} f(t, x) M(t, x)}_{2n}^2 \leq \ep C(n) \sum_{t = t_1}^{t_2 -1} \sum_{x \in \Xi(t)} \big|f_\infty (t) f(t, x)\big| \norm{Z(t, x)}_{2n}^2.
\end{equation*}
\end{lemma}
\begin{proof}
Using the previous lemma, the proof is the same as the one appeared in \cite[Lemma 5.3]{CGST18}.
\end{proof}
 Have prepared the preceding lemmas, we proceed to prove Proposition \ref{prop:tightness}. Here we use a slightly different approach compared with the proof of the moment bounds in \cite[Proposition 5.4]{CGST18}.
\begin{proof}[Proof of Proposition \ref{prop:tightness}]
Recall that $Z(t, x)$ is defined on $[0, \infty) \times \RR$ through linear interpolation. It suffices to prove the theorem for the lattice $t \in \ZZn$ and $x, x' \in \Xi(t)$. Generalization to continuum $t, x$ follows easily.
\bigskip
\\
Let us begin with proving \eqref{eq:tightness1}. We have by \eqref{eq:five:microSHE} 
\begin{equation*}
\norm{Z(t, x)}_{2n} \leq  \norm{\big(\tonepart(t, 0) * Z(0)\big)(x)}_{2n} + \norm{Z_{mg} (t)}_{2n}.
\end{equation*}
Using $(x+y)^2 \leq 2(x^2 + y^2)$, we get
\begin{equation}\label{temp1}
\norm{Z(t, x)}_{2n}^2 \leq 2 \norm{\big(\tonepart(t, 0) * Z(0)\big)(x)}_{2n}^2 + 2\norm{Z_{mg} (t)}_{2n}^2.
\end{equation}
For the first term on RHS of \eqref{temp1}, by Cauchy-Schwarz inequality,
\begin{equation}\label{eq:tightness:temp1}
\norm{\big(\tonepart(t, 0) * Z(0)\big)(x)}_{2n}^2 \leq \big(\tonepart(t, 0) * \norm{Z(0)}_{2n}^2\big)(x).
\end{equation}
For the second term $\norm{Z_{mg}(t)}_{2n}^2$, by \eqref{eq:five:temp8}
\begin{align*}
Z_{mg}(t) &= \sum_{s = 0}^{t-1} \big(\tonepart(t,  s + 1) * M(s)\big)(x+\mut{\mu}{s}) = \sum_{s = 0}^{t-1} \sum_{y \in \Xi(s)} \tonepart\big(t, s+1, x + \mut{\mu}{s} - y\big) M(s, y).
\end{align*}
Applying Lemma \ref{lem:burkholder}, there exists a constant $C_*$ so that
\begin{align*}
\norm{Z_{mg}(t)}_{2n}^2 &\leq C_* \ep \sum_{s = 0}^{t-1} \sum_{y \in \Xi(s)} \Big(\sup_{y \in \Xi(s)} \tonepart(t, s+1, x+\mut{\mu}{s} - y)\Big) \tonepart(t, s+1, x+ \mut{\mu}{s} - y) \norm{Z(s, y)}_{2n}^2,\\
\numberthis \label{eq:tightness:temp2}
&\leq \sum_{s= 0}^{t-1} \sum_{y \in \Xi(s)}  \frac{C_* \ep}{\sqrt{t-s}} \tonepart\big(t, s+1, x+\mut{\mu}{s} -y\big) \norm{Z(s, y)}_{2n}^2,
\end{align*}
where the last inequality follows from Theorem \ref{thm:onepartestimate} (i).
\bigskip
\\
Replacing the RHS of \eqref{temp1} by upper bound obtained in \eqref{eq:tightness:temp1} and \eqref{eq:tightness:temp2}, we obtain 
\begin{equation}\label{eq:tightness:temp3}
\norm{Z(t, x)}_{2n}^2 \leq (\tonepart(t, 0) * \norm{Z(0)}_{2n}^2)(x)  +  \sum_{s = 0}^{t-1} \frac{C_* \ep}{\sqrt{t-s}} \big(\tonepart(t, s+1) * \norm{Z(s)}_{2n}^2\big)(x+\mut{\mu}{s}). 
\end{equation}
Define the set $\Delta_n^+  = \{(s_1, \dots, s_n) \in \ZZn^n: 0 \leq s_n < \dots < s_1 < t\}$ for $n \in \NN$. Iterating \eqref{eq:tightness:temp3} yields
\begin{align*}
&\norm{Z(t, x)}_{2n}^2 \leq (\tonepart(t, 0) * \norm{Z(0)}_{2n}^2)(x) \\
\numberthis \label{eq:tightness:temp5}
&+\sum_{n=1}^{\infty}  \sum_{(s_1, \dots s_n) \in \Delta_n^+} \frac{(C_* \ep)^n}{\sqrt{t- s_1} \sqrt{s_1 - s_2} \dots \sqrt{s_{n-1} - s_n}} (\tonepart(t, s_1, \dots, s_n) * \norm{Z(0)}_{2n}^2) (x + \sum_{i=1}^{n} \mu(s_i)).
\end{align*}
where $\tonepart(t, s_1, \dots, s_n) = \tonepart(t, s_1 + 1) * \tonepart(s_1, s_2 + 1)* \dots *\tonepart(s_{n-1} +1, s_n)$.
Following Lemma \ref{thm:onepartestimate}, we bound 
\begin{align*}
&(\tonepart(t, 0) * \norm{Z(0)}_{2n}^2)(x)
\leq C e^{2 u\ep |x|},\\
\numberthis \label{eq:tightness:temp6}
&(\tonepart(t, s_1, \dots, s_n) * \norm{Z(0)}_{2n}^2) (x + \sum_{i=1}^{n} \mu(s_i)) \leq C e^{2 u \ep (|x| + n)}.
\end{align*}
For the second term on the RHS of \eqref{eq:tightness:temp5}, note that via integral approximation, we readily see that
\begin{align*}
&\sum_{(s_1, \dots, s_n) \in \Delta_n^+} \frac{(C_* \ep)^n }{\sqrt{t-s} \sqrt{s_1 - s_2} \dots \sqrt{s_{n-1} - s_n}} \leq \int_{0 \leq s_1 \leq \dots \leq s_n \leq t} \frac{(C_* \ep)^n ds_1 \dots ds_n}{\sqrt{t-s_1} \sqrt{s_1 - s_2} \dots \sqrt{s_{n-1} - s_n}} \\
\numberthis \label{eq:tightness:temp7}
&= (C_* \ep t^{\frac{1}{2}} )^n \int_{\tau_1 + \dots + \tau_n \leq 1} \frac{1}{\sqrt{\tau_1} \dots \sqrt{\tau_n}} d\tau_1 \dots d\tau_n = \frac{(\Gamma(\frac{1}{2}) C_* \ep t^{\frac{1}{2}} )^n  }{\Gamma(n/2)}
\end{align*}
where $\Gamma(z)$ is the Gamma function. Combining \eqref{eq:tightness:temp6} and \eqref{eq:tightness:temp7} yields 
\begin{equation*}
\norm{Z(t, x)}_2^2 \leq C e^{2 u\ep |x|} + \sum_{n=1}^{\infty} \frac{(\Gamma(\frac{1}{2}) C_* \ep t^{\frac{1}{2}} )^n  }{\Gamma(n/2)} e^{2 u\ep (|x| + n)} = e^{2 u\ep |x|} \bigg(C + \sum_{n=1}^{\infty} \frac{(\Gamma(\frac{1}{2}) C_* \ep t^{\frac{1}{2}} e^{2 u\ep })^n  }{\Gamma(n/2)} \bigg)
\end{equation*}
Note that $\ep t^{\frac{1}{2}} \leq \sqrt{T}$ (since $t \in [0, \ep^{-2} T]$), as the growth rate of $\Gamma(\frac{n}{2})$ is much faster than that of $x^n$, the infinite series in the parentheses above converge, which concludes \eqref{eq:tightness1}. 
\bigskip
\\
The proof for \eqref{eq:tightness2} and \eqref{eq:tightness3} relies on  \eqref{eq:tightness1}. We proceed to prove \eqref{eq:tightness2},  denote by $$\zt(t, x, x'): = Z(t, x) - Z(t, x'), \qquad \tonepart^\na (t, s,  x, x') := \tonepart(t, s, x) - \tonepart(t, s, x').$$ 
Using \eqref{eq:five:microSHE} (subtract $Z(t, x')$ from $Z(t, x)$)
, we have
\begin{equation*}
\zt(t, x, x') = \sum_{y \in \Xi(t)} \tonepart(t,0, y) \zt(0, x-y, x'-y) + Z_{mg}^{\nabla}(t),
\end{equation*}
where
\begin{align*}
Z_{mg}^{\nabla}(t) 
\numberthis \label{eq:five1:temp1}
=\sum_{s = 0}^{t-1} \sum_{y \in \Xi(s)} \tonepart^\na \big(t, s+1, x+ \mu(s) - y, x'+ \mu(s) -y\big) M(s, y).
\end{align*}
It is straightforward that 
\begin{equation*}
\norm{\zt(t, x, x')}_{2n}^2 \leq 2 \sum_{y \in \Xi(t)} \tonepart(t, 0, y) \norm{\zt(0, x-y, x'-y)}_{2n}^2 + 2 \norm{Z_{mg}^\nabla (t)}_{2n}^2.
\end{equation*}
By the definition of the near stationary initial data (Definition \ref{def:nearstationary}), for $a \in (0, \frac{1}{2})$, there exists $C$ such that 
\begin{align*}
\sum_{y \in \Xi(t)} \tonepart(t, 0, y) \norm{\zt(0, x-y, x'-y)}_{2n}^2 &\leq C \sum_{y \in \Xi(t)} \tonepart(t, 0, y) (\ep |x-x'|)^{2a} e^{2 u\ep (|x - y| + |x' - y|)}\\
&\leq  C (\ep |x - x'|)^{2a} e^{2u \ep (|x| +|x'|)}  \sum_{y \in \Xi(t)} \tonepart(t, 0, y) e^{4 u\ep |y|}
\end{align*}
Further applying Theorem \ref{thm:onepartestimate} (ii), one has 
\begin{align*}
\sum_{y \in \Xi(t)} \tonepart(t, 0, y) e^{4 u\ep |y|}
\leq C.
\end{align*} 
We conclude that 
\begin{align*}
\numberthis \label{eq:tightness:tempa}
\sum_{y \in \Xi(t)} \tonepart(t, 0, y) \norm{\zt(0, x-y, x'-y)}_{2n}^2 
\leq C (\ep |x - x'|)^{2a} e^{2 u\ep (|x| + |x'|)}.
\end{align*} 
To bound $\norm{Z_{mg}^\nabla (t)}_{2n} $, we appeal to Lemma \ref{lem:burkholder}. Note that due to  Lemma \ref{thm:onepartestimate} (iv),
\begin{equation*}
\sup_{y \in \Xi(s)} \big|\tonepart^\na (t, s+1, x + \mu(t-1) - y, x'+\mu(t-1) - y)\big| \leq C|x - x'|^{2a} (t-s)^{-\frac{2a +1}{2}},
\end{equation*}
Applying Lemma \ref{lem:burkholder} to \eqref{eq:five1:temp1} implies 
\begin{equation*}
\norm{Z_{mg}^{\na} (t)}_{2n}^2 \leq C \ep |x - x'|^{2a}  \sum_{s=0}^{t-1} (t-s)^{-\frac{a+1}{2}} \sum_{y \in \Xi(s)} \tonepart^\na (t-s-1, x+\mu(s) - y, x'+ \mu(s) -y) \norm{Z(s, y)}_{2n}^2.
\end{equation*}
Owing to Theorem \ref{thm:onepartestimate} (i), we observe that  
\begin{align*}
&\sum_{y \in \Xi(s)} \tonepart^\na (t-s-1, x+\mu(s)-y, x'+\mu(s) - y) \norm{Z(s, y)}_2^2 \\
&\leq C \sum_{y \in \Xi(s)} \tonepart^\na (t-s-1, x+ \mu(s) - y, x'+\mu(s) -y) e^{2u\ep |y|} \leq C e^{2 u\ep (|x| + |x'|)}.
\end{align*}
Consequently, 
\begin{align*}
\norm{Z_{mg}^\na (t)}_{2n}^2 \leq C \ep |x'-x|^{2a} e^{2 u\ep (|x|+ |x'|)} \sum_{s = 0}^{t-1} (t-s)^{-\frac{2a+1}{2}} &\leq  C (\ep |x-x'|)^{2 a}  (\ep^2 t)^{\frac{1-2 a}{2}} e^{2 u\ep (|x| + |x'|)},\\
\numberthis \label{eq:tightness:tempb}
&\leq C (\ep|x-x'|)^{2 a} e^{2 u\ep (|x| + |x'|)}.
\end{align*}
We conclude \eqref{eq:tightness2} via combining \eqref{eq:tightness:tempa} and \eqref{eq:tightness:tempb}.
\bigskip
\\
Finally, we justify \eqref{eq:tightness3}, we have
\begin{equation*}
Z(t, x) - Z(t', x) = \sum_{y \in \Xi(t')} \tonepart(t, t', x-y) (Z(t', y) -  Z(t', x)) + Z_{mg}(t, t'),
\end{equation*}
where $Z_{mg} (t, t') = \sum_{s = t'}^{t-1} \sum_{y \in \Xi(s)} \tonepart(t-s-1, x+ \mu(s) - y) M(s, y)$.
Similar to the previous proof, we have 
\begin{equation}\label{eq:tightness:tempc}
\norm{Z(t, x) - Z(t', x)}_{2n}^2 \leq 2 \sum_{y \in \Xi(t')} \tonepart(t, t', x-y) \norm{Z(t', y) - Z(t', x)}_{2n}^2 + 2 \norm{Z_{mg}(t, t')}_{2n}^2.
\end{equation}
For the first term on the RHS of \eqref{eq:tightness:tempc}, we apply \eqref{eq:tightness2} and Lemma \ref{thm:onepartestimate} (iii), for any $a \in (0, \frac{1}{2})$,
\begin{align*}
\sum_{y \in \Xi(t')} \tonepart(t, t', x-y) \norm{Z(t', y) - Z(t', x)}_{2n}^2 &\leq  C \ep^{2a}  \sum_{y \in \Xi(t')} \tonepart(t, t', x-y) |x-y|^{2a} e^{u\ep(|x| + |y|)}\\
&\leq C\ep^{2a} (t -t' +1)^{a} e^{2u\ep |x|}.
\end{align*}
For the second term, invoking Lemma \ref{lem:burkholder} gives 
\begin{align*}
\norm{Z_{mg}(t, t')}_{2n}^2 &\leq C\ep \sum_{s = t'}^{t-1} \frac{1}{\sqrt{t-s}} \sum_{y \in \Xi(s)}  \tonepart(t-s-1, x+\mu(s) -y) \norm{Z(s, y)}_{2n}^2\\
\numberthis \label{eq:five1:temp3}
&\leq C \ep e^{2 u\ep |x|} \sum_{s=t'}^{t-1} \frac{1}{\sqrt{t-s}} \leq C (\ep^2 (t-t'))^{\frac{1}{2}} e^{2u\ep |x|}.
\end{align*}
Combining \eqref{eq:tightness:tempc}-\eqref{eq:five1:temp3}, we obtain $\norm{Z(t, x) - Z(t', x)}_{2n} \leq C (\ep^2 (t-t'))^{\frac{a}{2}} e^{u\ep |x|}.$ We complete the proof of Proposition \ref{prop:tightness}. 
\end{proof}
Having shown the tightness of $Z_\epsilon (\epsilon^{-2} \cdot, \epsilon^{-1} \cdot)$, to prove Theorem 5.6, it suffices to show that any limit point $\mathcal{Z}$ of $Z_\epsilon (\epsilon^{-2} \cdot, \epsilon^{-1} \cdot)$ is the mild solution to the SHE \eqref{she}. This is the goal of the Section \ref{sec:martingalep} and Section \ref{sec:proofthm}, where we will formulate the notion of ``solution to the martingale problem" (which is equivalent to the mild solution) and prove that any limit point of $Z_\epsilon (\epsilon^{-2} \cdot, \epsilon^{-1} \cdot)$ satisfies the martingale problem. 
\subsection{The martingale problem}\label{sec:martingalep}
We recall the martingale problem of the SHE from \cite{BG97}.
\begin{defin}\label{def:martingalep}
We say that a $\funcspace$-valued process $\she(t, x)$ is a solution of martingale problem of the \ac{SHE} \eqref{she}
\begin{equation*}
\pa_t \she(t,x) = \frac{V_*}{2} \pa_x^2 \she(t,x) + \sqrt{D_*} \noise \she(t,x)
\end{equation*}
with initial condition $\she^{ic} \in C(\RR)$ if $\she(0, x) = \she^{ic} (x)$ in distribution and
\bigskip
\\
(i) Given any $T > 0$, there exists $u < \infty$ such that 
\begin{equation*}\label{eq:main1init1}
\sup_{t \in  [0, T]} \sup_{x \in \RR} e^{-u |x|} \EE\big[\she (t, x)^2\big] < \infty.
\end{equation*} 
(ii) For any test function $\psi \in C^\infty_c (\RR)$,
\begin{equation*}
\M_\psi (t) = \int_\RR \she(t, x) \psi(x) dx - \int_{\RR} \she(0, x) \psi(x) dx - \frac{V_*}{2} \int_0^t \int_{\RR} \she(s, x) \psi''(x) dx ds 
\end{equation*}
is a local martingale.
\bigskip
\\
(iii) For any test function $\psi \in C^\infty_c (\RR)$,
\begin{equation*}
\Q_\psi (t) = \M_\psi (t)^2 - D_* \int_0^t \int_\RR \she(s, x)^2 \psi(x)^2 dx ds
\end{equation*}
is a local martingale.
\end{defin}
\cite[Proposition 4.11]{BG97} proves the the solution $\mathcal{Z}$ to the martingale problem is also the weak solution (equivalently, the mild solution) to the SHE. Moreover, they show that there is a unique such solution.  
\bigskip
\\
To prove Theorem \ref{thm:main1}, it suffices to prove that any limit point of $Z_\ep (\ep^{-2} \cdot, \ep^{-1} \cdot)$ satisfies (i), (ii), (iii). We will do it in the next section. The main difficulty arises for justifying the quadratic martingale problem  (iii), we need the following proposition, whose proof is postponed to Section \ref{sec:timedecorr}.
\begin{prop}\label{prop:timedecorr}
For $s \in \mathbb{Z}_{\geq 0}$, define 
\begin{equation}\label{eq:tau}
\cas = \frac{\rho(I-\rho)}{I^2} \cdot \frac{b (I +2 \mod(s) +1) - (I+2\mod(s) - 1)}{b(I + 2\mod(s)) - (I + 2\mod(s) - 2)}.
\end{equation}
Start the unfused \ac{SHS6V} model from near stationary initial condition, for given $T > 0$, there exists constant $C$ and $u$ such that (recall the expressions $\Theta_1$ and $\Theta_2$ from \eqref{eq:Thetaone})
\begin{equation}\label{eq:five:timedecorr}
\bignorm{ \epsilon^2 \sum_{s = 0}^{t} \bigg(\epsilon^{-1} \Theta_1 \Theta_2 - \cas  Z^2 \bigg)(s, \xstar - \hmu(s) + \floor{\hmu(s)})}_2 \leq C \epsilon^{\frac{1}{4}} e^{u \epsilon |\xstar|}
\end{equation}
for all $t \in [0, \epsilon^{-2} T] \cap \ZZ$, $\xstar \in \ZZ$ and $\epsilon > 0$ small enough.  
\end{prop}
\begin{remark}
In \eqref{eq:five:timedecorr}, we compensate the space variable $\xstar \in \ZZ$ by $\hmu(s) - \floor{\hmu(s)} \in [0, 1)$ to ensure that $\xstar - \hmu(s) + \floor{\hmu(s)} \in \Xi(s)$. 
\end{remark}
\subsection{Proof of Theorem \ref{thm:main1}}\label{sec:proofthm} 
The entire section is devoted to the proof of Theorem \ref{thm:main1}. As we mentioned earlier, due to the tightness obtained in Proposition \ref{prop:tightness}, if suffices to prove that for any limit point $\ZZZ$ of $Z_\epsilon (\epsilon^{-2} \cdot, \epsilon^{-1} \cdot)$ satisfies the martingale problem. The proof is accomplished once we verify (i), (ii), (iii) for $\mathcal{Z}$. 
\bigskip
\\
For the ensuing discussion, we denote by $\err(t)$ to be a generic process (which may differ from line to line) satisfying for all fixed $T > 0$
$$\lim_{\ep \downarrow 0} \sup_{t \in [0, \ep^{-2} T] \cap \ZZ} \norm{\err(t)}_2 = 0.$$ 
We start by verifying (i). Due to  \eqref{eq:tightness1} and $Z_\ep (\ep^{-2} t, \ep^{-1} x) \wc \ZZZ(t, x)$, by Skorohod representation theorem and Fatou's lemma, (i) holds.
\bigskip
\\
We continue to prove (ii). To show that $\M_\psi (t)$ is a local martingale, we consider a discrete analogue. Define
\begin{equation}\label{eq:five:martingale}
M_\psi (t) := \epsilon \sum_{s = 0}^{t-1} \sum_{x \in \Xi(s)} M(s, x) \psi(\epsilon (x- \mu(s))).
\end{equation} 
Due to Proposition \ref{prop:microscopicshe}, $M(t, x)$ is a $\FF(t)$-martingale increment, which implies $M_\psi (t)$ is a $\FF(t)$-martingale. 
\bigskip
\\
Define $\langle Z(t), \psi \rangle_\epsilon := \sum_{x \in \Xi(t)} \epsilon \psi (\epsilon x) Z(t, x)$. By \eqref{eq:microshe},
\begin{equation*}
Z(t, x) = \sum_{y \in \Xi(t-1)} \tonepart_\epsilon (t, t-1, x - y) Z(t-1, y) + M(t-1, x+\mu(t-1)), \quad x \in \Xi(t),
\end{equation*}
we obtain  
\begin{align*}
&\anglebrack{Z(s), \psi}_\epsilon  - \anglebrack{Z(s-1), \psi}_\epsilon 
= \sum_{x \in \Xi(t)} \epsilon \psi(\epsilon x) Z(t, x) - \sum_{y \in \Xi(t-1)} \epsilon \psi(\epsilon y) Z(t-1, y) \\ 
&= \sum_{x \in \Xi(s)} \epsilon \psi(\epsilon x) \Big(\sum_{y \in \Xi(s-1)} \tonepart_\epsilon (s, s-1, x - y) Z(s-1, y) + M(s-1, x+\mu(s-1)  ) \Big) - \sum_{y \in \Xi(s-1)} \epsilon \psi(\epsilon y) Z(s-1, y)\\
\numberthis \label{eq:five:temp2}
&= \sum_{y \in \Xi(s-1)} \epsilon  Z(s-1, y) \Big(\sum_{x \in \Xi(s)} \tonepart_\epsilon (s, s-1, x  - y) \big(\psi(\epsilon x) - \psi(\epsilon y)\big)\Big) +  \sum_{x \in \Xi(s)} \epsilon \psi(\epsilon x) M(s-1, x+\mu(s-1))
\end{align*}
Summing \eqref{eq:five:temp2} over $s \in [1, t] \cap \ZZ$ yields
\begin{align}\label{eq:five:temp3}
M_\psi (t) = \langle Z(t), \psi \rangle_\epsilon - \langle Z(0), \psi \rangle_\epsilon - \sum_{s=0}^{t-1} \epsilon \sum_{y \in \Xi(s)} Z(s, y) \Big(\sum_{x \in \Xi(s+1)} \tonepart_\epsilon  (s+1, s, x-y) (\psi(\epsilon x) - \psi(\epsilon y)) \Big)
\end{align}
Recall that  $R_\ep (s)$ is the random variable  defined in \eqref{eq:rwdistribution}, as usual we put on the subscript $\ep$ to emphasize the dependence. Note that,  
\begin{equation*}
\EE\big[R_\ep (s)\big] = \sum_{x \in \Xi(1)} \toneparte(s +1, s, x) x = 0, \qquad 
\var\big[R_\ep (s)\big] = \sum_{x \in \Xi(1)} \toneparte(s+1, s, x) x^2.
\end{equation*}
By Taylor expansion 
$$\psi(\ep x) = \psi(\ep y) + \ep \psi'(\ep y) (x- y) + \frac{1}{2} \ep^2 \psi''(\ep y) (x - y)^2 + \ep^3 \OO(|x- y|^3),$$
whereby  \eqref{eq:five:temp3} becomes 
\begin{align*}
M_\psi (t) =\langle Z(t), \psi \rangle_\epsilon - \langle Z(0), \psi \rangle_\epsilon - \frac{1}{2} \epsilon^2  \sum_{s = 0}^{t-1} \var\big[R_\ep (s)\big] \langle  Z(s), \psi''\rangle_\epsilon + \err(t).
\end{align*} 
Furthermore, we have 
\begin{align*}
\var\big[R_\ep (s)\big] 
&=\lambda(s) \sum_{n=1}^{\infty} \frac{\alpha(s)(1-q)}{1 + \alphat{\alpha}{s}} \bigg(1- \frac{\nu + \alphat{\alpha}{s}}{1 + \alphat{\alpha}{s}}\bigg) \bigg(\frac{\nu +\alphat{\alpha}{s}}{1 + \alphat{\alpha}{s}}\bigg)^{n-1}  q^{\rho n} n^2 \\
&\quad- \bigg(\lambda(s) \sum_{n=1}^{\infty} \frac{\alpha(s)(1-q)}{1 + \alphat{\alpha}{s}} \bigg(1- \frac{\nu + \alphat{\alpha}{s}}{1 + \alphat{\alpha}{s}}\bigg) \bigg(\frac{\nu +\alphat{\alpha}{s}}{1 + \alphat{\alpha}{s}}\bigg)^{n-1}  q^{\rho n} n \bigg)^2\\
\numberthis \label{eq:five:temp6}
&= \frac{(I+1+2\mod(s)) b - (I + 2\mod(s) - 1)}{I^2 (1-b)} + \OO(\ep^{\frac{1}{2}}).
\end{align*}
In the last line, we used Lemma \ref{lem:wsc} to get asymptotics.
Denote by
\begin{equation*}
V(s) = \frac{(I+1+2\mod(s)) b - (I + 2\mod(s) - 1)}{I^2 (1-b)}
\end{equation*} 
Then 
\begin{align*}
M_\psi (t) =\langle Z(t), \psi \rangle_\epsilon - \langle Z(0), \psi \rangle_\epsilon - \frac{1}{2} \epsilon^2  \sum_{s = 0}^{t-1} V(s) \langle  Z(s), \psi''\rangle_\epsilon + \err(t).
\end{align*} 
Note that $\{V(s)\}_{s=0}^\infty$ is a periodic sequence with period $J$, by the time regularity of $Z(t, x)$ in \eqref{eq:tightness3}, we can replace $V(s)$ by  $$V_* = \frac{1}{J} \sum_{s = 0}^{J-1} V (s) = \frac{(I +J) b - (I +J-2)}{I^2 (1-b)}$$ 
as defined in \eqref{eq:vstar}. Consequently,
\begin{equation*}
M_\psi (t) =\langle Z(t), \psi \rangle_\epsilon - \langle Z(0), \psi \rangle_\epsilon - \frac{1}{2} \epsilon^2 V_* \sum_{s = 0}^{t-1} \langle Z(s), \psi''\rangle_\epsilon + \err(t).
\end{equation*}
Since $\lim_{\ep \downarrow 0} \sup_{t \in [0, \ep^{-2} T] \cap \ZZ} \norm{\err(t)}_2 = 0$, by a standard discrete to continuous argument from the martingale $M_{\psi}(t)$ to $\M_{\psi} (t)$,
we conclude that $\M_\psi (t)$ is a local martingale.
\bigskip
\\
We finish the proof of (iii) based on  Proposition \ref{prop:timedecorr}. 
Similar to what we did in proving (ii), we want to find a discrete approximation of $\Q_\psi (t)$. This is given by $M_\psi- \bracket{M_\psi }(t)$. Referring to \eqref{eq:five:martingale}, the martingale $M_\psi (t)$ possesses the quadratic variation 
\begin{align*}
\bracket{M_\psi }(t) &=\epsilon^2 \sum_{s = 0}^{t-1} \sum_{x, x' \in \Xi(s)} \psi(\epsilon (x-\mu(s))) \psi(\epsilon (x' - \mu(s))) \EE\big[M(s, x) M(s, x')\vv \FF(s)\big] \\
\numberthis \label{eq:five:temp4}
& = \epsilon^2 \sum_{s=0}^{t-1} \sum_{x, x' \in \Xi(s)}  \psi(\epsilon(x-\mu(s))) \psi(\epsilon(x' - \mu(s))) \bigg(\frac{\nu + \alpha(s)}{1 + \alpha(s)} q^{\rho}\bigg)^{|x-x'|} \Theta_1 (s, x\wedge x') \Theta_2 (s, x\wedge x')
\end{align*}
where the last equality follows from Proposition \ref{prop:microscopicshe}.
Since $\psi \in \test$, there exists a constant $C$ such that $$\big|\psi(\epsilon(x - \mu(s))) \psi(\epsilon (x' - \mu(s))) - \psi(\epsilon (x \wedge x'))^2 \big| \leq C \epsilon (|x - x'| +1) $$
Consequently, the expression \eqref{eq:five:temp4} is well-approximated with the corresponding term $\psi(\epsilon (x-\mu(s))) \psi(\epsilon (x'- \mu(s)))$ replaced by $\psi (\epsilon (x \wedge x')) \psi(\epsilon (x' \wedge x'))$, which yields
\begin{align*}
\bracket{M_\psi} (t) &= \epsilon^2 \sum_{s=0}^{t-1} \sum_{x, x' \in \Xi(s)} \psi(\epsilon (x \wedge x'))^2 \bigg(\frac{\nu + \alpha(s)}{1 + \alpha(s)} q^\rho\bigg)^{|x- x'|} \Theta_1 (s, x \wedge x') \Theta_2 (s, x \wedge x') + \err(t),\\
&= \epsilon^2 \sum_{s = 0}^{t-1} \sum_{x \in \Xi(s)}  \sum_{n=-\infty}^{\infty} \bigg(\frac{\nu + \alpha(s)}{1 + \alpha(s)} q^\rho\bigg)^{|n|} \psi(\epsilon x)^2 \Theta_{1} (s, x) \Theta_2 (s, x) + \err(t),
\\
&=\ep^2 \sum_{s=0}^{t-1} \sum_{x \in \Xi(s)} \frac{1 +\alpha(s) + (\nu + \alpha(s)) q^\rho}{1 + \alpha(s) - (\nu + \alpha(s)) q^\rho} \psi(\epsilon x)^2 \Theta_{1} (s, x) \Theta_2 (s, x) + \err(t),
\\
\numberthis
\label{eq:five:temp5}
&=\epsilon^2  \sum_{s = 0}^{t-1} \frac{b (I + 2 \mod(s))  -(I + 2 \mod(s)-2)}{I (1-b)} \sum_{x\in \Xi(s)}  \ep \psi(\epsilon x)^2 \big(\ep^{-1} \Theta_{1} (s, x) \Theta_2 (s, x)\big) + \err(t).
\end{align*}
Here, in the third equality we used $\sum_{n = -\infty}^{\infty} x^{-|n|} = \frac{1 + x}{1 - x}$. In the last equality, using Lemma \ref{lem:wsc} for asymptotics expansion of $\frac{\nu + \alpha(s)}{1 + \alpha(s)}$, one has
\begin{equation*}
\frac{1 + \alpha(s) + (\nu + \alpha(s)) q^{\rho}}{1 + \alpha(s) - (\nu+ \alpha(s)) q^{\rho} } =  \frac{1 + \frac{\nu + \alpha(s)}{1 + \alpha(s)} q^{\rho}}{1 - \frac{\nu + \alpha(s)}{1 + \alpha(s)} q^{\rho}} =  \frac{b (I + 2 \mod(s))  -(I + 2 \mod(s)-2)}{I (1-b)}  + \OO(\ep^{\frac{1}{2}}).
\end{equation*}
Using Proposition \ref{prop:timedecorr}, we replace the term $\ep^{-1} \Theta_1 (s, x) \Theta_2 (s, x)$ in  \eqref{eq:five:temp5} with $ \cas Z(s, x)^2$, 
\begin{align*}
\bracket{M_\psi} (t) &= \epsilon^2  \sum_{s=0}^{t-1} \frac{b (I + 2 \mod(s))  -(I + 2 \mod(s)-2)}{I (1-b)}  \sum_{x \in \Xi(s)} \epsilon \psi(\epsilon x)^2 \cas Z (s, x)^2 + \err(t),\\
&=\ep^2 \sum_{s = 0}^{t-1} \frac{\rho(I-\rho)}{I^2} \cdot \frac{b (I +2 \mod(s) +1) - (I+2\mod(s) - 1)}{I (1-b)} \sum_{x \in \Xi(s)} \epsilon\psi(\epsilon x)^2 Z (s, x)^2 + \err(t).
\end{align*} 
Using again the time regularity of $Z(t, x)$ in \eqref{eq:tightness3}, we conclude that 
\begin{equation*}
\bracket{M_\psi} (t) = D_* \sum_{s= 0}^{t-1} \sum_{x \in \Xi(s)} \epsilon \psi(\epsilon x)^2 Z(s, x)^2 + \err(t),
\end{equation*}
where $$D_* = \frac{1}{J} \sum_{s= 0}^{J-1} \frac{\rho (I - \rho)}{I^2} \cdot \frac{b (I +2 \mod(s) +1) - (I+2\mod(s) - 1)}{I (1-b)}  = \frac{\rho (I - \rho)}{I} \frac{(I +J) b - (I +J-2)}{I^2 (1-b)}$$  
as defined in \eqref{eq:dstar}. Via a standard discrete to continuous argument from the martingale $M_\psi(t) - \langle M_\psi \rangle(t)$ to $\Q_\psi (t)$,
we conclude that $\Q_\psi (t)$ is a local martingale. Since we have proved that for any limit point $\mathcal{Z}$ of $Z_{\ep} (\ep^{-2} \cdot, \ep^{-1} \cdot)$, it satisfies (i), (ii), (iii) in Definition \ref{def:martingalep}, this concludes the proof of Theorem \ref{thm:main1}.
\section{Estimate of the two particle transition probability}\label{sec:asymptotic analysis}
In this section, we prove a space-time estimate for the (tilted) two particle transition probability $\rhzrte$, using the integral formula  provided in Corollary \ref{cor:intformulav}. This technical result is crucial to the proof of Proposition \ref{prop:timedecorr}.
\bigskip
\\
Recall from Corollary \ref{cor:intformulav} that
\begin{align*}
&\rhzrte\big((x_1, x_2), (y_1, y_2), t, s\big) \\
&= c(y_1, y_2) \bigg[ \oint_{\lc} \oint_{\lc} \prod_{i=1}^2 \big(\coreej (z_i)\big)^{\floor{\frac{t-s}{J}}} \reme (z_i, t, s) z_i^{x_i - y_i } \frac{dz_i}{2\pi \im z_i} - \oint_{\lc} \oint_{\lc} \interacte(z_1, z_2) \prod_{i=1}^2 \big(\coreej (z_i)\big)^{\floor{\frac{t-s}{J}}} \reme (z_i, t, s) z_i^{x_{3-i} - y_i } \frac{dz_i}{2\pi \im z_i} \\
\numberthis \label{eq:six:temp8}
&\quad +\res_{z_1 = \pole(z_2)} \oint_{\lc} \oint_{\lc} \interacte(z_1, z_2) \big(\coreej (z_i)\big)^{\floor{\frac{t-s}{J}}} \reme (z_i, t, s) z_i^{x_{3-i} - y_i } \frac{dz_i}{2\pi \im z_i}\bigg],
\end{align*}
where $\C_R$ is a circle centered at zero with a large enough radius $R$ so as to include all the poles of the integrand, $c(y_1, y_2)$ is defined in \eqref{eq:cy1} and the functions in the integrand above are defined respectively in \eqref{eq:four2:temp2} - \eqref{eq:pole1expression}.
We put $\ep$ in the notation of $\rhzrte$ and other functions to emphasize the dependence on $\ep$ under the weakly asymmetry scaling. 
\bigskip
\\
We define the discrete gradients $\na_{x_1}, \na_{x_2}, \na_{y_1}, \na_{y_2}$
\begin{align*}
\nabla_{x_1} \rhzrte \big((x_1, x_2), (y_1, y_2), t, s\big) &= \rhzrte\big((x_1 + 1, x_2), (y_1, y_2), t, s\big) - \rhzrte\big((x_1, x_2), (y_1, y_2), t, s\big),\\
\nabla_{x_2} \rhzrte\big((x_1, x_2), (y_1, y_2), t) &= \rhzrte\big((x_1, x_2 + 1), (y_1, y_2), t, s\big) - \rhzrte\big((x_1, x_2), (y_1, y_2), t, s\big),\\
\nabla_{y_1} \rhzrte\big((x_1, x_2), (y_1, y_2), t, s\big) &= \rhzrte\big((x_1, x_2), (y_1 + 1, y_2), t, s\big) - \rhzrte\big((x_1, x_2), (y_1, y_2), t, s\big),\\
\nabla_{y_2} \rhzrte\big((x_1, x_2), (y_1, y_2), t, s\big) &= \rhzrte\big((x_1, x_2), (y_1, y_2 + 1), t, s\big) - \rhzrte\big((x_1, x_2), (y_1, y_2), t, s\big).
\end{align*}
Furthermore, we define the mixed discrete gradient 
\begin{align*}
\na_{x_1, x_2} \rhzrte\big((x_1, x_2), (y_1, y_2), t, s\big) &= \na_{x_2} \Big(\nabla_{x_1} \rhzrte \big((x_1, x_2), (y_1, y_2), t, s\big)\Big)\\
&= \rhzrte\big((x_1 +1, x_2 +1), (y_1, y_2), t, s\big) - \rhzrte\big((x_1 +1, x_2), (y_1, y_2), t, s\big) \\ 
&\quad -\rhzrte\big((x_1, x_2 +1), (y_1, y_2), t, s\big) +\rhzrte\big((x_1, x_2), (y_1, y_2), t, s\big)
\end{align*}
We define the $\na$-Weyl chamber (which is understood with respect to whichever gradient is taken) to be 
\begin{equation}
\begin{aligned}
&\{(x_1, x_2, y_1, y_2): x_1 +1 \leq x_2 \in \Xi(t), y_1 \leq y_2 \in \Xi(s)\} \qquad &\text{if } \na = \na_{x_1}, \\
&\{(x_1, x_2, y_1, y_2): x_1  \leq x_2 \in \Xi(t), y_1 \leq y_2 \in \Xi(s)\} \qquad &\text{if } \na = \na_{x_2},\\
&\{(x_1, x_2, y_1, y_2): x_1 \leq x_2 \in \Xi(t), y_1 +1 < y_2 \in \Xi(s)\} \qquad &\text{if } \na = \na_{y_1},\\
&\{(x_1, x_2, y_1, y_2): x_1 \leq x_2 \in \Xi(t), y_1 \leq y_2 \in \Xi(s)\} \qquad &\text{if } \na = \na_{y_2}.
\label{eq:deltaweyl}
\end{aligned}
\end{equation}
We remark that $\rhzrte\big((x_1, x_2), (y_1, y_2), t, s\big)$ is defined only for $x_1 \leq x_2 \in \Xi(t)$ and $y_1 \leq y_2 \in \Xi(s)$. In the definition of $\na$-Weyl chamber, when
$\na = \na_{x_1}, \na_{x_2}, \na_{y_2}$, the corresponding $\na$-Weyl chamber is exactly where the quantities $\na_{x_1} \rhzrte$, $\na_{x_2} \rhzrte$ or $\na_{y_2} \rhzrte$ are well defined.  But for $\na = \na_{y_1}$, we require $y_1 + 1 < y_2$, which is stronger than $y_1 + 1 \leq y_2$  (where $\na_{y_1} \rhzrte$ is well defined). The motivation of this requirement is to ensure that \eqref{eq:six:temp4} holds.
\bigskip
\\
The following result is the main technical contribution of our paper.
\begin{prop}\label{prop:semiestimate}
For all fixed $\beta, T > 0$, there exists positive constant $C(\beta), C(\beta, T)$  such that for $\ep > 0$ small enough and $s \leq t  \in [0, \ep^{-2} T] \cap \ZZ$ \\
$(a)$ For all $x_1 \leq x_2 \in \Xi(t)$ and $y_1 \leq y_2 \in \Xi(s)$,
\begin{align}
\numberthis \label{eq:six1:temp2}
\big|\rhzrte\big((x_1, x_2), (y_1, y_2), t, s\big)\big| \leq \frac{C(\beta, T)}{t -s +1} e^{-\frac{\beta(|x_1 - y_1| + |x_2 - y_2|)}{\sqrt{t - s +1} + C(\beta)}}.
\end{align}
$(b)$ For all $(x_1, x_2, y_1, y_2)$ in the $\na$-Weyl chamber,
\begin{align*}
& \big|\nabla_{x_i} \rhzrte\big((x_1, x_2), (y_1, y_2), t, s\big)\big| \leq \frac{C(\beta, T)}{(t - s+1)^{\frac{3}{2}}} e^{-\frac{\beta(|x_1 - y_1| + |x_2 - y_2|)}{\sqrt{t - s +1} + C(\beta)}}, \quad i=1, 2,\\
&\big|\nabla_{y_i} \rhzrte\big((x_1, x_2), (y_1, y_2), t, s\big)\big| \leq \frac{C(\beta, T)}{(t - s + 1)^{\frac{3}{2}}} e^{-\frac{\beta(|x_1 - y_1| + |x_2 - y_2|)}{\sqrt{t - s +1} + C(\beta)}}. \quad i=1, 2.
\end{align*}
(c) For all $x_1 < x_2 \in \Xi(t)$ and $y_1 \leq y_2 \in \Xi(s)$,
\begin{align*}
&\big|\nabla_{x_1, x_2}\rhzrte\big((x_1, x_2), (y_1, y_2), t, s\big) \big| \leq \frac{C(\beta, T)}{(t - s + 1)^2} e^{-\frac{\beta(|x_1 - y_1| + |x_2 - y_2|)}{\sqrt{t-s+1} + C(\beta)}}.
\end{align*}
\end{prop}
 It is helpful to divide the proof of Proposition \ref{prop:semiestimate} depending on whether the time increment $t - s$ is large enough. More precisely, we use the phrase $t- s$ is \emph{large enough} if the referred statement holds for all $t- s \geq t_0$, where $t_0$ is some generic time threshold which may change from line to line (depend on $\beta$ and $T$, but does not depend on $\ep$). Note that this is not to be confused with the global assumption $0 \leq s \leq t \leq \ep^{-2} T$, which implies $t - s \leq \ep^{-2} T$. 
 \bigskip
 \\
 Given arbitrary fixed $t_0 > 0$, let us first prove the proposition for $t - s \leq t_0$. 
\begin{proof}[Proof of Proposition \ref{prop:semiestimate} for $t-s \leq t_0$]
According to Lemma \ref{lem:wsc}, 
\begin{equation}\label{eq:temp16}
\lim_{\ep \downarrow 0} \sup_{t \in \ZZ_{\geq 0}} \frac{\nu + \alpha(t)}{1 + \alpha(t)} = \sup_{t \in \ZZ_{\geq 0}} \frac{(I+ \mod(t)) b - (I + \mod(t) - 1)}{\mod(t) b - (\mod(t) - 1)} < 1,
\end{equation}
here we used the condition $\frac{I+J-2}{I+J-1} < b < 1$ in \eqref{eq:scaling}.
Taking $k = 2$ in \eqref{eq:three:temp6} yields
\begin{equation}
\rhzr\big((x_1, x_2), (y_1, y_2), t, s\big) \leq C \prod_{i=1}^2 \binom{|x_i - y_i| + t - s}{t - s} \theta^{|x_i  -y_i|}
\end{equation}
where $\theta = \sup_{t \in \ZZ_{\geq 0}} \frac{\nu + \alpha(t)}{1 + \alpha(t)}$. So there exists $0< \delta <1$ such that for $\ep$ small enough and all $s \leq t$ such that  $t-s \leq t_0$
\begin{equation}\label{eq:six:temp7}
\rhzr \big((x_1, x_2), (y_1, y_2), t, s\big) \leq C \delta^{|x_i - y_i|},
\end{equation}
Referring to the relation \eqref{eq:four:tilttransition} between $\rhzrt$ and $\rhzr$.
By $\lim_{\ep \downarrow 0 } e^{\sqrt{\ep}}= 1$ along with \eqref{eq:six:temp7}, there exists $0 < \delta' < 1$ s.t.
\begin{equation*}
\rhzrte \big((x_1, x_2), (y_1, y_2), t, s\big) \leq C {\delta'}^{|x_1 - y_1| + |x_2 - y_2|}.
\end{equation*}
Consequently, we can take $C(\beta, T)$ and $C(\beta)$ in \eqref{eq:six1:temp2} large enough such that for $t - s \leq t_0$,
\begin{align*}
\rhzrte \big((x_1, x_2), (y_1, y_2), t, s\big) \leq C {\delta'}^{|x_1 - y_1| + |x_2 - y_2|} &\leq \frac{C(\beta, T)}{t_0 +1} e^{-\frac{\beta (|x_1 - y_1| + |x_2 - y_2|)}{\sqrt{t_0 +1} + C(\beta)}}\\ &\leq \frac{C(\beta, T)}{t - s+1} e^{-\frac{\beta (|x_1 - y_1| + |x_2 - y_2|)}{\sqrt{t - s +1} + C(\beta)}}
\end{align*}
For the gradients, let us consider $\na_{x_1} \rhzrte$ for example. Note that  
\begin{equation*}
\na_{x_1} \rhzrte \big((x_1, x_2), (y_1, y_2), t, s\big) = \rhzrte \big((x_1 + 1, x_2), (y_1, y_2), t, s\big) - \rhzrte \big((x_1, x_2), (y_1, y_2), t, s\big)
\end{equation*}
Using the same argument as above, there exists constant $C(\beta, T)$ and $C(\beta)$ such that for all $s \leq t$ satisfying $t - s \leq t_0$,
\begin{equation*}
\rhzrte \big((x_1, x_2), (y_1, y_2), t, s\big), \rhzrte \big((x_1 + 1, x_2), (y_1, y_2), t, s\big) \leq \frac{C(\beta, T)}{(t - s + 1)^{\frac{3}{2}}} e^{-\frac{\beta (|x_1 - y_1| + |x_2 - y_2|)}{\sqrt{t - s +1} + C(\beta)}},
\end{equation*}
which gives the desired bound for $\na_{x_1} \rhzrte \big((x_1, x_2), (y_1, y_2), t, s\big)$. The argument for the gradient $\na_{x_2} \rhzrte, \na_{y_1} \rhzrte, \na_{y_2} \rhzrte$ and $\na_{x_1, x_2} \rhzrte$ is similar. 
\end{proof}
Having proved Proposition \ref{prop:semiestimate} for $t- s \leq t_0$, it suffices to prove the same proposition for $t - s$ large enough. In other words, we need to show that there exists $t_0 > 0$ such that the proposition holds for $t-s \geq t_0$. We decompose $\rhzrte$ \eqref{eq:six:temp8} by $$\rhzrte = c(y_1, y_2)  \big(\rhzrfr - \rhzrin\big),$$ where 
\begin{align*}
\numberthis
\label{eq:six:freepart}
\rhzrfr\big((x_1, x_2), (y_1, y_2), t, s\big) &:=  \oint_{\lc} \oint_{\lc} \prod_{i=1}^2 \big(\coreej (z_i)\big)^{\floor{\frac{t-s}{J}}} \reme (z_i, t, s) z_i^{x_i - y_i } \frac{dz_i}{2\pi \im z_i},\\  
\rhzrin\big((x_1, x_2), (y_1, y_2), t, s\big) &:=  \oint_{\lc} \oint_{\lc} \big(\coreej (z_i)\big)^{\floor{\frac{t-s}{J}}} \reme (z_i, t, s) z_i^{x_{3-i} - y_i } \frac{dz_i}{2\pi \im z_i} \\
\numberthis \label{eq:six:intpart}
&\quad-\res_{z_1 = \pole(z_2)} \oint_{\lc} \oint_{\lc} \interacte(z_1, z_2) \prod_{i=1}^2 \big(\coreej (z_i)\big)^{\floor{\frac{t-s}{J}}} \reme (z_i, t, s) z_i^{x_{3-i} - y_i } \frac{dz_i}{2\pi \im z_i}.
\end{align*}
Referring to \eqref{eq:cy1}, $c(y_1, y_2)$ equals $1$ as long as $y_1 < y_2$. It is straightforward that for $(x_1, x_2, y_1, y_2)$ in the $\na$-Weyl chamber \eqref{eq:deltaweyl}, 
\begin{align*}
\na_{x_i} \rhzrte &= c(y_1, y_2) \big(\na_{x_i} \rhzrfr - \na_{x_i} \rhzrin\big),\\ 
\numberthis \label{eq:six:temp4}
\na_{y_i} \rhzrte &= c(y_1, y_2) \big(\na_{y_i} \rhzrfr - \na_{y_i} \rhzrin\big). 
\end{align*}
In addition, for $x_1 + 1 \leq x_2 \in \Xi(t)$ and $y_1 \leq y_2 \in \Xi(s)$,
\begin{equation*}
\na_{x_1, x_2} \rhzrte = c(y_1, y_2) \big(\na_{x_1, x_2} \rhzrfr - \na_{x_1, x_2} \rhzrin\big).
\end{equation*}
Note that under weakly asymmetric scaling, 
$$\lim_{\ep \downarrow 0} c(y_1, y_2) = \idc_{\{y_1 < y_2\}} + \frac{I-1}{2I} \idc_{\{y_1 = y_2\}},$$
which implies that $c(y_1, y_2)$ is uniformly bounded for $\ep$ small enough,
This being the case, to prove Proposition \ref{prop:semiestimate} for $t-s$ large enough, it suffices to prove the same result for $\rhzrfr$ and $\rhzrin$ respectively.
\begin{prop}\label{prop:semiestimatefr}
For all $\beta, T > 0$, there exists positive constant $t_0:= t_0 (\beta, T)$ and $C(\beta, T)$  such that for $\ep > 0$ small enough  and $0 \leq s \leq t  \in [0, \ep^{-2} T] \cap \ZZ$ satisfying $|t - s| \geq t_0$
\\ 
$(a)$ for all $x_1 \leq x_2 \in \Xi(t)$, $y_1 \leq y_2 \in \Xi(s)$
\begin{align*}
&\big|\rhzrfr\big((x_1, x_2), (y_1, y_2), t, s\big)\big| \leq \frac{C(\beta, T)}{t - s +1} e^{-\frac{\beta(|x_1 - y_1| + |x_2 - y_2|)}{\sqrt{t - s +1} }}
\end{align*}
$(b)$ For all $(x_1, x_2, y_1, y_2)$ in the $\na$-Weyl chamber,
\begin{align*}
& \big|\nabla_{x_i} \rhzrfr\big((x_1, x_2), (y_1, y_2), t, s\big)\big| \leq \frac{C(\beta, T)}{(t - s + 1)^{\frac{3}{2}}} e^{-\frac{\beta(|x_1 - y_1| + |x_2 - y_2|)}{\sqrt{t - s +1}}}, \quad i=1, 2,\\
&\big|\nabla_{y_i} \rhzrfr\big((x_1, x_2), (y_1, y_2), t, s\big)\big| \leq \frac{C(\beta, T)}{(t - s + 1)^{\frac{3}{2}}} e^{-\frac{\beta(|x_1 - y_1| + |x_2 - y_2|)}{\sqrt{t - s +1}}}, \quad i=1, 2.
\end{align*}
(c) For all $x_1 + 1 \leq x_2 \in \Xi(t)$ and $y_1 \leq y_2 \in \Xi(s)$,
\begin{align*}
&\big|\nabla_{x_1, x_2}\rhzrfr\big((x_1, x_2), (y_1, y_2), t, s\big) \big| \leq \frac{C(\beta, T)}{(t - s +1)^2} e^{-\frac{\beta(|x_1 - y_1| + |x_2 - y_2|)}{\sqrt{t - s +1}}}.
\end{align*}
\end{prop}
\begin{prop}\label{prop:semiestimateint}
For all $\beta, T > 0$, there exists positive constant $t_0:=t_0(\beta, T)$ and $C(\beta, T)$  such that for $\ep > 0$ small enough $0 \leq s \leq t  \in [0, \ep^{-2} T] \cap \ZZ$ such that $|t - s| \geq t_0$,\\
$(a)$ for all $x_1 \leq x_2 \in \Xi(t)$ and $y_1 \leq y_2 \in \Xi(s)$,
\begin{align*}
&\big|\rhzrin\big((x_1, x_2), (y_1, y_2), t, s\big)\big| \leq \frac{C(\beta, T)}{t - s +1} e^{-\frac{\beta(|x_2 - y_1| + |x_1 - y_2|)}{\sqrt{t - s +1}}}.
\end{align*}
$(b)$ For all $(x_1, x_2, y_1, y_2)$ in the $\na$-Weyl chamber, 
\begin{align*}
& \big|\nabla_{x_i} \rhzrin\big((x_1, x_2), (y_1, y_2), t, s\big)\big| \leq \frac{C(\beta, T)}{(t - s +1)^{\frac{3}{2}}} e^{-\frac{\beta(|x_2 - y_1| + |x_1 - y_2|)}{\sqrt{t - s +1}}}, \quad i=1, 2,\\
&\big|\nabla_{y_i} \rhzrin\big((x_1, x_2), (y_1, y_2), t, s\big)\big| \leq \frac{C(\beta, T)}{(t - s + 1)^{\frac{3}{2}}} e^{-\frac{\beta(|x_2 - y_1| + |x_1 - y_2|)}{\sqrt{t - s +1} }}, \quad i=1, 2.
\end{align*}
(c) For all $x_1 + 1 \leq x_2 \in \Xi(t)$ and $y_1 \leq y_2 \in \Xi(s)$, 
\begin{align*}
&\big|\nabla_{x_1, x_2}\rhzrin((x_1, x_2), (y_1, y_2), t) \big| \leq \frac{C(\beta, T)}{(t - s +1)^2} e^{-\frac{\beta(|x_2 - y_1| + |x_1 - y_2|)}{\sqrt{t+1}}}.
\end{align*}
\end{prop}
The reader might notice that in Proposition \ref{prop:semiestimateint}, we write $|x_2 - y_1| + |x_1 - y_2|$ on the RHS exponents (compared with $|x_1  - y_1| + |x_2  - y_2|$ in Proposition \ref{prop:semiestimate}). This in fact yields a stronger upper bound since by $x_1 \leq x_2$ and $y_1 \leq y_2$, one always has
\begin{equation*}
|x_1 - y_1| + |x_2 - y_2| \leq |x_2 - y_1| + |x_1 - y_2|.
\end{equation*}
Hence, combining Proposition \ref{prop:semiestimatefr} and Proposition \ref{prop:semiestimateint}, we conclude Proposition \ref{prop:semiestimate}.
\subsection{Estimate of $\rhzrfr$}\label{sec:estimatefree}
 In this section, we will prove Proposition \ref{prop:semiestimatefr}. 
Referring to \eqref{eq:intformulaonetpart}, 
\begin{align}\label{eq:six1:temp3}
\toneparte(t, s, x_i - y_i) 
=  \oint_{\lc} \big(\coreej (z_i)\big)^{\floor{(t-s)/J}} \reme(z_i, t, s) z_i^{x_i - y_i} \frac{dz_i}{2 \pi \im z_i}
\end{align}
where $R$ is large enough so that $\C_R$ encircles all the poles of the integrand. Therefore, from \eqref{eq:six:freepart} we have
\begin{align}\label{eq:six:temp3}
\rhzrfr\big((x_1, x_2), (y_1, y_2), t, s\big) = 
 \toneparte\big(t, s, x_1 - y_1\big) \toneparte\big(t, s, x_2 - y_2\big).
\end{align}
To estimate $\rhzrte\big((x_1, x_2), (y_1, y_2), t, s\big)$, it suffices to analyze $\toneparte(t, s, x_i - y_i)$. 
Referring to the expression \eqref{eq:four2:temp2} and \eqref{eq:four2:temp3}, 
\begin{align*}
\numberthis 
\label{eq:six:temp10}
\coreej (z) &:= \lambda z^{\mu}  \frac{(1 + \alpha q^J)q^{-\rho} z - (\nu + \alpha q^J)}{(1 + \alpha) q^{-\rho} z - (\nu + \alpha)},\\ 
\numberthis 
\label{eq:six:temp9}
\reme(z, t, s) &:= \prod_{k = s + J \floor{\frac{t-s}{J}}}^{t-1} \lambda(k) z^{\mu(k)}  \frac{(1 + \alpha(k) q) q^{-\rho} z - (\nu + \alpha(k) q) }{(1 + \alpha(k)) q^{-\rho} z - (\nu + \alpha(k))}.
\end{align*}
Define the set of poles of the integrand in \eqref{eq:six1:temp3} to be $\poleset$, it is clear that \[\mathcal{P} \subseteq \bigcup_{k=0}^{\infty} \{q^{\rho}\frac{\nu + \alpha(k)}{1 + \alpha(k)}\} \cup \{0\} = \bigcup_{k=0}^{J-1} \{q^{\rho}\frac{\nu + \alpha(k)}{1 + \alpha(k)}\} \cup \{0\}.\]
Due to Lemma \ref{lem:wsc},
\begin{equation*}
\lim_{\ep \downarrow 0} \frac{q^\rho (\alpha(k) + \nu)}{1 + \alpha(k)} 
= \frac{(I + \mod(k)) b - (I + \mod(k) - 1)}{b  \mod(k) - (\mod(k) - 1)} \in (0, 1).
\end{equation*}
Therefore, there exists $0 < \Theta < 1$ such that for $\ep$ small enough 
\begin{equation}\label{eq:temp17}
\poleset \subseteq [0, \Theta].
\end{equation}
To extract the spatial decay of $\toneparte(t, s, x_i - y_i)$, we deform the contour of $z_i$ from $\lc$ to $\C_{r_i}$ where 
\begin{equation}\label{eq:six:radius}
r_i  = \rfunc(t - s, -\sign(x_i  - y_i)\beta).
\end{equation}
Note that when $t - s$ is large enough, $r_i$ is close to $1$, thus  deforming the contour from $\lc$ to $\C_{r_i}$, we do not cross the poles in the integrand.
We parametrize $\C_{r_i}$ by $z_i(\theta_i) = r_i e^{\im \theta_i}, \theta \in (-\pi, \pi]$ and get
\begin{equation*}
\toneparte(t, s, x_i - y_i) = \frac{1}{2 \pi} \oint_{\C_{r_i}} \big(\coreej (z_i (\theta_i))\big)^{\floor{(t-s)/J}} \reme(z_i (\theta_i), t, s) z_i(\theta_i)^{x_i - y_i} d\theta_i 
\end{equation*}
We want to bound each terms that appear in the integrand above. Note that by \eqref{eq:six:radius}, $|z_i (\theta_i)|^{x_i - y_i}  = e^{-\frac{\beta}{\sqrt{t - s +1}} |x_i  - y_i|}$. 
\bigskip
\\
To estimate $\reme(z_i, t, s)$, referring to \eqref{eq:six:temp9}, $\reme(z, t, s)$ is a product of up to $J$ terms (since $t - s -  J\floor{\frac{t-s}{J}}\leq J$). For each term, by Lemma \ref{lem:wsc}
\begin{align*}
&\lim_{\ep \downarrow 0} \bigg|\lambda(k) z^{\mu(k)} \frac{(1 + \alpha(k) q) q^{-\rho} z - (\nu + \alpha(k) q)}{(1 + \alpha(k)) q^{-\rho} z  - (\nu + \alpha(k))}\bigg| \\
\numberthis \label{eq:six3:temp5}
&= |z|^{\frac{1}{I}} \frac{(b(1 + \mod(k)) - \mod(k)) z - (b (I + \mod(k) +1) - (I + \mod(k)) }{(b \mod(k) - (\mod(k) - 1)) z  - ((I + \mod(k)) b - (I + \mod(k) - 1))}.
\end{align*}
The singularities in \eqref{eq:six3:temp5} lie strictly inside the unit disk. Since $r_i$ is close to $1$ when $t-s$ is large, for $\ep$ small enough and $t-s$ large enough, there exists constant $C$ such that for $z \in \C_{r_i}$ and $k \in \ZZn$
\begin{equation*}
\bigg|\lambda(k) z^{\mu(k)} \frac{(1 + \alpha(k) q) q^{-\rho} z - (\nu + \alpha(k) q)}{(1 + \alpha(k)) q^{-\rho} z  - (\nu + \alpha(k))}\bigg| \leq C,
\end{equation*} 
 which implies 
\begin{equation}\label{eq:temp24}
|\reme(z_i, t, s)| \leq C.
\end{equation}
Consequently, 
\begin{equation}\label{eq:six:temp2}
\toneparte(t, s, x_i - y_i) \leq \int_{-\pi}^\pi |\coreej (z_i)|^{\floor{(t-s)/J}} |\reme(z_i(\theta), t, s)| |z_i (\theta)|^{x_i - y_i} d\theta\leq C e^{-\frac{\beta}{\sqrt{t - s +1}} |x_i - y_i|} \int_{-\pi}^{\pi} \big|\coreej (z_i(\theta))\big|^{\floor{(t-s)/J}} d\theta 
\end{equation}
We expect to extract the temporal decay $\frac{1}{\sqrt{t - s +1}}$ from the integral above. 
To this end, we need to the following lemma. 
\begin{lemma}\label{lem:free:std}
There exists positive constants $C(\beta, T)$, $C$ such that for $\theta \in (-\pi, \pi]$
\begin{equation*}
\big|\coreej (z(\theta))\big|^{t-s} \leq C(\beta, T) e^{- C (t-s+1) \theta^2}, \qquad z(\theta) = \rfunc(t-s, \pm \beta) e^{\im \theta}
\end{equation*}
holds for $\epsilon > 0$ small enough and large enough $t-s \leq \ep^{-2} T$ . 
\end{lemma}
As a remark, we see from \eqref{eq:six:temp10} that  the function $\coreej (z)$ is not globally analytic due to the factor $z^{\mu}$ ($\mu$ is not an integer), but it is analytic in a neighborhood of $1$. 
Furthermore, $\big|\coreej (z)\big|$ is a continuous function in a neighborhood of the unit circle.
\begin{proof}[Proof of Lemma \ref{lem:free:std}]
We only prove Lemma \ref{lem:free:std} for $z(\theta) = \rfunc(t-s, \beta) e^{\im \theta}$, the argument for $z(\theta) = \rfunc(t-s, -\beta) e^{\im \theta}$ is similar. 
By writing $\big|\coreej (z(\theta))\big|^{t-s} = e^{(t-s) \Re \log \coreej (z(\theta))}$, it suffices to show that there exists positive constants $C(\beta, T), C$ such that for $\epsilon > 0$ small enough and  $t-s \leq \ep^{-2} T$ large enough
\begin{equation*}
\Re \log \coreej (\rfunc(t-s, \beta) e^{\im \theta}) \leq \frac{C(\beta, T)}{t-s+1} - C \theta^2,
\end{equation*}
where $\Re z$ denotes the real part of a complex number $z$.
\bigskip
\\
We divide our proof into three cases. It suffices to show 
\begin{itemize}
\item $(\theta = 0): \log \coreej (\rfunc(t-s, \beta)) \leq \frac{C(\beta, T)}{t-s+1} $ 
\item ($\theta$ small): There exists $\zeta > 0$ s.t. 
\begin{equation*}
\Re \log \coreej (\rfunc(t-s, \beta) e^{\im \theta}) \leq \frac{C(\beta, T)}{t-s+1} - C \theta^2 \qquad \text{for }  |\theta| \leq \zeta.
\end{equation*}  
\item ($\theta$ large): There exists $\delta > 0$ such that $\big|\coree (\rfunc(t-s, \beta) e^{\im \theta})\big| < 1 -\delta $ for $|\theta| > \zeta$.
\end{itemize}
The proof for the first and second bullet point are done by using the local property of $\coreej (z)$ near $1$ (applying Taylor expansion). Let $O$ be a small neighborhood around $1$ such that $\coreej (z)$ is analytic inside $O$.
\bigskip
\\
\textbf{($\theta = 0$): } 
We write $\coreej(z)$ into terms of a telescoping product 
\begin{equation*}
\coreej (z) =  \prod_{k = 0}^{J-1} \lambdat{\lambda}{k} z^{\mut{\mu}{k}} \frac{1 + \alphat{\alpha}{k} q - (\nu + \alphat{\alpha}{k} q) q^{\rho} z^{-1}}{1 + \alphat{\alpha}{k} - (\alphat{\alpha}{k} + \nu) q^{\rho}}. 
\end{equation*}
By \eqref{eq:mgf}, we see that 
\begin{equation*}
\coreej (z)  = \prod_{k=0}^{J-1} \EE\big[z^{-R_\ep (k)}\big] = \EE\big[z^{-\sum_{k=0}^{J-1} R_\ep (k)}\big],
\end{equation*}
thus
\begin{equation*}
{\coreej}'(1) = -\EE\big[\sum_{k=0}^{J-1} R_\ep (k)\big] = 0, \qquad {\coreej}''(1) = \var\big[\sum_{k=0}^{J-1} R_\ep (k)\big] = \sum_{k=0}^{J-1} \var\big[R_\ep (k)\big].
\end{equation*}
Referring to \eqref{eq:five:temp6}, 
\begin{equation*}
\lim_{\ep \downarrow 0} \sum_{k=0}^{J-1} \var\big[R_\ep (k)\big] =  \sum_{k=0}^{J-1}  \frac{(I+1+2k) b - (I + 2k - 1)}{I^2 (1-b)} = J V_*,
\end{equation*}
where $V_*$ is given by \eqref{eq:vstar}.
The above discussion implies that  
\begin{equation*}
\log \coreej (1) = 0, \qquad (\log \coreej)' (1) = 0.
\end{equation*}  
Moreover, there exists constant $C$ such that uniformly for $z \in O$ and $\ep$ small enough,
\begin{equation*}
 \quad  |(\log \coreej)''(z)| \leq C.
\end{equation*}
Since $\lim_{t-s \to \infty} \rfunc(t-s, \beta) = 1$, we see that $\rfunc(t-s, \beta) \in O$ for $t-s$ large enough. Thus, we taylor expand $\coreej (z)$ around $z = 1$ and get
\begin{equation}\label{eq:six1:temp5}
\log \coreej (\rfunc(t-s, \beta)) \leq C \big|\rfunc(t-s, \beta) - 1\big|^2 \leq \frac{C(\beta, T)}{t-s+1},
\end{equation}
which justifies the first bullet point.
\bigskip
\\
\textbf{($\theta$ small):} 
Consider the function $\coreej (z (\theta))$, we calculate for $z(\theta) \in O$
\begin{align*}
\pa_\theta (\log \coreej (z(\theta))) \big|_{\theta = 0} &\in \im \RR,\\ 
\lim_{\epsilon \downarrow 0, t-s \to \infty} \pa_\theta^2 (\log \coreej (z(\theta))) \big|_{\theta = 0} 
&= - J V_*,
\\
\big|\pa_\theta^3 (\log \coreej (z(\theta))) \big| &\leq C.
\end{align*}
Given these properties, we taylor expand $\log \coreej (z(\theta))$ at $\theta = 0$, there exists $\zeta > 0$ such that
\begin{equation*}
\Re \log \coreej (z (\theta)) \leq \Re \log \coree (z (0)) - \frac{J V_*}{2} \theta^2 \qquad |\theta| \leq \zeta
\end{equation*}
In conjunction with $\Re \log \coree (z (0)) \leq \frac{C(\beta, T)}{t-s+1}$ (which is shown by \eqref{eq:six1:temp5}), we conclude the second bullet point.
\bigskip
\\
\textbf{($\theta$ large):} 
We set 
\begin{align}\label{eq:temp22}
\corelim(z) := 
z^{\frac{J}{I}} \frac{(b J - (J-1)) z - ((I+J)b - (I + J - 1))}{z - (I b - (I-1)) }
\end{align}
Referring to the expression of $\coreej$ in \eqref{eq:six:temp10} and  using Lemma \ref{lem:wsc}, one has
\begin{equation*}
\lim_{\epsilon \downarrow 0} \big| \coreej (z) \big| = \big|  \corelim (z) \big|.
\end{equation*}
The convergence is uniform in an open neighborhood of unit circle.
Thereby,
\begin{equation*}
\lim_{\epsilon \downarrow 0, t-s \to \infty} \big|\coreej (\rfunc(t-s, \beta) e^{\im \theta})\big| = \big|\core_*(e^{\im \theta})\big| \quad \text{uniformly over } (-\pi, \pi]. 
\end{equation*}
As a result, we conclude the third bullet point as long as we verify the following \emph{steepest descent condition}
\begin{equation}\label{sd:rone}\tag{SD.$\C_1$}
\big|\corelim(z)\big| < 1  \quad \text{for } z \in \C_1 \backslash \{1\}.
\end{equation}
To prove \eqref{sd:rone}, we compute 
\begin{align*}
\big|\corelim(e^{\im \theta})\big|^2 &= \bigg| \frac{(b J - (J-1)) e^{\im \theta} - ((I+J)b - (I + J - 1))}{e^{\im \theta} - (I b - (I-1)) }\bigg|^2
\\
&= \frac{(b J - (J-1))^2 + ((I + J)b - (I + J - 1))^2 - 2 (b J - (J-1)) ((I+J) b - (I + J -1)) \cos \theta}{1 + (I b - (I-1))^2 - 2(I b - (I-1)) \cos \theta} \\
&= 1 -   \frac{2J(1-b) (1- \cos\theta) ((I+J)b  - (I+J-2))}{1 + (Ib - (I-1))^2 - 2(I b -(I-1)) \cos \theta}  < 1, \qquad \theta \in (-\pi, \pi] \backslash \{0\}.
\end{align*}
In the last step, we used the condition $\frac{I +  J -2}{I + J - 1} < b < 1$.
\end{proof}
Having proved Lemma \ref{lem:free:std}, we proceed to finish the proof of Theorem \ref{prop:semiestimatefr}.
\begin{proof}[Proof of Theorem \ref{prop:semiestimatefr}]
Due to Lemma \ref{lem:free:std}, 
\begin{align*}
\int_{-\pi}^{\pi} \big|\coreej (z_i(\theta))\big|^{\floor{\frac{t - s}{J}}} d\theta 
\leq \int_{-\pi}^{\pi} C(\beta, T) e^{- C (\floor{\frac{t - s}{J}}+1) \theta^2} d\theta \leq \frac{C(\beta, T)}{\sqrt{t - s +1}}.
\end{align*}
This being the case, by \eqref{eq:six:temp2} we readily see that 
\begin{equation}\label{eq:six:onepartbound}
\toneparte(t, s, x_i - y_i) \leq \frac{C(\beta, T)}{\sqrt{t - s + 1}} e^{-\frac{\beta}{\sqrt{t - s + 1}} |x_i - y_i|}.
\end{equation}
Incorporating this bound into \eqref{eq:six:temp3} concludes Theorem \ref{prop:semiestimatefr} part (a). 
\bigskip
\\
For the gradient, notice that one has 
\begin{align*}
\numberthis \label{eq:six2:temp1}
\na_{x_1} \rhzrfr\big((x_1, x_2), (y_1, y_2), t, s\big) &= \na \tonepart(t, s, x_1- y_1) \tonepart(t, s, x_2 - y_2),\\ 
\na_{y_1} \rhzrfr\big((x_1, x_2), (y_1, y_2), t, s\big) &= \tonepart(t, s, x_1 - y_1) \na \tonepart(t, s, x_2 - y_2 - 1),\\
\numberthis \label{eq:six2:temp3}
\na_{x_1, x_2} \rhzrfr\big((x_1, x_2), (y_1, y_2), t, s\big) &= \na \tonepart(t, s, x_1 - y_1) \na \tonepart(t, s,  x_2 - y_2).
\end{align*}
The proof for gradients $\na_{x_2}$, $\na_{y_2}$ is similar to that for $\na_{x_1}, \na_{y_1}$ by symmetry. It suffices to analyze 
\begin{equation*}
\na \tonepart(t, x_1 - y_1) = \frac{1}{2 \pi} \int_{-\pi}^{\pi} \core(z_1 (\theta_1))^{\floor{\frac{t-s}{J}}} \reme(z_1(\theta_1), t, s) z_1 (\theta_1)^{x_1 - y_1} (z_1(\theta_1) - 1) d \theta_1 
\end{equation*}
By the fact 
$\big|z_1 (\theta_1) - 1\big|  = \big|e^{\pm \frac{\beta}{\sqrt{t - s +1}} + \im \theta_1} - 1\big| \leq C (\frac{1}{\sqrt{t - s +1}} + |\theta_1|)$, we conclude 
\begin{align}\label{eq:six:onepartgradbound}
\big|\na\tonepart(t, x_i - y_i)\big| 
\leq C(\beta, T) e^{-\frac{\beta}{\sqrt{t - s +1} } |x_i - y_i|} \int_{-\pi}^{\pi} e^{- C \floor{\frac{t-s}{J}} \theta_1^2} (\frac{1}{\sqrt{t - s +1}} + |\theta_1|) d \theta_1 \leq \frac{C(\beta, T)}{t - s +1} e^{-\frac{\beta}{\sqrt{t - s +1}} |x_i - y_i|},
\end{align}
where the last inequality follows by a change of variable $\theta_1 \to \frac{\theta_1}{\sqrt{t - s +1}}$. Incorporating this bound into \eqref{eq:six2:temp1} and \eqref{eq:six2:temp3}, we conclude the Theorem \ref{prop:semiestimatefr} (b), (c).
\end{proof}
\subsection{Estimate of $\rhzrin$, an overview.}
Recall from \eqref{eq:six:intpart} that
\begin{align*}
\rhzrin\big((x_1, x_2), (y_1, y_2), t, s\big) = &\oint_{\lc} \oint_{\lc} \interacte(z_1, z_2)  \prod_{i=1}^2 \coreej(z_i)^{\floor{\frac{t-s}{J}}} \reme(z_i, t, s) z_i^{x_{3-i} - y_i } \frac{dz_i}{2\pi \im z_i}  \\  
\numberthis \label{eq:six2:temp4}
& - \res_{z_1 = \pole(z_2)} \bigg[\oint_{\lc} \oint_{\lc} \interacte(z_1, z_2)  \prod_{i=1}^2 \coreej(z_i)^{\floor{\frac{t-s}{J}}} \reme(z_i, t, s) z_i^{x_{3-i} - y_i } \frac{dz_i}{2\pi \im z_i}\bigg].
\end{align*}
We study the double contour integral in \eqref{eq:six2:temp4}. Recall from \eqref{eq:interactexpression} and \eqref{eq:pole1expression} that 
\begin{equation}\label{eq:temp21}
\interacte(z_1 , z_2) = \frac{q\nu - \nu + (\nu-q) q^{-\rho} z_2 + (1-q\nu) q^{-\rho} z_1 + (q-1) q^{-2\rho} z_1 z_2}{q\nu - \nu + (\nu-q) q^{-\rho} z_1 + (1-q\nu) q^{-\rho} z_2 + (q-1) q^{-2\rho} z_1 z_2},
\end{equation}
which produces a pole at $z_1 = \pole(z_2)$ where 
\begin{equation*}
\pole (z) = \frac{(1-q \nu) q^{-\rho} z - \nu  (1-q)}{(q-\nu) q^{-\rho}  + (1-q) q^{-2\rho}  z}.
\end{equation*}
Referring to \eqref{eq:temp17}, the other poles of the integrand belong to $[0, \Theta]$ for some $0 < \Theta < 1$. 
\bigskip
\\
We say the contour $\Gamma$ is \emph{admissible} if 
\begin{equation}
\label{con:admissible}
(1): \Gamma \text{ contains } [0, \Theta] \text{ but does not contain } 1-I, \quad (2): d(1 - I, \Gamma) > \frac{1}{2I},  
\end{equation}
where the distance between a point $z \in \CC$ and a set $A$ is define by $d(z, A) := \inf \{|z- y|: y\in A\}$. Figure \ref{fig:admissible} below gives several graphical examples of admissible and not admissible contours.
\bigskip
\\ 
Define 
\begin{align*}
\sstar(z) := \lim_{\ep \downarrow 0} \pole(z) = \frac{(I-1) z + 1}{I+1 - z}.
\end{align*}
Note that 
\begin{align*}
\lim_{|z| \to \infty} \sstar(z) 
= 1-I.
\end{align*}
Note that $z_2 \in \C_R$, from above we have: For $R$ large enough  and $\ep$ small enough, if $\Gamma$ is admissible, deforming the $z_1$-contour from $\C_R$ to $\Gamma$ will cross the pole $s_\ep (z_2)$ for all $z_2 \in \C_R$. Moreover, such deformation does not cross any other poles in $\poleset$. Therefore, 
\begin{equation*}
\rhzrin\big((x_1, x_2), (y_1, y_2), t, s\big) = \oint_{\Gamma} \oint_{\lc} \interacte(z_1, z_2) \prod_{i=1}^2 \coreej (z_i)^{\floor{\frac{t-s}{J}}}  \reme(z_i, t, s) z_i^{x_{3 - i} - y_i} \frac{d z_i}{2 \pi \im z_i}.
\end{equation*}
In practice, we deform the $z_1$-contour to some contour $\contone$ which depends on both $t-s$ and $\ep$ so that it is admissible for $t-s$ large enough and $\ep$ small enough.
\begin{figure}[ht]
\includegraphics[width=.85\linewidth]{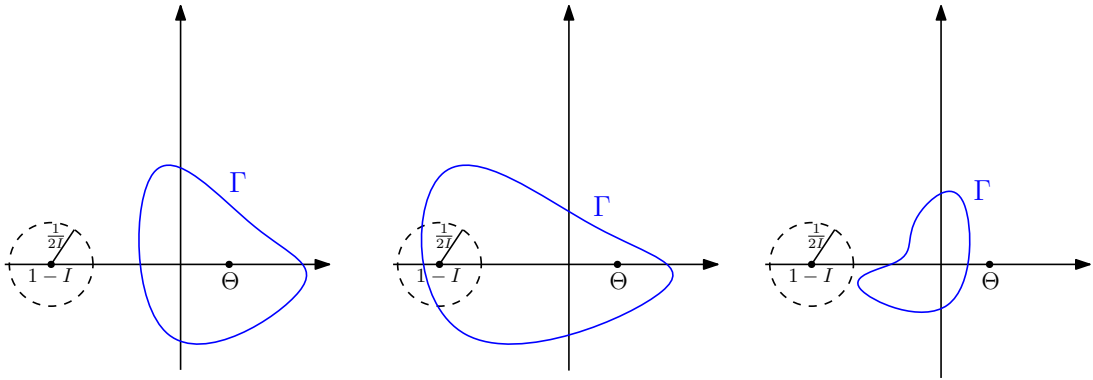}
\caption{Graphical examples of admissible and not admissible contour $\Gamma$.}
\label{fig:admissible}
\end{figure}
\bigskip
\\
Assuming that we have deformed $z_1$-contour to $\Gamma(t-s, \ep)$, which is admissible. The next step is to deform the $z_2$-contour. Note that given $z_1 \in \contone$, $\interacte(z_1, z_2)$ generates a pole at $z_2 = \ppole(z_1)$ ($\ppole$ is the inverse of $\pole$)
\begin{equation}\label{eq:six2:temp7}
\ppole(z_1) = \frac{(1-q) \nu + (q - \nu) q^{-\rho} z_1}{(q-1) q^{-2\rho} z_1 + (1-q\nu) q^{-\rho}}.
\end{equation}
We consider three potential radius
\begin{equation}\label{eq:temp26}
r_2 := \rfunc(t - s, \sign(x_1 - y_2) k_2 \beta),\quad r_2':= \rfunc(t - s, \sign(x_1 - y_2) 2 k_2 \beta),\quad r_2'':= \rfunc(t - s,  \sign(x_1 - y_2) 3 k_2 \beta),
\end{equation}
where $k_2 \geq 1$ is a constant which is irrelevant with the current discussion. We deform $z_2$-contour from $\lc$ to $\C_{\rad_2(z_1)}$, where 
\begin{equation*}
\rad_2 (z_1) = r_2 1_{\{\ppole(z_1) > r'_2 \}} + r_2'' 1_{\{\ppole(z_1) \leq r'_2 \}}.
\end{equation*} 
In other words, if the pole $\ppole(z_1)$ lies outside $\C_{r'_2}$, we choose $z_2$-contour to be a circle with radius $r_2 < r'_2$. If the pole $\ppole(z_1)$ lies inside $\C_{r'_2}$, we choose $z_2$-contour to be circle with radius $r''_2 > r'_2$. It is clear we always have for $t - s$ large enough that
\begin{equation}\label{eq:six:temp5}
|\ppole(z_1) - z_2| \geq \frac{\beta}{\sqrt{t - s +1}},  \qquad \forall z_2 \in \C_{\rad_2(z_1)}.
\end{equation}
Referring to the expression of $\interacte(z_1 , z_2)$
\eqref{eq:temp21}, we find that 
\begin{equation*}
\res_{z_2 = \ppole(z_1)} \interacte(z_1, z_2) = \frac{q\nu - \nu + (\nu - q) q^{-\rho} \ppole(z_1) + (1-q\nu) q^{-\rho} z_1 + (q-1) q^{-2\rho} z_1 \ppole(z_1)}{(q-1)q^{-2\rho} z_1 + (1-q\nu) q^{-\rho}}.
\end{equation*}
We set 
\begin{align*}
\pcoree (z_1) &= \coreej (z_1) \coreej (\ppole(z_1)), \\
\jprod(z_1)  &= \res_{z_2 = \ppole(z_1)} \interacte(z_1, z_2) z_1^{x_2 - y_1} \ppole(z_1)^{x_2 - y_1} \idc_{\{|\ppole(z_1)| > r'_2\}},\\
\numberthis \label{eq:jterm}
&= \frac{q\nu - \nu + (\nu - q) q^{-\rho} \ppole(z_1) + (1-q\nu) q^{-\rho} z_1 + (q-1) q^{-2\rho} z_1 \ppole(z_1)}{(q-1)q^{-2\rho} z_1 + (1-q\nu) q^{-\rho}} z_1^{x_2 - y_1} \ppole(z_1)^{x_2 - y_1} \idc_{\{|\ppole(z_1)| > r'_2\}}.
\end{align*}
From preceding discussion, we decompose $\rhzrin = \rhzrb + \rhzrr$, where 
\begin{align*}
&\rhzrb\big((x_1, x_2), (y_1, y_2), t, s\big) = \oint_{\contone}  \oint_{\C_{r_2 (z_1)}} \interacte(z_1, z_2) \prod_{i=1}^2 \coreej(z_i)^{\floor{\frac{t - s}{J}}} z_i^{x_{3 - i} - y_i} \frac{dz_i}{2 \pi \im z_i},\\
\numberthis \label{eq:six4:temp4}
&\rhzrr\big((x_1, x_2), (y_1, y_2), t, s\big) = \oint_{\contone} \idc_{\{|\ppole(z_1)| > r'_2\}} \jprod(z_1) \pcoree(z_1)^{\floor{\frac{t-s}{J}}} \frac{dz_1}{2 \pi \im z_1 \ppole(z_1)}.
\end{align*}
Note that we integrate under the indicator $1_{\{|\ppole(z_1)| > r'_2\}}$, which arises in the case that deforming the $z_2$-contour from $\C_R$ to $\C_{\rad_2 (z_1)}$ crosses the pole $\ppole(z_1)$. 
\bigskip
\\
We want to perform the steepest descent argument for $\rhzrb$ and $\rhzrr$, similar to what we have done in Section \ref{sec:estimatefree}. More precisely, as $t -s \to \infty$ and $\ep \downarrow 0$, $\contone$ converges to some fixed contour $\contonelim$.\footnote{We define the distance of two contours to be
$\dist \big(\Gamma_1, \Gamma_2\big) = \sup_{x \in \Gamma_1, y \in \Gamma_2}  \big(d(x, \Gamma_2) \vee d(y, \Gamma_1)\big).$
We say a sequence of contours $\Gamma_n$ converges to $\Gamma$ if 
$\lim_{n \to \infty} \dist(\Gamma_n, \Gamma) = 0.$
} We set 
\begin{equation}\label{eq:pstar}
\pstar(z) := \lim_{\ep \downarrow 0} \ppole (z) = \frac{(I+1) z - 1}{z + (I-1)}.
\end{equation}
Recall from \eqref{eq:temp22} that 
\begin{align*}
\corelim(z) 
=  z^{\frac{J}{I}} \frac{(J b - (J-1)) z - ((I+J)b - (I + J - 1))}{z - (I b - (I-1)) }.
\end{align*}
and set
\begin{align*}
\pcorelim(z) = 
\corelim(z) \corelim(\pstar(z)).
\end{align*} 
Note that 
\begin{equation*}
|\corelim(z)| = \lim_{\ep \downarrow 0} |\coreej (z)|, \qquad |\pcorelim(z)| =
\lim_{\ep \downarrow 0} |\pcoree(z)|.   
\end{equation*}
We require the contour $\contonelim$ satisfying the steepest descent condition. 
\begin{align}\label{eq:six:std}
(i) \big|\corelim (z)\big| <1, \quad z \in \contonelim \backslash \{1\}; \qquad  (ii) \big|\pcorelim (z)\big| < 1, \quad z \in \contonelim \backslash \{1\}.
\end{align} 
As we see from \eqref{sd:rone} that if we take $\Gamma_* = \C_1$,  $(i)$ holds. However, $(ii)$ does not hold. In truth, Figure \ref{fig:contline} indicates the region where $|\corelim(z)| \leq 1$ and $|\pcorelim(z)| \leq 1$ for $I = 2$ and $b = 0.8$. We see that $\C_1$ lies fully inside $|\corelim(z)| \leq 1$, but partially outside  $|\pcorelim(z)| \leq 1$.
\begin{figure}[ht]
\includegraphics[scale=0.45]{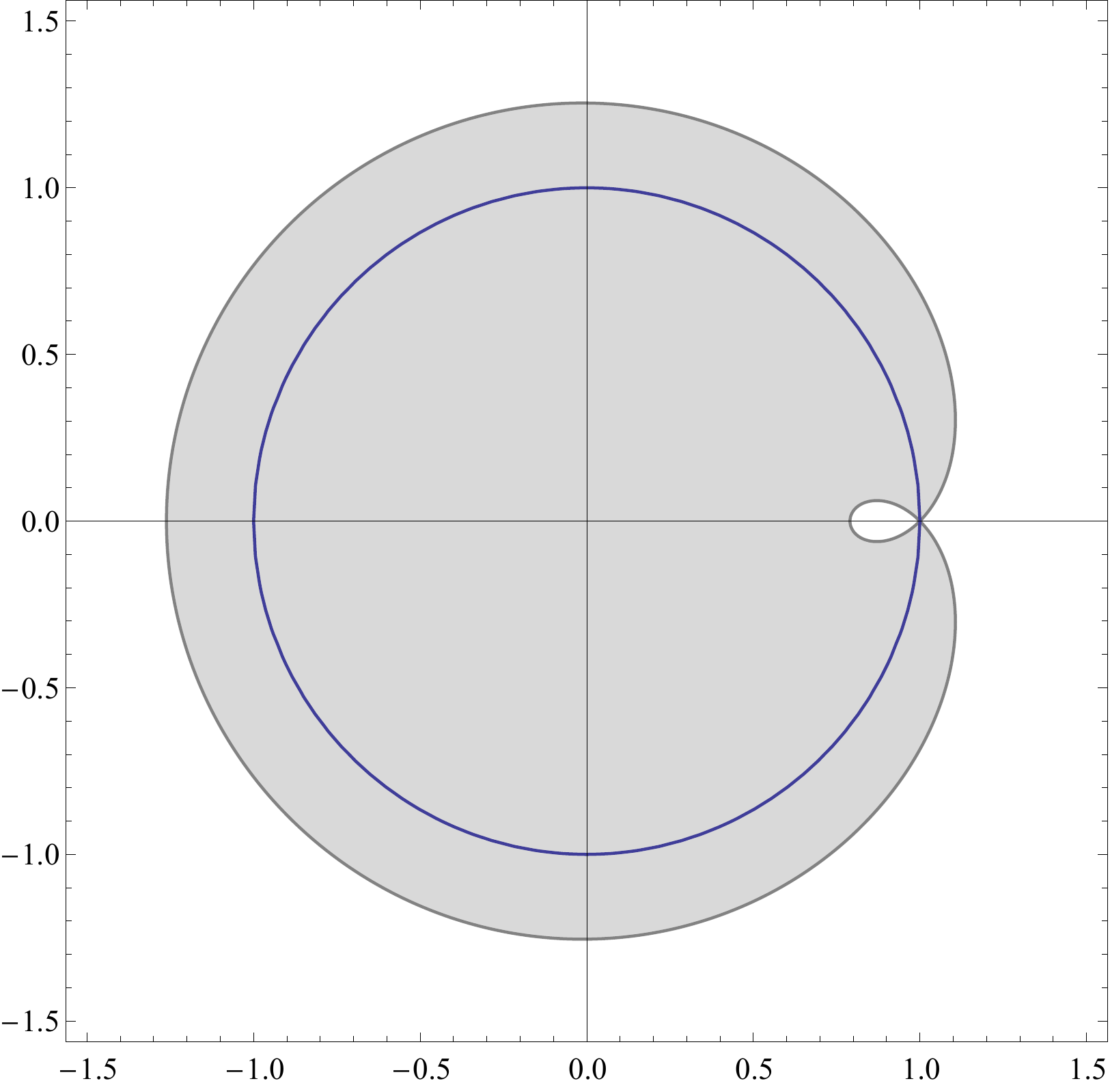}
\includegraphics[scale=0.45]{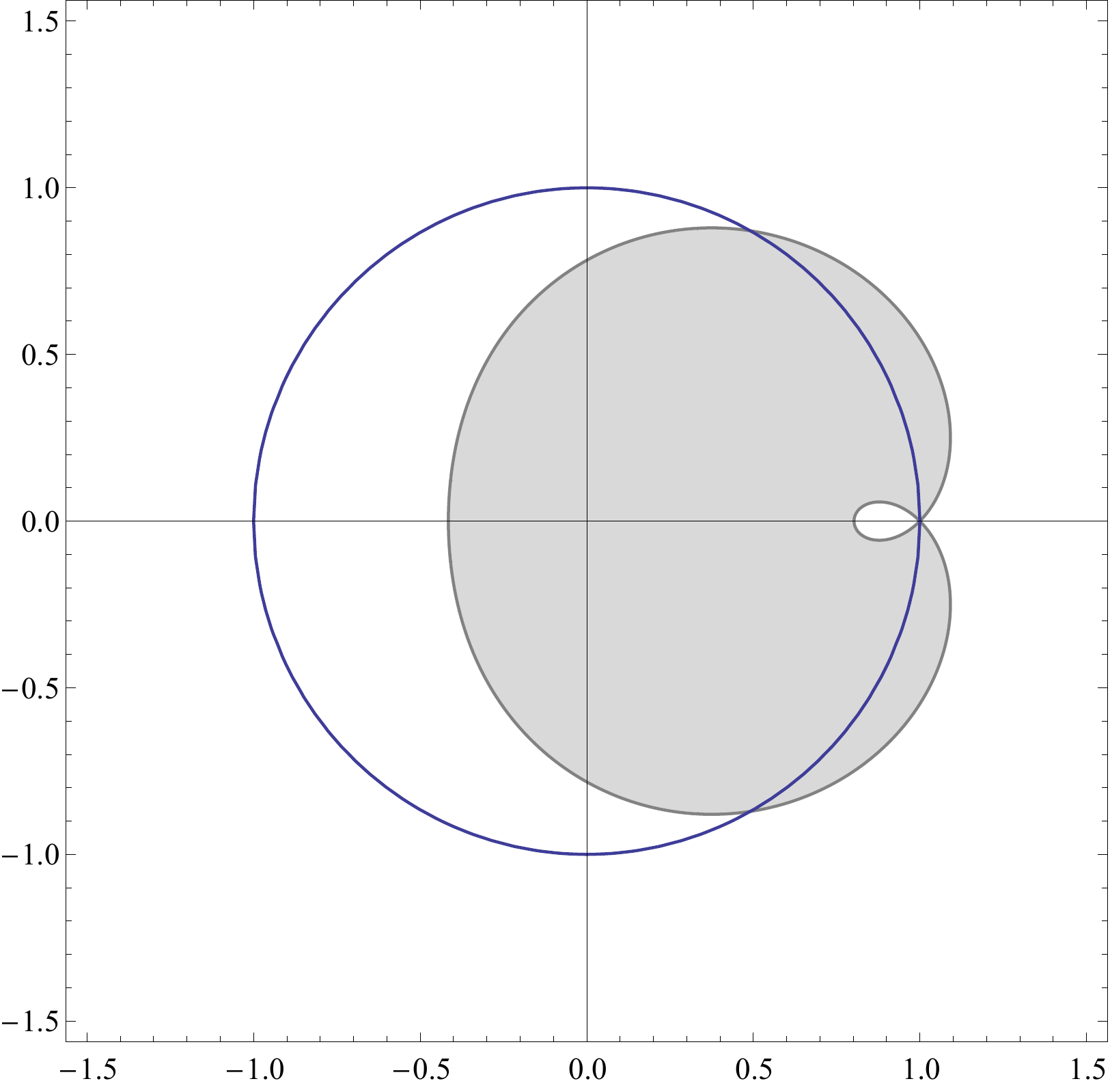}
\caption{We choose $b = 0.8$ and $I = 2$. The figures on the left and right show respectively the region where  $\big|\corelim(z)\big| \leq 1$ and $\big|\pcorelim(z)\big| \leq 1$, which is filled with gray color. The unit circle (with blue color) is drawn for comparison.}
\label{fig:contline}
\end{figure}
\bigskip
\\
Set 
$\mcont = \{\big|z - \frac{1}{I+1}\big| = \frac{I}{I+1}\}$, the following lemma says that $\mcont$ the satisfies the steepest descent condition \eqref{eq:six:std}. 
\begin{lemma}\label{lem:SDM}
We have 
\begin{equation}\tag{SD$\mcont$} \label{eq:SDM}
|\corelim(z)| < 1, z \in \mcont \backslash \{1\}, \qquad |\pcorelim(z)| < 1, z \in \mcont \backslash \{1\}.
\end{equation}
\end{lemma}
\begin{proof}
Parametrize $\mcont$ by $z(\theta) = \frac{1}{I+1} + \frac{I}{I+1} e^{\im \theta}, \theta \in (-\pi, \pi]$, we compute 
\begin{align*}
|\corelim(z(\theta))|^2 &\leq |z(\theta)|^{\frac{2J}{I}} \bigg|\frac{(J b - (J-1)) z(\theta) - ((I+J)b - (I + J - 1))}{z(\theta) - (I b - (I-1)) }\bigg|^2 \\
&\leq \bigg|\frac{(J b - (J-1)) z(\theta) - ((I+J)b - (I + J - 1))}{z(\theta) - (I b - (I-1)) }\bigg|^2\\
&= \bigg| \frac{(J b - (J-1)) (\frac{1}{I+1} + \frac{I}{I+1} e^{i\theta}) - ((I+J)b - (I + J - 1))}{\frac{I}{I+1} + \frac{I}{I+1} e^{\mathbf{i} \theta} - (Ib - (I-1))} \bigg|^2\\
&=1  - \frac{2 I^2 J (1 - b) ((I+J+1)b - (I+J-1))(1-\cos \theta)}{\big|\frac{1}{I+1} + \frac{I}{I+1} e^{\im \theta} - (I b - (I-1))\big|^2 (1+I)^2} < 1, \qquad \theta \in (-\pi, \pi] \backslash \{0\}. 
\end{align*}
where in the first line we used the fact $|z(\theta)| \leq 1$ and in the last line we used $\frac{I + J - 2}{I + J -1} < b < 1$, note that when $I \geq 2$ and $J \geq 1$, we have
\begin{equation*}
b \geq \frac{I + J - 2}{I + J -1} > \frac{I + J - 1}{I + J +1},
\end{equation*}
which concludes the last inequality.
\bigskip
\\
For $\pcorelim(z)$, note that 
\begin{align*}
\pcorelim(z) 
&=  z^{\frac{J}{I}} \frac{(b J - (J-1)) z - ((I+J)b - (I + J - 1))}{z - (I b - (I-1)) } \pstar(z)^{\frac{J}{I}}  \frac{(b J - (J-1)) \pstar(z) - ((I+J)b - (I + J - 1))}{\pstar(z) - (I b - (I-1)) }\\
&=  \big(z  \pstar(z)\big)^{\frac{J}{I}}  \frac{(b J - (J-1)) z - ((I+J)b - (I + J - 1))}{z - (I b - (I-1)) } \cdot \frac{(b J - (J-1)) \pstar(z) - ((I+J)b - (I + J - 1))}{\pstar(z) - (I b - (I-1)) }
\end{align*}
A crucial observation is that $\big|z - \frac{1}{I+1}\big| = \frac{I}{I+1}$ implies
\begin{equation*}
\big|z \pstar(z) \big| = \big|z \frac{(I+1) z - 1}{z + (I-1)}\big| = \big| \frac{Iz}{z + (I-1)} \big| = 1.
\end{equation*}
which can be verified by inserting $z(\theta) = \frac{1}{I+1} + \frac{I}{I+1} e^{\im \theta}$. Consequently, we see that 
\begin{align*}
\big|\pcorelim (z(\theta))\big|^2 &= \bigg|\frac{b z(\theta) - (I+1) b - 1}{z(\theta) - (Ib - (I-1))} \cdot \frac{b \pstar(z(\theta)) - ((I+1)b - I)}{\pstar(z(\theta)) - (Ib  - (I-1))}\bigg|^2\\
&= \big|\frac{I+J-(I+J+1) b + (J b - (J-1) ) e^{\im \theta}}{I-(I+1) b + e^{\im \theta}}
\cdot \frac{(I+J)b - (I+J-1)+ ((1-J)b + J-2) e^{\im \theta}}{I b - (I-1) + (b - 2) e^{\im \theta}}\bigg|^2
\\
\numberthis \label{eq:six2:temp5}
&= 1 + \frac{-4(b - 1) J (2- J - I + b (J+ I)) (\cos \theta - 1) (a_J - b_J \cos \theta)}{\big|(b - 2) e^{\im \theta} + (1 + (b - 1) I)\big|^2 \big|e^{\im \theta} - (b + (b -1) I)\big|^2}
\end{align*}
where 
\begin{align*}
a_J &= (J^2 + J I) (1-b)^2 + 2 + (2 b - 2) J + (b^2 - 1) I + (b - 1)^2 I^2\\
b_J &= (J^2 + J I)(1 - b)^2 + (2b -2) J + (1 + 2b - b^2) + (-3+ 4 b - b^2) I
\end{align*}
We claim that $|b_J| < a_J$, which implies $a_J - b_J \cos \theta > 0$. This claim is justified by showing 
\begin{align*}
a_J + b_J &=  (2J^2 + 2JI + I^2) (1-b)^2 + (4b - 4)(I+J)  + 3 +2b - b^2 = (J^2 - 1) (1-b)^2 + ((J+I)(b - 1) +2)^2 > 0,\\
a_J -  b_J &= (b - 1)^2 I^2 + 2(b - 1)^2 I + (b - 1)^2 = (b - 1)^2 (I + 1)^2 > 0.
\end{align*}
Therefore, by $\frac{I + J - 2}{I + J -1} < b < 1$ and \eqref{eq:six2:temp5}
\begin{equation*}
|\pcorelim(z(\theta))| < 1, \qquad \theta \in (-\pi, \pi] \backslash \{0\},
\end{equation*}
which concludes our proof.
\end{proof}
We need to consider the following modification of $\mcont$
\begin{equation*}
\mcontu := \pa \big(\{z: |z- \frac{1}{I+1}| = \frac{I}{I+1} + u\} \cap \{|z| \leq 1\}\big),
\end{equation*}
where $u$ is some positive real number. 
\begin{lemma}\label{lem:SDMu}
There exists $\delta > 0$ such that for all $0 < u < \delta$, one has
\begin{align*}\label{eq:SDMu} \tag{SD$\mcont(u)$}
\begin{split}
|\corelim(z)| < 1, \qquad z \in \mcontu \backslash \{1\},\\
|\pcorelim(z)| < 1, \qquad z \in \mcontu \backslash \{1\}.
\end{split}
\end{align*}
\end{lemma}
\begin{proof}
The proof of this lemma uses similar techniques which appear in \cite[Lemma 6.4]{CGST18}. By straightforward computation, one finds that 
\begin{align*}
&\corelim(1) = 1; \qquad \corelim'(1) = 0; \qquad \corelim''(1) = JV_*.\\
&\pcorelim(1) =1; \qquad \pcorelim'(1) = 0; \qquad \pcorelim''(1) = 2 J V_*.
\end{align*}
Here, $V_*$ is given by \eqref{eq:vstar}. We taylor expand $\corelim(z)$ and $\pcorelim(z)$ around $z  =1$ and get 
\begin{align*}
&\corelim(z) = 1 + \frac{1}{2} J V_* (z-1)^2 + \OO(|z-1|^3),\\
&\pcorelim(z) = 1 + J V_* (z-1)^2 + \OO(|z-1|^3).
\end{align*}
Notice that in the vertical direction where $z-1 \in \im \RR$, $\frac{1}{2} (z-1)^2$ is negative. This implies that
\begin{align}\label{eq:six2:temp6}
|\corelim(z)| < 1 \quad z \in \AA \backslash \{1\}; \qquad |\pcorelim(z)| < 1 \quad z \in \AA \backslash \{1\}.
\end{align}
where $\AA$ is a hourglass region centered at one, $\AA = \{z: z = 1+ v e^{\im \phi}, |\phi - \frac{\pi}{2}| < \phi_0, |\nu| <  \nu_0 \}$ with $\nu_0, \phi_0 > 0$ fixed. For $z \in \mcontu \backslash \AA$, due to
$\lim_{u \downarrow 0} \dist (\mcontu \backslash \AA,  \mcont \backslash \AA) = 0$ and Lemma \ref{lem:SDM}, we find that there exists a small $\delta$, such that for $0 < u < \delta$ 
\begin{equation*}
\sup_{z \in \mcontu \backslash \AA} |\corelim(z)| < 1,\qquad \sup_{z \in \mcontu \backslash \AA} |\pcorelim(z)| < 1.
\end{equation*}
Combining this with \eqref{eq:six2:temp6} concludes the proof of Lemma \ref{lem:SDMu}.
\end{proof}
We fix a constant $0 < u_* < \delta \wedge \frac{1}{4I},$ and set
$\mcontp  := \mcont(u_*)$. 
From our discussion above, $\mcontp$ is admissible and satisfies \eqref{eq:SDMu}.
\bigskip
\\
To prove Proposition \ref{prop:semiestimateint}, we need to choose our contour such that it controls both $\rhzrb$ and $\rhzrr$. The choice will 
depend on the sign of $x_2 - y_1$ and $x_1 - y_2$. We need to discuss separately for each of the following cases  
\bigskip
\\
\textbf{(i):} $\mathbf{(+-)}$ case: $x_2 - y_1 \geq 0$ and $x_1 - y_2 \leq 0$,\\
\textbf{(ii):} $\mathbf{(--)}$ case: $x_2 - y_1 \leq 0$ and $x_1 - y_2 \leq 0$,\\
\textbf{(iii):} $\mathbf{(++)}$ case: $x_2 - y_1 \geq 0$ and $x_1 - y_2 \geq 0$.
\bigskip
\\
Note that we don't need to consider the case where $x_2  -y_1 < 0$ and $x_1 - y_2 < 0$, since it contradicts our condition $x_1 \leq x_2$ and $y_1 \leq y_2$.
\subsection{Estimate of $\rhzrin$, the ($+-$) case}
In this case we shrink the $z_1$-contour from $\C_R$ to $$\mcont(t - s, -\beta) := \{z_1: \big|z_1 - \frac{1}{I+1}\big| = \frac{I}{I+1} - \frac{\beta}{\sqrt{t - s +1}}\}.$$ 
It is clear that for $t - s$ large enough, $\mcont(t - s, -\beta)$ is admissible. Consequently, we have
\begin{equation*}
\rhzrin\big((x_1, x_2), (y_1, y_2), t, s\big) = \rhzrb\big((x_1, x_2), (y_1, y_2), t, s\big) + \rhzrr\big((x_1, x_2), (y_1, y_2), t, s\big),
\end{equation*}
where 
\begin{align*}
\numberthis \label{eq:blk1}
&\rhzrb\big((x_1, x_2), (y_1, y_2), t, s\big) = \oint_{\C_{\rad_2 (z_1)}}  \oint_{\mcont(t-s, -\beta)}  \interacte(z_1, z_2) \prod_{i=1}^2 \coreej (z_i)^{\floor{\frac{t-s}{J}}} \reme(z_i, t, s) z_i^{x_{3 - i} - y_i} \frac{dz_i}{2 \pi \im z_i},\\
\numberthis \label{eq:res1}
&\rhzrr\big((x_1, x_2), (y_1, y_2), t, s\big) = \oint_{\mcont(t-s, -\beta)} \idc_{\{|\ppole(z_1)| > r'_2\}} \jprod(z_1) \pcoree(z_1)^{\floor{\frac{t-s}{J}}} \reme(z_1, t, s) \reme(\ppole(z_1), t, s)  \frac{dz_1}{2 \pi \im z_1 \ppole(z_1) }.
\end{align*}
Parametrizing $z_1 (\theta_1) = \frac{1}{I+1} + \big(\frac{I}{I+1} - \frac{\beta}{\sqrt{t - s +1}}\big) e^{\im \theta_1}$, we need the following lemma.
\begin{lemma}\label{lem:pm:std}
There exists positive  $C(\beta, T), C$ such that
\begin{equation*}
|\coreej (z(\theta))|^{t-s} \leq C(\beta, T) e^{-C (t-s + 1) \theta^2};  |\pcoree(z(\theta))|^{t-s} \leq C(\beta, T) e^{-C (t -s + 1) \theta^2} \text{ with } z(\theta) = \frac{1}{I+1} + \big(\frac{I}{I+1} - \frac{\beta}{\sqrt{t-s+1}}\big) e^{\im \theta}
\end{equation*}
for $\ep > 0$ small enough and $t - s \leq \ep^{-2} T$ large enough. 
\end{lemma}
\begin{proof}
Similar to the proof of Lemma \ref{lem:free:std}, it suffices to show there exists positive constants $C(\beta, T), C$ such that
\begin{equation}\label{eq:pm:stdequiv}
\Re \log \coreej (z(\theta)) \leq \frac{C(\beta, T)}{t-s+1} - C \theta^2; \qquad \Re \log \pcoree (z(\theta)) \leq \frac{C(\beta, T)}{t-s+1} - C \theta^2.
\end{equation}
We prove the lemma for $(\theta = 0)$, ($\theta$ small) and ($\theta$ large) respectively
\begin{itemize}
\item $(\theta = 0): \Re \coreej (z(0)), \Re \pcoree(z(0)) \leq \frac{C(\beta, T)}{t-s+1}$.
\item ($\theta$ small): There exists $\zeta > 0$ and constants $C(\beta, T)$ and  $C > 0$ such that \eqref{eq:pm:stdequiv} holds for $|\theta| \leq \zeta$.
\item ($\theta$ large): There exists $\delta > 0$ such that $\big|\coreej (z(\theta))\big|, \big|\pcoree (z(\theta))\big| < 1 -\delta $ for $|\theta| > \zeta$.
\end{itemize}
We consider the first two bullet points $(\theta = 0)$ and ($\theta$ small). The analysis of $(\theta = 0)$ and ($\theta$ small) case for $\coreej$ is similar to Lemma \ref{lem:free:std}, we do not repeat here.
For $\pcoree(z) = \coreej (z) \coreej (\ppole(z))$, by straightforward calculation, 
\begin{align*}
&\pcoree (1) =  \coreej(\ppole(1)),\\
&\pcoree'(1) = 
\coreej'(\ppole(1)) \ppole'(1),\\
\numberthis \label{eq:six3:temp1}
&\lim_{\ep \downarrow 0} \pcoree''(1) = 2J V_*.
\end{align*}
For the first equation above,  we taylor expand $\coreej (z)$ at $z=1$ and according to \eqref{eq:taylorpe1},
\begin{equation}\label{eq:six3:temp2}
\pcoree(1) = 1 + \frac{1}{2} \coreej''(1)(\ppole(1) - 1)^2 + \OO\big((\ppole(1) - 1)^3\big) = 1 + \frac{J V_* (\rho I -\rho^2)^2 }{2 I^2} \ep^2  + \OO(\ep^{\frac{5}{2}}).
\end{equation}
For $\pcoree'(1) = \coreej'(\ppole(1)) \ppole'(1)$, taylor expanding $\coreej'(z)$ around $z = 1$, according to \eqref{eq:taylorpe1},
\begin{equation*}
\coreej'(\ppole(1)) = \coreej'(1) + \coreej''(1) (\ppole(1) - 1) + \OO(\ppole(1) - 1)^2 = \frac{J V_* (\rho I - \rho^2)}{2 I} \ep + \OO(\ep^\frac{3}{2}). 
\end{equation*}
Combining this with $\ppole'(1) = 1 + \OO(\ep^{\frac{1}{2}})$ yields
\begin{equation}\label{eq:temp23}
\pcoree'(1) = \frac{J V_* (\rho I - \rho^2) }{2 I} \ep + \OO(\ep^{\frac{3}{2}}).
\end{equation}
Using \eqref{eq:six3:temp2}, \eqref{eq:temp23} and  \eqref{eq:six3:temp1}, we get
\begin{align}\label{eq:six3:temp4}
(\log \pcoree)(1) = \frac{J V_* (\rho I -\rho^2)^2}{2 I^2} \ep^2 + \OO(\ep^{\frac{5}{2}}), \quad (\log \pcoree)'(1) = \frac{J V_*  (\rho I - \rho^2)}{2 I} \ep + \OO(\ep^{\frac{3}{2}}), \quad \lim_{\ep \downarrow 0}, (\log \pcoree)''(1) = 2 J V_*.
\end{align}
Moreover, straightforward calculation gives $| (\log\pcoree)'''(z)| \leq C$ for $z \in O$ (which is a small neighborhood of $1$). Thereby, by Taylor expansion we find that 
\begin{equation*}
\log\pcoree(z(0)) = \log \pcoree(1) + (\log \pcoree)'(1) (z(0) - 1) +  (\log \pcoree)''(1) (z(0) - 1)^2 + \OO((z(0) - 1)^3).
\end{equation*}
Using \eqref{eq:six3:temp4}, $z(0) = 1 - \frac{\beta}{\sqrt{t - s +1}} $ and  $\ep^2 (t-s) \leq T$, we see that there exists $C(\beta, T)$ such that for $t - s$ large and $\ep$ small,
\begin{equation*}
\log \pcoree(z(0)) \leq \frac{C(\beta, T)}{t - s +1},
\end{equation*}
which gives the first bullet point.
\bigskip
\\
For ($\theta$ small), we readily calculate
\begin{align*}
\pa_\theta (\log \pcoree (z(\theta))) \big|_{\theta = 0} &\in \im \RR,\\ 
\lim_{\epsilon \downarrow 0, t-s \to \infty} \pa_\theta^2 (\log \pcoree (z(\theta))) \big|_{\theta = 0} &= -\frac{2I^2 J V_*}{(I+1)^2},\\
\big|\pa_\theta^3 (\log \pcoree (z(\theta)))\big| &\leq C, \qquad \text{ for } |\theta| \leq \zeta.  
\end{align*}
Thus, via Taylor expansion, we find that for $|\theta| \leq  \zeta$,
\begin{equation*}
\Re \log \pcoree (z(\theta)) \leq \Re \log \pcoree (z(0)) - \frac{I^2 J V_*}{2 (I+1)^2} \theta^2 \leq \frac{C(\beta, T)}{t-s+1} - \frac{I^2 J V_*}{2 (I+1)^2} \theta^2,
\end{equation*}
which conclude the second bulletin point.
\bigskip
\\
For ($\theta$ large), recall $z(\theta) = \frac{1}{I+1} + \big(\frac{I}{I+1} - \frac{\beta}{\sqrt{t - s +1}}\big) e^{\im \theta}$, we notice that
\begin{align*}
&\lim_{\ep \downarrow 0, t - s\to \infty} \big|\coreej (z(\theta))\big| = \big|\corelim(\frac{1}{I+1} + \frac{I}{I+1} e^{\im \theta})\big|, \quad \text{uniformly for } \theta \in (-\pi, \pi].\\
&\lim_{\ep \downarrow 0, t -s \to \infty} \big|\pcoree(z(\theta))\big| = \big|\pcorelim(\frac{1}{I+1} + \frac{I}{I+1} e^{\im \theta})\big|, \quad\  \text{uniformly for } \theta \in (-\pi, \pi].
\end{align*}  
Thanks to Lemma \ref{lem:SDM}, there exists $\delta > 0$ such that for $t-s$ large enough and $\ep > 0$ small enough, 
$$\big|\coreej (z(\theta))\big|, \big|\pcoree(z(\theta))\big| < 1 - \delta \text{ for } |\theta| > \zeta,$$
which completes our proof.
\end{proof}
For $\rhzrr$ \eqref{eq:res1}, we show that the indicator $\idc_{\{\ppole(z) > r'_2\}}$ prohibits $\theta$ to be too small.
\begin{lemma}\label{lem:pm:thetalb}
We can choose  $k_2$ large enough such that if $\big|\ppole(z(\theta))\big| > r'_2$ with $z(\theta) = \frac{1}{I+1} + \big(\frac{I}{I+1} - \frac{\beta}{\sqrt{t  -s +1}}\big) e^{\im \theta}$, then $|\theta| \geq (t - s +1)^{-\frac{1}{4}}$.
\end{lemma}
\begin{proof}
Note that $r'_2 = \rfunc(t-s, 2k_2 \beta) \geq 1 + \frac{2 k_2 \beta}{\sqrt{t - s +1}}$, it suffices to show that $$\big|\ppole(z(\theta))\big| > 1 + \frac{2k_2 \beta}{\sqrt{t  - s +1}} \text{ implies } |\theta| > C (t - s +1)^{-\frac{1}{4}}.$$
Referring to \eqref{eq:six2:temp7}, we taylor expand  $\ppole(1)$ around $\ep = 0$ 
\begin{equation}\label{eq:taylorpe1}
\ppole(1) = \frac{e^{-I \sqrt{\ep}}(1 - e^{\sqrt{\ep}}) + (e^{\sqrt{\ep}} - e^{-I\sqrt{\ep}}) e^{-\rho \sqrt{\ep}}}{(1-e^{(1-I)\sqrt{\ep}}) e^{-\rho \sqrt{\ep}} - (1- e^{\sqrt{\ep}})  e^{-2\rho \sqrt{\ep}}} = 1 + \frac{\rho I - \rho^2}{I} \ep + \OO(\ep^{\frac{3}{2}}).
\end{equation}
We highlight that there is no $\sqrt{\ep}$ term in the expansion, which is important for our proof. 
\bigskip
\\
We taylor expand $\ppole(z)$ at $z = 1$. Using \eqref{eq:taylorpe1}, $z(0) = 1 - \frac{\beta}{\sqrt{t - s +1}}$ and $\lim_{\ep \downarrow 0} \ppole'(1) = 1$, we find that for $t - s$ large enough and $\ep$ small enough, 
\begin{equation}\label{eq:taylorpe}
\ppole(z(0)) = \ppole(1) + \ppole'(1) (z(0) - 1) + \OO\big(z(0)- 1\big)^2 \leq 1 + \frac{ 2 (\rho I - \rho^2)}{I} \ep \leq 1 + \frac{C}{\sqrt{t - s +1}}. 
\end{equation}
In the last inequality, we used the condition $t -s \in [0, \ep^{-2} T]$. In addition, it is straightforward to see that $\frac{d}{d\theta} |\ppole(z(\theta))|\big|_{\theta = 0} = 0$ and there exists   $\zeta, C' > 0$ such that $\big|\frac{d^2}{d\theta^2} |\ppole(z(\theta))| \big| \leq C'$ for $|\theta| \leq \zeta$. Consequently, via Taylor expansion, for $|\theta| \leq \zeta$,
\begin{equation*}
\big|p_\ep (z(\theta)) \big| \leq \big|p_\ep (z(0)) \big| + \frac{C'\theta^2}{2}   \leq 1 + \frac{C}{\sqrt{t - s +1}}  + \frac{C' \theta^2}{2}.
\end{equation*}
Consequently, we have that when $|\theta| \leq \zeta$,
\begin{equation*}
\big|\ppole(z(\theta))\big| > 1 + \frac{2 k_2 \beta}{\sqrt{t - s +1}} \text{ implies } 1 + \frac{C}{\sqrt{t - s + 1}} + \frac{C' \theta^2}{2} \geq 1 + \frac{2  k_2 \beta}{\sqrt{t - s +1}}
\end{equation*}
By choosing $k_2$ large enough, we see that $|\theta| > (t - s +1)^{-1/4}$. 
\end{proof}
We are ready to prove Theorem \ref{prop:semiestimateint} for $(+-)$ case. As $\rhzrin = \rhzrb + \rhzrr$, it is enough to bound respectively $\rhzrb$ and $\rhzrr$. We begin with $\rhzrb$ \eqref{eq:blk1}. The proof consists a sequence of bounds on terms appearing in the integrand \eqref{eq:blk1}.
We parametrize by $z_1(\theta_1) = \frac{1}{I+1} + \big(\frac{I}{I+1} -\frac{\beta}{\sqrt{t - s +1}}\big) e^{\im \theta_1}$ and $z_2 (\theta_2) =\rad(z_1) e^{\im \theta}$. 
\bigskip
\\
($\rhzrb, z_1^{x_2 - y_1} z_2^{x_1 - y_2}$): \textbf{Show that} $|z_1^{x_2 - y_1} z_2^{x_1 - y_2}| \leq C e^{-\frac{\beta}{\sqrt{t - s +1}} (|x_1 - y_2| + |x_2 - y_1|)}$.\\
Observe that  $|z_1 (\theta_1)| = \big|\frac{1}{I+1} + \big(\frac{I}{I+1} -\frac{\beta}{\sqrt{t - s +1}}\big) e^{\im \theta_1}\big|$ reaches its maximum at $\theta_1 = 0$, hence $$|z_1 (\theta_1)| \leq |z_1 (0)| = 1 - \frac{\beta}{\sqrt{t - s +1}} \leq e^{-\frac{\beta}{\sqrt{t - s +1}}},$$ which gives  $|z_1|^{x_2 -  y_1} \leq e^{-\frac{\beta}{\sqrt{t - s +1}} |x_2 - y_1|}$. By $|z_2| \geq \rfunc(t - s , \beta) $, we deduce $|z_2|^{x_1 - y_2} \leq e^{-\frac{\beta}{\sqrt{t - s +1}} |x_1 - y_2|}$.
\bigskip
\\
($\rhzrb$, $\frac{1}{z_i}$): \textbf{Show that} $|\frac{1}{z_i}| \leq C$.\\
Clearly, $\frac{1}{|z_i|}$ is bounded for $z_1 \in \mcont(t-s, -\beta)$ and $z_2 \in \C_{\rad(z_1)}$.
\bigskip
\\
($\rhzrb$, $\interacte(z_1, z_2)$): \textbf{Show that} $\big|\interacte(z_1, z_2)\big| \leq C+ C \sqrt{t - s +1} (|\theta_1| + |\theta_2|)$. \\
To justify this claim, write
\begin{align*}
\interacte(z_1, z_2) &= \frac{q\nu - \nu + (\nu - q) q^{-\rho} z_2 + (1-q\nu) q^{-\rho}z_1 + (q-1) q^{-2\rho} z_1 z_2}{((q-1)q^{-2\rho} z_1 + (1-q\nu) q^{-\rho}) (z_2 - \ppole(z_1))}\\
\numberthis \label{eq:sixa:temp1}
&= 1 + \frac{q^{-\rho}(1+q)(\nu - 1)}{(q-1)q^{-2\rho} z_1 + (1-q\nu) q^{-\rho} } \cdot (z_2 - z_1) \cdot \frac{1}{z_2 - \ppole(z_1)}.
\end{align*}
Let us bound each factor on the RHS of \eqref{eq:sixa:temp1}. Referring to \eqref{eq:six:temp5}, we know that $\frac{1}{|z_2 - \ppole(z_1)|} \leq C \sqrt{t - s +1}$. Furthermore, we note that
\begin{equation*}
z_2 - z_1 = e^{\im \rad_2(z_1) \theta_2} - \big(\frac{1}{I+1} + (\frac{I}{I+1} - \frac{\beta}{\sqrt{t - s +1}} )e^{\im \theta_1}\big) = e^{\im \rad_2(z_1) \theta_2} - 1 - \big(\frac{I}{I+1} - \frac{\beta}{\sqrt{t - s +1}}\big) (e^{\im \theta_1} - 1) + \frac{\beta}{\sqrt{t - s +1}},
\end{equation*}
which implies $|z_2 - z_1| \leq C\big(\frac{1}{\sqrt{t - s +1}} + |\theta_1| + |\theta_2|\big)$. 
\bigskip
\\
In addition, we observe that 
\begin{equation*}
\lim_{\ep \downarrow 0} \frac{q^{-\rho} (1+q) (\nu - 1)}{(q-1)q^{-2 \rho} z_1 + (1-q\nu) q^{-\rho}} = - \frac{2I}{z_1 + I-1}.
\end{equation*}
Thus, $|\frac{q^{-\rho} (1+q) (\nu - 1)}{(q-1) q^{-2\rho} z_1 + (1- q\nu) q^{-\rho}}|$ is uniformly bound over  $\mcont(t -s, -\beta)$. Incorporating the bound for each factor on the RHS of \eqref{eq:sixa:temp1} gives the desired bound.
\bigskip
\\
($\rhzrb$, $\reme(z_i, t, s)$): \textbf{Show that} $|\reme(z_i, t, s)| \leq C$.\\
This is proved using the same reasoning for \eqref{eq:temp24}.
\bigskip
\\
($\rhzrb$, $\coreej (z_i)^{\floor{\frac{t-s}{J}}}$): \textbf{Show that}  $|\coreej (z_i(\theta_i))|^{\floor{\frac{t-s}{J}}} \leq C(\beta, T) e^{-C (t- s +1) \theta_i^2}$.\\
The result $\coree(z_1(\theta_1))|^{\floor{\frac{t-s}{J}}} \leq C(\beta, T) e^{-C(t - s +1)\theta_1^2}$ directly follows from Lemma \ref{lem:pm:std}. For $|\coreej(z_2(\theta_2))|^{\floor{\frac{t-s}{J}}}$, note that either $z_2 (\theta_2) = \rfunc(t, k_2\beta) e^{\im \theta_2}$ or $\rfunc(t, 3 k_2\beta) e^{\im \theta_2}$ (depending on the choice of $z_1$). Lemma \ref{lem:free:std} implies $|\coreej(z_2(\theta_2))|^{\floor{\frac{t-s}{J}}} \leq C(\beta, T) e^{-C (t  - s + 1) \theta_2^2}$.
\bigskip
\\
Via change of variable $z_1 = z_1 (\theta_1)$ and $z_2 = z_2 (\theta_2)$ and incorporating the preceding bounds, we arrive at 
\begin{align*}
&\big|\rhzrb\big((x_1, x_2), (y_1, y_2), t, s\big)\big|  \\
&\leq C(\beta, T) e^{-\frac{\beta}{\sqrt{t  -s +1}} (|x_2 -y_1| + |x_1 -y_2|)} \int_{-\pi}^{\pi} \int_{-\pi}^{\pi} (1 + \sqrt{t -s  +1} (|\theta_1| + |\theta_2|)) e^{-C (t -s +1) (\theta_1^2 + \theta_2^2)} d\theta_1 d\theta_2.
\end{align*}
Applying change of variable $\theta_i \to \frac{1}{\sqrt{t -s  +1}} \theta_i$, we conclude 
\begin{equation}\label{eq:blkbound1}
|\rhzrb\big((x_1, x_2), (y_1, y_2), t, s\big)| \leq  \frac{C(\beta, T)}{t -s +1} e^{-\frac{\beta}{\sqrt{t - s +1}} (|x_2 -y_1| + |x_1 - y_2|)}. 
\end{equation}
We turn to study $\rhzrr$ in \eqref{eq:res1}. The proof consists of bounds on terms involved in the integral \eqref{eq:res1}. In the following we parametrize $z_1 (\theta_1) =  \frac{1}{I+1} + \big(\frac{I}{I+1} -\frac{\beta}{\sqrt{t - s +1}}\big) e^{\im \theta_1}$.
\bigskip
\\
($\rhzrr$, $\frac{1}{z_1 \ppole(z_1)}$) \textbf{Show that} $\frac{1}{|z_1 \ppole(z_1)|} \leq C$. \\
By $\lim_{\ep \downarrow 0} \ppole(z_1) = \frac{(I+1) z_1 - 1}{z_1 + (I-1)}$, we deduce that $\frac{1}{|z_1 \ppole(z_1)|} \leq C$ for $z_1 \in \mcont(t-s, -\beta).$ 
\bigskip
\\
($\rhzrr$, $\reme(z_1, t, s) \reme(\ppole(z_1), t, s)$): \textbf{Show that} $|\reme(z_1, t, s) \reme(\ppole(z_1), t, s)| \leq C$.\\
By ($\rhzrb$, $\reme(z_i, t, s)$), we see that $|\reme(z_1, t, s)| \leq C$ for $z_1 \in \mcont(t-s, -\beta)$. We are left to show for $t-s$ large and $\ep$ small,
\begin{equation}\label{eq:six5:temp7}
|\reme(\ppole(z_1), t, s)| \leq C, \qquad z_1 \in \mcont(t-s, -\beta).
\end{equation}
Recall from \eqref{eq:temp17} that when $\ep>0$ is small enough, all the singularity of $\reme(z, t, s)$ belongs to the interval $[0, \Theta]$ for some $\Theta < 1$. As $\lim_{\ep \downarrow 0} \ppole(z) = \pstar(z)$, it suffices to show that 
\begin{equation*}
|\pstar(z_1)| \geq 1,\qquad z_1 \in \mcont. 
\end{equation*}
To justify this, we parametrize by $z_1(\theta) = \frac{1}{I+1} + \frac{I}{I+1} e^{\im \theta} \in \mcont$,
\begin{equation*}
|\pstar(z_1)|^2 = \frac{(I+1)^2}{I^2 +1 +2I \cos \theta} \geq 1.
\end{equation*}
Hence, we conclude \eqref{eq:six5:temp7}.
\bigskip
\\
($\rhzrr$, $\jprod(z_1)$): \textbf{Show that} $|\jprod(z_1)| \leq C e^{-\frac{\beta}{\sqrt{t -s +1}} (|x_2 - y_1| + |x_1 - y_2|)}$.\\
Referring to \eqref{eq:jterm}, 
\begin{align*}
\jprod(z_1) = \frac{q\nu - \nu + (\nu - q) q^{-\rho} \ppole(z_1) + (1-q\nu) q^{-\rho} z_1 + (q-1) q^{-2\rho} z_1 \ppole(z_1)}{(q-1)q^{-2\rho} z_1 + (1-q\nu) q^{-\rho}} z_1^{x_2 - y_1} \ppole(z_1)^{x_2 - y_1} \idc_{\{|\ppole(z_1)| > r'_2\}}.
\end{align*}
Let us first bound $z_1^{x_2 - y_1} \ppole(z_1)^{x_1 - y_2}  \idc_{\{|\ppole(z_1)| > r'_2\}}$. We know from the discussion in ($\rhzrb, z_1^{x_2 - y_1} z_2^{x_1 - y_2}$) that $|z_1| \leq e^{-\frac{\beta}{\sqrt{t -s +1}}}$. It is straightforward that $\big|\ppole(z_1)^{x_1 - y_2}  \idc_{\{|\ppole(z_1)| > r'_2\}}\big| \leq e^{-\frac{\beta}{\sqrt{t -s +1}} |x_1 - y_2|}$, which implies 
\begin{equation}\label{eq:six4:temp1}
|z_1 ^{x_2 - y_1} \ppole(z_1)^{x_1 - y_2}| \leq e^{-\frac{\beta}{\sqrt{t -s +1}}(|x_2 - y_1| + |x_1 - y_2|)}.
\end{equation}
In addition, one can compute 
\begin{equation*}
\lim_{\ep \downarrow 0} \frac{q \nu - \nu + (\nu - q) q^{-\rho} \ppole(z_1) + (1-q\nu) q^{-\rho}z_1 + (q-1)q^{-2\rho} z_1 \ppole(z_1)}{(q-1)q^{-2\rho} z_1 + (1-q\nu) q^{-\rho}} = \frac{1 - (1+I) \pstar(z) + (I-1) z + z\pstar(z)}{z + I-1},
\end{equation*}
recall $\pstar(z_1) = \frac{(I+1) z_1 - 1}{z_1 + (I-1)}$. This implies that 
\begin{equation}\label{eq:six4:temp2}
\big|\frac{q\nu - \nu + (\nu - q) q^{-\rho} \ppole(z_1) + (1-q\nu) q^{-\rho} z_1 + (q-1) q^{-2\rho} z_1 \ppole(z_1)}{(q-1)q^{-2\rho} z_1 + (1-q\nu) q^{-\rho}}\big| \leq C, \quad  z_1 \in \mcont(t-s, -\beta). 
\end{equation}
Combining \eqref{eq:six4:temp1} and \eqref{eq:six4:temp2} yields
\begin{equation*}
|\jprod(z_1)| \leq C e^{-\frac{\beta}{\sqrt{t - s +1}} (|x_2 - y_1| + |x_1 - y_2|)}.
\end{equation*}
\\
($\rhzrr$, $\pcoree(z_1(\theta_1))^{\floor{\frac{t-s}{J}}}$): \textbf{Show that} $|\pcoree(z_1(\theta_1))|^{\floor{\frac{t-s}{J}}} \leq C(\beta, T) e^{-C(t - s +1) \theta_1^2}$.\\
This directly follows from Lemma \ref{lem:pm:std}. 
\bigskip
\\
Consequently, we find that 
\begin{align*}
|\rhzrr\big((x_1, x_2), (y_1, y_2), t, s\big)| &\leq C \oint_{\mcont(t-s, -\beta)} \idc_{\{|\ppole(z_1 (\theta_1))| > r'_2\}} |\jprod(z_1(\theta_1))| |\pcoree(z_1 (\theta_1))|^{\floor{\frac{t-s}{J}}} \frac{d \theta_1}{ |\ppole(z_1(\theta_1))|},\\ 
&\leq C(\beta, T) e^{-\frac{\beta}{\sqrt{t - s +1}} (|x_2 - y_1| + |x_1 - y_2|)} \int_{-\pi}^{\pi} \idc_{\{\ppole(z_1(\theta_1)) > r'_2\}} e^{-C (t - s+1) \theta_1^2} d\theta_1, \\
&\leq C(\beta, T) e^{-\frac{\beta}{\sqrt{t - s +1}} (|x_2 - y_1| + |x_1 - y_2|)} \int_{|\theta_1| > (t - s +1)^{-\frac{1}{4}}} e^{-C (t - s + 1) \theta_1^2} d\theta_1,
\end{align*}
where the last inequality is due to Lemma \ref{lem:pm:thetalb}. Via change of variable $\theta_1 \to \frac{\theta_1}{\sqrt{t - s +1}}$, we get 
\begin{equation*}
\int_{|\theta_1| > (t - s +1)^{-\frac{1}{4}}} e^{-C (t - s + 1) \theta_1^2} d\theta_1 \leq \int_{|\theta_1| > (t - s +1)^{\frac{1}{2}}} e^{-C \theta_1^2} d\theta_1  \leq \frac{e^{-C(t - s +1)}}{\sqrt{t -s +1}} \leq \frac{C}{t - s +1}.
\end{equation*}
For the second inequality above, we used the fact $\int_{b}^{\infty} {e^{-x^2}} dx \leq \frac{C}{b} e^{-b^2}.$ Thereby, 
\begin{equation*}
|\rhzrr\big((x_1, x_2), (y_1, y_2), t, s\big)| \leq \frac{C(\beta, T)}{t - s  +1} e^{-\frac{\beta}{\sqrt{t - s +1}} (|x_2 - y_1| + |x_1 - y_2|)}.
\end{equation*}
Combining  this with the upper bound over $\rhzrb$ \eqref{eq:blkbound1} concludes Theorem \ref{prop:semiestimateint} part (a).
\bigskip
\\
For the gradient, note that applying $\na_{x_i}$ or $\na_{y_i}$ to \eqref{eq:blk1} and \eqref{eq:res1} will gives an additional $z_i^\pm - 1$ in the integrand of $\rhzrb$ and $\rhzrr$, we bound $|z_i (\theta_i) - 1| \leq C(\frac{1}{\sqrt{t -s +1}} + |\theta_i|)$ and perform the change of variable $\theta_i \to \frac{1}{\sqrt{t - s +1}} \theta_i$ produces an extra factor of $\frac{1}{\sqrt{t - s +1}}$.
Similarly, applying $\na_{x_1, x_2}$ will produce an additional factor $(z_1(\theta_1) - 1)(z_2(\theta_2) - 1)$. We bound $$|z_1(\theta_1) - 1| \cdot |z_2(\theta_2) - 1| \leq C \big(\frac{1}{\sqrt{t -s +1}} + |\theta_1|\big)\cdot \big(\frac{1}{\sqrt{t - s +1}} + |\theta_2|\big),$$ performing change of variable $\theta_i \to \frac{1}{\sqrt{t - s +1}} \theta_i$ produces an extra factor of $\frac{1}{t - s +1}$. This completes the proof of Theorem \ref{prop:semiestimateint} (b), (c). 
\subsection{Estimate of $\rhzrin$, the $(--)$ case}
We turn to prove Theorem \ref{prop:semiestimate} when  $x_2 - y_1 \leq 0$ and $x_1 - y_2 \leq 0$. This case is more involved than the previous one. One stumbling block is that we prefer to deform the $z_1$-contour to be $\C_{\rfunc(t-s, \beta)}$ to extract the spatial exponential decay. On the other hand, as depicted in Figure \ref{fig:contline},  the unit circle does not satisfy the steepest descent condition for $\pcoree(z)$. We resolve this issue by first shrinking the $z_1$-contour to $\mcont'(t-s, \beta)$,
then for $\rhzrb$, we re-deform the $z_1$-contour from $\mcont'(t-s, \beta)$ to $\C_{\rfunc(t-s, \beta)}$.
\bigskip
\\
We define $$\mcont'(t-s, \beta) = \partial\left\{ \{|z - \frac{1}{I+1}| \leq \frac{I}{I+1} + u_*\} \cap \{|z| \leq \rfunc(t-s, \beta)\}\right\},$$ 
recall $u_*$ is some fix constant which belongs to $(0, \delta \wedge \frac{1}{4I})$. Since $\mcontp(t-s, \beta) \to \mcontp$ as $t-s \to \infty$, it is clear that for $t-s$ large enough, $\mcont'(t-s, \beta)$ is admissible. Note that the parametrization of $\mcont'(t-s, \beta)$ is given by the right part of Figure \ref{fig:mp}.
\begin{figure}[ht]
\includegraphics[width=.85\linewidth]{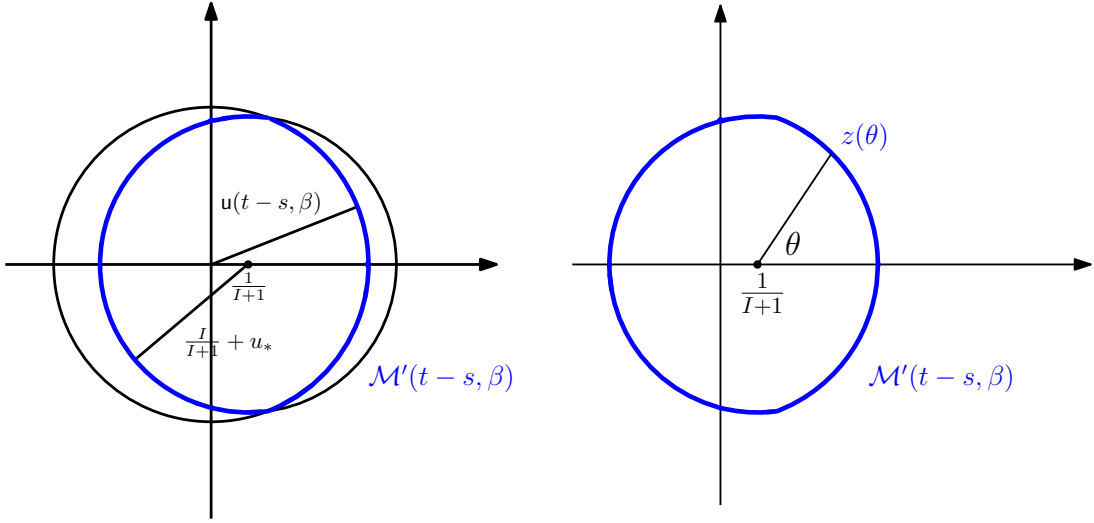}
\caption{The contour $\mcont'(t-s, \beta)$ and its parametrization}
\label{fig:mp}
\end{figure}
\bigskip
\\
We decompose $\rhzrin = \rhzrb + \rhzrr$,
\begin{align*}
\rhzrb\big((x_1, x_2), (y_1, y_2), t, s\big) &=  \oint_{\mcont'(t-s, \beta)} \oint_{C_{\rad_2(z_1)}} \interacte(z_1, z_2) \prod_{i=1}^2 \coreej (z_i)^{\floor{\frac{t-s}{J}}} \reme(z_i, t, s) z_i^{x_{3 - i} - y_i} \frac{dz_i}{2 \pi \im z_i},\\
\numberthis \label{eq:res2}
\rhzrr\big((x_1, x_2), (y_1, y_2), t, s\big) &= \oint_{\mcont'(t-s, \beta)} \idc_{\{|\ppole(z_1)| > r'_2\}} \jprod(z_1) \pcoree(z_1)^{\floor{\frac{t-s}{J}}} \reme(z_1, t, s) \reme(\ppole(z_1), t, s) \frac{dz_1}{2 \pi \im z_1 \ppole(z_1) }.
\end{align*}
Let us study $\rhzrb$ in the first place. As we mention at the beginning, when $x_2 -  y_1 \leq 0$,  $z_1$ does not favor the contour $\mcont'(t-s, \beta)$ to extract spatial decay. We prove in the following that we can re-deform the $z_1$-contour from $\mcont'(t-s, \beta)$ to $\C_{\rfunc(t -s, \beta)}$.
\begin{lemma}\label{lem:redeform}
For $t -s$ large enough and $\ep$ small enough, 
\begin{align*}
&\oint_{\mcont'(t-s, \beta)} \oint_{\C_{\rad_2(z_1)}} \interacte(z_1, z_2) \prod_{i=1}^2 \coreej (z_i)^{\floor{\frac{t-s}{J}}} \reme(z_i, t, s) z_i^{x_{3 - i} - y_i} \frac{dz_i}{2 \pi \im z_i} \\ 
&=\oint_{\C_{\rfunc(t-s, \beta)}} \oint_{\C_{\rad_2(z_1)}} \interacte(z_1, z_2) \prod_{i=1}^2 \coreej (z_i)^{\floor{\frac{t-s}{J}}} \reme(z_i, t, s) z_i^{x_{3 - i} - y_i} \frac{dz_i}{2 \pi \im z_i}.
\end{align*}
\end{lemma}  
\begin{proof} 
The contours $\mcont'(t-s, \beta)$ and $\C_{\rfunc(t-s, \beta)}$ share a common part $\Lambda(t-s) := \mcont'(t-s, \beta) \cap \C_{\rfunc(t-s, \beta)}$. We denote by $\Lambda_1(t-s) :=  \M'(t-s, \beta) \backslash \Lambda(t-s)$ and $\Lambda_2 (t-s) := \C_{\rfunc(t-s, \beta)} \backslash \Lambda(t-s)$.  Decompose the contour $\M'(t-s, \beta) = \Lambda(t-s) \cup \Lambda_1(t-s)$, $\C_{\rfunc(t-s, \beta)} = \Lambda(t-s) \cup \Lambda_2 (t-s)$,
it suffices to prove  
\begin{align*}
&\oint_{\Lambda_1(t-s)} \oint_{\C_{\rad_2(z_1)}} \interacte(z_1, z_2) \prod_{i=1}^2 \coreej (z_i)^{\floor{\frac{t-s}{J}}} \reme(z_i, t, s) z_i^{x_{3-i} - y_i} \frac{dz_i}{2 \pi \im z_i}  \\ 
\numberthis \label{eq:six:claim1}
&=\oint_{\Lambda_2(t-s)} \oint_{\C_{\rad_2(z_1)}}  \interacte(z_1, z_2) \prod_{i=1}^2 \coreej (z_i)^{\floor{\frac{t-s}{J}}} \reme(z_i, t, s) z_i^{x_{3-i} - y_i} \frac{dz_i}{2 \pi \im z_i}.
\end{align*}
To prove the above equation, we first claim that for $\ep$ small enough and $t - s \leq \ep^{-2} T$ large enough, 
\begin{equation}\label{eq:six3:temp6} 
\rad_2 (z_1) = \rfunc(t-s, k_2 \beta), \quad \forall\, z_1 \in \Lambda_1 (t-s) \cup \Lambda_2(t-s)
\end{equation}
That is to say, the $z_2$-contour is always $\C_{\rfunc(t-s, k_2 \beta)}$, which does not depend on the choice of $z_1$.
\bigskip
\\
To justify this claim, we need to prove for $\ep$ small enough and $t-s$ large enough
\begin{equation*}
|\ppole(z_1)| > \rfunc(t-s, 2k_2 \beta).
\end{equation*}	
We denote by $\Lambda^* = \mcont' \cap \C_1$, $\Lambda_1^* = \mcont' \backslash \Lambda^*$ and $\Lambda_2^* = \C_1 \backslash \Lambda^*$.
Note that as $t-s \to \infty$ and $\ep \downarrow 0$, 
\begin{equation*}
\Lambda_1 (t-s, \beta) \to \Lambda_1^*,\quad \Lambda_2 (t-s, \beta) \to \Lambda_2^*,\quad \ppole(z_1) \to \pstar(z_1),\quad \rfunc(t-s, 2k_2 \beta) \to 1.
\end{equation*}
Therefore, it suffices to consider the limit case and show that there exists $\delta > 0$ s.t.
\begin{equation*}
|\pstar(z_1)| = \bigg|\frac{(I+1) z_1 - 1}{z_1 + (I-1)}\bigg|> 1+\delta, \qquad z_1 \in \Lambda_1^* \cup \Lambda_2^*.
\end{equation*}
If $z_1 \in \Lambda_1^*$, we parametrize $z_1 (\theta) = \frac{1}{I+1} + \frac{I}{I+1} e^{\im \theta}$, where $|\theta| \geq \zeta$ for some positive constant $\zeta$. 
We readily compute 
\begin{equation*}
|\pstar(z_1(\theta))|^2 = \frac{(I+1)^2}{I^2 +1 +2I \cos \theta} \geq \frac{(I+1)^2}{I^2 +1 + 2I \cos \zeta} > 1. 
\end{equation*}
If $z_1 \in \Lambda_2^*$, we parametrize $z_1 (\theta) = e^{\im \theta}$ where $|\theta| \geq \zeta'$ for some positive constant $\zeta'$.
\begin{equation*}
|\pstar(z_1)|^2 
= \frac{(I+1)^2 + 1 - 2(I+1) \cos \theta}{(I-1)^2 +1 +2(I-1) \cos \theta} \geq \frac{(I+1)^2 + 1 - 2(I+1) \cos \zeta'}{(I-1)^2 +1 +2(I-1) \cos \zeta'} > 1,
\end{equation*}
where the first inequality above is due to the fact that $\frac{(I+1)^2 + 1 - 2(I+1) \cos \theta}{(I-1)^2 +1 +2(I-1) \cos \theta}$ increases as $|\theta| \in [0, \pi]$ increases.
\bigskip
\\
Having shown \eqref{eq:six3:temp6}, by Fubini's theorem,  the desired identity \eqref{eq:six:claim1} turns into 
\begin{align*}
&\oint_{\C_{\rfunc(t-s, k_2 \beta)}} \oint_{\Lambda_1(t-s)} \interacte(z_1, z_2) \prod_{i=1}^2 \coreej (z_i)^{\floor{\frac{t-s}{J}}} \reme(z_i, t, s) z_i^{x_{3-i} - y_i} \frac{dz_i}{2 \pi \im z_i} \\
&=\oint_{\C_{\rfunc(t-s, k_2 \beta)}} \oint_{\Lambda_2(t-s)}  \interacte(z_1, z_2) \prod_{i=1}^2 \coreej (z_i)^{\floor{\frac{t-s}{J}}} \reme(z_i, t, s) z_i^{x_{3-i} - y_i} \frac{dz_i}{2 \pi \im z_i}.
\end{align*}
In order to justify the identity above, it is sufficient to show that for all $z_2 \in \C_{\rfunc(t -s, k_2 \beta)}$, 
\begin{equation*}
\oint_{\Lambda_1(t-s)} \interacte(z_1, z_2)  \prod_{i=1}^2 \coreej (z_i)^{\floor{\frac{t-s}{J}}} \reme(z_i, t, s) z_i^{x_{3-i} - y_i} \frac{dz_i}{2 \pi \im z_i} = \oint_{\Lambda_2(t-s)}  \interacte(z_1, z_2)  \prod_{i=1}^2 \coreej (z_i)^{\floor{\frac{t-s}{J}}} \reme(z_i, t, s) z_i^{x_{3-i} - y_i} \frac{dz_i}{2 \pi \im z_i},
\end{equation*}
which is equivalent to 
\begin{equation}\label{eq:six:temp6}
\oint_{\pa \G(t-s)} \interacte(z_1, z_2)  \coreej (z_1)^{\floor{\frac{t-s}{J}}} \reme(z_1, t, s) z_1^{x_{2} - y_1} \frac{dz_1}{2 \pi \im z_1 } =  0,
\end{equation} 
where  $\pa \G(t-s)$ is the boundary of the crescent $\G(t-s) = \{|z| \leq \rfunc(t-s, \beta)\} \backslash \{|z - \frac{1}{I+1}| = \frac{I}{I+1} + u_*\}$, which is depicted in Figure \ref{fig:crescent} (note that  $\pa \G(t-s) = \Lambda_1 (t-s) \cup \Lambda_2 (t-s)$). 
\begin{figure}[ht]
\includegraphics[width=.65\linewidth]{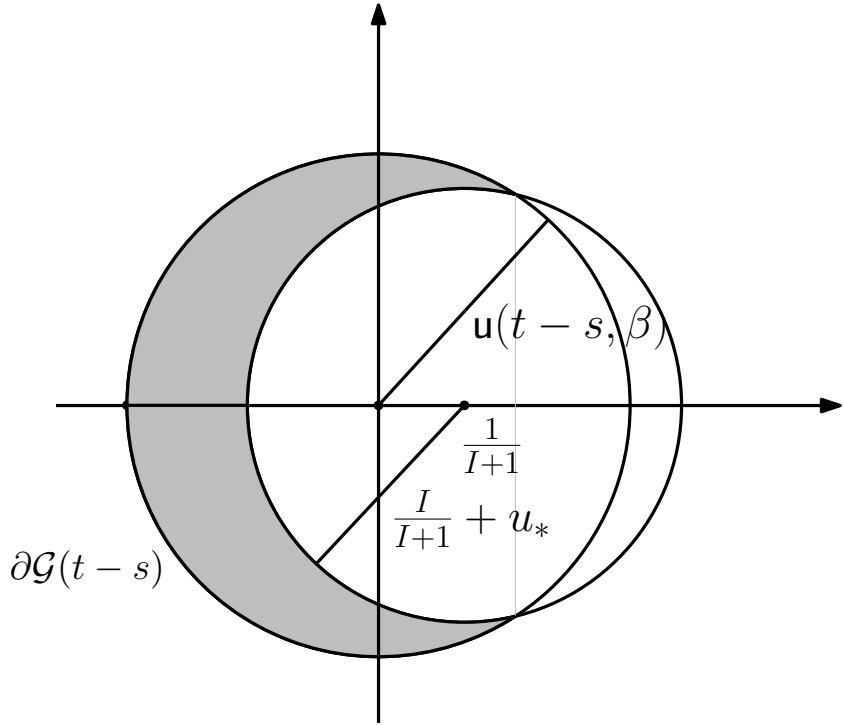}
\caption{The crescent $\mathcal{G}(t-s)$ and its boundary $\pa \G(t-s)$.}
\label{fig:crescent}
\end{figure}
\bigskip
\\
We set out proving \eqref{eq:six:temp6}. Since $\pa \G(t-s)$ is a closed curve, according to Cauchy's theorem, we only need to prove that no pole of the integrand \eqref{eq:six:temp6} lies inside of $\G(t-s)$. As we mentioned before, for $\ep$ small enough, the pole either equals $\pole(z_2)$ or belongs to $[0, \Theta]$. It is straightforward that  $[0, \Theta] \cap \G(t-s) = \emptyset$. Hence, we only need to show that $\pole(z_2) \notin \G(t - s)$ for all $z_2 \in \C_{\rfunc(t-s, k_2 \beta)}$. 
\bigskip
\\
We claim that for $t-s$ large enough and $\ep$ small enough,
$$\inf_{z_2 \in \C_{\rfunc(t-s, k_2 \beta)}}\Re \pole (z_2) > \sup_{z_1 \in \G(t-s)} \Re z_1.$$ 
Note that as $t- s \to \infty$ and $\ep \downarrow 0$,
\begin{equation*}
\C_{\rfunc(t-s, k_2 \beta)} \to \C_1, \quad \G(t-s) \to \G, \quad \pole(z) \to \sstar(z),
\end{equation*}
where  $\G := \{|z| \leq 1\} \backslash \{|z - \frac{1}{I+1}| = \frac{I}{I+1} + u_*\}$ and $\sstar(z) = \frac{(I-1) z +1}{I+1- z}$.
Therefore, it suffices to show that 
$$\inf_{z_2 \in \C_1}\Re \sstar (z_2) > \sup_{z_1 \in \G} \Re z_1.$$
To justify the inequality above, we first observe that $\sup_{z_1 \in \G} \Re z_1 < 1$. In addition, by setting $z_2 = e^{\im \theta}$, we see that 
\begin{equation*}
\Re \sstar(e^{\im  \theta})  = \Re \frac{(I-1) e^{\im \theta} + 1}{I+1 - e^{\im \theta}} = \frac{2 + (I^2 -2 ) \cos \theta}{(I+1)^2 + 1 - 2(I+1) \cos \theta} \geq 1.
\end{equation*}
Consequently, we proved $\pole(z_2) \notin \G(t-s)$, 
which completes the proof for Lemma \ref{lem:redeform}.
\end{proof} 
In summary,  we can write $\rhzrin = \rhzrb + \rhzrr$, where 
\begin{align*}
\numberthis \label{eq:blk2}
\rhzrb\big((x_1, x_2), (y_1, y_2), t, s\big) =  \oint_{\C_{\rfunc(t-s, \beta)}} \oint_{\C_{\rad_2(z_1)}} \interacte(z_1, z_2) \prod_{i=1}^2 \coreej (z_i)^{\floor{\frac{t-s}{J}}} \reme(z_i, t, s) z_i^{x_{3-i} - y_i} \frac{dz_i}{2 \pi \im z_i}
\end{align*}
and $\rhzrr$ is given by \eqref{eq:res2}.
\begin{lemma}\label{lem:mm:std}
For the parametrization $z(\theta)$ given in Figure \ref{fig:mp}, we have for $t-s \leq \ep^{-2} T$ large enough and $\ep > 0$ small enough
\begin{equation*}
|\coreej (z(\theta))|^{t-s} \leq C(\beta, T) e^{-C(t-s+1) \theta^2}, \qquad |\coreej (z(\theta))|^{t-s} \leq C(\beta, T) e^{-C(t-s+1) \theta^2}, \qquad z(\theta) \in \mcont'(t-s, \beta).
\end{equation*}
\begin{proof}
Similar to Lemma \ref{lem:pm:std}, it suffices to show that  there exists $C(\beta, T), C> 0$ s.t.
\begin{equation*}
\Re \log \coreej (z(\theta)) \leq \frac{C(\beta, T)}{t-s+1} - C \theta^2; \qquad \Re \log \pcoree (z(\theta)) \leq \frac{C(\beta, T)}{t-s+1} - C \theta^2.
\end{equation*}
We split out proof for $(\theta = 0)$, for ($\theta$ small) and for ($\theta$ large).
\begin{itemize}
\item $(\theta = 0): \Re \coreej (z(0)), \Re \pcoree(z(0)) \leq \frac{C(\beta, T)}{t-s+1}$.
\item ($\theta$ small): There exists $\zeta > 0$ and constants $C(\beta, T)$ and  $C > 0$ such that \eqref{eq:pm:stdequiv} holds for $|\theta| \leq \zeta$.
\item ($\theta$ large): We can find $\delta > 0$ such that $\big|\coreej (z(\theta))\big|, \big|\pcoree (z(\theta))\big| < 1 -\delta $ for $|\theta| > \zeta$.
\end{itemize}
Recall that $\mcont'(t-s, \beta)$ is the same as $\C_{\rfunc(t-s, \beta)}$ in a neighborhood of $1$, hence $z(\theta) \in \C_{\rfunc(t-s, \beta)}$ when $\theta$ is small. This being the case, the proof for $(\theta = 0)$ and ($\theta$ small) is the same as in Lemma \ref{lem:pm:std}.
For ($\theta$ large), since $\mcont'(t-s, \beta) \to \mcont'$ when $t-s \to \infty$ and $\mcont'$ satisfies the steepest descent condition, we find that for $t-s$ large and $\ep$ small,  
\begin{equation*}
|\coreej (z(\theta))| < 1 - \delta, \qquad |\pcoree (z(\theta))| < 1 - \delta, \qquad \text{ for } |\theta| \geq \zeta.
\end{equation*}
This completes our proof.
\end{proof}
\end{lemma}
We begin to estimate $\rhzrb$ in \eqref{eq:blk2}. In what follows, we check a sequence of bounds on terms involved in the integral \eqref{eq:blk2}, we  parametrize $z_1 = \rfunc(t-s, \beta) e^{\im \theta_1}$ and $z_2 = \rad_2 (z_1) e^{\im \theta_2}$.
\bigskip
\\
($\rhzrb, z_1^{x_2 - y_1} z_2^{x_1 - y_2}$): \textbf{Show that} $|z_1^{x_2 - y_1} z_2^{x_1 - y_2}| \leq e^{-\frac{\beta}{\sqrt{t-s+1}} (|x_2 - y_1| + |x_1  -y_2|)}$.\\
Since $z_1 \in \C_{\rfunc(t-s, \beta)}$ and $z_2 \in \C_{\rad_2 (z_1)}$, we have $|z_i| \geq \rfunc(t-s, \beta)$. Along with the condition $x_{3-i } - y_i \leq 0$ for $i = 1, 2$, we obtain $|z_1|^{x_2 - y_1} |z_2|^{x_1 - y_2} \leq e^{-\frac{\beta}{\sqrt{t-s+1}} (|x_2 - y_1| + |x_1 - y_2|)}$.
\bigskip
\\
($\rhzrb, \interacte(z_1, z_2)$): \textbf{Show that} $\big|\interacte(z_1, z_2)\big| \leq C + C \sqrt{t -s +1} (|\theta_1| + |\theta_2|)$.\\
The argument for this part is the same as in the $(+-)$ case.
\bigskip
\\
$(\rhzrb, \reme(z_i, t, s))$: \textbf{Show that} $|\reme(z_i, t, s)| \leq C$.\\
The argument is the same as $(+-)$ case $(\rhzrb, \reme(z_i, t, s))$.
\bigskip
\\
($\rhzrb, \coreej (z_i)^{\floor{\frac{t-s}{J}}}$): \textbf{Show that} $|\coreej (z_i(\theta_i))|^{\floor{\frac{t-s}{J}}} \leq C(\beta, T) \exp(- C (t - s +1) \theta_i^2 )$. 
\\
This is the content of Lemma \ref{lem:free:std}.
\bigskip
\\
As a consequence, we perform the same procedure as in the $(+-)$ case and get
\begin{align*}
|\rhzrb\big((x_1, x_2), (y_1, y_2), t, s\big)| &\leq C(\beta, T) e^{-\frac{\beta}{\sqrt{t  -s +1}} (|x_2 -y_1| + |x_1 -y_2|)} \int_{-\pi}^{\pi} \int_{-\pi}^{\pi} (1 +\sqrt{t - s +1} (|\theta_1| + |\theta_2|)) e^{- C (t - s +1) (\theta_1^2 + \theta_2^2)} d\theta_1 d\theta_2\\
\numberthis \label{eq:mm:inbound}
&\leq \frac{C(\beta, T)}{t -s +1} e^{-\frac{\beta}{\sqrt{t - s +1}}(|x_2 - y_1| + |x_1 - y_2|)}.
\end{align*}
We turn our attention to study $\rhzrr$, the proof similarly consists of bounds on terms involved in the integral \eqref{eq:res2}. In the following we parametrize $z_1 = z_1(\theta) \in \mcont'(t-s, \beta)$ as depicted in Figure \ref{fig:mp}. 
\bigskip
\\
($\rhzrr$, $\frac{1}{z_1 \ppole(z_1)}$): \textbf{Show that} $|\frac{1}{z_1 \ppole(z_1)}| \leq C$.\\
This is by the same argument as in the $(+-)$ case.
\bigskip
\\
$(\rhzrr, \reme(z_1, t, s) \reme(\ppole(z_1), t, s))$: \textbf{Show that} $|\reme(z_1, t, s) \reme(\ppole(z_1), t, s)| \leq C$.\\
The argument for this part is the same as $(\rhzrr, \reme(z_1, t, s) \reme(\ppole(z_1), t, s))$ in the $(+-)$ case.
\bigskip
\\
($\rhzrr$, $\pcoree(z_1)^{\floor{\frac{t-s}{J}}}$): \textbf{Show that}  $|\pcoree(z_1)|^{\floor{\frac{t -s }{J}}} \leq C(\beta, T) e^{- C(t - s +1) \theta^2}$. \\
This is the content of Lemma \ref{lem:mm:std}. 
\bigskip
\\
($\rhzrr, \jprod(z_1)$): \textbf{Show that} $|\jprod(z_1)| \leq C e^{-\frac{\beta}{2 \sqrt{t-s+1}} (|x_2  -y_1| + |x_1 - y_2|)}$.\\
Similar to the discussion in ($\rhzrr, \jprod(z_1)$) for the $(+-)$ case, it is sufficient to show  
$$|z_1^{x_2 - y_1} \ppole(z_1)^{x_1 - y_2} \idc_{\{|\ppole(z_1)| > r'_2\}}| \leq e^{-\frac{\beta}{2 \sqrt{t - s +1}} (|x_1- y_2| + |x_2 - y_1|)}.$$
Since for $z_1 \in \mcont(t -s, \beta)$, $|z_1|$ could be much less than $1$, we can not bound $z_1$ and $\ppole(z_1) $ separately. Instead, we write
\begin{equation}\label{eq:six4:temp3}
|z_1^{x_2  - y_1} \ppole(z_1)^{x_1 - y_2} \idc_{\{|\ppole(z_1)| > r'_2\}}| = |z_1 \ppole(z_1)|^{x_2 - y_1} |\ppole(z_1)|^{x_1 - x_2 +y_1 - y_2} \idc_{\{|\ppole(z_1)| > r'_2\}}.
\end{equation} 
Note that $x_1 - x_2 + y_1 - y_2 \leq 0$ (since $x_1 \leq y_1$ and $x_2 \leq y_2$), hence $$|\ppole(z_1)|^{x_1 - x_2 + y_1 -y_2} \idc_{\{|\ppole(z_1)| > r'_2\}} \leq  \rfunc(t-s, \beta)^{x_2 - x_1 + y_2 - y_1}.$$
We claim that
\begin{equation}\label{eq:mm:prodlb}
|z_1 \ppole(z_1)| > \rfunc(t-s, \beta), \quad z_1 \in \mcont'(t-s, \beta).
\end{equation} 
Once this is proved, by \eqref{eq:six4:temp3}
\begin{equation*}
|z_1^{x_2 - y_1} \ppole(z_1)^{x_1 - y_2} \idc_{\{|\ppole(z_1)| > r'_2\}}| \leq \rfunc(t -s, \beta)^{x_2 - y_1}  \rfunc(t -s, \beta)^{x_1 - x_2 + y_1  -y_2}  \leq e^{-\frac{\beta}{2 \sqrt{t- s +1}} (|x_1 - y_2| + |x_2 - y_1|)}.
\end{equation*}
\\
Let us justify \eqref{eq:mm:prodlb}. We decompose $\mcont'(t-s, \beta) = \Lambda(t -s) \cup \Lambda_1(t -s)$, where $\Lambda(t-s) = \mcont'(t-s, \beta) \cap \C_{\rfunc(t-s, \beta)}$ and $\Lambda_1(t-s) =  \M'(t-s, \beta) \backslash \Lambda(t-s)$.
If $z_1 \in \Lambda(t-s) \subseteq \C_{\rfunc(t-s, \beta)}$, we reparametrize by $z_1(\theta_1) = \rfunc(t-s, \beta) e^{\im \theta_1}$. It suffices to show that $$|\ppole(\rfunc(t-s, \beta) e^{\im \theta_1})| \geq 1.$$ By straightforward computation, one sees that $|\ppole(\rfunc(t-s, \beta) e^{\im \theta_1})|$ reaches its minimum at $\theta_1 = 0$. Hence we only need to prove that 
\begin{equation*}
\ppole(\rfunc(t -s, \beta)) \geq 1.
\end{equation*}
By \eqref{eq:taylorpe1}, $\ppole(1) = 1 + \frac{\rho I - \rho^2}{I} \ep + \OO(\ep^{\frac{3}{2}})$. In addition, direct computation yields $\lim_{\ep \downarrow 0} \ppole'(1) = 1$ and $|\ppole''(z)|$ uniformly bounded in a small neighborhood of $1$. Consequently, we taylor expand $\ppole(z)$ at $1$, 
\begin{equation*}
\ppole(\rfunc(t-s, \beta)) = \ppole(1) + \ppole'(1) (\rfunc(t-s, \beta) - 1) + \OO((\rfunc(t-s, \beta) -1)^2) \geq 1.
\end{equation*}
for $t -s$ large and $\ep$ small.
\bigskip
\\
If $z_1 \in \Lambda_1(t-s)$, which means that $|z_1 - \frac{1}{I+1}| = \frac{I}{I+1} + u_*$. We see that 
\begin{equation}\label{eq:mm:temp}
\lim_{\ep \downarrow 0} |z_1 \ppole(z_1)| = |z_1 \pstar(z_1)| = |(I+1) z_1 - 1| \cdot \big|\frac{z_1}{z_1 + I-1}\big| = (I + (I+1) u_*) \cdot \big|\frac{z_1}{z_1 + I-1}\big|  
\end{equation}
We claim that for $z_1 \in \Lambda_1 (t -s)$, $\big|\frac{z_1}{z_1 + I-1}\big| > \frac{1}{I}$. This could verify by inserting $z_1 = \frac{1}{I+1} + (\frac{I}{I+1} + u_*) e^{\im \theta}$ into \eqref{eq:mm:temp}. A geometric way to prove this inequality is that one has $|\frac{z}{z + I-1}| = \frac{1}{I}$ for all $z$ satisfying $|z - \frac{1}{I+1}| = \frac{I}{I+1}$. If ones increase the radius of circle $|z - \frac{1}{I+1}| = \frac{I}{I+1}\}$ (by $u_*$), the value of $\big|\frac{z}{z+I-1}\big|$ will also increase.
Thereby, 
\begin{equation*}
\lim_{\ep \downarrow 0} |z_1 \ppole(z_1)| \geq \frac{I + (I+1) u_*}{I} > 1. 
\end{equation*}
This implies when $z_1 \in \Lambda(t-s)$, $|z_1 \ppole(z_1)| > 1$ for $t-s$ large and $\ep$ small, which completes the proof of \eqref{eq:mm:prodlb}.
\bigskip
\\
Similar to the proof of Lemma \ref{lem:pm:thetalb} in the $(+-)$ case,  we find that $\{|\ppole(z_1(\theta))| > \rfunc(t-s, 2 k_2 \beta)\} \subseteq \{|\theta| > (t -s +1)^{-\frac{1}{4}}\}$, hence
\begin{align*}
&|\rhzrr\big((x_1, x_2), (y_1, y_2), t, s\big)| \leq C(\beta, T) e^{-\frac{\beta}{2 \sqrt{t -s +1}} (|x_2 - y_1| + |x_1 - y_2|)}\int_{-\pi}^{\pi} 1_{\{|\ppole(z_1(\theta))| \geq r'_2\}} e^{- C (t -s +1) \theta^2} d\theta\\ 
\numberthis \label{eq:mm:resbound}
&\leq C(\beta, T) e^{-\frac{\beta}{2\sqrt{t - s +1}} (|x_2 - y_1| + |x_1 - y_2|)} \int_{|\theta| > (t -s +1)^{-\frac{1}{4}}} e^{-C (t - s +1) \theta^2} d\theta  \leq \frac{C(\beta, T)}{t -s +1} e^{-\frac{\beta}{2\sqrt{t -s +1}} (|x_2 - y_1| + |x_1 - y_2|)} 
\end{align*}
Combining the bounds \eqref{eq:mm:inbound} and \eqref{eq:mm:resbound} implies Theorem \ref{prop:semiestimateint} (a).
\bigskip
\\
To estimate the gradient, the procedure is similar to in $(+-)$ case, note that applying $\na_{x_i}$ or $\na_{y_i}$ to \eqref{eq:blk2} and \eqref{eq:res2} gives an additional $z_i^\pm - 1$ factor, applying $\na_{x_1, x_2}$ produces an additional factor $(z_1 - 1)(z_2 - 1)$.  By  $|z_i (\theta_i) - 1| \leq C(\frac{1}{\sqrt{t -s +1}} + |\theta_i|)$, we conclude Theorem \ref{prop:semiestimateint} (b), (c). 

\subsection{Estimate of $\rhzrin$, the $(++)$ case}\label{sec:pp}
In this section, we fix $k_2 = 1$ in \eqref{eq:temp26}.
Note that $x_1 - y_2 \geq 0$, the difficulty for this case is to choose a suitable $z_1$-contour $\contone$ so as to extract the spatial decay from $z_1^{x_2 - y_1} \ppole(z_1)^{x_1 - y_2}$ in the integrand $\rhzrr$ \eqref{eq:six4:temp4}. 
Let us write 
\begin{equation*}
|z_1^{x_2  -y_1} \ppole(z_1)^{x_1 - y_2}| = |z_1 \ppole(z_1)|^{x_1 -y_2} |z_1|^{x_2 - x_1 + y_2 - y_1}.
\end{equation*}
We control respectively $|z_1 \ppole(z_1)|$ and $|z_1|$. We deform  the $z_1$-contour to 
\begin{equation*}
\mcont''(t-s, \ep, -k_1 \beta) = \{z_1: |z_1 \ppole(z_1)| = \rfunc(t-s, -k_1 \beta)\},
\end{equation*}
where $k_1$ is a positive constant that we will specify later. Note that when $I \geq 2$, this contour can only be implicitly defined (when $I = 1$ it is a circle). The following lemma provides a few properties of the contour.
\begin{lemma}\label{lem:pp:radexist}
For $t-s$ large enough and $\ep$ small enough, given $\theta \in (-\pi , \pi]$, there exists a unique positive $\pprad(\theta)$ such that 
\begin{equation}\label{eq:six4:temp7}
|z_1 \ppole(z_1)| = \rfunc(t-s, - k_1 \beta),  \quad z_1(\theta) = \frac{1}{I+1} + \pprad(\theta) e^{\im \theta}.
\end{equation}
$r_{\ep, t-s} (\theta)$ is infinitely differentiable with $\pprad' (0) = 0$. Moreover,  one has uniformly for $ \theta \in (-\pi, \pi]$, 
\begin{align*}
\lim_{\ep \downarrow 0, t-s \to \infty} \pprad (\theta) &= \frac{I}{I+1}, \\
\lim_{\ep \downarrow 0, t-s \to \infty} \pprad^{(n)} (\theta) &= 0, \quad \forall  n \in \NN. 
\end{align*}
where $f^{(n)}(\theta)$ represents the $n$-th derivative of $f(\theta)$. 
\end{lemma}
\begin{proof}
Let $w = t-s$, as $w \to \infty$ and $\ep \downarrow 0$, the equation $|z_1 \ppole(z_1)|  = \rfunc(w, -\beta)$ converges to  
\begin{equation}\label{eq:six4:temp5}
|z_1 \pstar (z_1)| = 
\bigg|\frac{z_1 ((I+1) z_1 - 1)}{z_1 + (I-1)}\bigg| = 1.
\end{equation} 
(note $\ppole(z_1) \to \pstar(z)$ and $\rfunc(w, \beta) \to 1$).
Setting $z_1  = \frac{1}{I+1} + r e^{\im \theta}$ in \eqref{eq:six4:temp5} yields
\begin{equation}\label{eq:six4:temp6}
(I+1)^4 r^4 + 2(I+1)^3  r^3 \cos \theta - 2 I^2 (I+1) r \cos \theta - I^4 = 0.
\end{equation}
Factorizing the LHS of \eqref{eq:six4:temp6} yields
\begin{equation*}
\big((I+1)^2 r^2 - I^2\big)\big((I+1)^2 r^2 + I^2 + 2(I+1) r \cos \theta\big) = 0.
\end{equation*}
Thus, \eqref{eq:six4:temp6} permits four root at 
\begin{equation}\label{eq:six:root}
r = \pm\frac{I}{I+1}, \frac{-1 \pm \im \sqrt{ \cos\theta^2 - I^2}  }{I+1}.
\end{equation}
We only care about positive root, thus the contour \eqref{eq:six4:temp5}
can be parametrized by $z_1 (\theta) = \frac{1}{I+1} + \frac{I}{I+1} e^{\im \theta}$.
\bigskip
\\
Similarly, inserting $z_1 = \frac{1}{I+1} + r e^{\im \theta}$ in \eqref{eq:six4:temp7} yields
\begin{equation*}
a_0 (\ep, w) r^4 + 2 a_1(\ep, w) r^3 \cos \theta + a_2(\ep, w) r^2 + a_3(\ep, w) r \cos \theta + a_4(\ep, w) = 0,
\end{equation*}
where $\{a_i(\ep, w)\}_{i=0}^4$ are constants depending on $\ep, w$ that converge to the coefficient in \eqref{eq:six4:temp6}:
\begin{equation}\label{eq:temp25}
\lim_{\ep \downarrow 0, w \to \infty} \big(a_0(\ep, w), a_1(\ep, w), a_2(\ep, w), a_3(\ep, w), a_4(\ep, w)\big) = \big((I+1)^4, 2(I+1)^3, 0, -2I^2(I+1), -I^4\big).
\end{equation}
Denote by 
\begin{align*}
P(\theta, r)  &= (I+1)^4 r^4 + 2(I+1)^3  r^3 \cos \theta - 2 I^2 (I+1) r \cos \theta - I^4 \\
P_{\ep, w}(\theta, r) &= a_0 (\ep, w) r^4 + 2 a_1(\ep, w) r^3 \cos \theta + a_2(\ep, w) r^2 + a_3(\ep, w) r \cos \theta + a_4(\ep, w).
\end{align*}
By \eqref{eq:temp25}, when $\ep$ is small and $w$ is large, $P_{\ep, w}(\theta, 0) < 0$ and $P_{\ep, w} (\theta, +\infty) = +\infty$. By continuity, for each $\theta \in (-\pi, \pi]$, $P_{\ep, w}(\theta, r) = 0$ admits a positive root. Since $P_{\ep, w}(\theta, r)$ is a perturbation of $P(\theta, r)$, as $\ep \downarrow 0$ and $w \to \infty$, the roots of $P_{\ep, w}(\theta, r)$ converge to those in \eqref{eq:six:root}, which implies the the positive root of $P_{\ep, w}(\theta)$ is unique for $\ep$ small and $t$ large. We denote this unique positive root by $r_{\ep, w}(\theta)$. It is also clear that for $\theta \in (-\pi, \pi]$
\begin{equation}\label{eq:six4:temp9}
\lim_{\ep \downarrow 0, w \to \infty} r_{\ep, w}(\theta)  = \frac{I}{I+1}\  \text{ uniformly}.
\end{equation}
Moreover, for all $\theta \in [-\pi, \pi]$, $r = \frac{I}{I+1}$ is a simple root of $P(\theta, r) = 0$. Hence, $\frac{\partial}{\partial r} P(\theta, r)\big|_{r = \frac{I}{I+1}} \neq 0$, using
implicit function theorem shows that for $\ep$ small and $w$ large, $\pprad(\theta)$ is smooth over $(-\pi, \pi]$. Furthermore, 
\begin{equation*}
r_{\ep, w}' (0) = -\frac{\frac{\pa}{\pa \theta} P_{\ep, w}(\theta, r_{\ep, w} (0))\big|_{\theta = 0}}{\frac{\pa}{\pa r} P_{\ep, w}(0, r)\big|_{r = r_{\ep, w}(0)}} = -\frac{\big(-2 a_1 (\ep, w) r_{\ep, w} (0)^3 \sin \theta + 2 I^2 (I+1) r_{\ep, w} (0) \sin \theta\big) \big|_{\theta = 0}}{\frac{\pa}{\pa r} P_{\ep, w}(0, r)\big|_{r = r_{\ep, w}(0)}} = 0.
\end{equation*}
In addition, by \eqref{eq:six4:temp9} and implicit function theorem, uniformly over $\theta \in (-\pi, \pi]$
\begin{equation*}
 \lim_{\ep \downarrow 0, w \to \infty} \ppradw^{(n)} (\theta) = \Big(\frac{I}{I+1}\Big)^{(n)} =  0,
\end{equation*}
this completes our proof.
\end{proof}
We adopt the parametrization $z_1(\theta_1) =  \frac{1}{I+1} + r_{\ep, t-s}(\theta_1) e^{\im \theta_1} \in \mcont''(t-s, \ep, -k_1 \beta)$. From the preceding lemma, as $t-s \to \infty$ and $\ep \downarrow 0$, $\mcont''(t-s, \ep, -k_1 \beta) \to \mcont$, thus the contour $\mcont''(t-s, \ep, -k_1 \beta)$ is admissible for $\ep$ small and $t-s$ large. As before, we decompose  $\rhzrin = \rhzrb + \rhzrr$, where
\begin{align}\label{eq:blk3}
\rhzrb\big((x_1, x_2), (y_1, y_2), t, s\big) &=  \oint_{\mcont''(t-s, \ep, -k_1 \beta)} \oint_{\C_{\rad_2(z_1)}} \interacte(z_1, z_2) \prod_{i=1}^2 \coreej (z_i)^{\floor{\frac{t-s}{J}}} \reme(z_i, t, s ) z_i^{x_{3-i} - y_i} \frac{dz_i}{2 \pi \im z_i},\\
\label{eq:res3}
\rhzrr\big((x_1, x_2), (y_1, y_2), t, s\big) &= \oint_{\mcont''(t-s, \ep, -k_1 \beta)} \idc_{\{|\ppole(z_1)| > r'_2\}} \jprod(z_1) \pcoree(z_1)^{\floor{\frac{t-s}{J}}} \reme(z_1, t, s ) \reme(\ppole(z_1), t, s ) \frac{dz_1}{2 \pi \im z_1 \ppole(z_1) }.
\end{align}
\begin{lemma}\label{lem:pp:radprop}
There exists $K > 0$ (which depends on $k_1$) such that for $t-s \leq \ep^{-2} T$ large enough and $\ep > 0$ small enough, 
we have 
\begin{align*}
z_1 (0) &\geq 1 - \frac{K \beta}{\sqrt{t-s+1}},\\ 
|z_1 (\theta)| &\leq 1 - \frac{k_1\beta}{5\sqrt{t-s+1}}.
\end{align*}
\end{lemma}
\begin{proof}
Consider an alternate parametrization $\widetilde{z}_1(\theta) = \ppradt(\theta) e^{\im \theta} \in \mcont''(t-s, \ep, -k_1 \beta)$, where the existence and uniqueness of $\ppradt(\theta)$ are confirmed by Lemma \ref{lem:pp:radexist}. It suffices  to show for $t-s \leq \ep^{-2} T$ large enough and $\ep > 0$ small enough, 
\begin{equation}\label{eq:six4:temp8}
\ppradt (0) \geq 1 - \frac{K \beta}{\sqrt{t-s+1}}; \qquad |\ppradt (\theta)| \leq 1 - \frac{k_1 \beta}{5\sqrt{t-s+1}}, \quad  \forall \theta \in (-\pi, \pi].
\end{equation}
We prove \eqref{eq:six4:temp8} in two steps. 
\begin{itemize}
\item First, $\frac{k_1\beta}{5\sqrt{t-s+1}} \leq 1 - \ppradt (0)  \leq \frac{K \beta}{ \sqrt{t-s+1}}$. 
\item Second, $|\ppradt(\theta)| \leq \ppradt (0)$ for $\theta \in (-\pi, \pi]$.
\end{itemize}
We verify the first bullet point. Note that uniformly in an neighborhood of $1$, 
\begin{equation*}
\lim_{\ep \downarrow 0} \ppole(z) = \pstar(z), \qquad \lim_{\ep \downarrow 0} \ppole'(z) = \pstar'(z).
\end{equation*}
Referring to \eqref{eq:pstar},  $\frac{d}{dz} z\pstar(z)\big|_{z=1}  = 2$. Thus, there exists $\delta > 0$ such that for $\ep$ small enough and $z \in (1-\delta, 1+ \delta)$, 
\begin{equation}\label{eq:six5:temp2}
|(z \ppole(z))' -2| < \frac{1}{2}.
\end{equation} 
We taylor expand  $z \ppole(z)$ around $z = 1$,  
\begin{equation}\label{eq:six5:temp1}
\rfunc(t-s, -k_1 \beta) =\ppradt(0) \ppole(\ppradt(0)) = \ppole(1) + \frac{d}{dz}(z \ppole(z))\bigg|_{z=x} \cdot 
(\ppradt(0)-1), \quad  x \in (1-\delta, 1+\delta).
\end{equation}
Referring to \eqref{eq:taylorpe1}, we see $\ppole(1) \geq 1$ for $\ep$ small enough, which implies  
\begin{equation*}
1 \geq \rfunc(t-s, -k_1 \beta) \geq 1 + \frac{d}{dz}(z \ppole(z))\big|_{z=x} \cdot 
(\ppradt(0)-1).
\end{equation*}
Hence, $\ppradt (0) \leq 1$. We have by \eqref{eq:six5:temp2} and \eqref{eq:six5:temp1}
\begin{align*}
\rfunc(t-s, -k_1 \beta) &\geq \ppole(1) + \frac{5}{2} (\ppradt(0) - 1),\\
\rfunc(t-s, -k_1 \beta) &\leq \ppole(1) + \frac{3}{2} (\ppradt(0) - 1).
\end{align*}
The first inequality yields 
\begin{equation*}
1 - \ppradt(0) \geq \frac{2}{5} \big(\ppole(1) - \rfunc(t-s, -k_1 \beta)\big) \geq \frac{2}{5} \big(1 - \rfunc(t-s, -k_1 \beta)\big) \geq \frac{k_1 \beta}{5\sqrt{t-s+1}}.
\end{equation*} 
which gives the lower bound. The second inequality indicates that (by \eqref{eq:taylorpe1})
\begin{equation*}
1 - \ppradt(0) \leq \frac{2}{3} \big(\ppole(1) - \rfunc(t-s, -k_1 \beta)\big) \leq \frac{2}{3}\big(1 - \rfunc(t-s, -k_1 \beta)\big) +  \frac{\rho I - \rho^2}{I} \ep.
\end{equation*}
Owing to $\ep \leq \sqrt{\frac{T}{t-s}}$, we see that $1 - \ppradt(0) \leq \frac{K \beta}{\sqrt{t-s+1}}$ for constant  $K$ large enough, which concludes the first bullet point.
\bigskip
\\
%
We move on proving the second bullet point. We set $F_{\theta} (r) = |r \ppole(re^{\im \theta})|$. When  $\ep$ small and $t-s$ large, we readily compute (note that $\ppradt(0)$ is nearly $\frac{I}{I+1}$ and $\ppole$ approximates $\pstar$)
\begin{equation*}
|F_{\theta }(\ppradt (0))|^2 =  \ppradt (0)^2 |\ppole(\ppradt (0) e^{\im \theta})|^2 = \frac{c_1^2 + c_2^2 - 2c_1 c_2 \cos \theta}{d_1^2 + d_2^2 + 2d_1 d_2 \cos \theta}, \qquad c_1, c_2, d_1, d_2 > 0,
\end{equation*}
which implies that $|F_{\theta }(\pprad (0))|$ reaches its minimum at $\theta = 0$. In other words, $F_\theta (\pprad(0)) \geq F_0 (\pprad(0)) = \rfunc(t-s, - k_1 \beta)$. In addition, $F_\theta (0) = 0 $. By intermediate value theorem, for each fixed $\theta \in (-\pi, \pi]$, the equation $F_\theta (r) = \rfunc(t-s, -k_1 \beta)$ admits a root $r \in (0, \ppradt(0)]$. By uniqueness, this root equals $\ppradt(\theta)$, thereby $\ppradt(\theta) \leq \ppradt(0)$ for all $\theta \in (-\pi, \pi]$. 
\end{proof}

\begin{lemma}\label{lem:pp:thetalb}
For $k_1$ large enough, $t-s \leq \ep^{-2} T$ large enough and $\ep > 0$ small enough, the condition $|\ppole(z(\theta))| > r'_2$ with $z(\theta) = \frac{1}{I+1} + \pprad(\theta)e^{\im \theta} \in \mcont''(t-s, \ep, \beta)$ implies $|\theta| \geq (t-s+1)^{-\frac{1}{4}}$.
\end{lemma}
\begin{proof}
The proof is similar to Lemma \ref{lem:pm:thetalb}. Since $k_2 = 1$, we have $r'_2 = \rfunc(t-s, -2\beta).$ Hence, $r'_2 \geq 1 - \frac{4 \beta}{ \sqrt{t-s+1}}$. It suffices to show that
\begin{equation*}
|\ppole(z(\theta))| \geq 1 - \frac{4 \beta}{ \sqrt{t-s+1}} \Rightarrow |\theta| \geq (t-s+1)^{-\frac{1}{4}}.
\end{equation*} 
Referring to \eqref{eq:taylorpe}, we see that 
\begin{equation*}
\ppole(z(0)) = \ppole(1) + \ppole'(1) (z(0) - 1) + \OO\big(z(0)- 1\big)^2. 
\end{equation*}
By \eqref{eq:taylorpe1}, we see $\ppole(1) \leq 1 + \frac{C}{\sqrt{t-s+1}}$ for some positive constant $C$, together with the fact 
\begin{equation*}
z(0) - 1 \leq \frac{-k_1 \beta}{5 \sqrt{t-s+1}}, \qquad \lim_{\ep \downarrow 0} \ppole'(1) = 1,
\end{equation*} 
we obtain 
\begin{equation*}
\ppole(z(0)) \leq 1 + \frac{C}{\sqrt{t-s+1}} - \frac{k_1 \beta}{10 \sqrt{t-s+1}}.
\end{equation*}
In addition, by Lemma \ref{lem:pp:radexist}, $\pprad'(0) = 0$. Using this, it is straightforward to compute $\frac{d}{d\theta} |\ppole(z(\theta))|\big|_{\theta = 0} = 0$ and there exists   $\zeta, C' > 0$ such that $\big|\frac{d^2}{d\theta^2} |\ppole(z(\theta))| \big| \leq C'$ for $|\theta| < \zeta$. Consequently, one has by taylor expansion 
\begin{equation*}
|\ppole(z(\theta))| \leq \ppole(z(0)) + C' \theta^2 \leq 1 + \frac{10 C - k_1 \beta}{10 \sqrt{t-s+1}} + C' \theta^2.
\end{equation*}
Thereby, we can pick $k_1$ large enough s.t. $|\ppole(z(\theta))| \geq 1 - \frac{4 \beta}{\sqrt{t-s+1}}$ implies $|\theta| \geq (t-s+1)^{-\frac{1}{4}}$.
\end{proof}
\begin{lemma}\label{lem:pp:std}
For $t-s$ large and $\ep$ small, there  exists positive constants $C(\beta, T), C$ such that  
\begin{equation*}
|\coreej(z(\theta))|^{t-s} \leq C(\beta, T) e^{-C(t-s+1) \theta^2}, \quad |\pcoree(z(\theta))|^{t-s} \leq C(\beta, T) e^{-C(t-s+1) \theta^2} \quad \text{with } z(\theta) = \frac{1}{I+1} + \pprad(\theta) e^{\im \theta}.
\end{equation*}
\end{lemma}
\begin{proof}
Similar to Lemma \ref{lem:pm:std}, it suffices to show that  there exists $C(\beta, T), C> 0$ s.t.
\begin{equation*}
\Re \log \coreej (z(\theta)) \leq \frac{C(\beta, T)}{t-s+1} - C \theta^2; \qquad \Re \log \pcoree (z(\theta)) \leq \frac{C(\beta, T)}{t-s+1} - C \theta^2.
\end{equation*}
We split out proof for $(\theta = 0)$, for ($\theta$ small) and for ($\theta$ large).
\begin{itemize}
\item $(\theta = 0): \Re \coreej (z(0)), \Re \pcoree(z(0)) \leq \frac{C(\beta, T)}{t-s+1}$.
\item ($\theta$ small): There exists $\zeta > 0$ and constants $C(\beta, T)$ and  $C > 0$ such that \eqref{eq:pm:stdequiv} holds for $|\theta| \leq \zeta$.
\item ($\theta$ large): There exists $\delta > 0$ such that $\big|\coreej (z(\theta))\big|, \big|\pcoree (z(\theta))\big| < 1 -\delta $ for $|\theta| > \zeta$.
\end{itemize}
Owing to Lemma \ref{lem:pp:radprop}, $\frac{K}{\sqrt{t-s+1}} \leq 1 - z(0) \leq \frac{k_1}{5\sqrt{t-s+1}}$, hence the argument for $(\theta = 0)$ is similar to  Lemma \ref{lem:mm:std}. 
\bigskip
\\
For ($\theta$ small), using Lemma \ref{lem:pp:radexist}, one has
\begin{equation*}
\pprad'(0) = 0, \qquad \lim_{\ep \downarrow 0, t-s \to \infty}  \pprad''(\theta) = 0, \qquad \lim_{\ep \downarrow 0, t-s \to \infty} \pprad'''(\theta) = 0.
\end{equation*}
Using this, after a tedious but straightforward calculation (recall $z(\theta) = \frac{1}{I+1} + \frac{I}{I+1} \pprad(\theta)$),
\begin{align*}
\pa_\theta (\log \coreej (z(\theta))) \big|_{\theta = 0} &\in \im \RR, \qquad \qquad \qquad \qquad \pa_\theta (\log \pcoree (z(\theta))) \big|_{\theta = 0} \in \im \RR\\ 
\lim_{\epsilon \downarrow 0, t-s \to \infty} \pa_\theta^2 (\log \coreej (z(\theta))) \big|_{\theta = 0} &= -\frac{I^2 J V_*}{(I+1)^2}, \quad \lim_{\epsilon \downarrow 0, t-s \to \infty} \pa_\theta^2 (\log \pcoree (z(\theta))) \big|_{\theta = 0} = -\frac{2I^2 J V_*}{(I+1)^2}\\
\big|\pa_\theta^3 (\log \coreej (z(\theta)))\big| &\leq C, \qquad \qquad \qquad \qquad \quad \  \big|\pa_\theta^3 (\log \pcoree (z(\theta)))\big| \leq C.
\end{align*}
The last line holds for all $|\theta| < \zeta$ where $\zeta > 0$ is a constant.
Hereafter, the argument is same as in Lemma \ref{lem:pm:std}, we do not repeat it here.
\bigskip
\\
For ($\theta$ large), since
$$\lim_{\ep \downarrow 0, t-s \to \infty} \pprad(\theta) = \frac{I}{I+1}, \qquad \text{uniformly for } \theta \in (-\pi, \pi],$$
we have
\begin{align*}
\lim_{\ep \downarrow 0, t-s \to \infty} \coreej(z(\theta)) &= \corelim(\frac{1}{I+1} + \frac{I}{I+1} e^{\im \theta}), \quad \text{ uniformly over } \theta \in (-\pi, \pi],\\
\lim_{\ep \downarrow 0, t-s \to \infty} \pcoree(z(\theta)) &= \pcorelim(\frac{1}{I+1} + \frac{I}{I+1} e^{\im \theta}), \quad \text{ uniformly over } \theta \in (-\pi, \pi].
\end{align*}
By the steepest descent condition \eqref{eq:SDM}, we conclude ($\theta$ large).
\end{proof}
Now we are ready to bound $\rhzrb$ and $\rhzrr$. We begin with $\rhzrb$ given by  \eqref{eq:blk3}. The proof consists of bounding each terms involved in the integrand \eqref{eq:blk3}. 
We parametrize $z_1 (\theta_1) = \pprads(\theta_1) e^{\im \theta_1}$, $z_2 (\theta_2) = \rad_2 (z_1) e^{\im \theta_2}$.
\bigskip
\\
($\rhzrb, z_1^{x_2 - y_1} z_2^{x_1 - y_2}$): \textbf{Show that} $|z_1^{x_2 - y_1} z_2^{x_1 - y_2}| \leq e^{-\frac{\beta}{\sqrt{t-s+1}} (|x_1 - y_2| + |x_2 - y_1|)}$.
\\
By Lemma \ref{lem:pp:radprop}, we see that $|z_1| \leq e^{-\frac{\beta}{\sqrt{t-s+1}}}$, since $\rad_2 (z_1)$ equals $\rfunc(t-s, -\beta)$ or $\rfunc(t-s, -3\beta)$, we find that $|z_2| \leq e^{-\frac{\beta}{\sqrt{t-s+1}}}$, which implies $|z_1^{x_2 - y_1} z_2^{x_1 - y_2}| \leq e^{-\frac{\beta}{\sqrt{t-s+1}} (|x_2 - y_1| + |x_1 - y_2|)}$.
\bigskip
\\
($\rhzrb, \interacte(z_1, z_2)$): \textbf{Show that} $\big|\interacte(z_1, z_2)\big| \leq C + C \sqrt{t - s +1} (|\theta_1| + |\theta_2|)$.\\
By the argument in ($\rhzrb, \interacte(z_1, z_2)$) in ($+-$) case. It suffices to show that $|z_2  - z_1| \leq C(\frac{1}{\sqrt{t - s +1}} + |\theta_1| + |\theta_2|)$. Note that
\begin{equation}\label{eq:temp27}
|z_2(\theta_2) - z_1(\theta_1)| \leq  |z_1 (\theta_1) - 1| + |z_2 (\theta_2) - 1|   \leq |\pprads(\theta_1)e^{\im \theta_1} - 1| + |\rad(z_1) e^{\im \theta_2} - 1|.
\end{equation} 
By Lemma \ref{lem:pp:radexist} and Lemma \ref{lem:pp:radprop}, we know that  $|\pprads(0) - 1| \leq \frac{C}{\sqrt{t-s+1}}$ and   $\lim_{\ep \downarrow 0, t-s \to \infty}\pprads'(\theta) = 0$ uniformly for $\theta \in (-\pi, \pi]$, we see that 
\begin{equation}\label{eq:temp28}
|\pprads(\theta_1) e^{\im \theta_1} - 1|  \leq |\pprads(\theta_1) - \pprads(0)| + |\pprads(0) -1| + |e^{-\im \theta_1} - 1| \leq C(\frac{1}{\sqrt{t-s+1}} + |\theta_1|)
\end{equation} 
Since $\rad(z_1) = \rfunc(t-s, \beta)$ or $\rad(z_1) = \rfunc(t-s, 3\beta)$, we have 
\begin{equation}\label{eq:temp29}
|\rad(z_1) e^{\im \theta_2} - 1| \leq C(\frac{1}{\sqrt{t-s+1}} + |\theta_2|)
\end{equation}
Incorporating the bound \eqref{eq:temp28} and \eqref{eq:temp29} into the RHS of \eqref{eq:temp27}, we conclude
$|z_2(\theta_2)  - z_1(\theta_1)| \leq C(\frac{1}{\sqrt{t-s+1}} + |\theta_1| + |\theta_2|)$.
\bigskip
\\
$(\rhzrb, \reme(z_i, t, s))$: \textbf{Show that} $|\reme(z_i, t, s)| \leq C$.\\
This is the same as $(+-)$ case $(\rhzrb, \reme(z_i, t, s))$.
\bigskip
\\
($\rhzrb, \coreej (z_i)^{\floor{\frac{t-s}{J}}}$): \textbf{Show that} $|\coreej (z_i)|^{\floor{\frac{t-s}{J}}} \leq C(\beta, T) \exp(- C (t-s+1)\theta_i^2 )$. \\
This is the content of Lemma \ref{lem:pp:std}.
\bigskip
\\
Consequently, we perform the same procedure as in the $(+-)$ case and get
\begin{align*}
|\rhzrb| &\leq C(\beta, T) \int_{-\pi}^{\pi} \int_{-\pi}^{\pi} (1 +\sqrt{t - s +1} (|\theta_1| + |\theta_2|)) e^{-C (t - s +1) (\theta_1^2 + \theta_2^2)} d\theta_1 d\theta_2\\
&\leq \frac{C(\beta, T)}{t - s +1} e^{-\frac{\beta}{\sqrt{t - s +1}}(|x_2 - y_1| + |x_1 - y_2|)}.
\end{align*}
Let us move on bounding $\rhzrr$ with integral expression \eqref{eq:res3}. We parametrize by $z_1(\theta) =  \pprads(\theta) e^{\im \theta} \in \mcont''(t-s, \ep, -k_1 \beta).$
\bigskip
\\
($\rhzrr$, $\frac{1}{z_1 \ppole(z_1)}$): \textbf{Show that} $|\frac{1}{z_1 \ppole(z_1)}| \leq C$.\\
This is by the same argument as in the $(+-)$ case.
\bigskip
\\
$(\rhzrr, \reme(z_1, t, s) \reme(\ppole(z_1), t, s))$: \textbf{Show that} $|\reme(z_1, t, s) \reme(\ppole(z_1), t, s)| \leq C$.\\
The argument for this is the same as $(\rhzrr, \reme(z_1, t, s) \reme(\ppole(z_1), t, s))$ in the $(+-)$ case.
\bigskip
\\
($\rhzrr$, $\pcoree(z_1)^{\floor{\frac{t-s}{J}}}$):\textbf{Show that}  $|\pcoree(z_1)|^{\floor{\frac{t-s}{J}}} \leq C(\beta, T) e^{-C (t-s+1) \theta^2}$. \\
This is the content of Lemma \ref{lem:pp:std}. 
\bigskip
\\
($\rhzrr, \jprod(z_1)$): \textbf{Show that} $|\jprod(z_1)| \leq C e^{-\frac{\beta}{2 \sqrt{t -s +1}} (|x_2  -y_1| + |x_1 - y_2|}$.\\
By the discussion in ($\rhzrr, \jprod(z_1)$), It is sufficient to show that $|z_1^{x_2 - y_1} \ppole(z_1)^{x_1 - y_2}| \leq e^{-\frac{\beta}{2 \sqrt{t -s +1}} (|x_1- y_2| + |x_2 - y_1|)}$. We write 
\begin{equation*}
|z_1^{x_2 - y_1} \ppole(z_1)^{x_1 - y_2}| = |z_1 \ppole(z_1)|^{x_1- y_2} |z_1|^{x_2 - x_1 + y_2  -y_1}
\end{equation*}
Since $z_1 \in \mcont''(t-s, \ep, -k_1 \beta)$, $|z_1 \ppole(z_1) | = \rfunc(t-s, -k_1 \beta) \leq e^{-\frac{\beta}{\sqrt{t-s+1}}}$. In addition, referring to Lemma \ref{lem:pp:radprop}, one has $|z_1| \leq e^{-\frac{\beta}{\sqrt{t -s +1}}}$. Consequently, 
\begin{equation*}
|z_1^{x_2 - y_1} \ppole(z_1)^{x_1 - y_2}| \leq e^{-\frac{\beta}{\sqrt{t-s+1}} (x_1 - y_2)} e^{-\frac{\beta}{\sqrt{t-s+1}} (x_2 - x_1 + y_2 - y_1)} = e^{-\frac{\beta}{\sqrt{t-s+1}} (x_2  -y_1)} \leq e^{-\frac{\beta}{2 \sqrt{t-s+1}} (|x_2 - y_1| + |x_1 - y_2|)}.
\end{equation*}
Thereby, using the same manner as $(+-)$ case, 
\begin{align*}
|\rhzrr| &\leq C(\beta, T) e^{-\frac{\beta}{2 \sqrt{t-s+1}} (|x_2 - y_1| + |x_1 - y_2|)}\int_{-\pi}^{\pi} 1_{\{|\ppole(z_1(\theta))| \geq r'_2\}} e^{- C (t-s+1) \theta^2} d\theta,\\ 
&\leq C(\beta, T) e^{-\frac{\beta}{2\sqrt{t-s+1}} (|x_2 - y_1| + |x_1 - y_2|)} \int_{|\theta| > (t-s+1)^{-\frac{1}{4}}} e^{-\frac{1}{C}(t-s+1) \theta^2} d\theta  \leq \frac{C(\beta, T)}{t -s +1} e^{-\frac{\beta}{2\sqrt{t-s+1}} (|x_2 - y_1| + |x_1 - y_2|)}.
\end{align*}  
We conclude Theorem \ref{prop:semiestimateint} (a). 
\bigskip
\\
To estimate the gradient, the procedure is similar to in $(+-)$ case, note that applying $\na_{x_i}$ or $\na_{y_i}$ will give an additional factor $z_i^\pm - 1$, while applying $\na_{x_1, x_2}$ will produce an additional factor $(z_1 - 1)(z_2 - 1)$.  By  $|z_i (\theta_i) - 1| \leq C(\frac{1}{\sqrt{t-s+1}} + |\theta_i|)$, we conclude Theorem \ref{prop:semiestimateint} (b), (c).

\section{Proof of Proposition \ref{prop:timedecorr} via self-averaging}\label{sec:timedecorr}
 In this section, we apply the two Markov dualities in Corollary \ref{cor:biduality} and the estimate of $\rhzrte$ in Theorem \ref{prop:semiestimate} to conclude Proposition \ref{prop:timedecorr}. The first step is to expand the term $\Theta_1(t, x)$ and $\Theta_2(t, x)$.
\subsection{Expanding $\Theta_1 (t, x)$ and $\Theta_2 (t, x)$}
 We use $\bb (t, x_1, \dots, x_n)$ to denote a generic uniformly bounded (random) process, which may differ from line to line. Define $$u_\epsilon (t, i) := \sum_{j=i}^{\infty} \toneparte (t+1, t, j-\mue(t)).$$
Referring to \eqref{eq:Thetaone} for the expression of $\Theta_1 (t, x)$
\begin{align*}
\epsilon^{-\frac{1}{2}} \Theta_1 (t, x) &= \epsilon^{-\frac{1}{2}} \qe \lambdae(t) Z(t, x) - \sum_{i=1}^{\infty} \epsilon^{-\frac{1}{2}} \tonepart_\epsilon (t+1, t, i- \mue(t))  Z(t, x-i),
\\
&= \epsilon^{-\frac{1}{2}} (q_\epsilon \lambdae(t) - 1) Z(t, x) + \sum_{i=1}^{\infty} \epsilon^{-\frac{1}{2}} \tonepart_\epsilon (t+1, t, i- \mue(t)) \big(Z(t, x) - Z(t, x-i)\big), \\
&= \epsilon^{-\frac{1}{2}} (\qe \lambdae(t) -1) Z(t, x) + \sum_{i=1}^{\infty} u_\epsilon (t, i) \big(\epsilon^{-\frac{1}{2}}\nabla Z(t, x-i)\big).
\end{align*}
Here, we used the relation $Z(t, x) - Z(t, x-i) = \sum_{j = 1}^{i} \na Z(t, x-j)$ and then changed the order of summation in the last equality.
\bigskip
\\
Likewise, by the expression \eqref{eq:Thetatwo} of $\Theta_2(t, x)$
\begin{align*}
\epsilon^{-\frac{1}{2}} \Theta_2 (t, x) 
= \epsilon^{-\frac{1}{2}} (1-\lambdae(t)) Z(t, x) - \sum_{i=1}^\infty u_\epsilon (t, i) (\epsilon^{-\frac{1}{2}} \nabla Z(t, x-i)).
\end{align*}
Using Lemma \ref{lem:wsc}, one has  $\epsilon^{-\frac{1}{2}} (q_\epsilon \lambda_\epsilon (t)  -1) = 1 - \frac{\rho}{I} + \OO(\epsilon^{\frac{1}{2}})$ and $\epsilon^{-\frac{1}{2}} (1 -\lambda_\epsilon (t)) = \frac{\rho}{I} + \OO(\epsilon^{\frac{1}{2}})$. Consequently, 
\begin{align}\label{eq:seven:Theta1}
\epsilon^{-\frac{1}{2}} \Theta_1 (t, x) &= \bigg(1 - \frac{\rho}{I}\bigg) Z(t, x) + \sum_{i=1}^\infty \ue (t, i) (\epsilon^{-\frac{1}{2}} \nabla Z(t, x-i))  + \epsilon^{\frac{1}{2}} \bb(t, x) Z(t, x),\\
\label{eq:seven:Theta2}
\epsilon^{-\frac{1}{2}} \Theta_2 (t, x) &= \frac{\rho}{I} Z(t, x) - \sum_{i=1}^\infty u_\epsilon (t, i) (\epsilon^{-\frac{1}{2}} \nabla Z(t, x-i))  + \epsilon^{\frac{1}{2}} \bb(t, x) Z(t, x).
\end{align}
For $x_1 \leq x_2 \in \Xi(t)$ and $x \in \Xi(t)$, we denote by
\begin{align*}
\onegradz{t, x_1, x_2} &:= \epsilon^{-\frac{1}{2}} \nabla Z(t, x_1) Z(t, x_2),\\ 
\twogradz{t, x_1, x_2} &:= \epsilon^{-1} \nabla Z(t, x_1) \nabla Z(t, x_2), \\
\numberthis \label{eq:onegradydefin}
\onegrad{t, x} &:=  \sum_{i \in \NN}  \ue(t, i) \onegradz{t, x-i, x},\\  
\numberthis \label{eq:twogradydefin}
\twograd{t, x} &:= \sum_{i > j \in \NN} \ue(t, i) \ue(t, j)  \twogradz{t, x-i, x-j},\\
\numberthis \label{eq:seven:temp8}
\gradsquare{t, x} &:= \sum_{i=1}^{\infty} \ue(t, i)^2 \Big(\twogradz{t, x-i, x-i} - \frac{\rho(I-\rho)}{I} Z(t, x-i)^2\Big).
\end{align*}
\begin{lemma}\label{lem:seven:Thetadecomp}
Recall from \eqref{eq:tau} that
$$\cat = \frac{\rho(I-\rho)}{I^2} \cdot \frac{b (I +2 \mod(t) +1) - (I+2\mod(t) - 1)}{b(I + 2\mod(t)) - (I + 2\mod(t) - 2)},$$ 
we have 
\begin{align*}
&\epsilon^{-1} \Theta_1 (t, x) \Theta_2 (t, x) - \cat Z(t, x)^2\\ &=
\bigg(\frac{2\rho}{I} -1\bigg) \onegrad{t, x} + 2 \twograd{t, x } +  
\gradsquare{t, x} + \eph \bb(t, x) Z(t, x)^2.
\end{align*}
\end{lemma}
\begin{proof}
We name the three terms on the RHS of \eqref{eq:seven:Theta1} (from left to right)  as $A_{1, Z}, A_{1, \nabla}, A_{1, \text{err}}$ respectively and those on the RHS of \eqref{eq:seven:Theta2} as $A_{2, Z}$, $A_{2, \nabla}$, $A_{2, \text{err}}$. Multiplying \eqref{eq:seven:Theta1} by \eqref{eq:seven:Theta2} gives
\begin{equation*}
\epsilon^{-1} \Theta_1 (t, x) \Theta_2 (t, x) = \big(A_{1, Z} + A_{1, \nabla} + A_{1, \text{err}}\big) \cdot \big(A_{2, Z} + A_{2, \nabla} + A_{2, \text{err}}\big).
\end{equation*}
Expanding this product, it is straightforward that
\begin{align*}  
&A_{1, Z} A_{2, Z} = \frac{\rho}{I} (1 - \frac{\rho}{I}) Z(t, x)^2, \qquad A_{1, \nabla} A_{2, Z} + A_{2, \nabla} A_{1, Z} = \bigg(\frac{2\rho}{I} - 1\bigg) \onegrad{t, x},\\
&A_{1, \nabla} A_{2, \nabla}  = -\twograd{t, x} -  \sum_{k=1}^{\infty} u_\epsilon (t, k)^2 \twogradz{t, x-k, x-k}.
\end{align*}
The sum of the rest of terms equals
\begin{align*}
&A_{1, Z} A_{2, \text{err}} + A_{1, \nabla} A_{2, \text{err}} + A_{1, \text{err}} A_{2, Z} + A_{1,\text{err}}  A_{2, \nabla} + A_{1, \text{err}} A_{2, \text{err}},\\ 
&= \epsilon^{\frac{1}{2}} \bb(t, x) Z(t, x) (\epsilon^{-\frac{1}{2}} \Theta_1 (t, x) + \epsilon^{-\frac{1}{2}} \Theta_2 (t, x)) - \epsilon \bb (t, x) Z (t, x)^2 = \ep^{\frac{1}{2}} \bb(t, x) Z(t, x)^2.
\end{align*}
Therefore, we find that 
\begin{align*}
\ep^{-1} \Theta_1 (t, x) \Theta_2 (t, x) = \frac{\rho}{I}(1- \frac{\rho}{I}) Z(t, x)^2 + \onegrad{t, x} - \twograd{t, x}  - \sum_{k=1}^\infty \ue(t, k)^2 \twogradz{t, x-k, x-k} + \eph \bb(t, x) Z(t, x)^2.
\end{align*} 
Thus,
\begin{align*}
\ep^{-1} \Theta_1 (t, x) \Theta_2 (t, x) - \frac{\rho}{I}(1- \frac{\rho}{I}) Z(t, x)^2=  \onegrad{t, x} - \twograd{t, x}  - \sum_{k=1}^\infty \ue(t, k)^2 \twogradz{t, x-k, x-k} + \eph \bb(t, x) Z(t, x)^2.
\end{align*}
Adding $\frac{\rho(I-\rho)}{I} \sum_{k=1}^{\infty} \ue(t, k)^2 Z(t, x-k)^2$ to both sides yields
\begin{align*}
&\ep^{-1} \Theta_1 (t, x) \Theta_2 (t, x) - \frac{\rho}{I}(1- \frac{\rho}{I}) Z(t, x)^2 + \frac{\rho(I-\rho)}{I} \sum_{k=1}^{\infty} \ue(t, k)^2 Z(t, x-k)^2  \\
&= \onegrad{t, x} - \twograd{t, x}  - \sum_{k=1}^\infty \ue(t, k)^2 \bigg(\twogradz{t, x-k, x-k} - \frac{\rho(I - \rho)}{I} Z(t, x-k)^2\bigg) + \eph \bb(t, x) Z(t, x)^2\\
\numberthis \label{eq:seven:temp3}
&= \onegrad{t, x} - \twograd{t, x}  - \gradsquare{t, x} + \eph \bb(t, x) Z(t, x)^2.
\end{align*} 
We claim that
\begin{equation}\label{eq:seven:temp1}
\sum_{k=1}^{\infty} \ue(t, k)^2 Z(t, x-k)^2 = \frac{1-b}{I(b(I + 2 \mod(t)) - (I + 2 \mod(t) - 2))} Z(t, x)^2 + \ep^{\frac{1}{2}} \bb(t, x) Z(t, x)^2.
\end{equation}
If \eqref{eq:seven:temp1} holds, note that
\begin{align*}
\cat = \frac{\rho}{I} (1-  \frac{\rho}{I}) - \frac{\rho(I-\rho)}{I} \frac{1-b}{I(b(I + 2\mod(t)) - (I + 2\mod(t) - 2))}.
\end{align*}
Replacing the term $\sum_{k=1}^{\infty} \ue(t, k)^2 Z(t, x-k)^2$ in the LHS  of \eqref{eq:seven:temp3} by the RHS of \eqref{eq:seven:temp1}, we prove Lemma \ref{lem:seven:Thetadecomp}. 
\bigskip
\\
To justify \eqref{eq:seven:temp1},  we write
\begin{equation}\label{eq:seven:temp2}
\sum_{k=1}^\infty \ue(t, k)^2 Z(t, x-k)^2 = \sum_{k=1}^\infty \ue(t, k)^2 \big(Z(t, x-k)^2 - Z(t, x)^2\big) + \sum_{k=1}^\infty \ue(t, k)^2 Z(t, x)^2.
\end{equation}
Let us analyze the first and second term on the RHS of \eqref{eq:seven:temp2} respectively. 
For the second term, we compute 
\begin{equation}\label{eq:seven1:temp1}
\ue(t, k) = \sum_{j=k}^\infty \toneparte(t+1, t, j) 
= \frac{\alpha(t)(1-q)}{1 + \alpha(t)} \bigg(\frac{\nu + \alpha(t)}{1 + \alpha(t)}\bigg)^{k-1}.
\end{equation}
Here, we used $\ppole(t+1, t, j) = \PP(R(t) = j)$, the expression of which is given in \eqref{eq:rwdistribution}.
Using the preceding equation, we find that 
\begin{align*}
\sum_{k=1}^\infty \ue(t, k)^2 &=
\frac{\big(1 - \frac{1 + q\alpha(t)}{1 + \alpha(t)}\big)^2}{1- \big(\frac{\nu + \alpha(t)}{1 + \alpha(t)}\big)^2}.
\end{align*}
Due to Lemma \ref{lem:wsc},
\begin{equation*}
\sum_{k=1}^{\infty} \ue(t, k)^2 = \frac{1-b}{I\big((I + 2\mod(t))b - (I + 2\mod(t) - 2)\big)} + \OO(\ep^{\frac{1}{2}}).
\end{equation*} 
Thereby, for the second term on the RHS of \eqref{eq:seven:temp2}, 
\begin{equation}\label{eq:seven:temp5}
\sum_{k=1}^{\infty} \ue(t, k)^2  Z(t, x)^2 = \frac{1-b}{I(b(I + 2\mod(t)) - (I + 2\mod(t) - 2))} Z(t, x)^2 + \eph \bb(t, x) Z(t, x)^2 
\end{equation}
For the first term on the RHS of \eqref{eq:seven:temp2}, noticing $Z(t, x-k) = e^{-\sqrt{\ep}\sum_{i=1}^k (\etat_{x-i+1} (t) - \rho)} Z(t, x)$ (recall $\widetilde{\eta}_x (t) = \eta_x (x + \hmu(t))$), hence
\begin{align*}
Z(t, x - k)^2  - Z(t, x)^2 &= Z(t, x)^2 \Big(e^{-2 \sqrt{\ep} \big((\etat_x (t) - \rho) + \dots + (\etat_{x- k +1}(t)  - \rho)\big)} - 1\Big) 
\end{align*}
Since $|\etat_x (t) - \rho| \leq I$, 
$$\bigg|\sum_{i=1}^k (\etat_{x-i+1} (t) - \rho) \bigg| \leq kI.$$
Note that for any $K > 0$, there exists a constant $C$ such that $$|e^{x} - 1| \leq C|x|, \text{ for } |x| \leq K.$$
Thus, if $k \leq \ep^{-\frac{1}{2}}$, one has
\begin{align*}
\big|e^{-2 \sqrt{\ep} \sum_{i=1}^k (\etat_{x-i+1} (t) - \rho)} - 1\big| \leq C \sqrt{\ep} k I.
\end{align*} 
If $k > \ep^{-\frac{1}{2}}$, one simply has
$$
\big|e^{-2 \sqrt{\ep} \sum_{i=1}^k (\etat_{x-i+1} (t) - \rho)} - 1\big| \leq  e^{2 k I \sqrt{\ep}}.
$$
Therefore,
\begin{equation}\label{eq:71temp1}
\big|e^{-2 \sqrt{\ep} \sum_{i=1}^k (\etat_{x-i+1} (t) - \rho)} - 1\big| \leq C \big(\sqrt{\ep} k I \idc_{\{k \leq \ep^{-\frac{1}{2}}\}} + e^{2 k I \sqrt{\ep}} \idc_{\{k > \ep^{-\frac{1}{2}}\}}\big).
\end{equation}
Referring to \eqref{eq:seven1:temp1} for the expression of $\ue(t, k)$, using \eqref{eq:temp16}
we see that  there exists $0< \delta < 1$ s.t. for $\ep$ small enough and for all $t, k$
\begin{equation}\label{eq:71expdecay}
\ue(t, k) \leq \delta^{k-1}.
\end{equation} 
Combining this with \eqref{eq:71temp1} gives 
\begin{align*}
\sum_{k=1}^{\infty} \ue(t, k)^2 \big(Z(t, x-k)^2 - Z(t,x)^2\big) &= Z(t, x)^2 \bigg(\sum_{k=1}^\infty \ue(t, k)^2 \big(e^{-2\sqrt{\ep} \sum_{i=1}^k (\etat_{x-i+1} (t) - \rho)} - 1\big)\bigg),\\
&\leq C Z(t, x)^2 \bigg(\sum_{k=1}^{\floor{\ep^{-\frac{1}{2}}}} \sqrt{\ep} k \delta^{k} + \sum_{k = \floor{\ep^{-\frac{1}{2}}} + 1}^\infty e^{2kI \sqrt{\ep}} \delta^k\bigg),
\\ 
&= \ep^{\frac{1}{2}} \bb(t, x) Z(t, x)^2.
\end{align*}
Combining this with \eqref{eq:seven:temp5}, we prove the desired claim \eqref{eq:seven:temp1}.
\end{proof}
 By Lemma \ref{lem:seven:Thetadecomp}, we reduce the proof of Proposition \ref{prop:timedecorr} to the following lemmas. 
\begin{lemma}\label{lem:timedecorr1}
For any given $T >0$, there exists positive constants $C$ and $u$ such that for all $t \in [0, \epsilon^{-2} T] \cap \ZZ$, $\xstar \in \ZZ$
\begin{align}\label{eq:onegrady}
\bignorm{\epsilon^{2} \sum_{s=0}^{t} \onegrad{s, \xstar (s)}}_2 &\leq C \ep^{\frac{1}{4}} e^{2u \ep |\xstar|}, \\
\label{eq:twogrady}
\bignorm{\epsilon^{2} \sum_{s=0}^{t} \twograd{s, \xstar(s)}}_2 &\leq C \epsilon^{\frac{1}{4}} e^{2u\ep |\xstar|},
\end{align} 
where we used the shorthand notation $\xstar(s) : = \xstar -\hmu(s) + \floor{\hmu(s)}$.
\end{lemma}
\begin{lemma}\label{lem:timedecorr2}
Fix $T > 0$, there exists positive constants $C$ and $u$ such that for all $t \in [0, \epsilon^{-2} T] \cap \ZZ$ and $\xstar \in \ZZ$, 
\begin{equation*}
\bignorm{\epsilon^{2} \sum_{s=0}^{t} \gradsquare{s, \xstar(s)}}_2 \leq C \epsilon^{\frac{1}{4}} e^{2u\ep |\xstar|}
\end{equation*}
\end{lemma} 
We will prove Lemma \ref{lem:timedecorr1}  and Lemma \ref{lem:timedecorr2} in the next two sections. Let us first conclude Proposition  \ref{prop:timedecorr} based on them.
\begin{proof}[Proof of Proposition \ref{prop:timedecorr}]
Referring to Lemma \ref{lem:seven:Thetadecomp}, we have 
\begin{align*}
\epsilon^2 \sum_{s = 0}^{t} \bigg(\epsilon^{-1} \Theta_1 \Theta_2 - \cas  Z^2 \bigg)(s, \xstar(s)) &= \epsilon^{2} \sum_{s=0}^{t} \onegrad{s, \xstar (s)} + \epsilon^{2} \sum_{s=0}^{t} \twograd{s, \xstar(s)} + \ep^2 \sum_{s = 0}^t \gradsquare{s, \xstar(s)}\\
&\quad + \ep^2 \sum_{s = 0}^t \ep^{\frac{1}{2}} \bb(s, x) Z(s, \xstar(s))^2.
\end{align*}
By Lemma \ref{lem:timedecorr1} and Lemma \ref{lem:timedecorr2}, together with the bound $\norm{Z(s, \xstar(s))}_2 \leq C e^{u\ep |\xstar|}$ (which follows from Proposition \ref{prop:tightness}), one has
\begin{align*}
&\bignorm{\epsilon^2 \sum_{s = 0}^{t} \Big(\epsilon^{-1} \Theta_1 \Theta_2 - \cas  Z^2 \Big)(s, \xstar(s))}_2 \\
&\quad \leq \bignorm{\epsilon^{2} \sum_{s=0}^{t} \onegrad{s, \xstar(s)}}_2  + \bignorm{\epsilon^{2} \sum_{s=0}^{t} \twograd{s, \xstar(s)}}_2 
+ \bignorm{\epsilon^{2} \sum_{s=0}^{t} \gradsquare{s, \xstar(s)}}_2 + \ep^2 \sum_{s=0}^t \ep^{\frac{1}{2}} \bb(s, x) \norm{Z(s, \xstar(s))}_2^2,
\\ 
&\quad\leq C \big(\ep^{\frac{1}{4}} e^{2u\ep |\xstar|} + \ep^{\frac{5}{2}} t e^{2u \ep |\xstar|}\big).
\end{align*}
Using $t \leq \ep^{-2} T$, we obtain
\begin{equation*}
\bignorm{\epsilon^2 \sum_{s = 0}^{t} \Big(\epsilon^{-1} \Theta_1 \Theta_2 - \cas  Z^2 \Big)(s, \xstar(s))}_2 \leq C \ep^{\frac{1}{4}} e^{2u\ep |\xstar|}
\end{equation*}
This completes the proof of Proposition \ref{prop:timedecorr}.
\end{proof}
\subsection{Proof of Lemma \ref{lem:timedecorr1}}
Recall the notation $\etat_x (t) = \eta_{x + \hmu(t)} (t) 
$,  we see that by Taylor expansion
$$\nabla Z(t, x) =  Z(t, x) \big(e^{-\sqrt{\epsilon} (\etat_{x+1} (t) - \rho)} - 1\big) = \sqrt{\epsilon} Z(t, x)  (\rho - \etat_{x+1} (t)) + \epsilon \bb (t, x) Z(t, x).$$
Hence,
\begin{align*}
\numberthis \label{eq:seven1:temp2}
\epsilon^{-\frac{1}{2}} \nabla Z(t, x) &= (\rho - \etat_{x+1} (t)) Z(t, x)  + \epsilon^{\frac{1}{2}} \bb(t, x) Z(t, x),\\
\numberthis \label{eq:seven1:temp5}
Z(t, x+1) &= Z(t, x) + \na Z(t, x) =  Z(t, x) + \eph \bb(t, x) Z(t, x).
\end{align*}
We will use these elementary relations frequently in the sequel. 
\bigskip\\
The following lemma is crucial for the proof of Lemma \ref{lem:timedecorr1}.
\begin{lemma}\label{lem:conditionexa}
Given $T > 0$ and $n \in \NN$, there exists constant $C$ and  $u$ such that for all $s \leq t \in [0, \epsilon^{-2} T] \cap \ZZ$  such that for $x_1 \leq x_2 \in \Xi(t)$,
\begin{align}\label{eq:conditionex1}
\bignorm{\EE\big[\onegradz{t, x_1, x_2}\vv \FF(s)\big]}_n \leq \frac{C \epsilon^{-\frac{1}{2}}}{\sqrt{t - s + 1}} e^{u \epsilon (|x_1| + |x_2|)}.
\end{align}
For $x_1 < x_2 \in \Xi(t)$,
\begin{align}\label{eq:conditionex2}
\bignorm{\EE\big[\twogradz{t, x_1, x_2}  \vv \FF(s)\big]}_n \leq \frac{C \ep^{-1}}{t - s + 1} e^{u \epsilon (|x_1| + |x_2|)}. 
\end{align}
\end{lemma}

\begin{proof}
Let us first justify \eqref{eq:conditionex1}. Recall the two point duality \eqref{eq:fdualt},
\begin{equation*}
\EE\big[Z(t, x_1) Z(t, x_2) \vv \FF(s)\big] = \sum_{y_1 \leq y_2 \in \Xi(s)^2} \rhzrt\big((x_1, x_2), (y_1, y_2), t, s\big) Z(s, y_1) Z(s, y_2).
\end{equation*}
As $\onegradz{t, x_1, x_2} = \epsilon^{-\frac{1}{2}} \nabla Z(t, x_1) Z(t, x_2)$, it is straightforward that by this duality, if $x_1 < x_2 $,
\begin{align*} 
\EE\big[\onegradz{t, x_1, x_2} \vv \FF(s) \big] 
\numberthis \label{eq:condexpan1}
&= \epsilon^{-\frac{1}{2}} \sum_{y_1 \leq y_2 \in \Xi(s)} \na_{x_1}  \rhzrte\big((x_1, x_2), (y_1, y_2), t, s\big) Z(s, y_1) Z(s, y_2).
\end{align*}
If $x_1 = x_2$, 
\begin{align*}
\EE\big[\onegradz{t, x_1, x_2} \vv \FF(s) \big] 
&= \epsilon^{-\frac{1}{2}} \sum_{y_1 \leq y_2 \in \Xi(s)} \na_{x_2}  \rhzrte \big((x_1, x_1), (y_1, y_2), t, s\big) Z(s, y_1) Z(s, y_2).
\end{align*}
We assume $x_1 < x_2$ without loss of generosity, the proof of \eqref{eq:conditionex1} for $x_1 = x_2$ will be similar (one only needs to replicate the estimate of $\na_{x_1} \rhzrte$ to $\na_{x_2} \rhzrte$).
By the estimate of $\na_{x_1} \rhzrte$ provided in Theorem \ref{prop:semiestimate} (b), we see that $$\big|\na_{x_1}  \rhzrte\big((x_1, x_2), (y_1, y_2), t, s\big)\big| \leq \frac{C(\beta, T)}{(t - s + 1)^{\frac{3}{2}}} e^{- \frac{\beta (|x_1 - y_1| + |x_2 - y_2 |) }{\sqrt{t-s +1} + C(\beta)}}. $$ 
This, together with the moment bound of $Z(t, x)$ in \eqref{eq:tightness1} 
yields 
\begin{equation*}
\bignorm{\sum_{y_1 \leq y_2} \na_{x_1} \rhzrte \big((x_1, x_2), (y_1, y_2), t, s\big) Z(s, y_1) Z(s, y_2)}_n \leq \sum_{y_1 \leq y_2 \in \Xi(s)}  \frac{C(\beta, T)}{(t-s+1)^{\frac{3}{2}}} e^{-\frac{\beta (|x_1 - y_1 | + |x_2 - y_2|)}{\sqrt{t-s+1} + C(\beta)}} e^{u \epsilon |y_1| } e^{u \epsilon |y_2 |}
\end{equation*}
Due to Lemma \ref{lem:five:usefullem}, we see that we can choose $\beta$ large enough so that 
\begin{align*}
\sum_{y_1, y_2 \in \Xi(s)} e^{- \frac{\beta (|x_1 -y_1| + |x_2 - y_2|)}{\sqrt{t-s+1} + C(\beta)}} e^{u\epsilon (|y_1| + |y_2|)}
&\leq \bigg(\sum_{y_1 \in \Xi(s)} e^{- \frac{\beta |x_1 -y_1|}{\sqrt{t-s+1} + C(\beta)}} e^{u\ep(|y_1|)}\bigg) \bigg(\sum_{y_2 \in \Xi(s)} e^{- \frac{\beta |x_2 -y_2|}{\sqrt{t-s+1} + C(\beta)}} e^{u\ep(|y_2|)}\bigg),\\
&\leq C (t-s +1) e^{u \epsilon (|x_1| + |x_2|)}.
\end{align*}
Thus, 
\begin{equation*}
\bignorm{\sum_{y_1 \leq y_2} \na_{x_1} \rhzrte \big((x_1, x_2), (y_1, y_2), t, s\big) Z(s, y_1) Z(s, y_2)}_n \leq \frac{C(\beta, T)}{\sqrt{t-s+1}} e^{u\ep (|x_1| + |x_2|)}.
\end{equation*}
Referring to  \eqref{eq:condexpan1}, 
we conclude \eqref{eq:conditionex1}.
\bigskip
\\
We turn our attention to prove \eqref{eq:conditionex2}. 
With the aid of \eqref{eq:fdualt}, one has for $x_1 < x_2 \in \Xi(t)$,
\begin{align*}
\EE\big[\twogradz{t, x_1, x_2} \vv \FF(s)\big] &= \epsilon^{-1} \EE\big[\na Z(t, x_1) \na Z(t, x_2) \vv \FF(s)\big],\\
\numberthis \label{eq:condexpan2}
&= \epsilon^{-1} \sum_{y_1 \leq y_2 \in \Xi(s)}
\na_{x_1, x_2} \rhzrte \big((x_1, x_2), (y_1, y_2), t, s\big) Z(s, y_1) Z(s, y_2).
\end{align*}
Note that \eqref{eq:condexpan2} does not hold when $x_1 = x_2$ (see Remark \ref{rmk:temp1} below). Theorem \ref{prop:semiestimate} $(c)$ implies $$\big|\na_{x_1, x_2} \rhzrte \big((x_1, x_2), (y_1, y_2), t, s\big)\big| \leq \frac{C(\beta, T)}{(t-s+1)^2} e^{-\frac{\beta(|x_1 - y_1| + |x_2 - y_2|)}{\sqrt{t-s+1} + C(\beta)}}.$$ 
By same argument used in proving \eqref{eq:conditionex1}, one has
\begin{align*}
\bignorm{\EE\big[\twogradz{t, x_1, x_2}  \vv \FF(s)\big]}_n \leq \frac{C \ep^{-1}}{t - s + 1} e^{u \epsilon (|x_1| + |x_2|)}. 
\end{align*}
This concludes the proof of the lemma.
\end{proof}
With the help of the preceding lemma, we proceed to prove Lemma \ref{lem:timedecorr1}.
\begin{proof}[Proof of Lemma \ref{lem:timedecorr1}]
Referring to \eqref{eq:onegradydefin}, \eqref{eq:twogradydefin} that 
\begin{align*}
\sum_{s=0}^t  \onegrad{s, \xstar (s)} &=  \bigg(\frac{2\rho}{I} -1\bigg) \sum_{i\in \NN}  \sum_{s=0}^t \ue(s, i) \onegradz{s, \xstar(s) - i, \xstar(s)},\\
\sum_{s = 0}^t \twograd{s, \xstar (s)} &= \sum_{i > j \in \NN}  \sum_{s=0}^t \ue(s, i) \ue(s, j)  \twogradz{s, x-i, x-j}. 
\end{align*}  
By triangle inequality, one has 
\begin{align*}
\bignorm{\ep^2 \sum_{s = 0}^t \onegrad{s, \xstar(s)}}_2 &\leq \bigg(\frac{2\rho}{I} -1\bigg) \sum_{i \in \NN} \bignorm{\ep^2 \sum_{s = 0}^{t} \ue(s, i) \onegradz{s, \xstar(s)-i, \xstar(s)}}_2\\
\bignorm{\ep^2 \sum_{s = 0}^t \twograd{s, \xstar(s)}}_2 &\leq \sum_{i > j \in \NN} \bignorm{\ep^2 \sum_{s = 0}^{t} \ue(s, i) \ue(s, j) \twogradz{s, \xstar(s)-i, \xstar(s) - j}}_2.
\end{align*}
To prove Lemma \ref{lem:timedecorr1}, it is sufficient to show that there exists constant $C, u$ such that for all $t \in [0, \epsilon^{-2} T] \cap \ZZ$, $\xstar \in \ZZ$ and some constant $0 < \delta < 1$,
\begin{align}\label{eq:72sts1}
&\bignorm{\ep^2 \sum_{s = 0}^{t} \ue(s, i) \onegradz{s, \xstar(s)-i, \xstar(s)}}_2 \leq C \epsilon^{\frac{1}{4}} e^{u \epsilon (2 |\xstar| + i) } \delta^{i}, \qquad \forall i \in \ZZn,\\
\label{eq:72sts2}
&\bignorm{\ep^2 \sum_{s = 0}^{t} \ue(s, i) \ue(s, j) \twogradz{s, \xstar(s)-i, \xstar(s) - j}}_2 \leq C \epsilon^{\frac{1}{4}} e^{u \epsilon (2 |\xstar| + i + j)} \delta^{i+j}, \qquad \forall i > j \in \NN.
\end{align}
Note that here, we include $i =0$ in \eqref{eq:72sts1}, which is not needed to prove Lemma \ref{lem:timedecorr1}. We are going to use this in the proof of Lemma \ref{lem:timedecorr2}.
\bigskip
\\
We begin with proving \eqref{eq:72sts1}, by writing
\begin{align*}
&\bignorm{\sum_{s = 0}^{t} \ue(s, i) \onegradz{s, \xstar(s) - i, \xstar(s)}}_2^2\\ 
&= 2 \sum_{0 \leq s_1 < s_2 \leq t}  \EE \bigg[\ue(s_1, i) \ue(s_2, i) \onegradz{s_1, \xstar(s_1) - i, \xstar(s_1)} 
 \onegradz{s_2, \xstar(s_2) - i, \xstar(s_2)}\bigg]\\ 
&\quad +\sum_{s=0}^t \EE\big[\ue(s, i)^2 \onegradz{s, \xstar(s)-i, \xstar(s)}^2\big]
\\
&= 2 \sum_{0 \leq s_1 < s_2 \leq t} \ue(s_1, i) \ue(s_2, i)  \EE \bigg[\onegradz{s_1, \xstar(s_1) - i, \xstar(s_1)} \EE\big[\onegradz{s_2, \xstar(s_2) - i, \xstar(s_2)} \vv \FF(s_1)\big]\bigg]   \\
&\quad+\sum_{s=0}^t \ue(s, i)^2 \EE\big[\onegradz{s, \xstar(s)-i, \xstar(s)}^2\big]
\end{align*}
Using \eqref{eq:71expdecay} to bound $\ue(s, i)$, one has 
\begin{align*}
\bignorm{\sum_{s = 0}^{t} \ue(s, i) \onegradz{s, \xstar(s) - i, \xstar(s)}}_2^2 &\leq C \delta^{2i} \sum_{0 \leq s_1 < s_2 \leq t} \bigg| \EE \bigg[\onegradz{s_1, \xstar(s_1) - i, \xstar(s_1)} \EE\big[\onegradz{s_2, \xstar(s_2) - i, \xstar(s_2)} \vv \FF(s_1)\big]\bigg] \bigg|   \\
\numberthis \label{eq:72expandl2a} 
&\quad + C \delta^{2i} \sum_{s=0}^t  \EE\big[\onegradz{s, \xstar(s)-i, \xstar(s)}^2\big]
\end{align*}
Let us analyze the two terms on the RHS of \eqref{eq:72expandl2a} respectively.  For the first term,
via Cauchy-Schwarz inequality $|\EE\big[XY\big]| \leq \norm{X}_2 \norm{Y}_2$, one has
\begin{align*}
&\bigg|\EE \bigg[\onegradz{s_1, \xstar(s_1) - i, \xstar(s_1)} \EE\big[\onegradz{s_2, \xstar(s_2) - i, \xstar(s_2)} \vv \FF(s_1)\big]\bigg]\bigg| \\
&\leq \norm{\onegradz{s_1, \xstar(s_1) - i, \xstar(s_1)}}_2 \norm{\EE\big[\onegradz{s_2, \xstar(s_2) - i, \xstar(s_2)} \vv \FF(s_1)\big]}_2
\end{align*}
By the moment bound in Proposition \ref{prop:tightness}, we have $\norm{\onegradz{s, x_1, x_2}}_2 \leq C e^{u \epsilon (|x_1| + |x_2|)} $. Combining this with \eqref{eq:conditionex1}, 
\begin{align*}
&\bigg|\EE \bigg[\onegradz{s_1, \xstar(s_1) - i, \xstar(s_1)} \EE\big[\onegradz{s_2, \xstar(s_2) - i, \xstar(s_2)} \vv \FF(s_1)\big]\bigg]\bigg|\\ 
&\leq   C e^{u \epsilon (|\xstar(s_1) - i| + |\xstar(s_1)|)}  \frac{\epsilon^{-\frac{1}{2}}}{\sqrt{s_2 -s_1 +1}} e^{u \epsilon (|\xstar(s_2) - i| + |\xstar(s_2)|) }\\ 
&\leq \frac{C \epsilon^{-\frac{1}{2}}}{\sqrt{s_2 - s_1 +1 }}  e^{2 u \epsilon (|\xstar| + |\xstar - i|) }.
\end{align*}
Consequently, the first term in \eqref{eq:72expandl2a} is upper bounded by 
\begin{align*}
&\bigg|\sum_{0 \leq s_1 < s_2 \leq t} \EE \bigg[\onegradz{s_1, \xstar(s_1) - i, \xstar(s_1)} \EE\big[\onegradz{s_2, \xstar(s_2) - i, \xstar(s_2)} \vv \FF(s_1)\big]\bigg] \bigg| \leq \sum_{0 \leq s_1 < s_2 \leq t} \frac{C \epsilon^{-\frac{1}{2}}}{\sqrt{s_2 - s_1 + 1}} e^{2 u \epsilon (|\xstar| + |\xstar - i|)} \\
\numberthis \label{eq:72firstba}
&\leq C \epsilon^{-\frac{1}{2}} t^{\frac{3}{2}} e^{2 u \epsilon (2 |\xstar| + i)} \leq C \epsilon^{-\frac{7}{2}} e^{2u \epsilon (2 |\xstar| + i )}.
\end{align*}
where in the second inequality above we used the integral approximation $$\sum_{0 \leq s_1 < s_2 \leq t} \frac{1}{\sqrt{s_2 - s_1 +1}} \leq C \int_{0 \leq s_1 \leq s_2 \leq t} \frac{ds_1 ds_2}{\sqrt{s_2 - s_1}} = C t^{\frac{3}{2}}$$
and in the last inequality we used  $t \leq \epsilon^{-2} T$.
\bigskip
\\
Using again $\norm{\onegradz{s, x_1, x_2}}_2 \leq C e^{u \epsilon (|x_1| + |x_2|)} $, the second term in \eqref{eq:72expandl2a} is readily upper bounded by
\begin{align}\label{eq:72secondba}
\bigg|\sum_{s = 0}^t \EE\big[\onegradz{s, \xstar(s) - i, \xstar(s)}^2\big]\bigg| \leq  C \sum_{s = 0}^t e^{2u \epsilon (|\xstar| + |\xstar - i|)} \leq C \epsilon^{-2} e^{2u \epsilon (2 |\xstar| + i) }.
\end{align}
Incorporating the bounds \eqref{eq:72firstba} and \eqref{eq:72secondba} into the RHS of \eqref{eq:72expandl2a}, we get \eqref{eq:72sts1}.
\bigskip
\\
We proceed to justify \eqref{eq:72sts2}, the method is similar to the proof of \eqref{eq:72sts1}. Write
\begin{align*}
&\bignorm{\sum_{s = 0}^{t} \ue(s, i) \ue(s, j) \twogradz{s, \xstar(s) - i, \xstar(s) - j}}_2^2 \\
&= 2 \sum_{0 \leq s_1 < s_2 \leq t} \ue(s_1, i) \ue(s_1, j) \ue(s_2, i) \ue(s_2, j)  \EE \bigg[\twogradz{s_1, \xstar(s_1) - i, \xstar(s_1) - j} \EE\big[\twogradz{s_2, \xstar(s_2) - i, \xstar(s_2) - j} \vv \FF(s_1)\big]\bigg]\\  
&\quad+\sum_{s=0}^t \ue(s, i)^2 \ue(s, j)^2 \EE\big[\twogradz{s, \xstar(s)-i, \xstar(s) - j}^2\big].
\end{align*} 
Using again \eqref{eq:71expdecay}, one has 
\begin{align*}
&\bignorm{\sum_{s = 0}^{t} \ue(s, i) \ue(s, j) \twogradz{s, \xstar(s) - i, \xstar(s) - j}}_2^2\\ &\leq C \delta^{2(i+j)} \sum_{0 \leq s_1 < s_2 \leq t} \bigg|\EE\bigg[\twogradz{s_1, \xstar(s_1) - i, \xstar(s_1) - j} \EE\big[\twogradz{s_2, \xstar(s_2) - i, \xstar(s_2) - j} \vv \FF(s_1)\big]\bigg] \bigg|\\
\numberthis \label{eq:72expandl2b}
&\quad + C \delta^{2(i+j)} \sum_{s=0}^t \EE\big[\twogradz{s, \xstar(s)-i, \xstar(s) - j}^2\big].
\end{align*}
Let us analyze the two terms on the RHS of \eqref{eq:72expandl2b} respectively.
For the first term, by Cauchy Schwarz, 
\begin{align*}
&\bigg| \EE\bigg[\twogradz{s_1, \xstar(s_1) - i, \xstar(s_1) - j} \EE\big[\twogradz{s_2, \xstar(s_2) - i, \xstar(s_2) - j} \vv 
\FF(s_1)\big]\bigg]\bigg|\\
&\leq \norm{\twogradz{s_1, \xstar(s_1) - i, \xstar(s_1) - j}}_2 \norm{\EE\big[\twogradz{s_2, \xstar(s_2) - i, \xstar(s_2) - j} \vv 
	\FF(s_1)\big]}_2
\end{align*}
Using the bound $\norm{\onegradz{s, x_1, x_2}}_2 \leq C e^{u \epsilon (|x_1| + |x_2|)} $ and  \eqref{eq:conditionex2}, we have
\begin{align*}
&\bigg| \EE\bigg[\twogradz{s_1, \xstar(s_1) - i, \xstar(s_1) - j} \EE\big[\twogradz{s_2, \xstar(s_2) - i, \xstar(s_2) - j} \vv 
\FF(s_1)\big]\bigg]\bigg|\\
&\leq  e^{u \epsilon (|\xstar - i| + |\xstar - j|)} \frac{C \epsilon^{-1}}{s_2  -s_1 +1} e^{u \epsilon (|\xstar - i| + |\xstar - j|)} = \frac{C \epsilon^{-1}}{s_2  -s_1 +1} e^{2u \epsilon (|\xstar - i| + |\xstar - j|)}.
\end{align*}
Therefore,
\begin{align*}
&\sum_{0 \leq s_1 < s_2 \leq t} \bigg| \EE \bigg[\twogradz{s_1, \xstar(s_1) - i, \xstar(s_1) - j} \EE\big[\twogradz{s_2, \xstar(s_2) - i, \xstar(s_2) - j} \vv \FF(s_1)\big]\bigg] \bigg| 
\\  \numberthis \label{eq:seven:temp7}
&\leq \sum_{0 \leq s_1 < s_2 \leq t} \frac{C\ep^{-1}}{s_2 - s_1 + 1} e^{2u \epsilon (|\xstar - i| + |\xstar - j|)}\\
&\leq C \epsilon^{-1} (t+1) \log(t + 1) e^{2u \epsilon (|\xstar - i| + |\xstar - j|)} \leq C \epsilon^{-\frac{7}{2}} e^{2u \epsilon (2 |\xstar| + i + j)}.
\end{align*}
In the second inequality above, we used the integral approximation 
\begin{equation*}
\sum_{0 \leq s_1 < s_2 \leq t} \frac{1}{s_2 - s_1 + 1} \leq C \int_{0 \leq s_1 \leq s_2 \leq t} \frac{1}{s_2 - s_1 + 1} ds_1 ds_2 \leq C (t+1) \log(t+1). 
\end{equation*}
For the second term in \eqref{eq:72expandl2b}, it is clear that 
\begin{equation}\label{eq:72secondbb}
\sum_{s = 0}^t \EE\big[\twogradz{s, \xstar(s)- i, \xstar(s) - j}^2\big] \leq C t e^{2u \epsilon (2 |\xstar| + i + j)} \leq C \ep^{-2} e^{2u \epsilon (2 |\xstar| + i + j)}. 
\end{equation}
Incorporating the bounds \eqref{eq:seven:temp7} and \eqref{eq:72secondbb} into the RHS of \eqref{eq:72expandl2b}, we prove the desired \eqref{eq:72sts2}.
\end{proof}
\begin{remark}\label{rmk:temp1}
In the argument above, we showed $\twogradz{t, x_1, x_2} = (\ep^{-\frac{1}{2}}\na Z(t, x_1)) (\ep^{-\frac{1}{2}} \na Z(t, x_2))$ vanishes after averaging over a long time interval when $x_1 \neq x_2$. The readers might wonder whether the same holds for $x_1 = x_2$? The answer is negative. In the case $x_1 \neq x_2$, we used two particle duality \eqref{eq:fdualt} to move the gradient from $Z$ to $\rhzrte$
\begin{equation*}
\EE\big[\twogradz{t, x_1, x_2} \vv \FF(s)\big] = \epsilon^{-1} \sum_{y_1 \leq y_2 \in \Xi(s)} \nabla_{x_1, x_2} \rhzrte \big((x_1, x_2), (y_1, y_2), t, s\big) Z(s, y_1) Z(s, y_2).
\end{equation*}
However, if $x_1 = x_2$,  the same two particle duality gives instead
\begin{align*}
&\EE\big[\twogradz{t, x_1, x_2} \vv \FF(s)\big] \\
&=\ep^{-1} \sum_{y_1 \leq y_2 \in \Xi(s)} \Big(\rhzrte\big((x_1 +1, x_1 +1), (y_1, y_2), t, s\big) - 2\rhzrte\big((x_1, x_1 +1), (y_1, y_2), t, s\big)  + 1\Big) Z(s, y_1) Z(s, y_2).
\end{align*} 
The same argument fails because we do not have an effective estimate of $$\rhzrte\big((x_1 +1, x_1 +1), (y_1, y_2), t, s\big) - 2\rhzrte\big((x_1, x_1 +1), (y_1, y_2), t, s\big)  + 1.$$ In fact, when $x_1 = x_2$, $\twogradz{t, x_1, x_2}$ does not vanish after averaging. One quick way to see this is to use 
\begin{align*}
\twogradz{t, x_1, x_1} &= \big(\epsilon^{-\frac{1}{2}} \nabla Z(t, x_1)\big)^2 \\
&= (\etat_{x_1 +1} (t) - \rho)^2 Z(t, x_1)^2 + \ep^{\frac{1}{2}} \bb Z(t, x_1)^2 \\
&\geq \min\big(1 - \{\rho\}, \{\rho\}\big)^2 Z(t, x_1)^2 + \ep^{\frac{1}{2}} \bb Z(t, x_1)^2
\end{align*}
where $\{\rho\}$ represents the fractional part of $\rho$. This implies that $\twogradz{t, x, x}$ is lower bounded by a constant times $Z(t, x)^2$, which does not vanish after averaging.
\end{remark}
\subsection{Proof of Lemma \ref{lem:timedecorr2}}
\label{sec:timedecorr2}
The aim of this section is to justify Lemma \ref{lem:timedecorr2}, which indicates that $\twogradz{t, x, x} - \frac{\rho(I-\rho)}{I} Z(t, x)^2$ vanishes after averaging over a long time interval. This was proved for the stochastic six vertex model \cite{CGST18} (which corresponds to $I=1, J=1$). Note that when $I = 1$, for all $t, x$ one has $\etat_x (t) \in \{0, 1\}$, which yields $\etat_x (t)^2 = \etat_x (t)$. \cite{CGST18} utilizes this crucial observation to show that 
\begin{align*}
\twogradz{t, x, x} &= (\etat_{x+1} (t) - \rho)^2 Z(t, x)^2 + \ep^{\frac{1}{2}} \bb(t, x) Z(t, x)^2,\\ 
&= \rho^2 Z(t, x)^2 + (1 - 2\rho) \etat_{x+1} (t) Z(t, x)^2,
\\
&= \rho (1- \rho) Z(t, x)^2 + (2 \rho - 1) \onegradz{t, x, x} + \ep^{\frac{1}{2}} \bb(t, x) Z(t, x)^2, 
\end{align*} 
where in the last line above, we used \eqref{eq:seven1:temp2}. 
We have seen in the previous section that $Z_\nabla (t, x, x)$ vanishes after averaging, which implies that  $ \twogradz{t, x, x} - \rho(1-\rho) Z(t, x)^2$ will also vanish. 
\bigskip
\\
When $I \geq 2$, $\etat_{x} (t)$ can takes more than two values so the $\etat_x (t)^2 = \etat_x (t)$ relation no longer holds. Notice that in the proof of Lemma \ref{lem:timedecorr1}, we have only leveraged the first duality \eqref{eq:fdualt} in the Lemma \ref{lem:dualt}. To conclude Lemma \ref{lem:timedecorr2}, we will combine both of the dualities \eqref{eq:fdualt} and \eqref{eq:sdualt}.
\bigskip
\\
Before moving to the proof, we first offer a heuristic argument to explain why the $\lambda = \frac{\rho(I-\rho)}{I}$ is the value which makes  $\twogradz{t, x, x} - \lambda Z(t, x)^2$ vanish after averaging. 
\begin{proof}[Heuristic argument for Lemma \ref{lem:timedecorr2}]
Note that   
$$\twogradz{t, x, x} = (\etat_{x+1} (t) - \rho)^2 Z(t, x)^2 + \ep^{\frac{1}{2}} \bb(t, x) Z(t, x)^2.$$ 
In Theorem \ref{thm:appendix:stat}, we find that the stationary distribution of the (bi-infinite) \ac{SHS6V} model is given by $\stat $, where $\statmargin$ is defined in \eqref{eq:stat}. It is straightforward to verify that $\stat$ is near stationary with density $\rho$ (Definition \ref{def:nearstationary}). Start the \ac{SHS6V} model from $\vec{\eta}(0) \sim \stat$, by stationarity $\eta_x (t) \sim \statmargin$ for all $t \in \ZZ_{\geq 0}$ and $x \in \ZZ$. Heuristically, we can approximate $(\etat_{x+1} (t) - \rho)^2 Z(t, x)^2$ by $\EE_{\pi_\rho}\big[(\etat_{x+1} (t)- \rho)^2\big] Z(t, x)^2$. Note that
\begin{equation*}
\EE_{\pi_\rho}\big[(\etat_{x+1} (t) - \rho)^2\big] Z(t, x)^2 =  \var\big[\pi_\rho\big]  Z(t, x)^2 
\end{equation*}
where $\var\big[\pi_\rho\big]$ represents the variance of the probability distribution $\pi_\rho$.
Referring to Lemma \ref{lem:appendix:stat}, we have
\begin{equation*}
\text{Var}\big[\pi_\rho\big] = \rho - \sum_{i=1}^I \frac{\chi^2}{(q^i - \chi)^2}.
\end{equation*}
where $\chi$ is  the unique negative real number satisfying 
$\sum_{i=1}^{I} \frac{\chi}{\chi  - q^{i}} = \rho.$ It is straightforward that under weak asymmetric scaling \eqref{eq:scaling}, 
one has $\lim_{\ep \downarrow 0} \chi_\ep = \frac{\rho}{\rho - I}$. Therefore,  
\begin{equation*}
\lim_{\ep \downarrow 0} \var\big[\pi_\rho\big] = \frac{\rho(I - \rho)}{I},
\end{equation*}
which explains $\lambda = \frac{\rho(I-\rho)}{I}$.
\end{proof}
We proceed to prove Lemma \ref{lem:timedecorr2} rigorously.
The first step is to express $\twogradz{t, x, x} - \frac{\rho(I - \rho)}{I} Z(t, x)^2$ in terms of the two duality functionals in Lemma \ref{lem:dualt},
\begin{align*}
&\twogradz{t, x, x} - \frac{\rho(I - \rho)}{I} Z(t, x)^2\\ 
&= \bigg((\etat_{x+1} (t) - \rho)^2 - \frac{\rho(I - \rho)}{I}\bigg) Z(t, x)^2 + \epsilon^{\frac{1}{2}} \bb(t, x) Z(t, x)^2  \\ 
&= \bigg((I - \etat_{x+1}(t)) (I - 1- \etat_{x+1}(t)) - \frac{(I-1)(I-\rho)^2}{I}\bigg) Z(t, x)^2 - (2\rho +1 - 2I ) \onegradz{t, x,x} + \epsilon^{\frac{1}{2}} \bb(t, x) Z(t, x)^2\\
\numberthis \label{eq:73temp3}
&=  \bigg((I - \etat_{x+1}(t)) (I - 1- \etat_{x+1}(t)) - \frac{(I-1)(I-\rho)^2}{I}\bigg) Z(t, x+1)^2 - (2\rho +1 - 2I ) \onegradz{t, x,x} + \epsilon^{\frac{1}{2}} \bb(t, x) Z(t, x)^2
\end{align*}
In the last equality, we replaced $Z(t, x)$ by $Z(t, x+1)$, according to \eqref{eq:seven1:temp5}, this procedure produces an error term which can be absorbed in the $\epsilon^{\frac{1}{2}} \bb(t, x) Z(t, x)^2$.
\bigskip
\\
Recall that $\qhalfint{n} = \frac{q^{\frac{n}{2}} - q^{-\frac{n}{2}}}{q^{\frac{1}{2}} - q^{-\frac{1}{2}}}$. Under weak asymmetric scaling, $q = e^{\sqrt{\epsilon}}$, one has 
\begin{equation}\label{eq:seven:temp9}
\qhalfint{n} = n + \OO(\epsilon^{\frac{1}{2}}), \qquad q^{\eta_x (t)} = 1 + \OO(\eph).
\end{equation}
These approximations imply that 
\begin{align*}
&(I - \etat_{x+1}(t)) (I - 1- \etat_{x+1}(t)) Z(t, x+1)^2 \\
&=\qhalfint{I - \etat_{x+1}(t)} \qhalfint{I - 1- \etat_{x+1}(t)} Z(t, x+1)^2 q^{\etat_{x+1}(t)} + \ep^{\frac{1}{2}} \bb(t, x) Z(t, x)^2,\\ 
\numberthis \label{eq:73temp4}
&= D(t, x+1, x+1) +  \ep^{\frac{1}{2}} \bb(t, x) Z(t, x)^2.
\end{align*}
where we recall the expression of the functional $D$ from \eqref{eq:dfunc}. Inserting \eqref{eq:73temp4} into the RHS of \eqref{eq:73temp3}
\begin{align*}
&\twogradz{t, x, x} - \frac{\rho (I- \rho)}{I} Z(t, x)^2\\ 
\numberthis \label{eq:73temp7}
&= D(t, x+1, x+1) - \frac{(I-1)(I-\rho)^2}{I} Z(t, x+1)^2 - (2\rho +1 - 2I) \onegradz{t, x, x} + \eph \bb(t, x) Z(t, x)^2.
\end{align*} 
Recall that our goal is to show 
$$\bignorm{\epsilon^{2} \sum_{s=0}^{t} \gradsquare{s, \xstar(s)}}_2 \leq C \epsilon^{\frac{1}{4}} e^{2u\ep |\xstar|}.$$
Referring to the expression of $\gradsquare{s, \xstar(s)}$ in \eqref{eq:seven:temp8}, we need to prove that there exists some $0 < \delta < 1$ such that for all $i \in \NN$,
\begin{equation*}
\bignorm{\ep^2 \sum_{s=0}^t \ue(s, i) \Big(\twogradz{s, \xstar(s)-i, \xstar(s)-i} - \frac{\rho(I-\rho)}{I} Z(s, \xstar(s)-i)^2\Big)}_2 \leq C \ep^{\frac{1}{4}} e^{2 u\ep |\xstar|} \delta^i.
\end{equation*}
Using \eqref{eq:73temp7}, it suffices to show that for all $i \in \NN$,
\begin{align}\label{eq:seven1:temp6}
\bignorm{\sum_{s = 0}^t \ue(s, i) \Big(D(s, \xstar(s) +1 -i, \xstar(s) + 1 -i) - \frac{(I-1)(I - \rho)^2}{I} Z(s, \xstar(s) +1 -i)^2\Big) }_2 \leq C \ep^{\frac{1}{4}} e^{2u\ep |\xstar|} \delta^i.
\end{align}
and
\begin{align}\label{eq:seven1:temp8}
\bignorm{\sum_{s = 0}^t \ue(s, i) \onegradz{s, \xstar(s), \xstar(s)}}_2 \leq C \ep^{\frac{1}{4}} e^{2u\ep |\xstar|} \delta^i.
\end{align}
Note that \eqref{eq:seven1:temp8} is proved by taking $i = 0$ in \eqref{eq:72sts1}. Therefore, we only need to prove \eqref{eq:seven1:temp6}.
Similar to the proof in Lemma \ref{lem:timedecorr1}, to conclude \eqref{eq:seven1:temp6}, it suffices to prove the following lemma for upper bounding the conditional expectation. We do not repeat the rest of the proof here. 
\begin{lemma}
For $T > 0$ and $n \in \NN$, there exists constant $C$ and $u$ such that for all $x \in \Xi(t)$ and $s \leq t \in [0, \ep^{-2} T] \cap \ZZ$, 
\begin{align}\label{eq:73conditionex}
\bignorm{\EE\bigg[D(t, x, x) - \frac{(I-1)(I-\rho)^2}{I} Z(t, x)^2 \vvv \FF(s)\bigg]}_n \leq \frac{C \ep^{-\frac{1}{2}}}{\sqrt{t - s + 1}} e^{2u \ep |x|}.
\end{align}
\end{lemma}
\begin{proof}
	Combining both of the dualities \eqref{eq:fdualt} and \eqref{eq:sdualt}, one has
\begin{align*}
&\EE\bigg[D(t, x, x) - \frac{(I-1)(I-\rho)^2}{I} Z(t, x)^2 \vvv \FF(s)\bigg] \\
&= \sum_{y_1 \leq y_2 \in \Xi(s)} \rhzrte \big((x, x), (y_1, y_2), t, s\big) \bigg(D(s, y_1, y_2) - \frac{(I-1)(I-\rho)^2}{I} Z(t, y_1) Z(t, y_2)\bigg)
\end{align*}
We split the summation above according to the range of the value of $|y_1 - y_2|$,
\begin{align*}
&\EE\bigg[D(t, x, x) - \frac{(I-1)(I-\rho)^2}{I} Z(t, x)^2 \vvv \FF(s)\bigg]\\ 
&= \sum_{\substack{y_1 < y_2 \in \Xi(s)\\ |y_1 - y_2| \geq 3}} \rhzrte\big((x, x), (y_1, y_2), t, s\big) \bigg(D(s, y_1, y_2) - \frac{(I-1)(I-\rho)^2}{I} Z(s, y_1) Z(s, y_2)\bigg) \\
\numberthis
\label{eq:73temp5}
&\quad+ \sum_{\substack{y_1 \leq y_2 \in \Xi(s) \\ 
|y_1 - y_2| \leq 2}} \rhzrte\big((x, x), (y_1, y_2), t, s\big) \bigg(D(s, y_1, y_2) - \frac{(I-1)(I-\rho)^2}{I} Z(s, y_1) Z(s, y_2)\bigg).
\end{align*}
We name the terms on the RHS of \eqref{eq:73temp5} $\mathbf{E_1}$ and $\mathbf{E_2}$ respectively  and we bound them separately. It follows from Proposition \ref{prop:tightness} that  
\begin{equation*}
\bignorm{D(s, y_1, y_2) - \frac{(I-1)(I-\rho)^2}{I} Z(s, y_1) Z(s, y_2)}_n \leq C e^{u \ep (|y_1| + |y_2|)}
\end{equation*}
Invoking Theorem \ref{prop:semiestimate} (a) and Lemma \ref{lem:five:usefullem}, we find that 
\begin{align}\label{eq:E2bound}
\norm{\mathbf{E_2}}_n \leq \sum_{\substack{y_1 \leq y_2 \in \Xi(s) \\ 
|y_1 - y_2| \leq 2}} \frac{C(\beta, T)}{t-s+1} e^{\frac{-\beta (|y_1-x| +|y_2 - x|)}{\sqrt{t-s+1} + C(\beta)}} e^{u \ep (|y_1| + |y_2|)}  \leq \frac{C}{\sqrt{t-s+1}} e^{2 u\ep |x|}.
\end{align}
We proceed to bound $\mathbf{E_1}$, recall that when $y_1 < y_2$,  
\begin{equation*}
D(s, y_1, y_2) = \frac{\qhalfint{I-1}}{\qhalfint{I}} Z(s, y_1) Z(s, y_2) \qhalfint{I - \etat_{y_1} (s)}  \qhalfint{I - \etat_{y_2}(s)}    q^{\frac{1}{2} \etat_{y_1}(s)} q^{\frac{1}{2} \etat_{y_2}(s)},
\end{equation*}
which could be rewritten as (using \eqref{eq:seven:temp9})
\begin{equation*}
D(s, y_1, y_2) = \frac{I-1}{I} (I - \etat_{y_1}(s)) (I - \etat_{y_2}(s)) Z(s, y_1) Z(s, y_2) + \eph \bb(s, y_1, y_2) Z(s, y_1) Z(s, y_2).
\end{equation*}
Consequently, we write
\begin{align*}
\mathbf{E_1} &= \frac{I-1}{I} \sum_{\substack{y_1 < y_2 \in \Xi(s) \\ |y_1 - y_2| \geq 3}} \rhzrte\big((x, x), (y_1, y_2), t, s\big) \Big((I - \etat_{y_1}(s)) (I - \etat_{y_2} (s)) - (I-\rho)^2\Big) Z(s, y_1) Z(s, y_2) \\
&\quad+\ep^{\frac{1}{2}} \sum_{\substack{y_1 < y_2 \in \Xi(s) \\ |y_1 - y_2| \geq 3}} \rhzrte\big((x, x), (y_1, y_2), t, s\big) \bb(s, y_1, y_2)  Z(s, y_1) Z(s, y_2)\\
&=\frac{I-1}{I} \sum_{\substack{y_1 < y_2 \in \Xi(s) \\ |y_1 - y_2| \geq 3}} \rhzrte\big((x, x), (y_1, y_2), t, s\big) \Big((\rho - \etat_{y_1}(s)) (I - \etat_{y_2}(s)) + (I - \rho)(\rho - \etat_{y_2}(s))\Big) Z(s, y_1) Z(s, y_2) \\
\numberthis \label{eq:73temp6}
&\quad +\ep^{\frac{1}{2}} \sum_{\substack{y_1 < y_2 \in \Xi(s) \\ |y_1 - y_2| \geq 3}} \rhzrte\big((x, x), (y_1, y_2), t, s\big) \bb(s, y_1, y_2)  Z(s, y_1) Z(s, y_2)
\end{align*}
It Is straightforward by \eqref{eq:seven1:temp2} and \eqref{eq:seven1:temp5} that
\begin{align*}
(\rho - \etat_{y_1}(s)) Z(s, y_1) &= (\rho - \etat_{y_1}(s)) Z(s, y_1 - 1) + \eph \bb(s, y_1) Z(s, y_1) = \ep^{-\frac{1}{2}} \nabla Z(s, y_1 - 1) + \eph \bb(s, y_1) Z(s, y_1), \\
(\rho - \etat_{y_2}(s)) Z(s, y_2) &= (\rho - \etat_{y_2}(s)) Z(s, y_2 - 1) + \eph \bb(s, y_2) Z(s, y_2) = \ep^{-\frac{1}{2}} \nabla Z(s, y_2 - 1) + \eph \bb(s, y_2) Z(s, y_2).
\end{align*}
Inserting these into the RHS of \eqref{eq:73temp6},
\begin{align*}
\mathbf{E_1} &= \frac{I-1}{I} \sum_{\substack{y_1 < y_2 \in \Xi(s)\\ |y_1 - y_2 | > 2}} \rhzrte\big((x, x), (y_1, y_2), t, s\big) (I - \etat_{y_2}(s)) (\ep^{-\frac{1}{2}} \nabla Z(s, y_1))  Z(s, y_2)
\\
&\quad+\frac{I-1}{I}  \sum_{\substack{y_1 < y_2 \in \Xi(s)\\ |y_1 - y_2 | > 2}} \rhzrte\big((x, x), (y_1, y_2), t, s\big) (I - \rho) (\ep^{-\frac{1}{2}} \nabla Z(s, y_2)) Z(s, y_1)  \\ &\quad +\sum_{\substack{y_1 < y_2 \in \Xi(s) \\ |y_1 - y_2| > 2}} \ep^{\frac{1}{2}} \rhzrte\big((x, x), (y_1, y_2), t, s\big) \bb(s, y_1, y_2)  Z(s, y_1) Z(s, y_2).
\end{align*}
Let us name respectively the three terms on the RHS above to be $\mathbf{J_1}$, $\mathbf{J_2}$, $\mathbf{J_3}$. Recall the summation by part formula (with notation $\nabla f(x)= f(x+1) - f(x) $)
\begin{align*}
\numberthis
\label{eq:sumbypart1}
&\sum_{x < y  } \na f(x) \cdot g(x) = f(y) g(y-1) - \sum_{x < y} f(x)  \cdot \na g(x-1),\\
&\sum_{x > y} \na f(x) \cdot g(x) = -f(y+1) g(y+1) - \sum_{x > y} f(x+1) \na g(x).
\end{align*}
Note that 
\begin{equation*}
\mathbf{J_1} = \frac{I-1}{I} \sum_{\substack{y_1 < y_2 \in \Xi(s)\\ |y_1 - y_2 | > 2}} \rhzrte\big((x, x), (y_1, y_2), t, s\big) (I - \etat_{y_2}(s)) (\ep^{-\frac{1}{2}} \nabla Z(s, y_1))  Z(s, y_2),
\end{equation*} 
by \eqref{eq:sumbypart1}, we move the gradient from $\na Z(s, y_1)$ to $\rhzrte$, 
\begin{align*}
\mathbf{J_1} &= \frac{I-1}{I} \bigg[\sum_{y_2 \in \Xi(s)} \ep^{-\frac{1}{2}} \rhzrte\big((x, x), (y_2-3, y_2), t, s\big) (I -\etat_{y_2}(s)) Z(s, y_2-3) Z(s, y_2)\\ &\quad-\sum_{\substack{y_1 < y_2 \in \Xi(s)\\ |y_1 - y_2| > 2}} \ep^{-\frac{1}{2}} \na_{y_1} \rhzrte \big((x, x), (y_1, y_2), t, s\big) (I - \etat_{y_2}(s)) Z(s, y_1) Z(s, y_2)\bigg].
\end{align*} 
Using Theorem \ref{prop:semiestimate} part (a) and part (b) to control $\rhzrte$ and $\nabla \rhzrte$ respectively, we see that for $n \in \NN$,
\begin{align*}
\norm{\mathbf{J_1}}_n &\leq C(\beta, T) \bigg(\sum_{y_2 \in \Xi(s)}  \frac{\ep^{-\frac{1}{2}}}{t - s +1} e^{\frac{-\beta (|y_2-x| + |y_2 - 3 - x|)}{\sqrt{t-s+1} + C(\beta)}} e^{u \ep (|y_2 - 3| + |y_2|)}\\ 
&\quad + \sum_{{\substack{y_1 \leq y_2} \in \Xi(s) }} \frac{\ep^{-\frac{1}{2}}}{(t-s+1)^{\frac{3}{2}}} e^{-\frac{\beta (|y_1 -x_1 |+ |y_2 - x_2|)}{\sqrt{t-s+1} + C(\beta)}} e^{u \ep (|y_1| + |y_2|)}\bigg).
\end{align*}
Applying Lemma \ref{lem:five:usefullem}
yields $\norm{\mathbf{J_1}}_n \leq \frac{C\ep^{-\frac{1}{2}}}{\sqrt{t-s+1}} e^{2 u\ep |x|}$.
Likewise,  we obtain $\norm{\mathbf{J_2}}_n \leq \frac{C\ep^{-\frac{1}{2}}}{\sqrt{t-s+1}} e^{2 u\ep |x|}$. 
\bigskip
\\
For $\mathbf{J_3}$, applying Theorem \ref{prop:semiestimate} part (a) and Lemma \ref{lem:five:usefullem} implies that
\begin{equation*}
\norm{\mathbf{J_3}}_n \leq \sum_{y_1 \leq y_2}  \frac{C(\beta, T) \ep^{\frac{1}{2}}}{t-s+1} e^{-\frac{\beta(|x  - y_1| + |x - y_2|)}{\sqrt{t-s+1} + C(\beta)}} e^{u \ep (|y_1| + |y_2|)} \leq C \eph  e^{2u \ep |x| } \leq  \frac{C \ep^{-\frac{1}{2}}}{\sqrt{t-s+1}} e^{2 u \ep |x|}. 
\end{equation*}
In the last inequality above, we used the fact $s \leq t \in [0, \ep^{-2} T]$, which implies $t-s \leq \ep^{-2} T$.
\bigskip
\\
Combining the bounds for $\norm{\mathbf{J_1}}_n$, $\norm{\mathbf{J_2}}_n$, $\norm{\mathbf{J_3}}_n$, we have
\begin{equation}\label{eq:seven1:temp7}
\norm{\mathbf{E_1}}_n \leq \frac{C \ep^{-\frac{1}{2}}}{\sqrt{t-s+1}} e^{2u \ep |x|}.
\end{equation}
Recall from \eqref{eq:73temp5} that
\begin{equation*}
\EE\bigg[D(t, x, x) - \frac{(I-1)(I-\rho)^2}{I} Z(t, x)^2 \vvv \FF(s)\bigg] = \mathbf{E_1} + \mathbf{E_2},
\end{equation*}
combining the bounds for $\mathbf{E_1}$ and $\mathbf{E_2}$ in \eqref{eq:seven1:temp7} and \eqref{eq:E2bound}, we conclude the desired \eqref{eq:73conditionex}.
\end{proof}

\appendix

\section{Stationary distribution of the SHS6V model}\label{sec:stat}
 In this section, we provide a one parameter family of stationary distribution for the unfused \ac{SHS6V} model. It is worth to remark that in the recent work of \cite{IMS19}, a translation-invariant Gibbs measure was obtained (using the idea from \cite{Agg16}) for the space-time inhomogeneous \ac{SHS6V} model on the full lattice, see Proposition 4.5 of \cite{IMS19}. However, It is not obvious that the dynamic of \ac{SHS6V} model under this Gibbs measure coincides with the one of the bi-infinite \ac{SHS6V} model specified in Lemma \ref{lem:biinfinite}. This being the case, we choose to proceed without relying on the result from \cite{IMS19}.
\bigskip
\\
We start with a well-known combinatoric lemma.
\begin{lemma}[q-binomial formula]\label{thm:qbinom}
Set $\nu  = q^{-I}$ as usual, the following identity holds for all $q \in \CC$,
\begin{equation*}\label{eq:qbinom}
\sum_{n=0}^I \frac{(\nu; q)_n}{(q; q)_n} z^n = \frac{(\nu z; q)_\infty}{(z; q)_\infty}.
\end{equation*}
\end{lemma}
\begin{proof}
According to $q$-binomial theorem \cite{GR00}, 
\begin{equation*}
\sum_{n=0}^\infty \frac{(\nu; q)_n}{(q; q)_n} z^n = \frac{(\nu z; q)_\infty}{(z; q)_\infty}.
\end{equation*} 
When $\nu = q^{-I}$, $(\nu, q) _n= 0$ for $n > I$. Therefore, 
\begin{equation*}
\sum_{n=0}^I \frac{(\nu; q)_n}{(q; q)_n} z^n = \sum_{n=0}^\infty \frac{(\nu; q)_n}{(q; q)_n} z^n= \frac{(\nu z; q)_\infty}{(z; q)_\infty}.
\end{equation*}
\end{proof}
\begin{lemma}\label{lem:appendix:stat}
Fix $q > 1$, $\nu = q^{-I}$ and $\rho \in (0, I)$, define a probability measure $\statmargin$ on $\iset$:
\begin{equation}\label{eq:stat}
\statmargin (i) = \frac{(\chi, q)_\infty}{(\chi \nu, q)_\infty} \frac{(\nu, q)_i}{(q, q)_i}  \chi^i, \quad i \in \iset,
\end{equation}
where 
$\chi$ is  the unique negative real number satisfying 
\begin{equation}\label{eq:crhodefin}
\sum_{i=1}^{I} \frac{\chi}{\chi  - q^{i}} = \rho.
\end{equation} 
Furthermore, we have 
\begin{equation*}
\EE\big[\statmargin\big] = \rho, \qquad \text{Var}\big[\statmargin\big] = \rho - \sum_{i=1}^I \frac{\chi^2}{(q^i - \chi)^2}.
\end{equation*}
\end{lemma}
\begin{proof}
We first show that $\statmargin$ is indeed a probability measure. It is clear that
$\statmargin (i) \geq 0$ for all $i \in \iset$. 
By Lemma \ref{thm:qbinom},
\begin{equation*}
\sum_{i=0}^I \statmargin (i) = \frac{(\chi, q)_\infty}{(\chi \nu, q)_\infty} \sum_{i=0}^I  \frac{(\nu, q)_i}{(q, q)_i}  \chi^i = \frac{(\chi, q)_\infty}{(\chi \nu, q)_\infty} \frac{(\nu \chi, q)_\infty}{(\chi, q)_\infty} = 1.
\end{equation*}
Next, we compute the expectation and the variance of $\statmargin$. Using again Lemma \ref{thm:qbinom}, the moment generating function is given by 
\begin{equation}\label{eq:temp2}
\Lambda(z) = \frac{(\chi, q)_\infty}{(\chi \nu, q)_\infty} \sum_{i=0}^I \frac{(\nu, q)_i}{(q, q)_i} \chi^i z^i =  \frac{(\chi, q)_\infty}{(\chi \nu, q)_\infty} \frac{(\nu \chi z, q)_\infty}{(\chi z, q)_\infty} =  \frac{(\chi, q)_\infty}{(\chi \nu, q)_\infty} \prod_{i=1}^{I} (1 - \nu q^{i-1} \chi z).
\end{equation}
It is clear that
\begin{align*}
\EE\big[\statmargin\big] &= \Lambda'(1),\\
\var\big[\statmargin\big] &= \Lambda''(1) + \Lambda'(1) - \Lambda'(1)^2. 
\end{align*}
Via \eqref{eq:temp2}, one has
\begin{align*}
&\Lambda'(z) = \frac{(\chi, q)_\infty}{(\chi \nu, q)_\infty}  \bigg(\prod_{i=1}^{I} (1 - \nu q^{i-1} \chi z)\bigg)  \bigg(\sum_{i=1}^I \frac{-\nu q^{i-1} \chi}{1 - \nu q^{i-1} \chi  z}\bigg),\\
&\Lambda''(z) =\frac{(\chi, q)_\infty}{(\chi \nu, q)_\infty} \bigg(\prod_{i=1}^I (1 - \nu q^{i-1} \chi z)\bigg)\bigg[\bigg(\sum_{i=1}^I \frac{-\nu q^{i-1}  \chi}{1 - \nu q^{i-1} \chi z}\bigg)^2 - \sum_{i=1}^I \frac{(\nu q^{i-1} \chi)^2}{(1 - \nu q^{i-1} \chi z)^2}\bigg].
\end{align*}
Note that $$ \frac{(\chi, q)_\infty}{(\chi \nu, q)_\infty}  \prod_{i=1}^{I} (1 - \nu q^{i-1} \chi) = 1,$$
combining this with \eqref{eq:crhodefin} yields
\begin{equation*}
\Lambda'(1) = \rho, \qquad \Lambda''(1) = \rho^2 - \sum_{i=1}^I \frac{\chi^2}{(q^i - \chi)^2},
\end{equation*}
which concludes the lemma.
\end{proof}
\begin{theorem}\label{thm:appendix:stat}
For $\rho \in (0, I)$, the product measure $\stat$ is stationary for the unfused \ac{SHS6V} model $\vec{\eta}(t)$ (Definition \ref{def:biinfinite}). 
\end{theorem}
\begin{proof}
It suffices to show that if $\vec{\eta}(t) \sim \stat$, then $\vec{\eta}(t+1) \sim \stat$.
\bigskip
\\
Recall that $K(t, y) = N(t, y) - N(t+1, y)$ records the number of particles (either zero or one) that move across location $y$ at time $t$. We first show that  $K(t, y) \sim \ber(\frac{\alpha(t) \chi}{\alpha(t) \chi +1})$ (recall that $\alpha(t) = \alpha q^{\mod(t)}$). To this end, referring to \eqref{eq:two:recur},
\begin{equation}\label{eq:appendix:ktemp}
K(t, y)  =\sum_{y' = -\infty}^y \prod_{z = y'+1}^y \bigg(B'(t, z, \eta_z (t)) - B(t, z, \eta_z (t))\bigg) B(t, y', \eta_{y'}(t)).
\end{equation}
Recalling from \eqref{eq:randomenvir}, $B(t, z, \eta) \sim \ber\big(\frac{\alpha(t)(1 - q^\eta)}{1 +\alpha(t)}\big), B'(t, z, \eta) \sim \ber\big(\frac{\alpha(t) + \nu q^\eta}{1 +\alpha(t)}\big)$. Since the random variables $B, B'$ are all independent,
\begin{equation*}
\EE\bigg[\prod_{z = y'+1}^y \bigg(B'(t, z, \eta_z (t)) - B(t, z, \eta_z (t))\bigg) B(t, y', \eta_{y'}(t)) \vvv \FF(t)\bigg] = \frac{\alpha(t) (1-q^{\eta_{y'}(t)})}{1 + \alpha(t)} \prod_{z = y'+1}^y \frac{(\alpha(t) + \nu) q^{\eta_z  (t)}}{1 + \alpha(t)}. 
\end{equation*}
Therefore, by tower property
\begin{align*}
\EE\big[K(t, y)\big] &= \sum_{y' = -\infty}^y \EE \bigg[\prod_{z = y'+1}^y \frac{\alpha(t) (1-q^{\eta_{y'}(t)})}{1 + \alpha(t)} \prod_{z = y'+1}^y \frac{(\alpha(t) + \nu) q^{\eta_z  (t)}}{1 + \alpha(t)}\bigg], \\
\numberthis \label{eq:appendix:temp1}
&= \sum_{y' = -\infty}^y  \frac{\alpha(t)}{1 + \alpha(t)} \bigg(\frac{\alpha(t) + \nu}{1 + \alpha(t)}\bigg)^{y-  y'} \big(\EE\big[q^{\eta_y(t)}\big]\big)^{y-  y'} (1 - \EE\big[q^{\eta_y(t)}\big]).
\end{align*}
As $\eta_y (t) \sim \statmargin$, we obtain using Lemma \ref{thm:qbinom} 
\begin{equation*}
\EE\big[q^{\eta_y (t)}\big] =\frac{(\chi, q)_\infty}{(\chi \nu, q)_\infty} \sum_{i=0}^{\infty}  \frac{(\nu, q)_i}{(q, q)_i} (\chi q)^i = \frac{(\chi \nu q; q)_\infty}{(\chi q; q)_\infty} \frac{(\chi; q)_\infty}{(\chi \nu; q)_\infty} = \frac{1 - \chi}{1 - \chi \nu}. 
\end{equation*}
Inserting the value of $\EE\big[q^{\eta_y (t)}\big]$  into the RHS of \eqref{eq:appendix:temp1} yields that $$\EE\big[K(t, y)\big] = \sum_{y' = -\infty}^y  \frac{\alpha(t)}{1 + \alpha(t)} \bigg(\frac{(\alpha(t) + \nu)(1-\chi)}{(1 + \alpha(t))(1-\chi\nu)}\bigg)^{y-  y'} \bigg(1 - \frac{1 - \chi}{1 - \chi \nu}\bigg) = \frac{\alpha(t) \chi}{\alpha(t) \chi +1}.$$ 
Since $K(t, y) \in \{0, 1\}$, we conclude that 
\begin{equation}\label{eq:app:temp3}
K(t, y) \sim \ber(\frac{\alpha(t) \chi}{\alpha(t) \chi +1}).
\end{equation}
The next step is to show that the marginal of  $\vec{\eta}(t+1)$ is distributed as  $\statmargin$ for each coordinate. Referring to \eqref{eq:appendix:ktemp}, it is straightforward that the following recursion holds
\begin{align*}
K(t, y)
\numberthis \label{eq:appendix:temp2}
= &B(t, y, \eta_{y} (t)) + \Big(B'(t, y, \eta_{y} (t)) - B(t, y, \eta_{y} (t))\Big) K(t, y-1).
\end{align*}
Therefore,
\begin{align*}
\eta_{y} (t) - \eta_{y} (t+1) &= N(t, y) - N(t, y-1) + N(t+1, y-1) - N(t+1, y) , \\
&= K(t, y) - K(t, y-1), \\
&= K(t, y-1) \Big(B'(t, y, \eta_{y}(t)) - B(t, y, \eta_{y}(t)) - 1\Big) + B(t, y, \eta_{y} (t)).
\end{align*}
For the second equality above, we used $K(t, y) = N(t, y) - N(t+1, y)$. Therefore,
\begin{equation}\label{eq:appendix:etarecur}
\eta_{y} (t+1)  = 
\begin{cases}
\eta_{y} (t) - B(t, y, \eta_{y} (t)), \qquad &K(t, y-1) = 0,\\
\eta_{y} (t) + 1 - B'(t, y, \eta_{y} (t)), \qquad &K(t, y-1) = 1.
\end{cases}
\end{equation}
Due to \eqref{eq:appendix:ktemp}, we see that $K(t, y-1) \in \sigma\Big(B(t, z, \eta), B'(t, z, \eta), \eta_z(t): z \leq y-1, \eta \in \iset \Big)$. Note that we have assumed $\vec{\eta}(t) \sim \stat$, which implies the independence between $\eta_y (t)$ and $\eta_z(t)$ for $z \neq y$. Therefore,  
$\eta_{y} (t)$ and $K(t, y-1)$ are independent. Using \eqref{eq:appendix:etarecur} we get
\begin{align*}
\PP\big(\eta_{y} (t+1) = i\big) &= \PP\big(K(t, y-1) = 0\big) \PP\big(\eta_{y} (t) - B(t, y, \eta_{y} (t)) = i\big)\\ 
&\quad +\PP\big(K(t, y-1) = 1\big) \PP\big(\eta_{y} (t) - B'(t, y, \eta_{y} (t)) = i-1\big).
\end{align*} 
By $K(t, y-1) \sim \ber(\frac{\alpha(t) \chi}{\alpha(t) \chi +1})$ and $\eta_{y} (t) \sim \statmargin$, one readily has 
\begin{align*}
&\PP\big(\eta_{y} (t+1) = i\big)\\
&= \frac{1}{1 + \alpha(t) \chi} \bigg[\statmargin (i) \frac{1 + \alpha(t) q^i}{1 + \alpha(t)} + \statmargin (i+1) \frac{\alpha(t) (1 - q^{i+1})}{1 +\alpha(t)}\bigg] + \frac{\alpha(t) \chi}{1 + \alpha(t) \chi} \bigg[\statmargin (i) \frac{\alpha(t) + \nu q^i}{1 +\alpha(t)} + \statmargin (i-1) \frac{1 - \nu q^{i-1}}{1 + \alpha(t)}\bigg]\\ 
&= \statmargin (i).
\end{align*}
To conclude Theorem \ref{thm:appendix:stat}, it suffices to show the independence among $\eta_y (t+1)$ for different value of  $y$. It is enough to show that
\begin{equation}\label{eq:appendix:desiredind}
\eta_{y} (t+1) \independent \{\eta_{y+1} (t+1), \eta_{y+2}(t+1), \dots\} \text{ for all } y \in \ZZ.
\end{equation} 
We need the following lemma.
\begin{lemma}\label{lem:appendix:independent}
For all $y \in \ZZ$, $\eta_{y} (t+1)$ is independent with $K(t, y)$.
\end{lemma}
Let us first see how this lemma leads to  \eqref{eq:appendix:desiredind}. We have via \eqref{eq:appendix:ktemp},  
$$K(t, y) \in \sigma\Big(B(t, z, \eta), B'(t, z, \eta), \eta_{z} (t): z \leq y, \eta \in \iset\Big).$$ 
Combining this with \eqref{eq:appendix:etarecur},
$$\eta_{y}(t+1) \in \sigma\Big(B(t, z, \eta), B'(t, z, \eta), \eta_{z} (t): z \leq y, \eta \in \iset\Big).$$
Since $\eta_i (t)$ are all independent for different $i$, one has 
\begin{equation*}
\Big(B(t, z, \eta), B'(t, z, \eta), \eta_{z} (t): z \leq y, \eta \in \iset\Big) \independent (\eta_{y+1}(t), \eta_{y+2}(t), \dots).
\end{equation*}  
We achieve
\begin{equation*}
\big(K(t, y), \eta_{y} (t+1)\big) \independent \big(\eta_{y+1}(t), \eta_{y+2} (t), \dots \big).
\end{equation*}
Using Lemma \ref{lem:appendix:independent}, we conclude  
\begin{equation*}
\eta_{y}(t+1) \independent \big(K(t, y), \eta_{y+1}(t), \eta_{y+2}(t), \dots \big).
\end{equation*}
Therefore, 
\begin{equation}\label{eq:app:temp1}
\eta_{y}(t+1) \independent  \sigma\Big(K(t, y),  \eta_z(t), B(t, z, \eta), B'(t, z, \eta): z \geq y+1, \eta \in \iset\Big).
\end{equation}
On the other hand, by \eqref{eq:appendix:temp2} and \eqref{eq:appendix:etarecur}, we conclude for all $y \in \ZZ$
\begin{equation}\label{eq:app:temp2}
\big(\eta_{y+1}(t+1), \eta_{y+2}(t+1), \dots \big) \in \sigma\Big(K(t, y), B(t, z, \eta), B'(t, z, \eta), \eta_z(t): z \geq y+1, \eta \in \iset\Big).
\end{equation}
Combining \eqref{eq:app:temp1} and \eqref{eq:app:temp2}, we find that for all $y \in \ZZ$
\begin{equation*}
\eta_{y} (t+1) \independent \big(\eta_{y+1}(t+1), \eta_{y+2}(t+1), \dots \big),
\end{equation*}
which concludes \eqref{eq:appendix:desiredind}.
\end{proof}
\begin{proof}[Proof of Lemma \ref{lem:appendix:independent}]
As $K(t, y) \in \{0, 1\}$, it suffices to show that for all $j \in \iset$, one has
\begin{equation*}
\PP\big(\eta_{y}(t+1) = j, K(t, y) = 1\big) = \PP\big(\eta_{y}(t+1) = j\big) \PP\big(K(t, y) = 1 \big).
\end{equation*}
Due to \eqref{eq:appendix:temp2},
\begin{equation*}
K(t, y) = 
\begin{cases}
B(t, y, \eta_{y} (t)), \qquad &K(t, y-1) = 0,\\
B'(t, y, \eta_{y} (t)), \qquad &K(t, y-1) = 1.
\end{cases}
\end{equation*}
Together with \eqref{eq:appendix:etarecur},
we obtain that if $K(t, y-1) = 0$, 
\begin{equation*}
\big(\eta_{y} (t+1), K(t, y)\big) = (j, 1) \text{ is equivalent to } \big(\eta_{y} (t), B(t, y, \eta_{y} (t))\big) = (j+1, 1).
\end{equation*}
If $K(t, y-1) = 1$, 
\begin{equation*}
\big(\eta_{y} (t+1), K(t, y)\big) = (j, 1) \text{ is equivalent to } \big(\eta_{y} (t), B(t, y, \eta_{y} (t))\big) = (j, 1).
\end{equation*}
The discussion above yields (using the independence between  $\eta_{y} (t)$ and $K(t, y-1)$)
\begin{align*}
&\PP\big(\eta_{y} (t+1) = j, K(t, y) = 1\big), \\
&= \PP\big(K(t, y-1) = 0\big) \PP\big(\eta_{y} (t) = j+1, B(t, y, \eta_{y} (t)) = 1\big) + \PP\big(K(t, y-1) = 1\big) \PP\big(\eta_{y} (t) = j, B'(t, y, \eta_{y} (t)) = 1\big),\\
&= \frac{1}{1 + \alpha(t) \chi}  \frac{\alpha(t) (1-  q^{j+1})}{1 + \alpha(t)} \statmargin (j+1) + \frac{\alpha(t) \chi}{1 + \alpha(t) \chi}  \frac{\alpha(t) + \nu q^j}{1 + \alpha(t)} \statmargin (j), \\
&= \frac{\alpha(t) \chi \statmargin(j) }{\alpha(t) \chi +1} = 
\PP\big(\eta_{y+1} (t+1) = j\big) \PP\big( K(t, y) = 1\big),
\end{align*}
which concludes Lemma \ref{lem:appendix:independent}.
\end{proof}
\begin{remark}\label{rmk:temp2}
Since $\vec{g}(t) = \vec{\eta}(Jt)$, it is clear that for all $\rho \in (0, I)$, $\stat$ is also stationary for the fused \ac{SHS6V} model $\vec{g}(t)$. 
\end{remark}
\section{KPZ scaling theory}\label{app:kpzscaling}
The \ac{KPZ} scaling  theory has been developed in a landmark contribution by \cite{krug92}. The scaling
theory is a physics approach which makes prediction for the non-universal coefficients of the \ac{KPZ} equation. 
In this appendix, we show how the coefficients of the \ac{KPZ} equation \eqref{eq:KPZ equation} arise from the microscopic observables of the fused \ac{SHS6V} model using the \ac{KPZ}  scaling theory. 
\bigskip
\\
Recall that Theorem \ref{thm:main} reads
\begin{align*}
\sqrt{\epsilon} \big(\N_\epsilon (\epsilon^{-2} t, \epsilon^{-1} x + \epsilon^{-2} \mu_\epsilon t ) - \rho (\epsilon^{-1} x + \epsilon^{-2} \mu_\epsilon  t )  - t \log \lambda_\epsilon \big) \wc \HH(t, x) \text{ in } C([0, \infty), C(\RR)) \text{ as } \ep \downarrow 0.
\end{align*}
Here, $\N_\ep (t, x)$ is the fused height function and $\kpz(t, x)$ solves the \ac{KPZ} equation
\begin{align*}
\partial_t \kpz(t, x) = \frac{\alpha_1}{2} \pa_x^2 \kpz(t, x) - \frac{\alpha_2}{2} \big(\pa_x \kpz(t, x) \big)^2 + \sqrt{\alpha_3} \xi(t, x),
\end{align*}
where
\begin{align*}
&\alpha_1 = \alpha_2 = JV_* = \frac{J\big((I +J)b - (I + J - 2)\big)}{I^2 (1-b)},\\
&\alpha_3 = J D_* = \frac{\rho(I-\rho)}{I} \cdot \frac{J \big((I + J) b - (I + J -2)\big)}{I^2 (1-b)}.
\end{align*}
%
The first step in the \ac{KPZ} scaling theory is to derive the stationary distribution of the fused \ac{SHS6V} model, which is exactly what we did in Appendix \ref{sec:stat} (see Remark \ref{rmk:temp2}). Under stationary distribution $\stat$, 
we proceed to define two natural quantities of the models:
\begin{itemize}
	\item The \emph{average steady state current} $j (\rho)$ is defined as  
\begin{equation}\label{app:temp1}
j(\rho) = \ep^{-\frac{1}{2}} \big(\big\langle \N(t, x) - \N(t, x+1) \big\rangle_\rho - \rho \mu_\ep\big),
\end{equation}
where $\langle \cdot \rangle_\rho$ means that we are taking the expectation under stationary distribution $\bigotimes\pi_\rho$ and $\mu$ is given in \eqref{eq:temp1}. Note that under stationary distribution, the average steady state current $j (\rho)$ depends neither on space or time. 
Let us explain the meaning of \eqref{app:temp1}. Note that $\N(t, x) - \N(t+1, x)$ records the number of particles in the fused \ac{SHS6V} model that move across location $x$ at time $t$, we subtract $\rho \mu_\ep$ here because we are in a reference frame that moves to right with speed $\rho \mu_\ep$.
\\
\item\emph{The integrated covariance} is defined as
\begin{align*}
&A(\rho):= \lim_{r \to \infty} \frac{1}{2 r} \bigg\langle \N (t, x+r) - \N (t, x-r) - \big\langle \N (t, x+r) - \N (t, x-r) \big\rangle_\rho  \bigg\rangle_{\!\rho}.
\end{align*}
\end{itemize}
The KPZ scaling theory (equation (12) and (15) of \cite{krug92}) predicts that 
$$(i)\  \alpha_2 = -\lim_{\ep \downarrow 0} j''_\ep (\rho), \qquad (ii)\  \frac{\alpha_3}{\alpha_1} = \lim_{\ep \downarrow 0} A_\ep (\rho), $$ 
$A_\ep (\rho)$ and $j_\ep (\rho)$ depend on $\ep$ under weakly asymmetry scaling \eqref{eq:scaling}.
\bigskip
\\
Let us first verify $(ii)$, note that under stationary distribution, $\N_\ep (t, x+ r) - \N_\ep (t, x-r)$ is the sum  of $2r$ i.i.d. random variables with the same distribution $\statmargin$, hence
$A_\ep (\rho) = \var\big[\statmargin\big].$
By Lemma \ref{lem:appendix:stat}, we know that 
\begin{equation*}
\text{Var}\big[\statmargin\big] = \rho - \sum_{i=1}^I \frac{\chi^2}{(q^i - \chi)^2},
\end{equation*}
where $\chi$ is the unique negative solution of 
\begin{equation}\label{eq:temp4}
\sum_{i=1}^{I} \frac{\chi}{\chi  - q^{i}} = \rho.
\end{equation} 
Under weakly asymmetric scaling, one has $q = e^{\sqrt{\ep}}$, which yields $\lim_{\ep \downarrow 0} \chi_\ep = \frac{\rho}{\rho - I}$. Therefore, 
\begin{equation*}
\lim_{\ep \downarrow 0} A_\ep (\rho) = \lim_{\ep \downarrow 0} \var\big[\statmargin\big] = \frac{\rho(I - \rho)}{I}.
\end{equation*}
This matches with the value of $\frac{\alpha_3}{\alpha_1}$.
\bigskip
\\
We proceed to verify $(i)$. First, note that by $\N(t, x) = N(Jt, x)$,
\begin{align*}
\N(t, x) - \N (t+1, x) = N(Jt, x) - N((J+1)t, x) = \sum_{s = Jt}^{(J+1)t - 1} K(s, x),
\end{align*} 
where $K(s, x) = N(s, x) - N(s+1, x)$.
We have shown in \eqref{eq:app:temp3} that $K(s, x) \sim \ber(\frac{\alpha(s) \chi}{1 + \alpha(s) \chi})$, where $\alpha(s) = \alpha q^{\mod(s)}$. Therefore,
$$\EE\big[\N(t, x) - \N (t+1, x)\big] = \EE\bigg[\sum_{s = Jt}^{(J+1)t - 1} K (s, x)\bigg] = \sum_{k=0}^{J-1} \frac{\alpha q^k \chi}{1 + \alpha q^k \chi},$$
which yields 
\begin{equation*}
j(\rho) = \ep^{-\frac{1}{2}}  \bigg(\sum_{k=0}^{J-1} \frac{\alpha q^k \chi}{1 + \alpha q^k \chi} - \rho \mu\bigg).
\end{equation*}
We proceed to taylor expand $j_\ep (\rho)$ around $\ep = 0$. Note that $\chi$ is implicitly defined through \eqref{eq:temp4}, we expand $\chi_\ep$ around $\ep = 0$ 
\begin{equation*}
\chi_\ep = \frac{\rho}{\rho - I} + \frac{(I+1)\rho}{2(\rho  - I)} \sqrt{\ep} + \OO(\ep).
\end{equation*} 
Note that $\alpha$ depends on $\ep$ through $\alpha_\ep  = \frac{1 - b}{b - e^{\sqrt{\ep}}}$. Via straightforward calculation, one has
\begin{equation*}
\frac{\alpha q^k \chi}{1 + \alpha q^{k} \chi} = \frac{\alphae e^{k\sqrt{\ep}} \chi_\ep}{1 + \alphae e^{k \sqrt{\ep}} \chi_\ep} = \frac{\rho}{I} + \frac{(I\rho - \rho^2)((2k+I+1)b + 1 -I - 2k)}{2(b-1) I^2} \sqrt{\ep} + \OO(\ep),
\end{equation*}
which implies 
\begin{equation*}
\sum_{k=0}^{J-1} \frac{\alpha q^k \chi}{1 + \alpha q^k \chi} = \frac{J\rho}{I} + \frac{J(I\rho - \rho^2)\big((I + J) b - (I+J-2)\big)}{2(b-1) I^2} \sqrt{\ep} + \OO(\ep).
\end{equation*}
Referring to the expression of $\mu$ in \eqref{eq:temp1}, one has the asymptotic expansion
\begin{align*}
\mu_\ep 
=\frac{J}{I} + \frac{J (I - 2\rho) (2 + (b-1)(I+J)) }{2(b -1) I^2} \sqrt{\ep} + \OO(\ep).
\end{align*}
Consequently, 
\begin{align*}
j_\ep(\rho) &= \ep^{-\frac{1}{2}}  \bigg(\sum_{k=0}^{J-1} \frac{\alpha q^k \chi}{1 + \alpha q^k \chi} - \rho \mu\bigg)
= 
\frac{\rho^2 J(b(I + J) - (I+J-2)}{2(b -1) I^2}  + \OO(\ep^{\frac{1}{2}}).
\end{align*}
We have
\begin{equation*}
\lim_{\ep \downarrow 0} -j_\ep'' (\rho) = \frac{J (b(I+J) - (I+J-2))}{(1-b) I^2},
\end{equation*}
which coincides with the value of $\alpha_2$.

\bibliographystyle{alpha}
\bibliography{higherspin}
\end{document}